\PassOptionsToPackage{pdfauthor={},pdftitle={},hypertexnames=false}{hyperref}
\documentclass[aos]{imsart}

\setpkgattr{journal}{name}{\empty}
\setpkgattr{author}{prefix}{}

\RequirePackage{amsthm,amsmath,amsfonts,amssymb, yhmath}
\RequirePackage[numbers,sort]{natbib}
\RequirePackage[colorlinks,citecolor=blue,urlcolor=blue]{hyperref}%
\RequirePackage{graphicx}%

\usepackage{natbib} 
\usepackage[usenames,dvipsnames]{color}
\definecolor{RefColor}{rgb}{0,0,.85}

\makeatletter
\def\acks{\gdef\@thefnmark{}\@footnotetext}
\makeatother

\usepackage{mathtools}
\usepackage[capitalize]{cleveref}

\usepackage{dsfont}
\usepackage{mathdots}
\usepackage{nccmath} 
\usepackage{scalerel}
\usepackage[hang,flushmargin,para]{footmisc}
\usepackage{caption}
\usepackage{titlesec}
\titleformat*{\section}{\large\bfseries}
\titleformat*{\subsection}{\bfseries}

\RequirePackage[OT1]{fontenc}

\usepackage[english]{babel}

\usepackage{enumitem}
\setlist[itemize]{leftmargin=1.5em}

\usepackage{booktabs}
\usepackage{tikz}
\usetikzlibrary{calc}
\usetikzlibrary{decorations.markings,decorations.pathreplacing}
\tikzstyle{mybraces}=[mirrorbrace/.style={
          decoration={brace, mirror},
          decorate},brace/.style={
          decoration={brace},
          decorate}]

\usepackage{amsmath,amssymb,amscd,amsfonts,amsthm,mathtools}
\usepackage{thmtools}

\usepackage{times}

\startlocaldefs

\DeclareMathOperator{\Tr}{Tr}

\renewcommand{\P}{\mathbb{P}}

\renewcommand{\c}[1]{\{{#1}\}}

\newcommand{\abs}[1]{\lvert #1 \rvert}

\newcommand{\norm}[1]{\lVert #1 \rVert}

\DeclareMathOperator{\diag}{diag}

\newcommand{\2} {\mspace{2 mu}}

\renewcommand{\P}{\mathbb{P}}

\newcommand{\Var}{\mathrm{Var}}
\newcommand{\Cov}{\mathrm{Cov}}

\newcommand{\eps}{\varepsilon}

\newcommand{\R}{\mathbb{R}}  %
\ifdefined\C\renewcommand{\C}{\mathbb{C}}\else\newcommand{\C}{\mathbb{C}}\fi %
\renewcommand{\Im}{\mathrm{Im}\,} %
\renewcommand{\Re}{\mathrm{Re}\,} %
\newcommand{\N}{\mathbb{N}}  %
\renewcommand{\P}{\mathbb{P}}  %
\newcommand{\Z}{\mathbb{Z}}  %

\newcommand{\df}{\mathfrak d}

\newcommand{\ff}{\mathfrak f}
\newcommand{\gf}{\mathfrak g}

\newcommand{\Bf}{\mathfrak B}
\newcommand{\Ff}{\mathfrak F}
\newcommand{\Mf}{\mathfrak M}
\newcommand{\Sf}{\mathfrak S}

\newcommand{\Hb}{\mathbb H}

\newcommand{\normtwo}[1]{\lVert #1 \rVert_{2}}

\newcommand{\normtwoinf}[1]{\lVert #1 \rVert_{2\to\infty}}
\newcommand{\norminf}[1]{\lVert #1 \rVert_{\infty}}

\allowdisplaybreaks

\usepackage{amsthm}
\usepackage{thmtools}

\theoremstyle{plain}
\declaretheoremstyle[postheadspace=.4em,headfont=\bfseries,bodyfont=\itshape,spaceabove=8pt,
spacebelow=10pt]{basic}
\theoremstyle{basic}
\declaretheorem[style=basic,name={Theorem}]{theorem}
\declaretheorem[style=basic,sibling=theorem,name={Lemma}]{lemma}

\declaretheorem[style=basic,sibling=theorem,name={Proposition}]{proposition}
\declaretheorem[style=basic,sibling=theorem,name={Corollary}]{corollary}
\theoremstyle{definition}
\newtheorem{definition}[theorem]{Definition}
\declaretheorem[style=definition,name={Remark}]{remark}
\declaretheorem[style=definition,name={Remark},numbered=no]{remark*}
\declaretheorem[style=definition,name={Example}]{example}
\declaretheorem[style=definition,name={Assumption}]{assumption}
\declaretheorem[style=definition,name={Model}]{model}

\crefname{assumption}{Assumption}{Assumptions}
\crefname{model}{Model}{Models}

\usepackage{mathtools}
\usepackage{tikz}
\usetikzlibrary{shapes.geometric}
\usetikzlibrary{arrows.meta}

\makeatletter 
\newcommand{\htarget}[1]{\Hy@raisedlink{\hypertarget{#1}{}}} 
\makeatother
\usepackage{amsmath,amssymb,amscd,amsfonts,amsthm,mathtools}

\usepackage{nccmath,tikz}
\usepackage{stmaryrd,scalerel} %

\usepackage{float}

\usetikzlibrary{calc}
\usetikzlibrary{decorations.markings,decorations.pathreplacing}
\usetikzlibrary{shapes.misc}

\usepackage{enumitem}
\newenvironment{proplist}{\begin{enumerate}[
    leftmargin=2.5em,
    labelwidth=0em,
    label=(\roman{enumi}),
    topsep=.1em,
    partopsep=0em,
    itemsep=0em
    ]}
{\end{enumerate}}

 \DeclareMathOperator*{\argmax}{argmax}

\def\P{{\mathbb P}}

\def\Var{\text{\rm Var}}
\def\Cov{\text{\rm Cov}}

\def\Tr{\text{\rm Tr}}

\def\bA{\boldsymbol{A}}

\def\bF{\boldsymbol{F}}
\def\bG{\boldsymbol{G}}
\def\bH{\boldsymbol{H}}
\def\bI{\boldsymbol{I}}
\def\bJ{\boldsymbol{J}}

\def\bS{\boldsymbol{S}}

\def\ba{\boldsymbol{a}}

\def\bq{\boldsymbol{q}}
\def\br{\boldsymbol{r}}

\def\bu{\boldsymbol{u}}
\def\bv{\boldsymbol{v}}

\def\bx{\boldsymbol{x}}
\def\by{\boldsymbol{y}}
\def\bz{\boldsymbol{z}}

\def\bzero{\boldsymbol{0}}
\def\bone{\boldsymbol{1}}

\def\fB{\mathfrak{B}}

\def\fF{\mathfrak{F}}
\def\fG{\mathfrak{G}}

\def\fM{\mathfrak{M}}

\def\fS{\mathfrak{S}}

\def\fU{\mathfrak{U}}

\def\C{\mathbb{C}}
\def\D{\mathbb{D}}

\def\bbH{\mathbb{H}}

\def\N{\mathbb{N}}

\def\P{{\mathbb P}}

\def\R{\mathbb{R}}

\def\Z{\mathbb{Z}}

\def\c{\mathbb{c}}

\def\r{\mathbb{r}}

\def\cA{\mathcal{A}}

\def\cE{\mathcal{E}}
\def\cF{\mathcal{F}}
\def\cG{\mathcal{G}}
\def\cH{\mathcal{H}}

\def\cN{\mathcal{N}}

\def\cP{\mathcal{P}}

\def\cR{\mathcal{R}}

\def\cV{\mathcal{V}}

\DeclareMathOperator{\msum}{\medmath\sum}

\DeclareMathOperator{\mint}{\medmath\int}
\newcommand{\argdot}{{\,\vcenter{\hbox{\tiny$\bullet$}}\,}}%
\newcommand{\tagaligneq}{\refstepcounter{equation}\tag{\theequation}}

\newcommand{\msup}{\sup\nolimits}
\newcommand{\minf}{\inf\nolimits}

\newcommand{\ind}{\mathbb{I}}

\newcommand{\mean}{\mathbb{E}}
\usepackage[dvipsnames]{xcolor}

\def\diag{\text{\rm diag}}
\def\Tr{\text{\rm Tr}}

\usepackage{makecell}

\expandafter\def\expandafter\normalsize\expandafter{%
    \normalsize%
    \setlength\abovedisplayskip{4pt}%
    \setlength\belowdisplayskip{4pt}%
    \setlength\abovedisplayshortskip{-8pt}%
    \setlength\belowdisplayshortskip{2pt}%
}

\usepackage{CJKutf8}

\endlocaldefs

\makeatletter
\newcommand{\righttagsonce}{%
  \tagsleft@false
  \let\veqno\@@eqno
}
\makeatother

\begin{document}

    \begin{frontmatter}
    \title{
     The zero pattern of a design matrix drives multiple descent in over-parameterized regression
    }

        \author{\fnms{Kevin Han}~\snm{Huang}$^{*,\dagger}$}
        \author{\fnms{Haoyu}~\snm{Ye}$^{*,\ddagger}$}
        \author{\fnms{Somak}~\snm{Laha}$^{\ddagger}$}
        \author{\fnms{Morgane}~\snm{Austern}$^{\ddagger}$}
        \\[1em]
        \normalfont\small $^\dagger$University of Warwick \qquad  $^\ddagger$Harvard University

        \footnotetext[1]{\hspace{-.3em}Equal contributions.}
    
    \begin{abstract}
        Over-parameterized linear regression has been widely studied over the last decade. However, most existing works assume that the covariates are independent and that their covariance matrices are non-degenerate. In this paper, we relax both assumptions and derive deterministic equivalents for the prediction risk in a vanishing-ridge regime. We show that degeneracy of the covariance matrices and dependence can lead to multiple descent, and characterize where the corresponding peaks can occur. Our proofs use a novel graph representation of the variance profile. We show that maximum matchings and the Dulmage--Mendelsohn decomposition of the associated bipartite graph identify the configurations at which the variance becomes singular.
    \end{abstract}
                
    \end{frontmatter}

\section{Introduction}

Over the last decade, our understanding of over-parameterized models has been reshaped by a
simple observation: they often predict well even when they interpolate the training data.
Plotted against model size, the test error need not follow the classical U-shape: it can
fall, rise to a peak at the interpolation threshold, and fall again. This \emph{double
descent} is well-understood for linear and kernel models. In the isotropic case it comes from
a single spectral event \citep{belkin2019reconciling, hastie2022surprises}: at the threshold the design {matrix} loses rank, the minimum-norm interpolator amplifies its smallest singular values,
and the risk spikes. In these analyses the {data} covariance is non-degenerate, either isotropic or
positive definite, so the spectrum reaches the hard edge at a single { dimension-to-sample-size} ratio and the
curve has one peak. More recent works { observe} \emph{several} peaks, but usually place them by
design through the architecture, separate feature blocks or distinct { scaling regimes}
\citep{d2020triple, adlam2020neural, meng2024multiple}. We ask when, why, and \emph{where}
multiple descent can arise from data structure alone {in the proportional scaling regime}.

This paper gives such a mechanism, together with an exact rule for its peaks. We study ridge
least squares with a vanishing penalty, which approximates the minimum $\ell_2$-norm
(ridgeless) estimator. The { covariates} are Gaussian and
either \emph{heterogeneous}, with observations of differing covariances, or \emph{dependent},
with correlation across observations.  A single design and estimator
can produce several risk peaks, and their number and locations are fixed by a combinatorial invariant of the design, {which} reads the zero pattern in its
covariance structure.

\paragraph*{Setup.}
We observe $(X_i, y_i)_{i\le n}$ generated by $y_i = X_i^\top\beta + \epsilon_i$. Let $X \in \mathbb{R}^{n\times p}$ be the matrix of covariates whose $i$-th row is $X_i^\top$; $X$ is known as the design matrix. 
We estimate the signal $\beta$ by \emph{ridge} least squares with a vanishing
penalty $\lambda=\lambda_n\downarrow0$,
\[
    \hat\beta_\lambda \;=\; \Big(\mfrac{1}{n}\msum_{i\le n} X_iX_i^\top + \lambda I_p\Big)^{\!-1}\,
    \mfrac{1}{n}\msum_{i\le n} X_i y_i \;,
\]
a computable proxy for the minimum-$l_2$-norm (ridgeless) interpolator. As
$\lambda\downarrow0$ the ridge estimator converges to that interpolator, and a small positive
penalty is what numerical implementations use. We work in the proportional regime $p/n\to\gamma\in(0,\infty)$
and measure performance by the out-of-sample prediction risk $\cR^\lambda(X)$. Let
$W_n = \frac1n X^\top X$ be the sample covariance. Since the observations are heterogeneous,
there is no single population covariance to compare against; we therefore use the
\emph{averaged covariance}
\[
    \bar\Sigma \;\coloneqq\; \mfrac1n\msum_{i\le n}\Var[X_i]\;,
\]
the mean of the $n$ per-observation covariances. The risk splits into bias and variance
(\Cref{lem:bias-var-decomp}): $\cR^\lambda(X) = \cR_B^\lambda(X) + \cR_V^\lambda(X)$ with
\[
    \cR_B^\lambda(X) = \big\|\bar\Sigma^{1/2}\big((W_n+\lambda I_p)^{-1}W_n - I_p\big)\beta\big\|^2\;,
    \qquad
    \cR_V^\lambda(X) = \mfrac{\sigma_\epsilon^2}{n}\Tr\big(\bar\Sigma\,W_n(W_n+\lambda I_p)^{-2}\big)\;.
\] 
Both are governed by the smallest eigenvalues of $W_n$, near zero --- \emph{the hard edge of the sample-covariance spectrum} in the terminology of random-matrix theory; see, e.g. \citep{HachemHardyNajim2016HardEdge}, and the classical Bessel-kernel analysis of \citep{TracyWidom1994Bessel}. As the penalty vanishes, {
the variance diverges when the smallest nonzero eigenvalues of $W_n$ approach this edge, and the bias concentrates on the directions outside the row space of $X$, with a zero bias if $W_n$ is asymptotically full-rank.
}
In the isotropic model, {the variance diverges at the single threshold $\gamma =1$, where the bias also changes from zero to nonzero.}

\paragraph*{Gaussian designs and variance profile.}
We consider two Gaussian designs. In the \emph{heterogeneous} design (\Cref{mod:het}) the
observations are independent and $X_i = \Sigma_i^{1/2} Z_i$ with $Z_i\sim\cN(0,I_p)$, so each
$X_i$ has its own covariance $\Sigma_i$. In the \emph{dependent} design (\Cref{mod:dep}) the
observations are correlated through a finite sum of separable parts,
\[
    X \;=\; \msum_{m=1}^M \big(\tilde\Sigma^{(m)}\big)^{1/2} Z^{(m)} \big(\Sigma^{(m)}\big)^{1/2}\;;
\]
we call this \emph{finite-rank dependence}, and data augmentation is a natural instance. In both
designs{, we assume that} the covariances share an eigenbasis but may be rank-deficient, so some observations
carry no variance along some directions. (Positive-definite covariances without a shared
eigenbasis are the subject of \Cref{sec:non:sim:diag:but:pos}.) This is the regime of
interacting degeneracies: the design can lose rank in several places, so a single design with a
single estimator can show several risk peaks, contributing to triple descent or more.

{ Both} the heterogeneous and the dependent designs can be \emph{completely
characterized} through a single, simpler object. After rotating
to the shared eigenbasis, each becomes a Gaussian matrix with independent entries and a fixed
variance matrix $S\in\R^{n\times p}$ --- what we call the \emph{variance profile} --- with
$\Var[X_{ij}]=nS_{ij}$ (\Cref{lem:model:equivalence,lem:reduction:var:profile} in
\Cref{appendix:setup}). A zero entry $S_{ij}=0$ marks a direction that observation $i$ cannot see, so $S$ records the degeneracies as a pattern of zeros. The profile is explicit in every case. Write
$\Sigma_i=U^\top D^{(i)}U$ for the shared eigendecomposition of the heterogeneous covariances,
and $\tilde\Sigma^{(m)}=\tilde U^\top \tilde D^{(m)}\tilde U$, $\Sigma^{(m)}=U^\top D^{(m)}U$
for the eigendecompositions in \Cref{mod:dep}, with $\tilde D^{(m)}\in\R^{n\times n}$ and
$D^{(m)}\in\R^{p\times p}$ {as the diagonal matrix of eigenvalues}. 
The
reductions of \Cref{lem:reduction:var:profile,lem:model:equivalence} in \Cref{appendix:setup} give:

\begin{center}
\begin{tabular}{@{}ll@{}}
\toprule
Design & Induced variance profile $S$ \\
\midrule
\Cref{mod:het} (heterogeneous) & $S_{il}=\tfrac1n\,D^{(i)}_{ll}$ 
\\[2mm]
\Cref{mod:dep} (dependent)     & $S_{il}=\tfrac1n\sum_{m=1}^M \tilde D^{(m)}_{ii}\,D^{(m)}_{ll}$ \\
\bottomrule
\end{tabular}
\end{center}

\noindent The zero pattern of $S$, and { how it characterizes} the  multiple-descent phenomenon, is
detailed in \Cref{subsec:vp-model,sec:vp-risk-main}.

\paragraph*{The limiting risk.}
Fix the variance profile $S$. The risk then has one description, through a vector
$\br(\lambda)=(r_l(\lambda))_{l\le p}$ with one entry per coordinate, defined for $\lambda>0$ by
the fixed-point system
\[
    1 \;=\; r_l(\lambda) + \msum_{i=1}^n \mfrac{S_{il}\, r_l(\lambda)}{\lambda + \sum_{j=1}^p S_{ij}\, r_j(\lambda)}\;,
    \qquad 1\le l\le p\;.
\]
Write $r_l(0)$ and $\partial r_l(0)$ for its value and $\lambda$-derivative as $\lambda {= \lambda_n} \downarrow0$.
Under a mild flatness and connectivity condition on $S$ (\Cref{asst:flat}), the bias and variance
match explicit deterministic equivalents (\Cref{thm:master}). The variance equivalent is a
weighted sum of the derivatives,
\[
    \cR_V^{
         \lambda
    }(X) \;-\; \mfrac{\sigma_\epsilon^2}{n}\msum_{l=1}^p \bar\Sigma_{ll}\,\partial r_l({\lambda}) \;\xrightarrow{\rm a.s.}\; 0\;.
\]
{ The} bias equivalent stays bounded and is an explicit function of $\br( \lambda)$ alone;
in the vanishing-ridge limit, its mass can only sit on the coordinates with $r_l(0)>0$
(\Cref{thm:bias-support}). 
Two numbers per coordinate { enter the risk}: {t}he mass
$r_l(0)$ that coordinate $l$ places at zero marks where bias can persist, and the rate
$\partial r_l(0) = \lim_{\lambda\downarrow0}\partial r_l(\lambda)$, which measures how fast small
nonzero eigenvalues pile up at zero, decides whether the variance stays bounded. 
{
    Notably at different coordinates of $\br$, bias transition and variance divergence can happen at \emph{different} values of $\gamma$. This generalizes the bias--variance split of the isotropic theory \citep{hastie2022surprises}, with all coordinates equivalent and risk peaking at a single ratio $\gamma$, to a multiple-descent theory described by $S$.
}

\vspace{.5em}

\paragraph*{Main findings.}
{ (i) A} single design and a single estimator can already produce several
risk peaks. {
This can arise from data heterogeneity, dependence, or both.}
\\[.2em]
{ (ii) Multiple descents are driven by} spectral
degeneracy, which heterogeneity and dependence can each create. {By contrast, p}ositive-definite covariances that share an eigenbasis are non-degenerate, { creating} only the classical peak at $\gamma=1$  no matter how heterogeneous they are; we conjecture that the same holds when they do not commute (\Cref{sec:non:sim:diag:but:pos}).
\\[.2em]
{ (iii) The peak locations are determined by the zero pattern of the variance profile $S$, through the maximum matchings of the bipartite graph of $S$. In particular,  non-zero biases come from the
coordinates that some maximum matching leaves unmatched, and a variance peak occurs when the
residual matching becomes rigid (see \Cref{sec:vp-risk-main} for a precise definition).}

{\paragraph*{Proof techniques.} The key technical ingredient is a novel combination of techniques from random matrix theory and graph theory. First, we} extend the local law {of Alt, Erdős and Krüger} \citep{alt2017local} {for a sample covariance matrix under the variance profile model} by explicitly tracking the spectral parameter {of the matrix resolvent}. {This enables ridge least squares analysis} at a vanishing ridge level
$\lambda=\lambda_n\downarrow0$ (\Cref{appendix:local:law}). {Second, we characterize the bias and variance terms of the risk through the Dulmage--Mendelsohn decomposition of the bipartite graph of the variance profile $S$, yielding a precise graph-theoretic description.}

\paragraph*{Motivating examples.}
Data augmentation, which enlarges a training set with random transformations of each sample,
is one of the most common heuristics in machine learning and a natural source of dependence.
The augmented copies of one sample are strongly dependent, while copies of different samples
are independent, so augmentation is a natural test of our mechanism. As we show
(\Cref{ex:augmentation}), it is an instance of the dependent design (\Cref{mod:dep}) with $M=2$
separable components; by our reduction, it carries the limiting risk of a heterogeneous
variance profile.  
\Cref{fig:intro-aug} previews the consequence. The unaugmented design has
the single classical interpolation peak. A single augmented design already shows
\emph{multiple} risk peaks: multiple descent produced by dependence alone, with the peaks
where the {zero pattern} of the induced profile predicts (\Cref{sec:vp-risk-main}). We return to this
experiment in \Cref{subsec:sim-aug}.

{A second example uses real data and synthetic observations from a 
generative model \citep{lacan2023gan}. It has been observed that synthetic data can be significantly less diverse than real data, and in certain settings have a much lower-rank covariance than the real data \citep{lacan2023gan}. Hence, the combined data is heterogeneous and can have different degeneracy patterns. We observe exactly this in
\Cref{subsec:sim-gan},  where we consider gene expression data pooled with samples from a generative adversarial
network (GAN), and observe triple descent with two peaks.
}

\begin{figure}[t]
    \centering
    \includegraphics[width=0.72\linewidth]{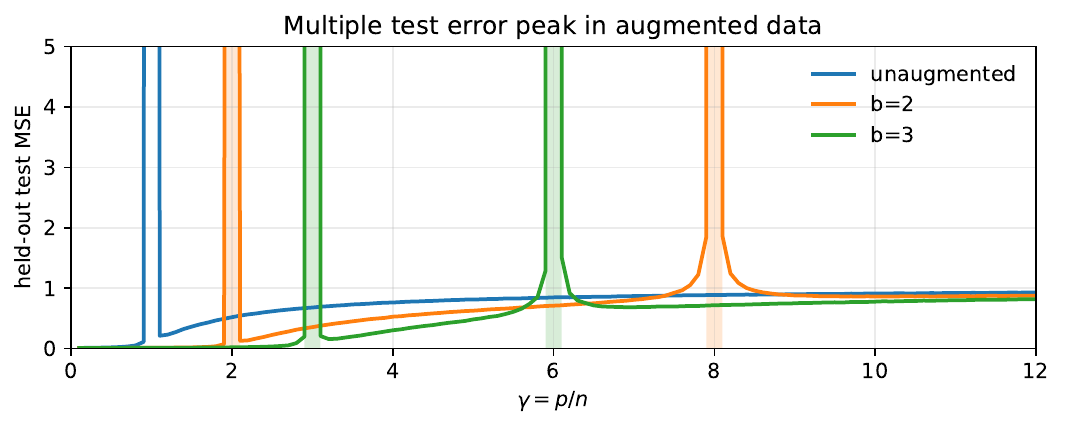}
    \caption{   
    Multiple descent from data augmentation. Held-out test error
    against $\gamma=p/n$ for an unaugmented design and for random within-block permutations of
    block size $b=2,3$ ($n=100$, $k=5$ copies). The unaugmented curve has a single
    interpolation peak; the augmented designs show additional peaks. Full setup in
    \Cref{subsec:sim-aug}.}
    \label{fig:intro-aug}
\end{figure}

\paragraph*{Organization.} \Cref{sec:setup-main} introduces the two designs {with heterogeneity and dependence}, presents the variance profile as their
common reduction, states the regularity assumption, and gives the limiting risk.
\Cref{sec:vp-risk-main} introduces the bipartite graph induced by the variance profile, and develops the graph-theoretic tools that locates the bias and the variance peaks.
\Cref{sec:non:sim:diag:but:pos} considers the positive-definite case with no shared eigenbases, where a single threshold
persists. \Cref{sec:simulations} verifies our messages on synthetic experiments.
\section{Related literature}

\paragraph*{Interpolation, benign overfitting, and double descent.}
Double descent was first identified by \citep{belkin2019reconciling}: the test risk can fall
again in the over-parameterized regime, past the interpolation peak. It was observed in
deep models by \citep{nakkiran2021deep} and made precise for ridgeless regression by
\citep{hastie2022surprises}. The over-parameterized regime has since been
studied for ridge regression \citep{dobriban2018high, richards2021asymptotics, tsigler2023benign},
for the minimum-norm interpolator \citep{bartlett2020benign, muthukumar2020harmless}, and for
random-features regression, a solvable model of a two-layer network, where
\citep{mei2022generalization} derived the exact double-descent curve;
see \citep{bartlett2021deep} for a survey. {In the isotropic case, a single spectral degeneracy gives one peak:} at the
interpolation threshold {$\gamma=1$,} the smallest singular values of the design approach zero and the
pseudo-inverse amplifies them, so the variance diverges. The bias comes from the signal
directions the data never sees{, and becomes non-zero when $\gamma > 1$}.

{Our theory builds on ridge and ridgeless regression}, which we now recall.
\citep{dobriban2018high} gave the
exact limiting prediction risk of ridge regression and classification for a general
population covariance, and \citep{hastie2022surprises} gave exact bias and variance formulas
for the ridgeless interpolator; both work in the proportional regime $p/n\to\gamma$.
\citep{cheng2024dimension} instead gave dimension-free deterministic equivalents that bound
the risk through the covariance spectrum rather than the ratio $\gamma$, from
finite-dimensional up to trace-class features; their bounds are sharp while the smallest
nonzero sample eigenvalue stays bounded away from zero, that is, away from the hard edge we
study. 
{Optimal regularization has also been studied here}: \citep{wu2020optimal} find the optimal
anisotropic ridge penalty, \citep{kobak2020optimal} show it can be zero or negative, and
\citep{nakkiran2021optimal} show that optimal regularization removes the peak; see also
\citep{bibas2021distribution}. 
These results share the same setting: the observations are identically distributed under a
single positive-definite, well-conditioned population covariance
(as \citep{hastie2022surprises} requires). This homogeneity and non-degeneracy force a
single interpolation threshold. We drop both. Keeping the bias-variance framework of
\citep{hastie2022surprises} and its anisotropic local law
\citep{knowles2017anisotropic}, we replace the isotropic input with
variance-profile tools and let the ridge parameter vanish, so that heterogeneous or dependent
observations with rank-deficient covariances produce the several thresholds the
single-threshold theory does not reach.

\paragraph*{From double to multiple descent.}
{Another} line of work shows that the risk curve can carry more than one peak. \citep{d2020triple} analysed a
random-features regression model, and confirmed the effect in neural networks, separating a
peak tied to the input dimension from one tied to the number of features.
\citep{adlam2020neural} gave exact asymptotics for kernel regression with the neural tangent
kernel and traced a further peak to a second, quadratic overparameterization scale.
\citep{liang2020multiple} proved multiple-descent risk bounds for kernel interpolation across
the scalings $d = n^\alpha$. \citep{chen2021multiple} showed, for linear regression, that
appending coordinates drawn from suitably chosen distributions can raise or lower the risk at
each step, so the curve can be made to have any number of peaks at prescribed locations.
\citep{pmlr-v139-mel21a} derived multiple descent in ridge regression when the input covariance
has several well-separated eigenvalue scales aligned with the target, one peak appearing as
each scale becomes resolved. Closest to us, \citep{meng2024multiple} gave a theory for
multi-component random-feature models, where each component contributes an
interpolation peak. Almost all these works introduce extra peaks by design,
through the architecture, extra feature blocks, several kernel scales, or a constructed
covariance. Our peaks instead come from {the support, {or equivalently} the zero pattern,} of a single
heterogeneous or dependent design. Our contribution is not that several peaks can
occur, but a rule for their number and locations.

\paragraph*{Heterogeneity and dependence.}
The closest comparators study ridge regression with a variance profile. \citep{bigot2024high}
treat independent but non-identically distributed covariates given by a variance profile, and
\citep{dabo2025high} extend this to random features.
Both rely on a global deterministic equivalent driven by the bulk of the limiting spectrum; \citep{bigot2024high} reach the ridgeless limit only under a non-degeneracy condition that keeps the variance profile entries away from zero. We instead let the ridge vanish, so that the bias and variance singularities are governed by the {zero pattern} of the profile, rather than by the bulk of the limiting spectrum. This regime is harder to analyze: the constants in a bulk deterministic equivalent blow up as the penalty vanishes and the spectrum reaches the hard edge, so the resolvent must be controlled through a local law with explicit spectral-parameter dependence down to zero.
{Other works consider heterogeneity and dependence in more specific setups:} \citep{song2024generalization} studies pooled minimum-norm interpolation
under covariate shift, \citep{moniri2025asymptotics} treats linearly dependent covariates, and
\citep{huang2022data} study the dependence induced by data augmentation. A parallel route uses
Gaussian universality and the convex Gaussian min-max theorem rather than the resolvent:
universality for empirical risk minimization \citep{montanari2022universality, han2023universality}
and random features \citep{hu2022universality}, the Gaussian comparison inequalities of
\citep{Gordon1985SomeIF} in the modern form of \citep{thrampoulidis2018precise}, and extensions
to dependent data \citep{lahiry2023universality, mallory2025universality}.

\paragraph*{{Random matrix analysis under the variance profile model}.} Our random-matrix {tool comes from} the theory of Gram matrices with a variance profile; see
\citep{couillet2022random} for random-matrix methods in machine learning.
\citep{hachem2007deterministic} gave deterministic equivalents for such matrices, and
\citep{alt2017local} proved entrywise and averaged local laws for an arbitrary variance
profile, covering both the rectangular soft edge and the hard edge at zero. The supporting
quadratic-vector-equation machinery, including the stability and singularity analysis of
the matrix Dyson equation, is due to \citep{Ajanki_2019, ajanki2017universalitygeneralwignertypematrices};
see also \citep{yin2020singularvaluedistributionlarge}. Near interpolation the risk is
controlled not by the bulk spectrum but by the resolvent at a small spectral parameter. We
re-derive the local law with explicit spectral-parameter dependence, trading
the sharp uniform bounds of \citep{alt2017local} for bounds valid at a vanishing
ridge level. A small penalty still exposes the degeneracies that a fixed ridge
would smooth out, and our analysis departs from the anisotropic local law of
\citep{knowles2017anisotropic} used in the single-threshold theory. \citep{kruger2025singularity} considers a sequential limit with $n\rightarrow\infty$ first followed by the spectral parameter tending to $0$, whereas our setting requires a joint limit. Moreover, \citep{kruger2025singularity} constructs a directed graph to characterize the blowup of spectral density at zero, which is insufficient for characterizing the hard-edge quantities in our regression risk.

\paragraph*{Structural rank, matchings, and the Dulmage--Mendelsohn decomposition.}
The combinatorial half of our analysis links this random-matrix theory to classical
structural-rank theory, through the bipartite graph of {the variance profile}.
Hall's theorem~\citep{hall1935representatives} says when a bipartite graph
has a matching that saturates one side, and Berge's theorem~\citep{berge1957two}
describes maximum matchings through alternating paths. The Dulmage--Mendelsohn
decomposition~\citep{Dulmage_Mendelsohn_1958} organizes all maximum matchings into a
canonical block form, long used in sparse numerical linear algebra to read off structural
rank and the structure of pseudo-inverses~\citep{pothen1990computing, lovasz2009matching}.
The strong Hall property~\citep{brualdi1994strong} is a form of non-singularity in these settings. Here these tools carry a statistical meaning: a
maximum matching identifies which coordinates the design can miss, and the strong-Hall
condition separates the nonsingular regimes from the boundaries where the risk diverges. 
{We use them to} locate the bias support and the variance peaks.

\section{Main results} \label{sec:setup-main}

\subsection{The ridgeless prediction problem}

We work in the
 high-dimensional linear model. Given $n$ mean-zero
observations $X_1, \ldots, X_n \in \R^p$, 
possibly dependent across
observations and heterogeneously distributed, we observe responses generated by
\begin{align*}
    y_i \;\coloneqq\; X_i^\top \beta + \epsilon_i\;,
    \qquad
    \epsilon_i \ \text{ i.i.d.~with } \ \mean[\epsilon_i] = 0 \ \text{ and } \ \Var[\epsilon_i] = \sigma^2_\epsilon\;,
\end{align*}
where $\beta \in \R^p$ is an unknown signal and the noise $(\epsilon_i)_{i \le n}$ is
independent of the design. We estimate $\beta$ by \emph{ridge} least squares with a
deterministic penalty $\lambda = \lambda_n \downarrow 0$,
\begin{align*}
    \hat \beta_\lambda
    \;\coloneqq\;
    \Big( \mfrac{1}{n} \msum_{i \le n} X_i X_i^\top + \lambda I_p \Big)^{\!-1}
    \mfrac{1}{n} \msum_{i \le n} X_i y_i\;.
\end{align*}
As $\lambda\downarrow 0$, $\hat\beta_\lambda$ converges to the minimum {$\ell_2$}-norm
(ridgeless) {estimator}.
Let
$(X_{\rm new}, Y_{\rm new})$ be an independent test pair with
$Y_{\rm new} = X_{\rm new}^\top\beta + \epsilon_{\rm new}$, where $\epsilon_{\rm new}$
is an independent copy of the noise and $X_{\rm new}$ is drawn as
in~\eqref{eq:test-gen-main} below. Its test error splits into an irreducible noise
{variance} and a reducible {component},
\begin{align*}
    \mean\!\big[\,(Y_{\rm new} - X_{\rm new}^\top\hat\beta_\lambda)^2 \;\big|\; (X_i)_{i \le n}\big]
    \;=\;
    \sigma^2_\epsilon \;+\; \cR^\lambda(X)\;,
\end{align*}
and we study the reducible \emph{prediction risk}, the random scalar
\begin{align*}
    \cR^\lambda(X)
    \;\coloneqq\;
    \mean\!\Big[\, \big( X_{\rm new}^\top (\beta - \hat\beta_\lambda) \big)^2
    \;\Big|\; (X_i)_{i \le n} \Big]\;,
    \quad \text{ in the proportional regime }
    \mfrac{p}{n} \to \gamma \in (0,\infty)\;.
\end{align*}
Because the observations are heterogeneous, there is no single population
covariance against which to measure prediction error; instead, we generate the
test point $X_{\rm new}$ from the {uniform} mixture of the training marginals. Writing $\mu_i$
for the law of $X_i$, we draw
\begin{align*}
    U \;\sim\; \mathrm{Uniform}\{1, \ldots, n\}
    \qquad\text{ and }\qquad
    X_{\rm new}\mid U \;\sim\; \mu_U\;,
    \tagaligneq \label{eq:test-gen-main}
\end{align*}
independently of the training data $(X_i, y_i)_{i\le n}$. This naturally generalizes the test mechanism used in the homogeneous theory of ridgeless
regression~\citep{hastie2022surprises}: when the $X_i$ are identically distributed,
it reduces to drawing $X_{\rm new}$ from the common law.

The analysis centers on the $\R^{n\times p}$ data matrix and the
$\R^{p\times p}$ sample covariance matrix
\begin{align*}
    X \;\coloneqq\;
    \begin{psmallmatrix}
        \leftarrow X_1^\top \rightarrow \\[-.1em]
        \vdots \\[-.1em]
        \leftarrow X_n^\top \rightarrow
    \end{psmallmatrix}\;,
    \qquad\qquad
    W_n \;\coloneqq\; \mfrac{1}{n}\msum_{i\le n} X_i X_i^\top \;=\; \mfrac{1}{n} X^\top X\;.
\end{align*}
The vanishing-ridge risk is controlled by the hard edge of \(W_n\). Rank deficiency creates bias, through directions outside the row space of \(X\), while small nonzero singular values make the variance blow up as the penalty vanishes. In the homogeneous, non-degenerate model this mechanism has a single threshold at \(\gamma=1\). In heterogeneous or dependent designs, different parts of the design may become singular at different effective thresholds. The following bias--variance decomposition isolates the effects through \(W_n\) and the averaged covariance.

\begin{proposition}[Bias--variance decomposition] \label{lem:bias-var-decomp}
Under the test mechanism~\eqref{eq:test-gen-main}, 
with
\begin{align*}
    \bar\Sigma \;\coloneqq\; \mfrac{1}{n}\msum_{i\le n}\Var[X_i]\;,
\end{align*}
the prediction risk decomposes as
\(\cR^\lambda(X) = \cR_B^\lambda(X) + \cR_V^\lambda(X)\), where
\begin{align*}
    \cR_B^\lambda(X)
    &\;\coloneqq\;
    \big\| \bar\Sigma^{1/2}\big( (W_n+\lambda I_p)^{-1} W_n - I_p \big)\beta \big\|^2,\\
    \cR_V^\lambda(X)
    &\;\coloneqq\;
    \mfrac{\sigma^2_\epsilon}{n}\Tr\big( \bar\Sigma\, W_n (W_n+\lambda I_p)^{-2} \big)\;.
\end{align*}
\end{proposition}

The bias $\cR_B^\lambda(X)$ collects the signal in the directions where $W_n$ is small. Since
$(W_n+\lambda I_p)^{-1}W_n - I_p = -\lambda(W_n+\lambda I_p)^{-1}$, the bias shrinks the signal
component along each eigenvector of $W_n$ with eigenvalue $s$ by the factor $\lambda/(s+\lambda)$:
components on eigenvalues well above $\lambda$ are recovered, {while} components on eigenvalues of order
$\lambda$ or below are not. {As $\lambda$ gets small}, $(W_n+\lambda I_p)^{-1}W_n$ approaches the projection $W_n^\dagger W_n$ onto
the row space of $X$, so in the limit the bias is carried by the directions outside the row
space; we do not let $\lambda=\lambda_n$ tend to zero faster than {what}
\Cref{thm:master} permits. The
variance $\cR_V^\lambda(X)$ is a $\bar\Sigma$-weighted trace of $W_n(W_n+\lambda I_p)^{-2}$,
the term through which the smallest {nonzero} eigenvalues of $W_n$, of order $\lambda_n$ or below, drive
the risk to diverge. The decomposition is standard; a derivation from the test
mechanism~\eqref{eq:test-gen-main} is given in \Cref{appendix:setup}.

\subsection{Heterogeneous and dependent designs}

\label{subsec:practical-models}

We study two Gaussian designs. The first models heterogeneity:
observations are independent but {may have different} covariances.

\begin{model}[Heterogeneous data] \label{mod:het}
$X_1, \ldots, X_n$ are independent $\R^p$ Gaussian vectors, each generated from a symmetric
positive semi-definite {covariance} matrix $\Sigma_i \in \R^{p\times p}$ as
\begin{align*}
    X_i \;=\; \Sigma_i^{1/2} Z_i\;,
    \qquad Z_i \overset{\rm i.i.d.}{\sim} \cN(0, I_p)\;.
\end{align*}
We assume that the matrices $(\Sigma_i)_{i\le n}$ are simultaneously diagonalizable.
\end{model}

Two features of \Cref{mod:het} {are important for our analysis}. First, it permits rank deficiency, at
the cost of requiring a common eigenbasis: some coordinates may carry zero variance in some
observations. 
 This is the regime of interacting
degeneracies, and hence of multiple descent. {A natural example is} highly collinear data, where
a few common factors carry almost all the variance in a shared basis and leave many
coordinate directions degenerate. Second, the emphasis on positive
\emph{semi}-definiteness, rather than the positive definiteness
assumed throughout much of the literature, is what lets the \emph{{zero} pattern} of the
design govern the risk. A zero eigenvalue of some $\Sigma_i$ means observation $i$ does not
see a coordinate direction, and which observations see which directions {is what sets} where
the design loses rank and where the peaks {occur}. Covariances that are positive definite but
share no eigenbasis fall outside this model; we {consider them} in \Cref{sec:non:sim:diag:but:pos},
where we conjecture that they produce only the familiar single-threshold behavior.

The second model keeps the coordinate structure but introduces dependence \emph{across}
observations, through finitely many separable components.

\begin{model}[Dependent data: finite-rank dependence] \label{mod:dep}
Fix $M\in\N$, a sequence of positive semi-definite $\R^{p \times p}$ matrices $(\Sigma^{(m)})_{m \in \N}$ and another sequence of positive semi-definite simultaneously diagonalizable $\R^{n \times n}$ matrices $(\tilde \Sigma^{(m)})_{m \in \N}$. The data matrix is
\begin{align*}
    X \;=\; \msum_{m=1}^M \big(\tilde\Sigma^{(m)}\big)^{1/2} Z^{(m)} \big(\Sigma^{(m)}\big)^{1/2}\;,
\end{align*}
where $(Z^{(m)})_{m\le M}$ are i.i.d.~$\R^{n\times p}$ matrices with i.i.d.~$\cN(0,1)$
entries. We assume that the matrices $(\Sigma^{(m)})_{m\le M}$ are simultaneously
diagonalizable.
\end{model}

We call \Cref{mod:dep} \emph{finite-rank dependence} because the full $np\times np$ covariance of
$\operatorname{vec}(X)$ is a sum of only $M$ separable terms:
\begin{align*}
    \Cov[X_{ij}, X_{i'j'}]
    \;=\;
    \msum_{m=1}^M \tilde\Sigma^{(m)}_{ii'}\,\Sigma^{(m)}_{jj'}\;,
    \tagaligneq \label{eq:cov-dep-main}
\end{align*}
so the $\tilde\Sigma^{(m)}$ carry sample-to-sample dependence and the $\Sigma^{(m)}$ carry
coordinate-to-coordinate dependence. 
Taking each $\tilde\Sigma^{(m)}$ block-diagonal in a common partition of the observations
gives a separable subclass of \emph{block dependence}: observations in the same block may be
correlated, those in different blocks are independent. We keep this stronger finite-rank,
separable structure, with $M$ fixed independent of $n$ and $p$, because it keeps the
model tractable{; as we show next, it is} reducible to a heterogeneous variance profile.

\begin{remark}
The restriction that $M$ is fixed is not essential. Our results also
allow $M=M_n$ to vary with $n$, provided the induced variance profile satisfies
\Cref{asst:flat}.
\end{remark}

\begin{example}[{Gaussian affine} data augmentation and finite-rank dependence] \label{ex:augmentation}
Data augmentation is a natural source of \Cref{mod:dep}. Let $V_1,\dots,V_n$ be i.i.d.\
mean-zero base samples in $\R^p$ {with $V_i\sim\cN(0,\Sigma_V)$}, and {
let $\{E_{ic}\}_{i\le n,\,c\le k}$ be i.i.d.\
$\cN(0,I_p)$ vectors, independent of the base samples. Let $A,B\in\R^{p\times p}$ be deterministic matrices such that
$A\Sigma_VA^\top$ and $BB^\top$ are simultaneously diagonalizable.
Define the i.i.d.\ random affine transformations
\begin{align*}
    T_{ic}(v)
    \;\coloneqq\;
    A v + B E_{ic}\;.
\end{align*}
} Augmentation forms the $nk$ rows
\begin{align*}
    X_{(i,c)} \;=\; T_{ic}\,(V_i)\;,\qquad 1\le i\le n,\ 1\le c\le k\;,
\end{align*}
which play the role of the sample in \Cref{mod:dep}. Rows from different base samples are
independent, while the $k$ rows sharing a base sample $V_i$ are dependent. {Moreover, the augmented rows are jointly Gaussian.}
Writing
\begin{align*}
    \Sigma_{11}
    \;\coloneqq\;
    \Var[T_{ic}(V_i)]
    {=A\Sigma_VA^\top+BB^\top},
    \quad
    \Sigma_{12}
    \;\coloneqq\;
    \Cov[T_{ic}(V_i),\,T_{ic'}(V_i)]
    {=A\Sigma_VA^\top}
    (c\ne c'),
\end{align*}
the row covariance is, for $a=(i,c)$ and $b=(i',c')$,
\begin{align*}
    \Cov[X_a, X_b]
    \;=\;
    \ind\{i=i'\}\,\Sigma_{12}
    \;+\;
    \ind\{a=b\}\,(\Sigma_{11}-\Sigma_{12})\;,
    \quad\text{with}\quad
    \Sigma_{11}-\Sigma_{12}{\;=\;BB^\top}\;\succeq\;0\;.
\end{align*}
This is the finite-rank covariance~\eqref{eq:cov-dep-main} with $M=2$: a \emph{shared}
component carrying the base draw, $\tilde\Sigma^{(1)}=I_n\otimes\bone_k\bone_k^\top$ and
$\Sigma^{(1)}=\Sigma_{12}$, and an \emph{idiosyncratic} component carrying the copy-specific
fluctuation, $\tilde\Sigma^{(2)}=I_{nk}$ and $\Sigma^{(2)}=\Sigma_{11}-\Sigma_{12}$. 
{Since} $\Sigma_{11}$ and $\Sigma_{12}$ are
simultaneously diagonalizable, this is an example of \Cref{mod:dep}, so \Cref{lem:model:equivalence} in
\Cref{appendix:setup} reduces it
to a heterogeneous variance profile, and the support of that profile predicts the extra risk
peaks. 

{
As a simple special case, taking
$A=\sqrt{\rho}\,I_p$ and
$B=\sqrt{1-\rho}\,I_p$ for any fixed
$\rho\in(0,1)$ gives
$T_{ic}(v)=\sqrt{\rho}\,v+
\sqrt{1-\rho}\,E_{ic}$.
}
\end{example}
\subsection{The variance-profile model and the reduction} \label{subsec:vp-model}
Our results are proved for a Gaussian design whose entries are independent but have unequal
variances{,}{ known as }a \emph{variance-profile} matrix 
\citep{hachem2007deterministic, alt2017local}{.} {We transfer them} to the two designs above
via \Cref{subsec:reduction}.

\begin{model}[Variance profile] \label{mod:vp}
The $\R^{n\times p}$ matrix $X$ has independent mean-zero Gaussian entries with variance
profile $nS\in\R^{n\times p}$:
\begin{align*}
    \Var\Big[\mfrac{1}{\sqrt n}\,X_{ij}\Big] \;=\; S_{ij}\;.
\end{align*}
\end{model}

The normalization is chosen so that the rescaled design $\frac{1}{\sqrt n}X$ carries the
profile $S$ used in \citep{alt2017local}. 
 Under \Cref{mod:vp} the row covariances are diagonal,
$\Var[X_i] = \diag(nS_{i1}, \ldots, nS_{ip})$, so the averaged covariance is
\begin{align*}
    \bar\Sigma
    \;=\;
    \mfrac1n\msum_{i\le n}\Var[X_i]
    \;=\;
    \msum_{i\le n}\diag(S_{i1}, \ldots, S_{ip})\;,
    \qquad\text{i.e.}\qquad
    \bar\Sigma_{ll} \;=\; \msum_{i\le n} S_{il}\;.
    \tagaligneq \label{eq:Sigmabar-S}
\end{align*}

Although \Cref{mod:vp} has independent entries, it is more general than the
i.i.d.~isotropic design, and both practical designs reduce to it. {The \emph{zero pattern} of $S$} is what
produces several interacting degeneracies, and hence multiple descent.

The reduction is what makes $S$ the right object. For \Cref{mod:het} and \Cref{mod:dep},
passing to the common eigenbasis turns
each design into a Gaussian matrix with independent entries --- that is, into \Cref{mod:vp}
with an explicit profile $S$ built from the covariance eigenvalues
(\Cref{lem:model:equivalence,lem:reduction:var:profile} in \Cref{appendix:setup}).
It therefore suffices to analyze the variance-profile model,
which we do in \Cref{subsec:master} and in \Cref{sec:vp-risk-main}; the resulting limiting
risk for \Cref{mod:het} and \Cref{mod:dep} is recorded in \Cref{cor:practical-risk}. The
positive-definite designs of \Cref{sec:non:sim:diag:but:pos} do \emph{not} reduce
to a variance profile; we discuss what can still be said about them there.

\subsection{The variance-profile assumption}

{
Our limiting-risk theorem asks that the profile $S$ be uniformly small and
sufficiently connected, which are the standing assumption of
\citep{alt2017local,huang2022data}. Both are mild. When a profile violates the connectivity condition by decomposing into non-interacting components, the
risk also decomposes accordingly. Then it suffices that each component be uniformly small and connected
(\Cref{rem:decomposable}).
}

\begin{assumption}[Variance profile] \label{asst:flat}
There exist constants $s_*, \psi_1, \psi_2 > 0$ and integers $L_1, L_2\in\N$ such that
\begin{align*}
    S_{ij} \;\le\; \mfrac{s_*}{p+n}\;,
    \qquad
    \big[(S S^\top)^{L_1}\big]_{ii'} \;\ge\; \mfrac{\psi_1}{n+p}\;,
    \qquad
    \big[(S^\top S)^{L_2}\big]_{jj'} \;\ge\; \mfrac{\psi_2}{n+p}\;,
\end{align*}
for all $1\le i,i'\le n$ and $1\le j,j'\le p$, where $(\cdot)^L$ denotes the $L$-th matrix power above.
\end{assumption}

 The flatness bound $S_{ij}\le s_*/(p+n)$ ensures that no coordinate's variance is abnormally large, so
the limiting spectrum is shaped by the whole profile rather than a few dominant
entries. 

The lower bounds on powers of $SS^\top$ and $S^\top S$ are connectivity conditions.
Observe first that
\[
    (SS^\top)_{ii'} \;=\; \msum_{j=1}^p S_{ij}\,S_{i'j}
\]
So $SS^\top$ is the \emph{overlap matrix} of the rows: its $(i,i')$ entry is positive exactly
when observations $i$ and $i'$ place variance on a common coordinate. Higher powers record
indirect overlap. The entry $\big[(SS^\top)^{L_1}\big]_{ii'}$ is the total weight of all chains
$i = i_0, i_1, \ldots, i_{L_1} = i'$ in which consecutive observations share a coordinate, and it
is positive exactly when $i$ and $i'$ are linked by such a chain of length at most $L_1$.
The bound $\big[(SS^\top)^{L_1}\big]_{ii'}\ge \psi_1/(n+p)$ for \emph{every} pair $(i,i')$ asks
for two things: \\
{(i)} any two observations are joined by such a chain of length at most $L_1$,
so the profile is connected with no two observations more than $L_1$ overlap-steps apart; \\
{(ii) } the weight along these chains is at least $\psi_1/(n+p)$, uniformly in $n$ and $p$: the
overlaps stay bounded below rather than merely nonzero. The bound on $(S^\top S)^{L_2}$ is the
same with observations and coordinates exchanged. Together with the flatness bound, these are
exactly assumptions (A) and (B) of \citep{alt2017local}.

The support of $S$ alone fixes the smallest power $L_1$ for which every entry of $(SS^\top)^{L_1}$
is positive. The next three patterns illustrate this; in each we take the nonzero entries equal,
so that $(SS^\top)_{ii'}$ counts the coordinates shared by observations $i$ and $i'$, and $L_1$
is the number of steps needed to link every pair of rows.
\begin{itemize}
    \item \emph{Diameter one} ($L_1=1$). For the pattern
    \begin{align*}
        S \;=\;
        \begin{psmallmatrix} 1 & 1 & 0 \\ 1 & 0 & 1 \\ 0 & 1 & 1 \end{psmallmatrix}\;,
        \qquad\text{so}\qquad
        SS^\top \;=\;
        \begin{psmallmatrix} 2 & 1 & 1 \\ 1 & 2 & 1 \\ 1 & 1 & 2 \end{psmallmatrix}\;,
    \end{align*}
    every off-diagonal entry of $SS^\top$ is already positive: each pair of observations
    shares a coordinate directly, so $L_1=1$.
    \item \emph{Diameter two} ($L_1=2$). For the chain pattern
    \begin{align*}
        S \;=\;
        \begin{psmallmatrix} 1 & 1 & 0 & 0 \\ 0 & 1 & 1 & 0 \\ 0 & 0 & 1 & 1 \end{psmallmatrix}\;,
        \qquad\text{so}\qquad
        SS^\top \;=\;
        \begin{psmallmatrix} 2 & 1 & 0 \\ 1 & 2 & 1 \\ 0 & 1 & 2 \end{psmallmatrix}\;,
    \end{align*}
    we have $(SS^\top)_{13}=0$: observations $1$ and $3$ share no coordinate. They are
    linked only through observation $2$, and $\big[(SS^\top)^2\big]_{13}=1>0$, so $L_1=2$ but
    not $L_1=1$.
    \item \emph{Reducible} (no finite $L_1$). For the block pattern
    \begin{align*}
        S \;=\;
        \begin{psmallmatrix}
            1 & 1 & 0 & 0 \\ 1 & 1 & 0 & 0 \\ 0 & 0 & 1 & 1 \\ 0 & 0 & 1 & 1
        \end{psmallmatrix}\;,
        \qquad\text{so}\qquad
        SS^\top \;=\;
        \begin{psmallmatrix}
            2 & 2 & 0 & 0 \\ 2 & 2 & 0 & 0 \\ 0 & 0 & 2 & 2 \\ 0 & 0 & 2 & 2
        \end{psmallmatrix}\;,
    \end{align*}
    $SS^\top$ is block diagonal, and so is every power; the cross-block entry
    $\big[(SS^\top)^{L}\big]_{13}$ is $0$ for \emph{all} $L$. No finite $L_1$ works: the design
    splits into two groups that never interact, and any sequence of profiles built on this
    pattern falls outside \Cref{asst:flat}. While this example does not satisfy \Cref{asst:flat} itself, each of its two blocks satisfies the assumption. The remark below explains how our theory applies here.
\end{itemize}

{

\begin{remark}[Decomposable profiles] \label{rem:decomposable}
{\Cref{asst:flat}} can hold for finite $L_1, L_2$ only if the variance profile is sufficiently connected. In particular, a variance profile that splits into components, like the reducible pattern above, fails \Cref{asst:flat} for every
$L_1$ and $L_2$. Nonetheless, in this case, the risk decomposes into the same components as the variance profile, with each component  satisfying \Cref{asst:flat} and thereby addressable by our theory. To see this, let $I_1,\ldots,I_K$ and $J_1,\ldots,J_K$ be the
observations and the coordinates of the components of the profile $S$, and let $n_k \coloneqq |I_k|$,
$p_k \coloneqq |J_k|$, $S^{(k)} \coloneqq (S_{ij})_{i\in I_k,\, j\in J_k}$,
$X^{(k)} \coloneqq (X_{ij})_{i\in I_k,\, j\in J_k}$ and $\beta_{J_k} \coloneqq (\beta_l)_{l\in J_k}$.
Components with $n_k=0$ are coordinates no observation sees, and are discarded. As $S_{ij}=0$
across components, permuting rows and columns makes $S$, and hence $X$, $W_n$ and $\bar\Sigma$,
block diagonal along $J_1,\ldots,J_K$, so both terms of \Cref{lem:bias-var-decomp} split into
their $J_k$ blocks $W^{(k)}$ and $\bar\Sigma^{(k)}$:
\begin{align*}
    \cR_B^\lambda(X)
    \;&=\;
    \msum_{k\le K}\big\| (\bar\Sigma^{(k)})^{1/2}\big( (W^{(k)}+\lambda I_{p_k})^{-1}W^{(k)} - I_{p_k}\big)\beta_{J_k}\big\|^2\;,
    \\
    \cR_V^\lambda(X)
    \;&=\;
    \mfrac{\sigma^2_\epsilon}{n}\msum_{k\le K}\Tr\big(\bar\Sigma^{(k)}\,W^{(k)}(W^{(k)}+\lambda I_{p_k})^{-2}\big)\;.
\end{align*}
Each summand is the risk of its component. $X^{(k)}$ follows \Cref{mod:vp} with $n_k$ observations
and profile $\frac{n}{n_k}S^{(k)}$, and the $J_k$ block of $\hat\beta_\lambda$ is its ridge
estimator at ridge level $\frac{n}{n_k}\lambda$, so
$\cR^\lambda(X;\beta) = \sum_{k\le K}\frac{n_k}{n}\,\cR^{n\lambda/n_k}(X^{(k)};\beta_{J_k})$. If
each component satisfies \Cref{asst:flat} in its own normalization, with $p_k/n_k$ converging in
$(0,\infty)$ and $n_k/n$ bounded away from $0$, then {our theory (}\Cref{thm:master}) applies to each summand and
the equivalents add, every component keeping its own thresholds. We assume \Cref{asst:flat} for $S$, without loss of generality. 
\end{remark}
}

\subsection{The limiting risk} \label{subsec:master}

We now give the limiting prediction risk $\cR^\lambda(X)$ under the
variance-profile model {(\Cref{mod:vp})}.
The deterministic equivalent is governed by a vector $\br(\lambda) = (r_l(\lambda))_{l\le p}$
indexed by the coordinates.
 For $\lambda > 0$ it solves
the fixed-point system
\begin{align*}
    1 \;=\; r_l(\lambda) + \msum_{i=1}^n \mfrac{S_{il}\, r_l(\lambda)}{\lambda + \sum_{j=1}^p S_{ij}\, r_j(\lambda)}\;,
    \qquad 1\le l\le p\;.
    \tagaligneq \label{eq:r-fixedpoint}
\end{align*}
The system~\eqref{eq:r-fixedpoint} is the coordinate-wise reduction, to the diagonal of the
resolvent, of the matrix Dyson equation for the variance profile. 
By the quadratic-vector-equation theory of \citep{Ajanki_2019, alt2017local}, the system has a \emph{unique} solution $\br$ in the domain $\lambda \in \C \setminus (-\infty,0]$, which is also a solution in the domain $\lambda > 0$. The next lemma provides some useful properties of $\br$ (see  \Cref{lem:r:eqns,lem:r:values} in \Cref{appendix:properties:r} for the proof).

\begin{lemma}[Spectral representation of $\br$] \label{lem:r-spectral}
For each $1\le l\le p$ there is a symmetric probability measure $\nu_l$ on $\R$ such that,
for all $\lambda\ge 0$,
\begin{align*}
    r_l(\lambda) \;=\; \mint_\R \mfrac{\lambda}{t^2+\lambda}\,\nu_l(dt)\;.
\end{align*}
Consequently
\begin{align*}
    r_l(0) \;=\; \nu_l(\{0\})
    \qquad\text{ and }\qquad
    \partial r_l(0) \;\coloneqq\; \lim_{\lambda\downarrow 0}\partial_\lambda r_l(\lambda)
    \;=\; \mint_{\R\setminus\{0\}} \mfrac{1}{t^2}\,\nu_l(dt) \;\in\; [0,\infty]\;.
\end{align*}
\end{lemma}

This representation exposes the two mechanisms behind double descent: $r_l(0) = \nu_l(\{0\})$ is the mass that the
$l$-th coordinate places exactly at the hard edge, and it fixes the coordinates on which the \emph{bias} is supported (\Cref{thm:bias-support});
$\partial r_l(0) = \int_{\R\setminus \{0\}} t^{-2}\nu_l(dt)$ measures how strongly small but nonzero
eigenvalues accumulate near the edge, and it drives the \emph{variance}. The latter
is finite when $\nu_l$ has no mass concentrating at zero from outside, and
infinite when the spectrum touches the edge.

Write $\cR_B^{\lambda}(X;\beta)$ {and $\cR_V^{\lambda}(X;\beta)$ } for the bias and variance term of \Cref{lem:bias-var-decomp}, { where we make the dependence on} signal $\beta\in\R^p$ {explicit}. \Cref{thm:master} below evaluates it along the vanishing sequence
$\lambda=\lambda_n$ {and under \Cref{mod:vp}}.

\begin{theorem}[Limiting vanishing-ridge risk under a variance profile] \label{thm:master}
Suppose \Cref{mod:vp} and \Cref{asst:flat} hold, and let $p/n\to\gamma\in(0,\infty)$. {For every fixed 
{$\lambda_0 \in (0,1)$}
and $\delta>0$, }there
is a constant $\alpha_*>0$, depending only on {$\delta$ and} the constants $s_*,\psi_1,\psi_2,L_1,L_2$ of
\Cref{asst:flat}, such that the following holds for every  deterministic ridge sequence $\lambda=\lambda_n$ with
$\lambda_n\in(n^{-\alpha},\lambda_0)$ for some fixed $0<\alpha<\alpha_*$: for any deterministic
signal $\beta$ with 
{ 
    a fixed $\| \beta \| \in (0, \infty)$} 
and $\|\beta\|_{\ell_1}\le n^{\frac14-\delta}$, the bias and
variance terms $\cR_B^{\lambda_n}(X;\beta)$ and $\cR_V^{\lambda_n}(X)$ of
\Cref{lem:bias-var-decomp} admit deterministic equivalents, almost surely:
\begin{align*}
    \cR_B^{\lambda_n}(X;\,\beta) - \msum_{l=1}^p \beta_l^2\; e_l^\top\, \bA(\lambda_n)^{-1}\,\bar\Sigma\,\br(\lambda_n) \;\xrightarrow{\rm a.s.}\; 0, \tagaligneq \label{eq:bias-limit}
\end{align*}
where $e_l$ is the $l$-th standard basis vector of $\R^p$ and
\begin{align*}
    \bA(\lambda) \;\coloneqq\; \diag\{\br(\lambda)\}^{-1} - \diag\{\br(\lambda)\}\, S^\top \diag\{\lambda\bone_n+S\br(\lambda)\}^{-2} S .
\end{align*}
Moreover,
\begin{align*} \cR_V^{\lambda_n}(X) - \mfrac{\sigma_\epsilon^2}{n} \msum_{l=1}^p \bar\Sigma_{ll}\,\partial r_l(\lambda_n) \;\xrightarrow{\rm a.s.}\; 0\,. \tagaligneq \label{eq:var-limit}
\end{align*}
\end{theorem}

\begin{remark}[Beyond the Gaussian case] \label{rem:non-gaussian}
{
The proof of \Cref{thm:master} relies on an adaptation of the local law of \citep{alt2017local}, which holds also for independent \emph{non-Gaussian} entries with matching variance profile and uniformly bounded moments.
}
 The
limiting risk~\eqref{eq:bias-limit}--\eqref{eq:var-limit} therefore continues to describe
any such variance-profile design. 
{
    Gaussianity is used only for reducing Model \ref{mod:het} and Model \ref{mod:dep} to the variance profile model by an orthogonal-invariance argument; see Section \ref{subsec:reduction}.
}
For non-Gaussian heterogeneous data{,} the
master theorem still applies whenever the design is already entrywise independent. This is
{supported numerically}{; see \Cref{fig:sim-latent-nongauss}.
}
\end{remark}

The analogous result for the heterogeneous and dependent designs (\Cref{mod:het,mod:dep}) is deferred to Section \ref{subsec:reduction}. As a sanity check, we first examine the isotropic cases.

\begin{example}[Isotropic i.i.d.\ design] \label{ex:iso}
When the observations are i.i.d.\ isotropic Gaussians, $S_{ij} = 1/n$, $\bar\Sigma = I_p$,
and all coordinates share a common scalar $r(\lambda)$. In the limit $p/n\to\gamma$ the
system~\eqref{eq:r-fixedpoint} becomes $1 = r(\lambda) + r(\lambda)/(\lambda + \gamma r(\lambda))$,
i.e.\ $\gamma r(\lambda)^2 + (1+\lambda-\gamma) r(\lambda) - \lambda = 0$; the root with
$r(0)\ge 0$ gives
\begin{align*}
    r(0) \;=\; \Big(1 - \mfrac1\gamma\Big)_{\!+}
    \qquad\text{ and }\qquad
    \partial r(0)
    \;=\;
    \begin{cases}
        \frac{1}{1-\gamma}, & \gamma < 1,\\[.3em]
        \frac{1}{\gamma(\gamma-1)}, & \gamma > 1.
    \end{cases}
\end{align*}
Here $r(0) = \nu(\{0\})$ is exactly the atom of the Mar\v{c}enko--Pastur law at the
origin. Substituting into~\eqref{eq:bias-limit}--\eqref{eq:var-limit} recovers the
familiar ridgeless formulas of \citep{hastie2022surprises}: the bias is $0$ for
$\gamma<1$ and $\|\beta\|^2(\gamma-1)/\gamma$ for $\gamma>1$, while the variance is
$\sigma_\epsilon^2\gamma/(1-\gamma)$ for $\gamma<1$ and $\sigma_\epsilon^2/(\gamma-1)$
for $\gamma>1$, both diverging as $\gamma\to 1$. There is a single spectral degeneracy,
located at the single aspect ratio $\gamma = 1$, and hence a single peak.
\end{example}

\begin{example}[Isotropic design with unseen coordinates] \label{ex:iso-unseen}
Let every observation see only the first $p_1$ coordinates, with $p_1/p\to\theta\in(0,1)$:
that is, $S_{il} = 1/n$ for $l\le p_1$ and $S_{il}=0$ otherwise. This is \Cref{mod:het} with
the rank-deficient covariances $\Sigma_i = \diag(1,\ldots,1,0,\ldots,0)$. It is decomposable {(\Cref{rem:decomposable})}, so \Cref{asst:flat} is applied blockwise.
The system~\eqref{eq:r-fixedpoint} then decouples. For an unseen coordinate $l>p_1$ the sum on
the right-hand side of \eqref{eq:r-fixedpoint} is empty, so $r_l(\lambda)=1$ for all $\lambda$:
the coordinate carries full hard-edge mass, $r_l(0)=1$, and $\partial r_l(0)=0$. The seen
coordinates share a common scalar $r(\lambda)$ that satisfies the isotropic equation of
\Cref{ex:iso} with the smaller ratio $\gamma_1 \coloneqq \lim p_1/n = \theta\gamma$ in place of
$\gamma$; in particular $r(0)=(1-1/\gamma_1)_+$. In the language of \Cref{sec:vp-risk-main},
the unseen block belongs to the degenerate set $\bJ_S$ at every aspect ratio. 
The seen block enters $\bJ_S$ only once $\gamma_1>1$, and the variance diverges
at the single shifted threshold $\gamma = 1/\theta$, where the seen block becomes square. One
rank-deficient design thus already moves the peak away from $\gamma=1$; two such blocks with different ratios give two peaks (\Cref{ex:two-group-disjoint}).
\end{example}

We defer the proof of \Cref{thm:master} to \Cref{appendix:proof:bias} (bias term) and \Cref{appendix:proof:var} (variance term). The proof adapts the ridge-resolvent
framework of \citep{hastie2022surprises} to the variance profile, substituting the
deterministic Dyson solution of \citep{alt2017local} for the empirical resolvent.
The {main technical} ingredient concerns the vanishing penalty. The local law of \citep{alt2017local}
is stated for spectral parameters bounded away from $0$, so to follow the resolvent as
$\lambda_n\downarrow0$ we re-derive it with explicit dependence on the spectral parameter{. This gives} a quantitative refinement of \citep{alt2017local, Ajanki_2019} {(\Cref{appendix:local:law}), which may be} of independent interest. {Our result} lets $\lambda_n$ approach zero as fast as the
resolution of the local law permits, which exposes the structural degeneracies that
a fixed ridge penalty would smooth over; we read them off the support of $S$ in
\Cref{sec:vp-risk-main}.

\begin{remark}[The exact ridgeless limit]\label{rem:exact-ridgeless}
Our results are stated for a vanishing penalty $\lambda_n\downarrow0$ rather than the exact
ridgeless interpolator ($\lambda=0$). Reaching $\lambda=0$ would require controlling the
smallest \emph{nonzero} singular value of $X$ (the soft lower edge of the non-zero bulk)
uniformly down to zero, a sharper input than our local law provides. We expect the limiting
risk to be unchanged, and leave this refinement open.
\end{remark}

\subsection{Reduction to a variance profile} \label{subsec:reduction}

The limiting risk for the two designs now follows from \Cref{thm:master} with no extra
work. As explained in \Cref{subsec:vp-model} and proved in \Cref{appendix:setup}, each \Cref{mod:het}
design becomes, in the common eigenbasis of its covariances, a variance profile (\Cref{mod:vp})
with an explicit $S$ (\Cref{lem:reduction:var:profile}), and each \Cref{mod:dep} design reduces
to this case after regrouping, up to a rotation of the signal (\Cref{lem:model:equivalence}).
Feeding the induced profile into \Cref{thm:master} gives the risk for both designs.

\begin{corollary}[Limiting risk for the two designs] \label{cor:practical-risk}
Suppose \Cref{mod:het} or \Cref{mod:dep} holds, and let $S$ be the variance profile it
induces through \Cref{lem:reduction:var:profile} (for \Cref{mod:dep}, after the reduction of
\Cref{lem:model:equivalence}), with signal rotation $U$; both lemmas are in \Cref{appendix:setup}.
Assume $S$ satisfies \Cref{asst:flat},
$p/n\to\gamma\in(0,\infty)$, and take a ridge sequence $\lambda_n\in(n^{-\alpha},\lambda_0)$ with
$0<\alpha<\alpha_*$ as in \Cref{thm:master}, and a signal $\beta$ whose rotation $U\beta$ meets
the signal conditions of \Cref{thm:master} (namely $\sup_p\|U\beta\|<\infty$ and
$\|U\beta\|_{l_1}\le n^{\frac14-\delta}$; recall $\|U\beta\|=\|\beta\|$). Then the
vanishing-ridge risk $\cR^{\lambda_n}(X;\beta)$ obeys the conclusion of \Cref{thm:master} with the
profile $S$ and the rotated signal $U\beta$: its bias and variance terms agree almost surely with
the deterministic equivalents \eqref{eq:bias-limit}--\eqref{eq:var-limit}. 
\end{corollary}
\section{The variance-profile regime: degeneracy and multiple descent} \label{sec:vp-risk-main}

The deterministic equivalents in \Cref{thm:master} are explicit functions of the fixed-point
solution~\eqref{eq:r-fixedpoint}, and their behavior along a vanishing ridge is captured by its
value $\br(0)$ and derivative $\partial\br(0)$ at the hard edge. We
determine both from the support of $S$ by classical matching theory: the bias support in
\Cref{subsec:bias-support} and the variance blow-up in \Cref{subsec:switching}. Together they
locate the descent peaks through the Dulmage--Mendelsohn structure of the variance graph.
{
Throughout, $X$ obeys \Cref{mod:vp} with profile $S$ satisfying \Cref{asst:flat}, and
$\gamma = \lim p/n$. When 
{ the variance profile} is disconnected, \Cref{asst:flat} is imposed on each component
and every statement below is read componentwise (\Cref{rem:decomposable}).
}

{We read both $\br(0)$ and $\partial \br(0)$} from a single bipartite graph attached to $S$. The
\emph{variance graph} $\cG_S$ (\Cref{def:var-graph}) joins observation $i$ to coordinate $j$
whenever $S_{ij}>0$. The bias is carried by the \emph{degenerate set} $\bJ_S$
(\Cref{def:Jset}), the coordinates that some maximum matching leaves unmatched: we show that
$r_l(0)>0$ if and only if $l\in\bJ_S$ (\Cref{thm:bias-support}). The variance is governed by a
\emph{switching condition} (\Cref{def:switching}) on the design that remains once $\bJ_S$ and
its adjacent rows are removed; the derivative $\partial\br(0)$ is finite if and only if this
condition fails, so each design at which it holds has divergent variance, {giving} a
descent peak (\Cref{thm:var-switch}).

\subsection{The variance graph and the support of the bias} \label{subsec:bias-support}

The relevant object is the bipartite graph of the profile $S$. Since
permuting rows or columns with identical patterns leaves the problem invariant, we
quotient out such permutations by passing to the graph.

\begin{definition}[Variance graph] \label{def:var-graph}
The variance graph $\cG_S$ of $S$ is the bipartite graph on the vertex set
$\cV_S = \cV_{\rm row}\cup\cV_{\rm col}$, with $\cV_{\rm row} = \{1,\ldots,n\}$,
$\cV_{\rm col} = \{1,\ldots,p\}$, and an edge $i\sim j$ if and only if $S_{ij} > 0$. We
write $\cE_S$ for its edge set.
\end{definition}

A \emph{matching} is a set of edges no two of which share a vertex; a \emph{maximum
matching} $\cE^{\rm max}_S$ is one of largest cardinality. The number $|\cE^{\rm max}_S|$
is not realized by a unique matching but is itself unique, and equals the
\emph{structural rank} of $S$---the maximal rank over all matrices with the same support
as $S$, equivalently the almost-sure rank of $X$. The combinatorial invariant that
governs the bias is the set of columns that can be left out.

\begin{definition}[Degenerate column set] \label{def:Jset}
The degenerate column set of $S$ is
\begin{align*}
    \bJ_S
    \;\coloneqq\;
    \big\{\, 1\le j\le p \;\big|\; \text{some maximum matching of } \cG_S \text{ leaves } j \text{ unmatched} \,\big\}
    \;\subseteq\; \cV_{\rm col}\;.
\end{align*}
\end{definition}

The set $\bJ_S$ is a classical object: it is the column part of the
Dulmage--Mendelsohn (DM) decomposition of the bipartite graph
$\cG_S$~\citep{Dulmage_Mendelsohn_1958, pulleyblank1996matchings, pothen1990computing},
the canonical partition encoding the structure of all maximum matchings, and it is
computed in linear time by alternating-path search from any maximum matching.
It is widely used in sparse numerical linear algebra to read off structural rank and
the support of pseudoinverses, and is available in standard routines (for instance
\texttt{dmperm} in \textsc{Matlab}). The next two examples illustrate it.

\begin{example}[All-one variance profile] \label{ex:all-ones}
When the observations are i.i.d.~isotropic Gaussians, the variance profile $S$ in \Cref{mod:vp} is the all-one $n \times p$ matrix. In this case, we can compute the size of any maximum matching as
\begin{align*}
    | \cE^{\rm max}_S | \;=\; \min\{n,p\}\;,
\end{align*}
which is also the asymptotic rank of the sample covariance matrix $X^\top X$. For any given matching, the number of unmatched vertices in $\cV_{\rm col}$ is $p - | \cE^{\rm max}_S | = p - \min\{n,p\}$, which captures the dimension of the null space of $S$. While the set of unmatched vertices is matching-specific, our degenerate column set $\bJ_S$ captures possible degeneracies induced by \emph{all} matchings. Indeed, since the columns of $S$ are identical, the columns of Gaussian data matrix $X$ are exchangeable, and we do not expect the quantity that characterizes degeneracies to be matching specific. In this case, $\bJ_S$ is either the empty set or the full set:
\begin{align*}
    \bJ_S \;=\; \begin{cases}
        \emptyset & \text{ if } p \leq n\;,
        \\
        \{1, \ldots, p\} & \text{ if } p > n\;.
    \end{cases}
\end{align*}
Using our result on the connection between $\bJ_S$ and bias, the statistical intuition is the following: In the underparameterized regime, signals in all dimensions are fully described by the data, the set of degenerate columns is empty and the bias term of the linear regression is $0$. In the overparameterized regime, the signal along any dimension has a non-negligible probability to be missed out by the data, so the set of degenerate columns is the full set and a bias is introduced along every dimension.
\end{example}

\begin{example}[Two-group profile with overlapping column sets]
\label{ex:two-group-overlapping}
Consider the following example of $S$, where the observations consist of two
independent heterogeneous groups with overlapping column sets:
\begin{align*}
    S
    \;=&\;
    \begin{psmallmatrix}
        \bone_{n_1\times(p_1-p_{12})}
        & \bone_{n_1\times p_{12}}
        & 0
        & 0
        \\
        0
        & \bone_{n_2\times p_{12}}
        & \bone_{n_2\times(p_2-p_{12})}
        & 0
    \end{psmallmatrix}
    \;,
\end{align*}
where $\bone_{n'\times p'}$ denotes the $n'\times p'$ matrix consisting of
all ones. Let $\cV_1,\cV_2\subseteq\cV_{\rm col} \coloneq [p]$ denote the column index
sets of the two groups, where \begin{align*}
    \cV_1
    \coloneqq \{1,\ldots,p_1\}, \qquad
    \cV_2
    \coloneqq
    \{p_1-p_{12}+1,\ldots,p_1+p_2-p_{12}\}.
\end{align*}
Then
\begin{align*}
    |\cV_1|=p_1,
    \qquad
    |\cV_2|=p_2,
    \qquad
    |\cV_1\cap\cV_2|=p_{12}>0.
\end{align*}
We further write
\begin{align*}
    \cV_0
    \;\coloneqq\;
    \{p_1+p_2-p_{12}+1,\ldots,p\}
    \;=\;
    \cV_{\rm col}\setminus(\cV_1\cup\cV_2).
\end{align*}
The four groups of columns in $S$ are indexed,
from left to right, by $\cV_1\setminus\cV_2$, $\cV_1\cap\cV_2$,
$\cV_2\setminus\cV_1$, and $\cV_0$. In particular, the coordinates in
$\cV_1\cap\cV_2$ are adjacent to both observation groups, while those in $\cV_0$ have no incident edges.
The degenerate column set is given by four cases.
\begin{proplist}
    \item If $p_1-p_{12}\le n_1$, $p_2-p_{12}\le n_2$, and
    $p_1+p_2-p_{12}\le n_1+n_2$, then
    $\bJ_S=\cV_0$.

    \item If $p_1-p_{12}>n_1$ and $p_2\le n_2$, then
    $\bJ_S=\cV_0\cup(\cV_1\setminus\cV_2)$.

    \item If $p_2-p_{12}>n_2$ and $p_1\le n_1$, then
    $\bJ_S=\cV_0\cup(\cV_2\setminus\cV_1)$.

    \item If $p_1>n_1$, $p_2>n_2$, and
    $p_1+p_2-p_{12}>n_1+n_2$, then
    $\bJ_S=\cV_0\cup\cV_1\cup\cV_2$.
\end{proplist}

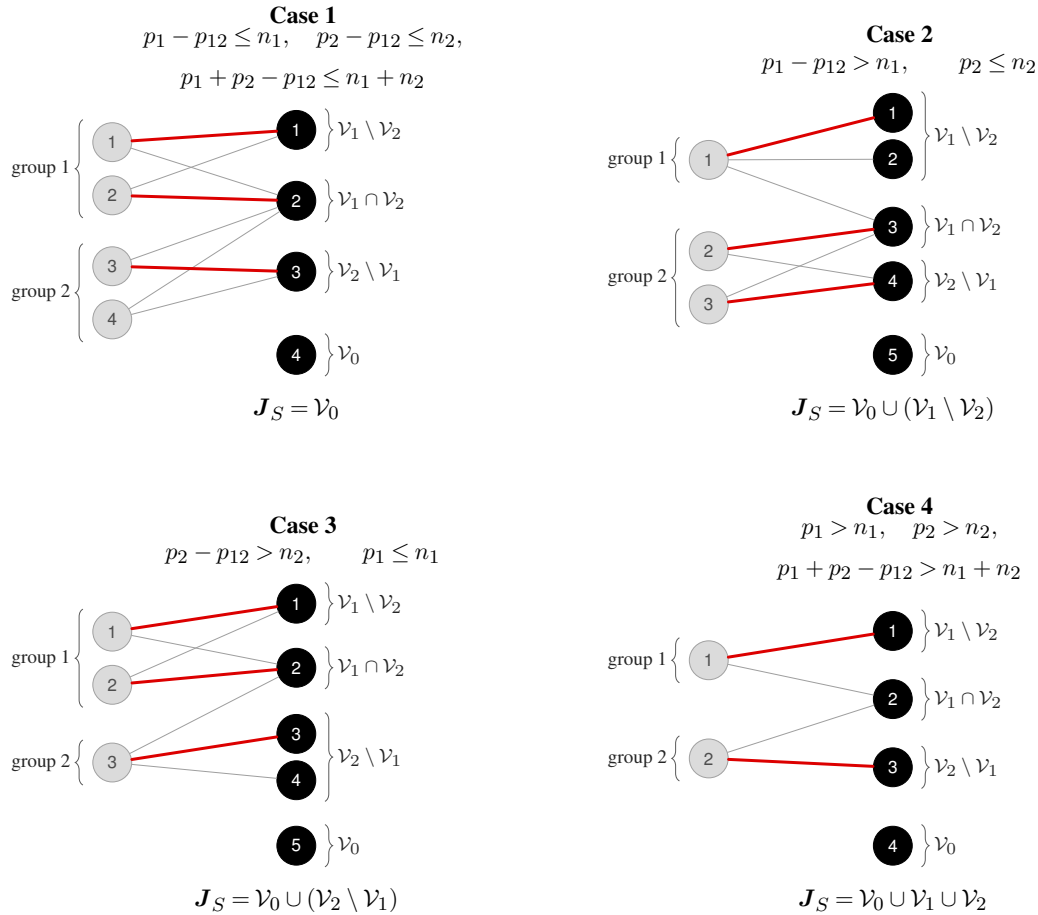
\begin{figure}[h]
    \centering
    \resizebox{\linewidth}{!}{%
        \begingroup
\providecommand{\cV}{\mathcal V}
\providecommand{\bJ}{\mathbf J}

\definecolor{edgegray}{gray}{0.62}
\definecolor{rowfill}{gray}{0.86}
\definecolor{rowdraw}{gray}{0.63}
\definecolor{matchred}{RGB}{220,0,0}

\tikzset{
  rowvertex/.style={
    circle,
    draw=rowdraw,
    fill=rowfill,
    text=black!70,
    line width=0.35pt,
    minimum size=5.7mm,
    inner sep=0pt,
    font=\sffamily\scriptsize
  },
  colvertex/.style={
    circle,
    draw=black,
    fill=black,
    text=white,
    line width=0.35pt,
    minimum size=5.7mm,
    inner sep=0pt,
    font=\sffamily\scriptsize
  },
  varianceedge/.style={
    draw=edgegray,
    line width=0.38pt,
    line cap=round
  },
  matchingedge/.style={
    draw=matchred,
    line width=1.25pt,
    line cap=round
  },
  groupbrace/.style={
    decorate,
    decoration={brace,amplitude=3.2pt},
    draw=black!65,
    line width=0.32pt
  },
  leftgroupbrace/.style={
    decorate,
    decoration={brace,amplitude=3.2pt,mirror},
    draw=black!65,
    line width=0.32pt
  },
  paneltitle/.style={
    align=center,
    font=\small
  },
  grouplabel/.style={
    font=\scriptsize,
    text=black!85,
    inner sep=1pt
  },
  jslabel/.style={
    font=\small
  }
}

\begin{tikzpicture}[x=1cm,y=1cm]

\begin{scope}[shift={(0,6.3)}]
  \node[paneltitle] at (4.25,3.48) {\textbf{Case 1}\\[-1pt]
    $\begin{gathered}
      p_1-p_{12}\le n_1,\quad p_2-p_{12}\le n_2,\\[-1pt]
      p_1+p_2-p_{12}\le n_1+n_2
    \end{gathered}$};

  \node[rowvertex] (c1r1) at (1.45, 2.12) {1};
  \node[rowvertex] (c1r2) at (1.45, 1.34) {2};
  \node[rowvertex] (c1r3) at (1.45, 0.30) {3};
  \node[rowvertex] (c1r4) at (1.45,-0.48) {4};

  \node[colvertex] (c1a1) at (4.15, 2.30) {1};
  \node[colvertex] (c1s1) at (4.15, 1.26) {2};
  \node[colvertex] (c1b1) at (4.15, 0.22) {3};
  \node[colvertex] (c1z1) at (4.15,-1.00) {4};

  \foreach \r in {c1r1,c1r2}{
    \draw[varianceedge] (\r) -- (c1a1);
    \draw[varianceedge] (\r) -- (c1s1);
  }
  \foreach \r in {c1r3,c1r4}{
    \draw[varianceedge] (\r) -- (c1s1);
    \draw[varianceedge] (\r) -- (c1b1);
  }

  \draw[matchingedge] (c1r1) -- (c1a1);
  \draw[matchingedge] (c1r2) -- (c1s1);
  \draw[matchingedge] (c1r3) -- (c1b1);

  \draw[leftgroupbrace] (1.02,2.45) -- (1.02,1.01)
    node[midway,left=4pt,grouplabel] {group 1};
  \draw[leftgroupbrace] (1.02,0.63) -- (1.02,-0.81)
    node[midway,left=4pt,grouplabel] {group 2};

  \draw[groupbrace] (4.57,2.61) -- (4.57,1.99)
    node[midway,right=4pt,grouplabel] {$\cV_1\setminus\cV_2$};
  \draw[groupbrace] (4.57,1.57) -- (4.57,0.95)
    node[midway,right=4pt,grouplabel] {$\cV_1\cap\cV_2$};
  \draw[groupbrace] (4.57,0.53) -- (4.57,-0.09)
    node[midway,right=4pt,grouplabel] {$\cV_2\setminus\cV_1$};
  \draw[groupbrace] (4.57,-0.69) -- (4.57,-1.31)
    node[midway,right=4pt,grouplabel] {$\cV_0$};

  \node[jslabel] at (4.15,-1.80) {$\bJ_S=\cV_0$};
\end{scope}

\begin{scope}[shift={(8.75,6.3)}]
  \node[paneltitle] at (4.25,3.48) {\textbf{Case 2}\\[-1pt]
    $p_1-p_{12}>n_1,\qquad p_2\le n_2$};

  \node[rowvertex] (c2r1) at (1.45, 1.86) {1};
  \node[rowvertex] (c2r2) at (1.45, 0.52) {2};
  \node[rowvertex] (c2r3) at (1.45,-0.26) {3};

  \node[colvertex] (c2a1) at (4.15, 2.55) {1};
  \node[colvertex] (c2a2) at (4.15, 1.87) {2};
  \node[colvertex] (c2s1) at (4.15, 0.88) {3};
  \node[colvertex] (c2b1) at (4.15, 0.08) {4};
  \node[colvertex] (c2z1) at (4.15,-1.00) {5};

  \foreach \c in {c2a1,c2a2,c2s1}{
    \draw[varianceedge] (c2r1) -- (\c);
  }
  \foreach \r in {c2r2,c2r3}{
    \draw[varianceedge] (\r) -- (c2s1);
    \draw[varianceedge] (\r) -- (c2b1);
  }

  \draw[matchingedge] (c2r1) -- (c2a1);
  \draw[matchingedge] (c2r2) -- (c2s1);
  \draw[matchingedge] (c2r3) -- (c2b1);

  \draw[leftgroupbrace] (1.02,2.19) -- (1.02,1.53)
    node[midway,left=4pt,grouplabel] {group 1};
  \draw[leftgroupbrace] (1.02,0.85) -- (1.02,-0.59)
    node[midway,left=4pt,grouplabel] {group 2};

  \draw[groupbrace] (4.57,2.86) -- (4.57,1.56)
    node[midway,right=4pt,grouplabel] {$\cV_1\setminus\cV_2$};
  \draw[groupbrace] (4.57,1.19) -- (4.57,0.57)
    node[midway,right=4pt,grouplabel] {$\cV_1\cap\cV_2$};
  \draw[groupbrace] (4.57,0.39) -- (4.57,-0.23)
    node[midway,right=4pt,grouplabel] {$\cV_2\setminus\cV_1$};
  \draw[groupbrace] (4.57,-0.69) -- (4.57,-1.31)
    node[midway,right=4pt,grouplabel] {$\cV_0$};

  \node[jslabel] at (4.15,-1.80)
    {$\bJ_S=\cV_0\cup(\cV_1\setminus\cV_2)$};
\end{scope}

\begin{scope}[shift={(0,0)}]
  \node[paneltitle] at (4.25,3.48) {\textbf{Case 3}\\[-1pt]
    $p_2-p_{12}>n_2,\qquad p_1\le n_1$};

  \node[rowvertex] (c3r1) at (1.45, 2.14) {1};
  \node[rowvertex] (c3r2) at (1.45, 1.36) {2};
  \node[rowvertex] (c3r3) at (1.45, 0.22) {3};

  \node[colvertex] (c3a1) at (4.15, 2.55) {1};
  \node[colvertex] (c3s1) at (4.15, 1.61) {2};
  \node[colvertex] (c3b1) at (4.15, 0.64) {3};
  \node[colvertex] (c3b2) at (4.15,-0.04) {4};
  \node[colvertex] (c3z1) at (4.15,-1.00) {5};

  \foreach \r in {c3r1,c3r2}{
    \draw[varianceedge] (\r) -- (c3a1);
    \draw[varianceedge] (\r) -- (c3s1);
  }
  \foreach \c in {c3s1,c3b1,c3b2}{
    \draw[varianceedge] (c3r3) -- (\c);
  }

  \draw[matchingedge] (c3r1) -- (c3a1);
  \draw[matchingedge] (c3r2) -- (c3s1);
  \draw[matchingedge] (c3r3) -- (c3b1);

  \draw[leftgroupbrace] (1.02,2.47) -- (1.02,1.03)
    node[midway,left=4pt,grouplabel] {group 1};
  \draw[leftgroupbrace] (1.02,0.55) -- (1.02,-0.11)
    node[midway,left=4pt,grouplabel] {group 2};

  \draw[groupbrace] (4.57,2.86) -- (4.57,2.24)
    node[midway,right=4pt,grouplabel] {$\cV_1\setminus\cV_2$};
  \draw[groupbrace] (4.57,1.92) -- (4.57,1.30)
    node[midway,right=4pt,grouplabel] {$\cV_1\cap\cV_2$};
  \draw[groupbrace] (4.57,0.95) -- (4.57,-0.35)
    node[midway,right=4pt,grouplabel] {$\cV_2\setminus\cV_1$};
  \draw[groupbrace] (4.57,-0.69) -- (4.57,-1.31)
    node[midway,right=4pt,grouplabel] {$\cV_0$};

  \node[jslabel] at (4.15,-1.80)
    {$\bJ_S=\cV_0\cup(\cV_2\setminus\cV_1)$};
\end{scope}

\begin{scope}[shift={(8.75,0)}]
  \node[paneltitle] at (4.25,3.48) {\textbf{Case 4}\\[-1pt]
    $\begin{gathered}
      p_1>n_1,\quad p_2>n_2,\\[-1pt]
      p_1+p_2-p_{12}>n_1+n_2
    \end{gathered}$};

  \node[rowvertex] (c4r1) at (1.45, 1.72) {1};
  \node[rowvertex] (c4r2) at (1.45, 0.28) {2};

  \node[colvertex] (c4a1) at (4.15, 2.15) {1};
  \node[colvertex] (c4s1) at (4.15, 1.15) {2};
  \node[colvertex] (c4b1) at (4.15, 0.15) {3};
  \node[colvertex] (c4z1) at (4.15,-1.00) {4};

  \draw[varianceedge] (c4r1) -- (c4a1);
  \draw[varianceedge] (c4r1) -- (c4s1);
  \draw[varianceedge] (c4r2) -- (c4s1);
  \draw[varianceedge] (c4r2) -- (c4b1);

  \draw[matchingedge] (c4r1) -- (c4a1);
  \draw[matchingedge] (c4r2) -- (c4b1);

  \draw[leftgroupbrace] (1.02,2.05) -- (1.02,1.39)
    node[midway,left=4pt,grouplabel] {group 1};
  \draw[leftgroupbrace] (1.02,0.61) -- (1.02,-0.05)
    node[midway,left=4pt,grouplabel] {group 2};

  \draw[groupbrace] (4.57,2.46) -- (4.57,1.84)
    node[midway,right=4pt,grouplabel] {$\cV_1\setminus\cV_2$};
  \draw[groupbrace] (4.57,1.46) -- (4.57,0.84)
    node[midway,right=4pt,grouplabel] {$\cV_1\cap\cV_2$};
  \draw[groupbrace] (4.57,0.46) -- (4.57,-0.16)
    node[midway,right=4pt,grouplabel] {$\cV_2\setminus\cV_1$};
  \draw[groupbrace] (4.57,-0.69) -- (4.57,-1.31)
    node[midway,right=4pt,grouplabel] {$\cV_0$};

  \node[jslabel] at (4.15,-1.80)
    {$\bJ_S=\cV_0\cup\cV_1\cup\cV_2$};
\end{scope}

\end{tikzpicture}
\endgroup%
    }
    \caption{
        Four regimes for the degenerate column set of the two-group
        variance profile.
        Gray and black vertices represent rows and columns, respectively,
        thin gray edges indicate nonzero entries of \(S\), and red edges
        show one maximum matching.
        Below each panel is \(\mathbf J_S\), the set of
        columns left unmatched by some maximum matching.
    }
    \label{fig:two-group-degenerate-columns}
\end{figure}

The complete matching calculation is given in \Cref{sec:two-group-alt}.
\end{example}

To connect $\bJ_S$ to the system~\eqref{eq:r-fixedpoint} we use a Hall-type
characterization. For a column set $J\subseteq\{1,\ldots,p\}$ write
$I_J(S) = \{i : S_{ij} > 0 \text{ for some } j\in J\}$ for its neighboring rows and
$S(J) = (S_{ij})_{i\in I_J(S),\, j\in J}$ for the induced submatrix.

\begin{definition}[Row-side strong Hall property] \label{def:row-shp}
A matrix $A\in\R^{n'\times p'}$ has the row-side strong Hall property if, for every
$1\le k\le n'$, each set of $k$ rows of $A$ has nonzero entries in at least $k+1$
distinct columns.
\end{definition}

\begin{lemma}[Matching characterization of $\bJ_S$] \label{lem:J-shp}
$\bJ_S$ is the unique maximal column set $J$ whose induced submatrix $S(J)$ has the
row-side strong Hall property.
\end{lemma}

{Lemma~\ref{lem:J-shp} shows that $\bJ_S$ does not depend on the choice of
maximum matching, but is determined by the Dulmage--Mendelsohn structure of
the variance graph.}
{S}tarting from any maximum matching, \(\bJ_S\) can be recovered as
the set of columns reachable by alternating paths from unmatched columns. 
{We use the row-side strong Hall formulation to analyze the fixed
point equation for \(r(0)\).} The proof is deferred to \Cref{app:bias-support}.

{The next result connects the matching structure to the bias.}
The
estimator is biased exactly along the degenerate coordinates, in the vanishing-penalty limit.

\begin{theorem}[Bias support] \label{thm:bias-support}
For every $n$ and $p$, the support of $\br(0)$ is $\bJ_S$; that is,
\begin{align*}
    \bJ_S \;=\; \{\, 1\le l\le p \;|\; r_l(0) > 0 \,\}\;.
\end{align*}
\end{theorem}

\Cref{thm:bias-support} identifies the hard-edge atom in the deterministic equivalent with a structural-rank defect of the support graph. In particular, the coordinates with \(r_l(0)>0\) are those that can be left unmatched by some maximum matching. 
Thus the hard-edge mass behind the deterministic bias equivalent of \Cref{thm:master} is confined to \(\bJ_S\): coordinates outside \(\bJ_S\) are matched in every maximum matching and carry no hard-edge mass, while coordinates in \(\bJ_S\) are precisely the directions along which the vanishing-ridge estimator may retain bias. 
{In \Cref{ex:two-group-overlapping}, this means that the bias lives on \(\cV_0\) in all regimes, and, depending on the regime, additionally on \(\cV_1\setminus\cV_2\), \(\cV_2\setminus\cV_1\), or all of \(\cV_1\cup\cV_2\). The proof is deferred to \Cref{app:bias-support}.}

\subsection{The switching condition and the location of the peaks} \label{subsec:switching}

The bias support $\bJ_S$ records \emph{which} coordinates are degenerate; the variance
records \emph{when} the degeneracy is critical, i.e.\ when small eigenvalues pile up at
the hard edge and $\partial\br(0)$ diverges. After deleting the
  degenerate coordinates and the rows that feed them, what remains is the residual
  design. Write
$\bI_S \coloneqq I_{\bJ_S}(S) = \{i : S_{ij} > 0 \text{ for some } j\in\bJ_S\}$ and
$\bJ_S^c, \bI_S^c$ for the complements.

\begin{definition}[Residual profile] \label{def:residual}
The residual profile is the submatrix $\bS^{\rm res} \coloneqq (S_{ij})_{i\notin\bI_S,\, j\notin\bJ_S}$,
with variance graph $\cG_{\bS^{\rm res}}$ defined as in \Cref{def:var-graph}.
\end{definition}

\begin{definition}[Switching condition] \label{def:switching}
The switching condition holds if some row $i\notin\bI_S$ is matched by \emph{every}
maximum matching of $\cG_{\bS^{\rm res}}$.
\end{definition}

At the switching configuration there is an unavoidable
square subdesign forcing the smallest nonzero eigenvalues of $X$ to the edge  so
that $\bJ_S$ is on the verge of growing and the variance is about to blow up. As with
the bias, the condition has a Hall-type reading.

\begin{definition}[Column-side strong Hall property] \label{def:col-shp}
A matrix $A\in\R^{n'\times p'}$ has the column-side strong Hall property if, for every
$1\le q\le p'$, each set of $q$ columns of $A$ has nonzero entries in at least $q+1$
distinct rows.
\end{definition}

\begin{lemma}[Hall form of switching] \label{lem:switch-shp}
The switching condition is equivalent to the failure of the column-side strong Hall
property for $\bS^{\rm res}$.
\end{lemma}
The proof is deferred to \Cref{app:var-switch}.  The next theorem shows that $\partial\br(0)$ is finite
exactly when the residual profile satisfies the column-side strong Hall
property.

\begin{theorem}[Variance divergence and the residual fixed point] \label{thm:var-switch}
The vector $\partial\br(0)$ is finite, i.e.\ $\partial\br(0)\in\R^p$, if and only if the
switching condition fails. When it does, the residual coordinates
$v_k^0 = \partial_\lambda r_k(0)$, $k\in\bJ_S^c$, are the unique positive solution of
\begin{align*}
    1 \;=\; \msum_{i\in\bI_S^c}\mfrac{S_{ik}\, v_k^0}{1 + \sum_{k'\in\bJ_S^c} S_{ik'} v_{k'}^0}\;,
    \qquad k\in\bJ_S^c\;,
    \tagaligneq \label{eq:partial-r-residual}
\end{align*}
and the degenerate coordinates $u_j^0 = \partial_\lambda r_j(0)$, $j\in\bJ_S$, solve the
linear system obtained by linearizing~\eqref{eq:r-fixedpoint} about $\br(0)$,
\begin{align*}
    \mfrac{u_j^0}{r_j(0)}
    \;=\;
    \msum_{i\in\bI_S}\mfrac{S_{ij}\, r_j(0)}{\big(\sum_{j'\in\bJ_S} S_{ij'} r_{j'}(0)\big)^2}
    \Big(1 + \msum_{j'\in\bJ_S} S_{ij'} u_{j'}^0 + \msum_{k\in\bJ_S^c} S_{ik} v_k^0\Big)\;,
    \quad j\in\bJ_S\;.
    \tagaligneq \label{eq:partial-r-support}
\end{align*}
\end{theorem}

The proof is deferred to \Cref{app:var-switch}. Combining \Cref{thm:bias-support}, \Cref{lem:switch-shp} and \Cref{thm:var-switch} gives the structural rule announced above. The bias is supported on the columns left unmatched by some maximum matching of \(\cG_S\). Meanwhile variance blow-up occurs exactly when the residual graph, obtained after deleting \(\bJ_S\) and its neighboring rows, loses column-side strong Hall slack. Both are determined by the Dulmage--Mendelsohn structure of the variance graph. Along a path of models, each switching configuration encountered by the residual graph gives one candidate peak of the ridgeless risk.

\begin{example}[Multiple descent from two groups] \label{ex:two-group-peaks}

Return to the two-group with overlap profile of \Cref{ex:two-group-overlapping}, with
\(n_1=n_2=n/2\). If $\bJ_S=\cV_0$, the residual graph is the full graph on
$\cV_1\cup\cV_2$. The column-side strong Hall property fails when
$|\cV_1\setminus\cV_2|=n/2$, $|\cV_2\setminus\cV_1|=n/2$, or
$|\cV_1\cup\cV_2|=n$.
If $\bJ_S=\cV_0\cup(\cV_1\setminus\cV_2)$, removing $\bJ_S$ and its neighboring rows leaves the complete bipartite graph between the second group and $\cV_2$, which satisfies the switching condition exactly when $p_2=n/2$. Similarly, if $\bJ_S=\cV_0\cup(\cV_2\setminus\cV_1)$, the switching condition holds exactly when $p_1=n/2$. If $\bJ_S=\cV_0\cup\cV_1\cup\cV_2$, the residual graph is empty, so the switching condition does not hold.

Therefore, the switching condition holds in the following five cases:
\begin{align*}
    \textnormal{(i)}   &&& p_1-p_{12}=n/2,
        & &\text{with } p_2\le n/2;\\
    \textnormal{(ii)}  &&& p_2-p_{12}=n/2,
        & &\text{with } p_1\le n/2;\\
    \textnormal{(iii)} &&& p_1+p_2-p_{12}=n,
        & &\text{with } p_1-p_{12}\le n/2
            \text{ and } p_2-p_{12}\le n/2;\\
    \textnormal{(iv)}  &&& p_2=n/2,
        & &\text{with } p_1-p_{12}>n/2;\\
    \textnormal{(v)}   &&& p_1=n/2,
        & &\text{with } p_2-p_{12}>n/2.
\end{align*}
The complete matching calculation is given in \Cref{sec:two-group-alt}.
\end{example}

\subsection{Risk peaks in the limit} 
The results so far in this section hold for finite $n$ and $p$. We allow
the variance profile to change with $n$ as long as it satisfies
\Cref{asst:flat}. To study when risk peaks
occur asymptotically, we consider a sequence of variance profiles obtained
by replicating a fixed prototype: Let $\widetilde S\in\R_+^{m\times k}$, where $m$ and $k$ are fixed. For
$t\in\N$, set
\begin{equation}\label{eqn:prototype_definition}
    n_t\coloneqq tm,
    \qquad
    p_t\coloneqq tk,
    \qquad
    S^{(t)}
    \coloneqq
    \frac1t\,\widetilde S\otimes(\bone_t\bone_t^\top)
    \in\R_+^{n_t\times p_t}.
\end{equation}
Then $p_t/n_t=k/m$ for every $t$, and the zero pattern of $S^{(t)}$ is determined by
that of $\widetilde S$.

For $\lambda>0$, let
$\widetilde{\br}(\lambda)=(\widetilde r_b(\lambda))_{b=1}^k$
denote the fixed-point solution in \eqref{eq:r-fixedpoint} corresponding
to $\widetilde S$. Also, we write
$\partial_\lambda\widetilde r_b(0)
\coloneqq\lim_{\lambda\downarrow0}\partial_\lambda\widetilde r_b(\lambda)$.
\begin{corollary}[Variance along the prototype sequence]
\label{cor:prototype-variance}
Suppose $\widetilde S$ satisfies \Cref{asst:flat}, let $S^{(t)}$ be defined by
\eqref{eqn:prototype_definition}, and let $X^{(t)}$ follow \Cref{mod:vp} with profile
$S^{(t)}$. If $\lambda_t\downarrow0$ is given as in \Cref{thm:master} with $n=n_t$, then
\begin{align*}
    \cR_V^{\lambda_t}(X^{(t)})
    -\frac{\sigma_\epsilon^2}{m}
      \sum_{b=1}^k
      \left(\sum_{a=1}^m\widetilde S_{ab}\right)
      \partial_\lambda\widetilde r_b(\lambda_t)
    \xrightarrow{\mathrm{a.s.}}0.
\end{align*}
Consequently, if $\widetilde S$ satisfies the switching condition, then
\begin{align*}
    \cR_V^{\lambda_t}(X^{(t)})
    \xrightarrow{\mathrm{a.s.}}+\infty.
\end{align*}
Otherwise,
\begin{align*}
    \cR_V^{\lambda_t}(X^{(t)})
    \xrightarrow{\mathrm{a.s.}}
    \frac{\sigma_\epsilon^2}{m}
    \sum_{b=1}^k
    \left(\sum_{a=1}^m\widetilde S_{ab}\right)
    \partial_\lambda\widetilde r_b(0)
    <\infty.
\end{align*}
\end{corollary}
The proof is deferred to \Cref{app:prototype-variance}.
\begin{remark}\label{rem:prototype-general}
For the sequence in \eqref{eqn:prototype_definition}, the asymptotic variance diverges if
and only if $\widetilde S$ satisfies the switching condition. We expect variance to diverge
for more general sequences with $n,p\to\infty$ and $p/n\to\gamma\in(0,\infty)$ that
approach a switching configuration, even if switching never holds exactly at finite $n,p$.
Proving this requires an additional stability analysis of \eqref{eq:r-fixedpoint}. We leave
this for future work. \end{remark}

\section{Perspectives on positive-definite designs} \label{sec:non:sim:diag:but:pos}

The designs of \Cref{sec:setup-main} share an eigenbasis, and their multiple descent comes from
rank-deficient covariances. This section instead considers positive-definite covariances that are
full rank and need not share an eigenbasis. We first set up the
positive-definite variants of the two designs and explain why the reduction to a variance
profile fails.
We then present a conjecture, supported by the simulations of
\Cref{subsec:sim-pd}: no matter how different the covariances are, the risk curve keeps the single
classical peak at $\gamma = 1$.

\begingroup
\renewcommand{\themodel}{1$^\star$}%
\begin{model}[Heterogeneous data, positive definite] \label{mod:het:pd}
$X_1, \ldots, X_n$ are independent $\R^p$ Gaussian vectors with
{
    $
    X_i \;=\; \Sigma_i^{1/2} Z_i$ and $Z_i \overset{\rm i.i.d.}{\sim} \cN(0, I_p)
    $
}
as in \Cref{mod:het}, except that each $\Sigma_i \in \R^{p\times p}$ is positive definite and
no simultaneous diagonalizability is required of $(\Sigma_i)_{i\le n}$.
\end{model}
\endgroup

\begingroup
\renewcommand{\themodel}{2$^\star$}%
\begin{model}[Dependent data, positive definite] \label{mod:dep:pd}
$X = \sum_{m=1}^M \big(\tilde\Sigma^{(m)}\big)^{1/2} Z^{(m)} \big(\Sigma^{(m)}\big)^{1/2}$ as in
\Cref{mod:dep}, except that each $\tilde\Sigma^{(m)}$ is diagonal with at least one positive
entry, each $\Sigma^{(m)}$ is positive definite, and no simultaneous diagonalizability is
required of $(\Sigma^{(m)})_{m\le M}$.
\end{model}
\endgroup

\Cref{mod:het:pd} is natural when data are pooled from heterogeneous sources --- records from
several sites, each full rank but with its own correlation structure. It suffices to treat
\Cref{mod:het:pd}: the regrouping of \Cref{lem:model:equivalence}(ii) in \Cref{appendix:setup} turns
\Cref{mod:dep:pd} into \Cref{mod:het:pd} with row covariances
$\Sigma_i = \sum_{m=1}^M \tilde D^{(m)}_{ii}\,\Sigma^{(m)}$, which are positive definite.

{Unlike \Cref{subsec:reduction}, \Cref{mod:het:pd} cannot be reduced to a variance profile, due to the lack of a shared eigenbasis. Meanwhile, positive definiteness prevents additional degeneracies from being created from the structure of the data alone. To make this precise, s}ince $n$ and $p$ are growing, we first impose that all covariance matrices
remain positive definite in the limit:

\begin{assumption}[Positive-definite covariances] \label{asst:hetero:eigen:bound} For the
$\Sigma_i$'s given in \Cref{mod:het:pd}, there exist some universal constants
$\alpha_{\rm min}, \alpha_{\rm max}$ with $0 < \alpha_{\rm min} \leq \alpha_{\rm max}$ such that
\begin{align*}
     \alpha_{\rm min} 
     \;\leq\; 
      \minf_{1 \leq i \leq n} \, \lambda_{\rm min}(\Sigma_i) 
      \;\leq\; 
      \msup_{1 \leq i \leq n} \, \lambda_{\rm max}(\Sigma_i)
    \;\leq\; 
    \alpha_{\rm max}\;,
\end{align*}
where $\lambda_{\rm min}(\argdot)$ and $\lambda_{\rm max}(\argdot)$ denote the smallest and largest
eigenvalues.
\end{assumption}

{Recall} the risk components $\cR_B^{\lambda}(X)$ and $\cR_V^{\lambda}(X)$ of
\Cref{lem:bias-var-decomp}. The bias term is bounded uniformly in
$\lambda$, with $\cR_B^{\lambda}(X) \le
\|\bar\Sigma\|_{op}\|\beta\|^2$ (\Cref{lem:bias:bounded} in \Cref{appendix:setup}), so the location of any risk peak is decided by the variance term. Our first lemma
shows that we can further ignore $\sigma^2_\epsilon$ and $\bar\Sigma$ and focus on the quantity
\begin{align*}
    \tilde \cR_V^{\lambda}(X)
    \;\coloneqq\;
    \mfrac{1}{n} \, \Tr\big( W_n (W_n + \lambda I_p)^{-2} \big)\;,
    \qquad
    W_n \;=\; \mfrac{1}{n} \msum_{i \leq n} \Sigma_i^{1/2} Z_i Z_i^\top \Sigma_i^{1/2}\;.
\end{align*}

\begin{lemma} \label{lem:var:bounds} Under \Cref{mod:het:pd} and \Cref{asst:hetero:eigen:bound},
for every $\lambda>0$, the following holds: 
\begin{align*}
    \sigma^2_\epsilon \alpha_{\rm min} \tilde \cR_V^{\lambda}(X) 
    \;\leq\; 
    \cR_V^{\lambda}(X)
    \;\leq\; 
    \sigma^2_\epsilon \alpha_{\rm max} \tilde \cR_V^{\lambda}(X)
    \qquad 
    \text{ almost surely\;.}
\end{align*}
\end{lemma}

\noindent The proof is a one-line trace bound and is deferred to \Cref{appendix:setup}.

Notice that we can re-express $\tilde \cR_V^{\lambda}(X)$ through the resolvent of the sample
covariance $W_n=\frac1n X^\top X$, namely $G_{W_n}(z) \coloneqq (W_n - z I_p)^{-1}$ (distinct
from the {Hermitialized} resolvent $G(z)$ used in the appendix), evaluated at the fixed negative
point $z=-\lambda$:
\begin{align*}
    \tilde \cR_V^{\lambda}(X)
    \;=\;
    \mfrac{1}{n} \Tr\big( W_n \, (W_n + \lambda I_p)^{-2} \big)
    \;=\;
    \partial_z \Big[ \mfrac{z}{n} \Tr\big( G_{W_n}(z) \big) \Big]\Big|_{z=-\lambda}
    \;.
\end{align*}
Since this only concerns the trace of the resolvent and its derivative at $z=-\lambda$, the
study of the location of the risk-curve peaks is reduced to the analysis of the \emph{global} law
of the matrix $\frac{1}{n} X^\top X$, which is substantially simpler than a local law analysis.
Indeed, we expect that a full characterization of $\tilde \cR_V^{\lambda}(X)$, and subsequently
$\cR_V^{\lambda}(X)$, can be established by taking global law results for Gaussian matrices with
independent but non-identically distributed columns, and tracking the dependence on the spectral
parameter $z$ explicitly as $z=-\lambda \equiv - \lambda_n\uparrow0$, as we do for the local law in
\Cref{appendix:local:law}. We do not pursue this in rigorous technical detail, but remark that
a recent work \citep{louart2021spectral} establishes the approximation of $G_{W_n}(z)$ by a
matrix $\Gamma_z$ that solves the deterministic system of equations 
\begin{align*}
    \mfrac{1}{\big( \Gamma_z\big)_i} \;=\; - z + \mfrac{1}{n} \Tr\Big( \Sigma_i \Big(\mfrac{1}{n} \msum_{j \leq n} (\Gamma_z)_j \, \Sigma_j + I_p  \Big)^{-1} \Big) 
    \qquad 
    \text{ for } 1 \leq i \leq n\;.
    \tagaligneq \label{eq:hetero:det:sys}
\end{align*}
The stability of the system \eqref{eq:hetero:det:sys} of equations does not change, whether we
take all $\Sigma_j$'s to be $I_p$ or we take general $\Sigma_j$'s under
\Cref{asst:hetero:eigen:bound}. This leads us to conjecture that under \Cref{mod:het:pd} (and,
by the reduction above, under \Cref{mod:dep:pd}), the location of the peak is exactly the same
as in the isotropic case, i.e.~when
\begin{align*}
    \gamma \;=\; \lim \mfrac{p}{n} \;\text{ equals } \; 1\;.
\end{align*}
We note that this is only a theoretical conjecture: because \citep{louart2021spectral} does not
make the dependence on $z$ explicit, it does not by itself yield a rigorous statement about
$\tilde \cR_V^{\lambda_n}(X)$ as $\lambda_n\downarrow0$. The simulations in \Cref{subsec:sim-pd}
are consistent with the conjecture: with positive-definite, non-commuting covariances the risk
still peaks only at $\gamma=1$.
\section{Simulations} \label{sec:simulations}

{We test our theory in three synthetic and one real designs.} \Cref{subsec:sim-latent} 
{affirms that} a diagonalizable, low-rank heterogeneous design produces several risk peaks, 
{with locations exactly predicted by our theory.}
\Cref{subsec:sim-pd} {checks the negative conjecture in \Cref{sec:non:sim:diag:but:pos}}: {for positive definite covariances}, arbitrary heterogeneity { and non-diagonalizability} leave only the classical
threshold at $p/n = 1$. \Cref{subsec:sim-aug} verifies our theory for 
{a dependent design generated by a data
augmentation.}
{ \Cref{subsec:sim-gan} concludes with multiple descent under a real design, gene expression data pooled
with samples from a generative adversarial network.}
Throughout, $p$ is the ambient dimension, the empirical aspect ratio is $p/n$, and we
estimate the prediction risk $\cR^\lambda(X)$ of the (near-)ridgeless estimator
\(\hat\beta_\lambda\) by Monte Carlo over independent design draws and an independent test
point from the training mixture~\eqref{eq:test-gen-main}. The signal is flat in the
eigenbasis, $\beta = \bone_p/\sqrt p$, so that no coordinate direction is privileged.

\subsection{Multiple descent from latent covariance support} \label{subsec:sim-latent}

The first experiment is a dense, latent-subspace realization of the two-group
profile of \Cref{ex:two-group-overlapping}. There are two groups of observations, each of size
$n/2$; conditional on the group, the rows are independent centered Gaussian vectors
in $\R^p$. Fix an orthogonal matrix $U_p \in \R^{p\times p}$ whose columns are the
latent factor directions, and choose mutually orthogonal latent subspaces
$A, B, C, O \subset \R^p$, where $A$ is specific  to group~1, $B$ is specific  to
group~2, $C$ is common to both, and $O$ is inactive. Group~1 draws
$X_i \sim \cN(0, \Sigma_1)$ and group~2 draws $X_i \sim \cN(0, \Sigma_2)$, with
\begin{align*}
    \operatorname{range}(\Sigma_1) \;=\; E_1 \;\coloneqq\; A \oplus C\;,
    \qquad
    \operatorname{range}(\Sigma_2) \;=\; E_2 \;\coloneqq\; B \oplus C\;,
\end{align*}
and all nonzero eigenvalues of $\Sigma_1, \Sigma_2$ set to $1$. The signal is flat
in the latent basis, $\beta = U_p\,\bone_p/\sqrt p$.

{
Because $\Sigma_1$ and $\Sigma_2$ share the eigenbasis $U_p$ and are only positive
\emph{semi}-definite, this is exactly the simultaneously diagonalizable heterogeneous
design of \Cref{mod:het}. By \Cref{cor:practical-risk} its risk is that of a
variance profile whose variance graph joins each group to the latent coordinates it
activates; the bias support is then given by \Cref{thm:bias-support} and the
locations of the variance peaks by the switching condition of \Cref{thm:var-switch}.
The inactive subspace $O$ leaves its coordinates isolated in that graph, so the profile is
decomposable and \Cref{cor:practical-risk} is applied componentwise (\Cref{rem:decomposable}).
}

We fix
\begin{align*}
    \dim A \;=\; 0.45\,p\;, \qquad \dim C \;=\; 0.10\,p\;,
    \qquad\text{and sweep}\qquad
    \dim B \;=\; \rho\, p\;,
\end{align*}
so that $\dim E_1 = 0.55\,p$ is held fixed while $\dim E_2 = (\rho + 0.10)\,p$ grows.
The theory then predicts one stationary peak from the block $A$ specific to group~1, at
$p/n = 1/(2\cdot 0.45) = 10/9$, together with a moving peak from the group-2 span at
$p/n = 1/\big(2(\rho + 0.10)\big)$. The two coincide at $\rho = 0.35$ and, beyond
it, merge into a single 
 peak at $p/n = 1/(0.55 + \rho)$, where the
combined active span becomes square with all $n$ samples. In \Cref{fig:sim-latent},
the dashed vertical lines mark these predicted switching locations; the
empirical risk curves exhibit two well-separated peaks that merge exactly as the rule
dictates.

\begin{figure}[t]
    \centering
    \includegraphics[width=0.90\linewidth,trim=0 7bp 0 6bp,clip]{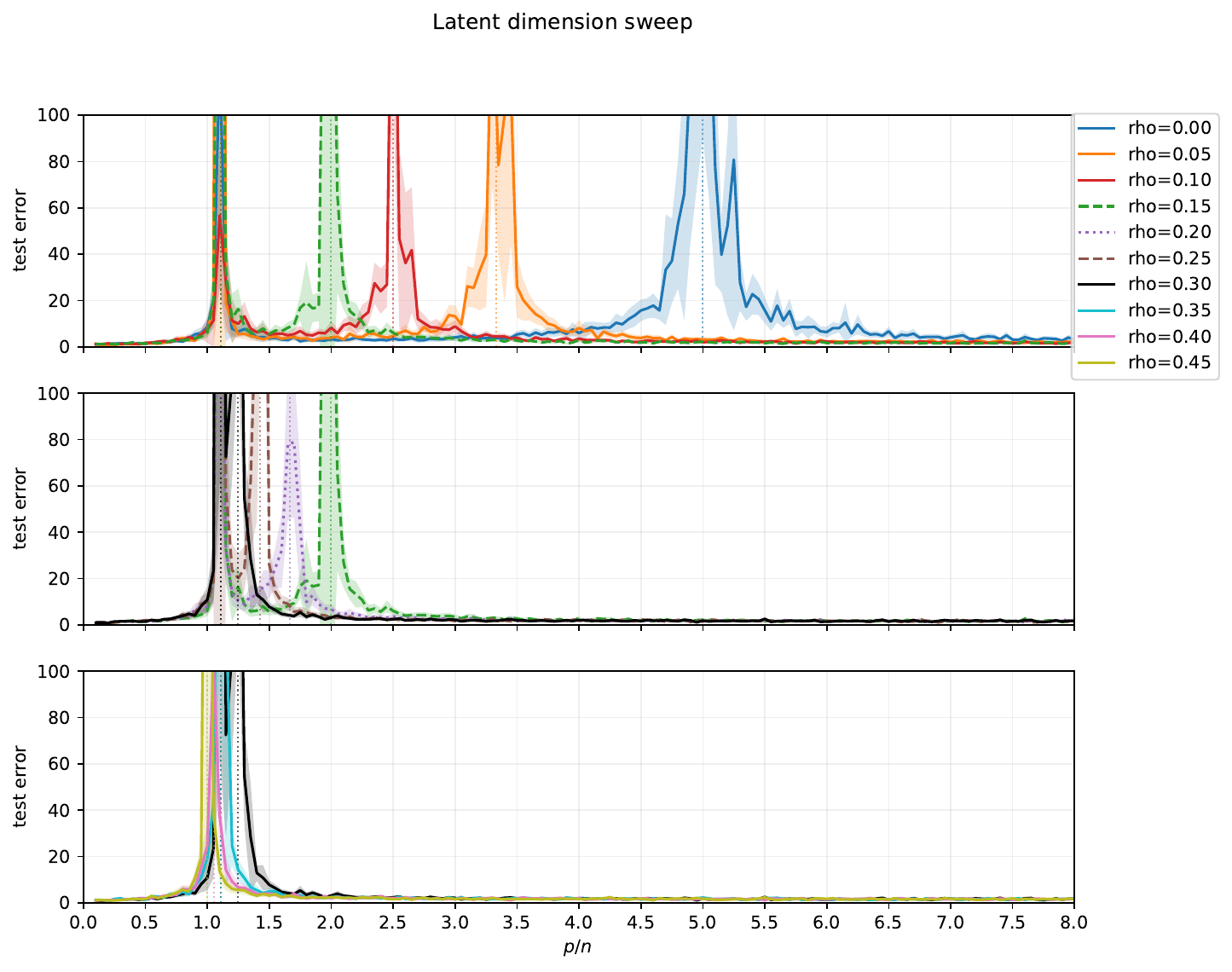}
    \caption{Multiple descent from latent covariance support (\Cref{mod:het}). Two
    groups of $n/2$ observations with low-rank covariances sharing a common latent
    basis. Prediction risk against $p/n$; each colour is one value of the sweep
    parameter $\rho$. Dashed vertical lines mark the switching thresholds predicted by
    \Cref{thm:var-switch}.
    }
    \label{fig:sim-latent}
\end{figure}

\paragraph*{Beyond the Gaussian design.} The reduction to a variance profile in
\Cref{subsec:reduction} is Gaussian, but \Cref{thm:master} itself is not: as noted in
\Cref{rem:non-gaussian}, it holds for any design with independent entries and a matching
variance profile. To probe how far this reaches, we repeat the above experiment of
\Cref{fig:sim-latent} verbatim---identical latent subspaces, group sizes and signal---but
replace the Gaussian entries of the latent coordinates by i.i.d.\ draws from a standardized
Student-$t_3$ law (variance one, obtained as $t_3/\sqrt3$). This law has finite variance,
so the variance profile is unchanged{. However,} its third and higher moments are infinite, so it
falls strictly outside the bounded-moment hypothesis of the local law of \citep{alt2017local}
on which \Cref{thm:master} rests. \Cref{fig:sim-latent-nongauss} shows that the risk curves
are essentially indistinguishable from the Gaussian ones.

\begin{figure}[t]
    \centering
    \includegraphics[width=0.90\linewidth,trim=0 7bp 0 6bp,clip]{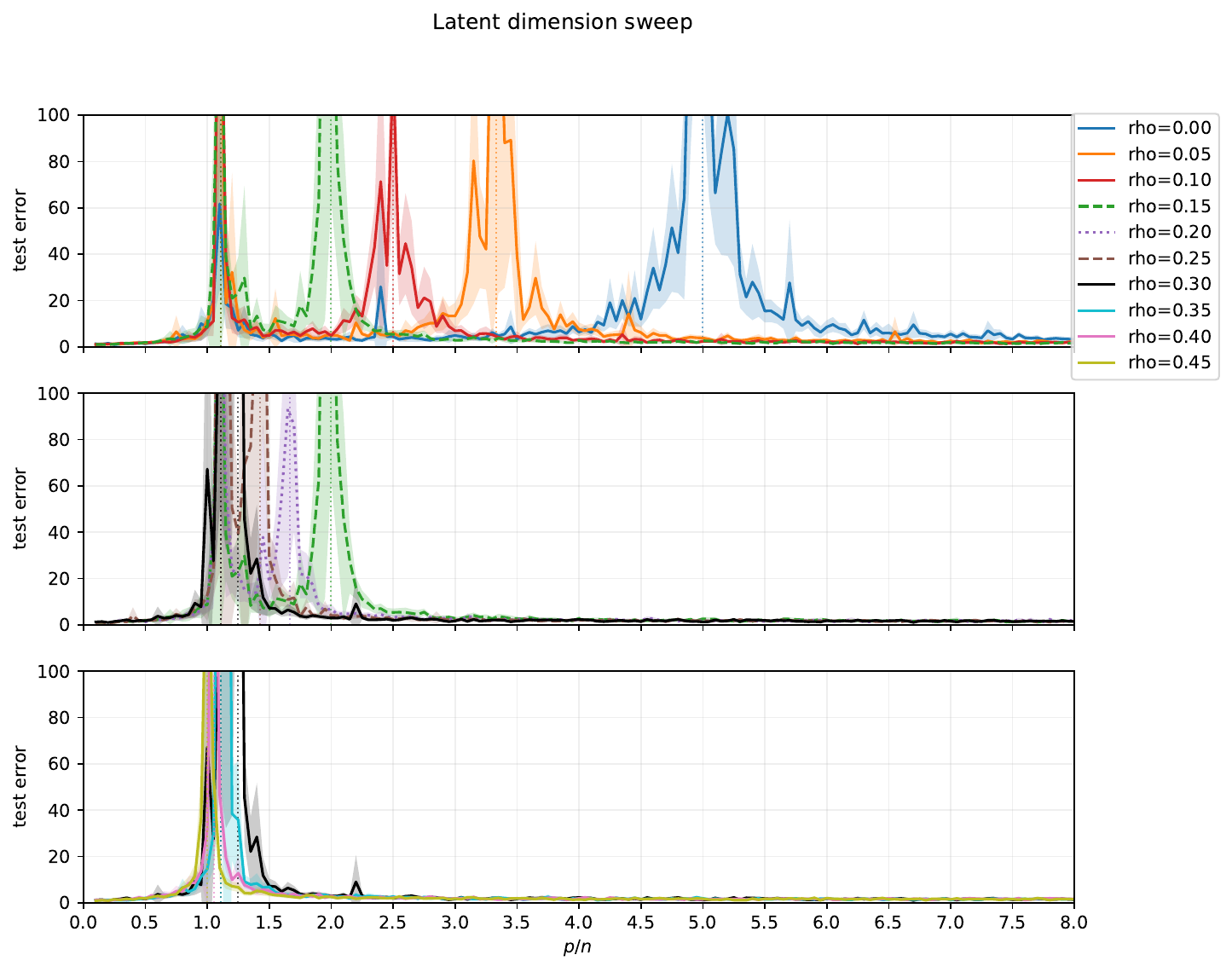}
    \caption{Non-Gaussian universality of multiple descent. The same sweep of
    \Cref{fig:sim-latent} (\Cref{mod:het}) with i.i.d.\ \emph{standardized
    Student-$t_3$} entries replacing Gaussian ones; the variance profile, group
    sizes ($n/2$ each), latent subspaces and signal are identical. Prediction risk against
    $p/n$, one colour per $\rho$; dashed vertical lines mark the same switching
    thresholds of \Cref{thm:var-switch} as in \Cref{fig:sim-latent}. Although $t_3$ has
    infinite third moment, violating the bounded-moment hypothesis of
    \citep{alt2017local}, the peak locations and merging match the Gaussian case,
    consistent with \Cref{rem:non-gaussian}.}
    \label{fig:sim-latent-nongauss}
\end{figure}

\subsection{Positive-definite heterogeneity has only the classical peak}
\label{subsec:sim-pd}

The second experiment is the control predicted by \Cref{sec:non:sim:diag:but:pos}:
when every covariance is positive definite, heterogeneity does not add peaks.
 We use three groups of observations with a
fixed total sample size $n = 150$, split evenly as $n_1 = n_2 = n_3 = 50$. For an
even ambient dimension $p$ write $q = p/2$ and set the anisotropic profile and the
block rotation
\begin{align*}
    D \;=\;
    \begin{psmallmatrix}
        I_q & 0 \\ 0 & 0.1 I_q
    \end{psmallmatrix}\;,
    \qquad
    Q(\theta) \;=\;
    \begin{psmallmatrix}
        \cos\theta\, I_q & -\sin\theta\, I_q \\
        \sin\theta\, I_q & \cos\theta\, I_q
    \end{psmallmatrix}\;.
\end{align*}
Group~$g$ has covariance $\Sigma_g = Q(\theta_g)\, D\, Q(\theta_g)^\top$ with the
three angles fixed at $\theta_1 = 0$, $\theta_2 = \pi/12$ and $\theta_3 = \pi/6$.
Each $\Sigma_g$ is positive definite, with eigenvalues $1$ and $0.1$ of common
multiplicity $q$, so this is a full-rank instance of \Cref{mod:het:pd}. The three
covariances are \emph{not} simultaneously diagonalizable: since the eigenvalues of
$D$ are unequal, rotating $D$ by angle differences outside $(\pi/2)\Z$ produces
covariances that do not commute. The design is therefore an instance of
\Cref{mod:het:pd}, strictly outside the
diagonalizable regime of \Cref{subsec:sim-latent}, and the bounded-spectrum
condition of \Cref{asst:hetero:eigen:bound} holds with $\alpha_{\min} = 0.1$ and
$\alpha_{\max} = 1$.

\Cref{fig:sim-pd} plots the empirical risk against $p/n$. As anticipated in
\Cref{sec:non:sim:diag:but:pos}, the curve shows a single interpolation peak at
$p/n = 1$ (dashed line), in the same location as the isotropic baseline of
\Cref{ex:iso}. The rotations change the orientation of the covariance ellipsoids but
neither the number nor the location of the risk singularities: positive-definite
heterogeneity alone does not create additional descents.

\begin{figure}[t]
    \centering
    \includegraphics[width=0.95\linewidth,trim=0 7bp 0 6bp,clip]{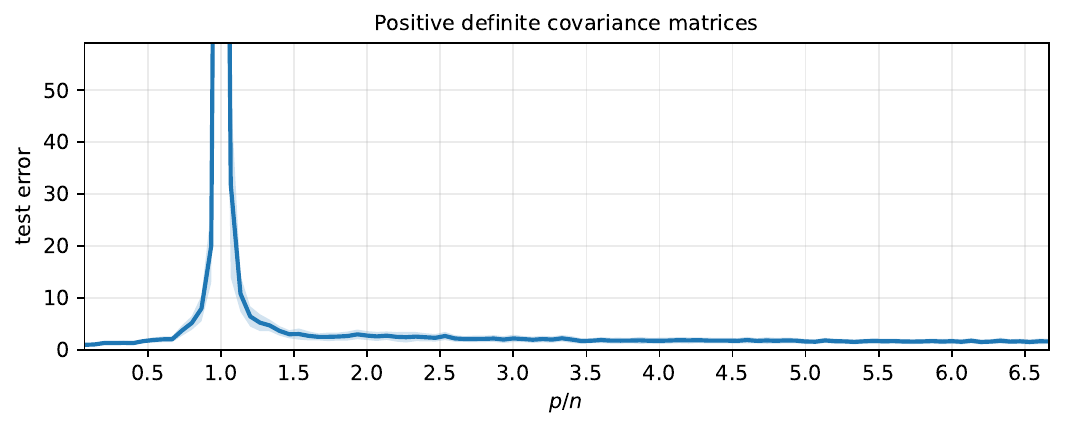}
    \caption{
    Positive-definite heterogeneity has only the classical peak
    (\Cref{mod:het:pd}). Three full-rank groups with anisotropic covariances
    rotated by angles $0, \pi/12, \pi/6$ and do not commute, with
    $n = 150$. Prediction risk against $p/n$ has a single peak at $p/n = 1$,
     matching the isotropic threshold of \Cref{ex:iso} regardless of
    orientation.}
    \label{fig:sim-pd}
\end{figure}

\subsection{Dependent designs generated by geometric augmentation}
\label{subsec:sim-aug}

The {next} experiment leaves the independent-row setting: the design is made dependent
\emph{across} observations by a data augmentation, in the spirit of
\citep{huang2022data}. We start from a base sample $Y_i = V_i^\top\beta + \epsilon_i$
with isotropic covariates $V_i \sim \cN(0, I_p)$, noise
$\epsilon_i \sim \cN(0, \sigma_\epsilon^2)$, and flat signal $\beta = \bone_p/\sqrt p$,
and partition the $p$ coordinates into $p/b$ consecutive blocks of size $b$. From each
base sample we form $k$ augmented copies; copy $j$ permutes the coordinates
\emph{within each block} by an independent random permutation $\pi_{ij}$, and attaches
a freshly drawn response,
\begin{align*}
    \big(\pi_{ij} V_i,\; \tau_{ij} Y_i\big)\;,
    \qquad
    \tau_{ij} Y_i \;=\; (\pi_{ij} V_i)^\top \beta + \tilde\epsilon_{ij}\;,
    \qquad \tilde\epsilon_{ij} \overset{\rm i.i.d.}{\sim} \cN(0, \sigma_\epsilon^2)\;.
\end{align*}
Because $\beta$ is flat, permuting coordinates leaves the mean response unchanged,
$(\pi_{ij} V_i)^\top\beta = V_i^\top\beta$, so the augmentation acts on the design
alone and the fresh noise is its only effect on the labels. We fit the (near-)ridgeless estimator \(\hat\beta_\lambda\) on the $nk$ augmented pairs and measure its held-out test error---the
mean squared empirical risk on an independent test set drawn from the same
distribution---as a function of $\gamma = p/n$; here $n = 100$ training points,
$k = 5$ augmented copies, and the error is averaged over $1000$ held-out points.

The $k$ copies of a sample share the same $V_i$, so the $nk$ rows are dependent, {
and each
augmented row is marginally Gaussian. However, the augmented rows are generally not jointly
Gaussian, so \Cref{mod:dep} does not apply directly. Nevertheless, this is a more realistic augmentation, and we still observe several thresholds rather than one. We conjecture that the persistence of multiple descent beyond the jointly Gaussian setting is because of Gaussian universality.
} \Cref{fig:intro-aug} bears this out: the unaugmented
baseline shows the single classical peak, whereas the augmented designs with $b = 2$
and $b = 3$ each display additional peaks, so that a single augmented design already
exhibits multiple descent. Refreshing the label with independent noise per
copy---rather than reusing the response, as in the oracle relabeling of
\citep[Figure~13]{huang2022data}---brings out the extra peak.

{
\subsection{Multiple descent from pooling real and generated data}
\label{subsec:sim-gan}

The previous experiments used synthetic covariances; this one uses real data, pooled with generated
samples. The real observations are gene expression profiles from The Cancer Genome Atlas
\citep{weinstein2013cancer}, as preprocessed by \citep{lacan2023gan}: $9920$ deduplicated tumor and normal-tissue
samples with expression levels of $20531$ genes. We train a WGAN-GP using the architecture and
training setup of~\citep{lacan2023gan} on these $9920$ observations, and then generate $9920$
synthetic observations. 
Since the final hidden layer of the generator has dimension $1024$ and the
output layer is linear \citep{lacan2023gan}, the covariance of the centered generated observations has rank at most
$1024$. This is consistent with the more concentrated covariance spectrum of the generated samples, as reported in
\citep{lacan2023gan}. Approximately all of the empirical variance of their WGAN-GP samples is
explained by the leading $500$ principal components, whereas the leading $2000$ principal
components explain only about $90\%$ of the variance in the real data. 

The two-group example (\Cref{ex:two-group-overlapping,ex:two-group-peaks}) gives some guidance for choosing the number of retained principal components.
Indeed, in the special case $V_2\subset V_1$, so that $p_2=p_{12}$, the two switching
configurations in \Cref{ex:two-group-peaks} occur at $p_1-p_{2}=n/2$ and $p_2=n/2$. Thus, if
$p_1-p_{2}$ is larger than the number of observations available from the first group, the
first configuration cannot be reached. Similarly, here the generated
observations have covariance rank at most $1024$, while only $9920$ real observations are
available, which motivates reducing the dimension by PCA rather than using all $20531$.
We take $p=6144$ and choose the {top $p$ principal components of the real data once}, and use the same {covariates} for the real and synthetic groups and every value
of $n$. This avoids selecting covariates by differences between the real and generated
covariance structures or by the resulting test-error curve.

We standardize the real and generated samples separately, subtracting the empirical mean and dividing
by the empirical standard deviation of each selected covariates within each group. We then combine the two groups
into a pool of $19840$ observations. With a flat signal $\beta=\bone_p/\sqrt{p}$, we generate
$y=X\beta+\epsilon$, and sweep $\gamma=p/n$ by drawing $n$ training observations at random
from this pool. The reported test errors are averaged over $B=60$ random draws for each $n$.

\Cref{fig:sim-gan} shows two distinct test-error peaks. Thus, after pooling the real observations
with the lower-dimensional generated observations, the test error exhibits the same multiple-descent
behavior suggested by the heterogeneous low-rank examples of \Cref{sec:vp-risk-main}, although this experiment
is not an exact instance of the models analyzed there.

\begin{figure}[t]
    \centering
    \includegraphics[width=0.90\linewidth,trim=0 7bp 0 6bp,clip]{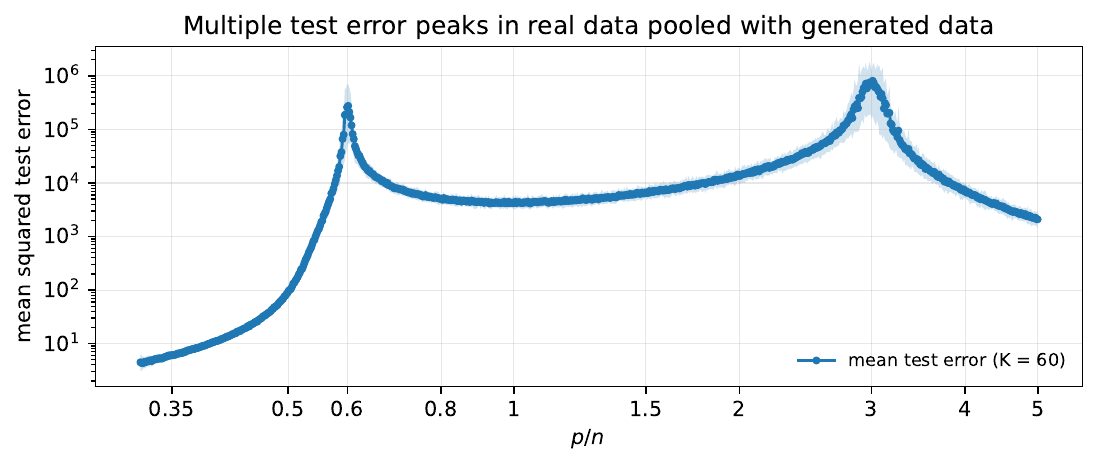}
    \caption{Multiple descent from pooling real and generated data. Test error against
    $\gamma = p/n$, both axes logarithmic. Training sets are drawn from a pool of
    $9920$ TCGA gene expression samples and $9920$ samples from a generative adversarial network trained
    on them, with $p = 6144$ genes. Mean over $60$ random training sets for each $n$; the band covers the
    $5$th to $95$th percentiles.}
    \label{fig:sim-gan}
\end{figure}
}

\medskip

Additional synthetic simulations appear in \Cref{sec:additional-experiments}, including further diagonal-support examples, approximate-degeneracy experiments, and full-rank non-commuting controls.

\section{Discussion}

We have shown that multiple descent can come from the structure of a single design. In the
simultaneously diagonalizable, possibly rank-deficient designs \Cref{mod:het} and
\Cref{mod:dep}, both the heterogeneous and the finite-rank-dependent Gaussian designs reduce
to a common variance-profile model, and the risk of the vanishing-ridge estimator admits the
deterministic equivalents of \Cref{thm:master}. Their hard-edge behavior is summarized by two
coordinatewise
quantities attached to the fixed-point solution $\br$ of~\eqref{eq:r-fixedpoint}: the value
$r_l(0)=\nu_l(\{0\})$, which fixes the coordinates on which the bias is supported, and the
derivative $\partial r_l(0)=\int_{\R\setminus\{0\}}t^{-2}\,\nu_l(dt)\in[0,\infty]$, which
determines whether the variance equivalent stays bounded or diverges.
Both are read off the support of
$S$: through the Dulmage--Mendelsohn structure of the variance graph, it fixes the bias support
and the locations at which the vanishing-ridge variance becomes singular
(\Cref{thm:bias-support,thm:var-switch}). The mechanism is spectral degeneracy, the zeros of the
covariance; by contrast, in the uniformly positive-definite regimes we conjecture, and our
simulations support, that no threshold appears beyond the classical one at $\gamma=1$
(\Cref{sec:non:sim:diag:but:pos}). Reaching the hard edge also required a random-matrix tool of
independent interest, an extension of the local law of \cite{alt2017local} with explicit
spectral-parameter dependence, sufficient to handle the vanishing ridge levels
$\lambda_n\downarrow0$ (\Cref{appendix:local:law}).

Several questions remain open. Proving the positive-definite conjecture rigorously needs a
global law for non-i.i.d.\ columns tracked to the hard edge, which \cite{louart2021spectral}
does not provide. Our rule also treats exact zeros; the theoretical behavior when these are
replaced by eigenvalues that are small, possibly vanishing with $n$, remains open. Finally, our
results hold for a vanishing ridge $\lambda_n\downarrow0$ rather than the exact
minimum-$l_2$-norm interpolator at $\lambda=0$, which would require controlling the smallest
nonzero singular value uniformly.

\subsection*{Acknowledgements}
KHH acknowledges support from the UK Engineering and Physical Sciences Research Council (EPSRC) (Grant No. EP/Y028783/1, Prob AI Hub). MA and SL were supported by NSF CAREER DMS-2441652.

\bibliographystyle{imsart-number}
\bibliography{ref.bib}

\newpage
\setcounter{page}{1}
\thispagestyle{empty}
\appendix 

\crefalias{section}{appendix}
\crefalias{subsection}{subappendix}
\crefalias{subsubsection}{subsubappendix}

\numberwithin{theorem}{section}
\numberwithin{lemma}{section}
\numberwithin{proposition}{section}
\numberwithin{corollary}{section}
\numberwithin{definition}{section}
\numberwithin{remark}{section}
\numberwithin{example}{section}
\numberwithin{assumption}{section}
\numberwithin{equation}{section}
\numberwithin{figure}{section}

\crefalias{section}{appendix}
\crefalias{subsection}{appendix}
\crefalias{subsubsection}{appendix}

\begin{center}
    \textbf{Appendices}
\end{center}

\vspace{.5em}

The appendices are organized as follows:
\begin{itemize}
    \item \Cref{appendix:setup} restates the setup and proves the equivalence of the models;
    \item \Cref{sec:risk:interpret} proves properties of $\br(0)$ and $\partial \br(0)$, which appear in the deterministic approximation of the risk, 
    together with the two-group and the asymptotic variance result for the prototype sequence in \Cref{sec:vp-risk-main};
    \item \Cref{appendix:local:law} states and proves the local law for the variance profile model with explicit $z$-dependence;
    \item \Cref{appendix:risk} proves the main results, which approximate the bias and the variance of the risks;
    \item \Cref{sec:additional-experiments} includes additional simulations.
\end{itemize}

\vspace{.5em}

\section{Setup} \label{appendix:setup}
\noindent
\textbf{\texorpdfstring{Equivalence of \Cref{mod:het} and \Cref{mod:dep}}{Equivalence of the heterogeneous and dependent models}. } 
We use the notation and models introduced in
\Cref{sec:setup-main,sec:non:sim:diag:but:pos}.

We first observe that \Cref{mod:het} is a special case of \Cref{mod:dep}. Indeed, by choosing $M=n$ and the $\R^{n \times n}$ covariance matrices as
\begin{align*}
    (\tilde \Sigma^{(m)})_{i,i'} \;=\;  \ind_{\{ i = i' = m \}} \qquad \text{ for } 1 \leq i, i' \leq n\;,
\end{align*}
i.e.~all but one entries of $\tilde \Sigma^{(m)}$ are zero, the data matrix $X$ in \Cref{mod:het} has independent row vectors each with covariance $\Sigma^{(m)}$. Then, the assumptions on $\Sigma^{(m)}$ in \Cref{mod:dep} and \Cref{mod:dep:pd} correspond exactly to those on $\Sigma_i$ in \Cref{mod:het} and \Cref{mod:het:pd} respectively. 

\vspace{.5em}

Next we give the proof of the bias--variance decomposition.
\begin{proof}[Proof of \Cref{lem:bias-var-decomp}]
By the test mechanism in \eqref{eq:test-gen-main},
\begin{align*}
    \cR^\lambda(X)
    \;=&\;
    \mfrac{1}{n} \msum_{i \leq n} 
    \mean\big[ (\beta - \hat \beta_\lambda)^\top \, \Var[ X_i]  \,
    (\beta - \hat \beta_\lambda) \,\big|\, (X_i)_{i \leq n} \big]
    \\
    \;=&\;
    \mean\Big[
        \Big\| \bar \Sigma^{1/2} \Big(  
            \Big(  \mfrac{1}{n} \msum_{i \leq n} X_i X_i^\top + \lambda I_p \Big)^{-1}  
            \mfrac{1}{n} \msum_{i \leq n} X_i ( X_i^\top \beta + \epsilon_i)
            - \beta 
        \Big) \Big\|^2
    \,\Big|\, 
        (X_i)_{i \leq n} 
    \Big]
    \\
    \;=&\;
    \underbrace{
        \big\| \bar \Sigma^{1/2}
        \big( (W_n + \lambda I_p)^{-1}  W_n - I_p \big)
        \beta \big\|^2 
    }_{\eqqcolon \cR_B^\lambda(X)}
    +
    \underbrace{
    \mfrac{\sigma^2_\epsilon}{n}
    \Tr\big( 
        \bar \Sigma \, W_n (W_n + \lambda I_p)^{-2}
    \big)
    }_{\eqqcolon \cR_V^\lambda(X)}
    \;.
    \tagaligneq \label{eq:R:formula} 
\end{align*}
\end{proof}
We first note that, regardless of which model we choose, the bias term is always bounded provided that the test data covariance $\bar \Sigma$ and the true signal $\beta$ are both bounded in norm:

\begin{lemma} \label{lem:bias:bounded} For every $\lambda > 0$, $\cR_B^\lambda(X) \;\leq\; \| \bar \Sigma \|_{op} \| \beta\|^2$. 
\end{lemma}

\begin{proof}[Proof of \Cref{lem:bias:bounded}] The result follows directly from observing that almost surely 
\begin{align*}
    \big\| (W_n + \lambda I_p)^{-1} W_n - I_p \big\|_{\mathrm{op}} 
    \;=\;
    \lambda \, \big\| (W_n + \lambda I_p)^{-1} \big\|_{\mathrm{op}} 
    \;\leq\; 1\;.
\end{align*}
\end{proof}

As we will see, the variance term $\cR_V^\lambda(X)$ will be the key quantity for characterizing the location of the peak. We next prove the inequality in \Cref{lem:var:bounds}.

\vspace{.5em}

\begin{proof}[Proof of \Cref{lem:var:bounds}] Under \Cref{asst:hetero:eigen:bound}, all eigenvalues of $\bar \Sigma = \frac{1}{n} \sum_{i \leq n} \Sigma_i$ lie in $[\alpha_{\rm min}, \alpha_{\rm max}]$. Since $W_n (W_n + \lambda_n I_p)^{-2}$ is non-negative almost surely, the desired bound holds.
\end{proof}

\vspace{.5em}

Meanwhile, we also observe that \Cref{mod:het} and \Cref{mod:dep} are in fact equivalent:
\begin{lemma} \label{lem:model:equivalence} Let $\cR^\lambda(X) = \cR^\lambda(X; \beta)$ be the test risk defined for \Cref{mod:dep} or
\Cref{mod:dep:pd}, where we have made the dependence on the ground-truth signal $\beta$ explicit. Then the following holds:
\begin{proplist}
    \item Consider \Cref{mod:dep} and express, for every $1 \leq m \leq M$,
    \begin{align*}
        \tilde \Sigma^{(m)}
        \;=&\;
        \tilde U^\top \tilde D^{(m)} \tilde U 
        &\text{ and }&&
        \Sigma^{(m)}
        \;=&\;
        U^\top  D^{(m)} U 
        \tagaligneq \label{eq:sim:diag}
    \end{align*}
    for some orthogonal matrices $\tilde U \in \R^{n \times n}$, $U \in \R^{p \times p}$ and diagonal matrices $\tilde D^{(m)} \in \R^{n \times n}$, $D^{(m)} \in \R^{p \times p}$. Then 
    \begin{align*}
        \cR^\lambda(X ;\beta) \;=\; \cR^\lambda(X_\star; U \beta)\;,
    \end{align*}
    where $X_\star$ is generated as in \Cref{mod:het} such that the covariance of its $i$-th row vector is given by
    \begin{align*}
        \Sigma_i \;=&\; \msum_{m=1}^M \tilde D^{(m)}_{ii} D^{(m)}\;;
    \end{align*}
    \item Consider \Cref{mod:dep:pd}, in which case $ \tilde \Sigma^{(m)} = \tilde D^{(m)}$ for some diagonal matrix for every $1 \leq m \leq M$. Then
    \begin{align*}
        \cR^\lambda(X; \beta) \;=\; \cR^\lambda(X_{\star\star}; \beta) \;,
    \end{align*}
    where 
    $X_{\star\star}$ is generated as in \Cref{mod:het:pd} such that the covariance of its $i$-th row vector is given by
    \begin{align*}
        \Sigma_i \;=&\; \msum_{m=1}^M \tilde D^{(m)}_{ii} \Sigma^{(m)}\;.
    \end{align*}
\end{proplist}
    
\end{lemma}

\begin{proof}[Proof of \Cref{lem:model:equivalence}] To prove (i), we first use \eqref{eq:sim:diag} to express 
\begin{align*}
    X \;=&\; 
    \msum_{m=1}^M 
    \,
    \tilde U^\top (\tilde D^{(m)})^{1/2} \tilde U 
    \,
    Z^{(m)}
    U^\top  (D^{(m)})^{1/2} U 
    \\
    \;\overset{d}{=}&\;
    \tilde U^\top 
    \Big(
    \underbrace{
        \msum_{m=1}^M 
        (\tilde D^{(m)})^{1/2} 
        Z^{(m)}
        (D^{(m)})^{1/2}
    }_{\eqqcolon X_\star}    
    \Big) U 
     \;,
\end{align*}
where we have used the orthogonal invariance of the standard Gaussian matrix $Z^{(m)}$. In this case 
\begin{align*}
    W_n 
    \;=&\; 
    \mfrac{1}{n} \, X^\top X
    \;=\;
     U^\top \Big( \underbrace{\mfrac{1}{n} X_\star^\top X_\star}_{\eqqcolon W_\star} \Big) U
    \;,
    \\
    \bar \Sigma 
    \;\overset{\eqref{eq:cov-dep-main}}{=}&\;
    \mfrac{1}{n} \msum_{i \leq n} \msum_{m=1}^M \tilde \Sigma_{ii}^{(m)} \Sigma^{(m)}
    \\
    \;=&\;
    U^\top 
    \Big(
            \mfrac{1}{n} \msum_{i \leq n} \msum_{m=1}^M \tilde \Sigma_{ii}^{(m)} D^{(m)}
    \Big) U
    \\
    \;=&\;
    U^\top 
    \Big(
            \mfrac{1}{n}  \msum_{m=1}^M \Tr\big( \tilde \Sigma^{(m)} \big) D^{(m)}
    \Big) U
    \;=\;
    U^\top 
    \Big(
        \underbrace{
            \mfrac{1}{n} \msum_{i \leq n} \msum_{m=1}^M \tilde D_{ii}^{(m)} D^{(m)}
        }_{\eqqcolon \bar \Sigma_\star}
    \Big) U
    \;.
\end{align*}
We can then express
\begin{align*}
    \cR_B^\lambda(X)
    \;=&\;
     \Big\| 
        U^\top
        \bar \Sigma_\star^{1/2}
        U
        \big( 
            U^\top ( W_{\star} + \lambda I_p)^{-1}  W_{\star} U - I_p \big)
        \beta 
    \Big\|^2 
    \\
    \;=&\;
    \Big\| 
        \bar \Sigma_\star^{1/2}
        \big( 
            ( W_{\star} + \lambda I_p)^{-1}  W_{\star} - I_p
        \big)
        (U \beta )
    \Big\|^2 
    \;,
    \\
    \cR_V^\lambda(X)
    \;=&\;
     \mfrac{\sigma^2_\epsilon}{n}
    \Tr\big( 
        U^\top \bar \Sigma_\star U \, U^\top W_{\star} ( W_{\star} + \lambda I_p)^{-2}  U
    \big)
    \;=\;
    \mfrac{\sigma^2_\epsilon}{n}
    \Tr\big( 
        \bar \Sigma_\star \, W_{\star} ( W_{\star} + \lambda I_p)^{-2}
    \big)
    \;.
\end{align*}
We now analyze $X_\star$. Denote the row vectors of $X_\star$ as $(X_{\star;i})_{i \leq n}$. They are independent mean-zero Gaussians by construction, each with covariance 
\begin{align*}
    \Var[X_{\star;i}] \;=\; \msum_{m=1}^M \tilde D^{(m)}_{ii} D^{(m)}
    \;.
\end{align*}
Then we see that $\cR^\lambda(X) = \cR_B^\lambda(X) + \cR_V^\lambda(X)$ under \Cref{mod:dep} is exactly the same as the test risk under \Cref{mod:het} and the test data generation mechanism \eqref{eq:test-gen-main} by identifying 
\begin{align*}
    \Sigma_i \;=&\; \msum_{m=1}^M \tilde D^{(m)}_{ii} D^{(m)}
\end{align*}
and replacing the ground-truth signal $\beta$ by $U \beta$.

\vspace{.5em}

To prove (ii), we note that under \Cref{mod:dep:pd},
\begin{align*}
    X
    \;=&\; 
    \msum_{m=1}^M (\tilde D^{(m)})^{1/2} Z^{(m)} (\Sigma^{(m)})^{1/2}\;.
\end{align*}
The row vectors of $X$ are again independent, each with covariance 
\begin{align*}
    \Var[X_i] \;=&\; \msum_{m=1}^M \tilde D^{(m)}_{ii} \, \Sigma^{(m)} 
    \;,
\end{align*}
which is positive definite since each $\Sigma^{(m)}$ is positive definite and $\tilde D^{(m)}_{ii} \geq 0$ with at least one $\tilde D^{(m)}_{ii} > 0$. This is exactly the setup of \Cref{mod:het:pd}. 
\end{proof}

\noindent
\textbf{\texorpdfstring{Reduction of Models \ref{mod:het} and \ref{mod:dep} to the variance profile model in \cite{hachem2007deterministic,alt2017local}}{Reduction of the heterogeneous and dependent models to a variance profile}.}

We next give the reduction of \Cref{mod:het} and \Cref{mod:dep} to \Cref{mod:vp}.

\begin{lemma} \label{lem:reduction:var:profile} The following statements hold:
\begin{proplist}
    \item Let $\cR^\lambda(X) = \cR^\lambda(X; \beta)$ be the test risk defined under \Cref{mod:het}, where we have made the dependence on the ground-truth signal $\beta$ explicit, and express
    \begin{align*}
        \Sigma_i
        \;=&\;
        U^\top  D^{(i)} U 
        \tagaligneq \label{eq:sim:diag:model:1}
    \end{align*}
    for some orthogonal matrix $U \in \R^{p \times p}$ and diagonal matrices $D^{(i)} \in \R^{p \times p}$. Then
    \begin{align*}
        \cR^\lambda(X;\beta) 
        \;=&\; 
         \cR^\lambda( X_{\rm VP}; U\beta )
        \;,
    \end{align*}
    where $X_{\rm VP}$ is generated as in \Cref{mod:vp} with the variance profile 
    \begin{align*}
        S \;=\; (S_1, \ldots, S_n)^\top \;=\; \mfrac{1}{n} ( D^{(1)} , \ldots, D^{(n)} )^\top   \;;
    \end{align*}
    \item Let $\cR^\lambda(X) = \cR^\lambda(X; \beta)$ be the test risk defined under \Cref{mod:dep}, where we have made the dependence on the ground-truth signal $\beta$ explicit, and inherit the notation $\tilde D^{(m)}$, $U$ and $D^{(m)}$ from \Cref{lem:model:equivalence}(i). Then
    \begin{align*}
        \cR^\lambda(X) 
        \;=&\; 
        \cR^\lambda( X_{\rm VP}'; U\beta )
        \;,
    \end{align*}
    where $X_{\rm VP}'$ is generated as in \Cref{mod:vp} with the variance profile 
    \begin{align*}
        S' \;=\; 
        (S_1', \ldots, S_n')^\top 
        \;=\; 
        \mfrac{1}{n}
        \msum_{m=1}^M 
        \big( 
            \tilde D^{(m)}_{11} D^{(m)}, \ldots, 
            \tilde D^{(m)}_{nn} D^{(m)}  
        \big)^\top  
        \;.
    \end{align*}
\end{proplist}
\end{lemma}
\begin{proof}[Proof of \Cref{lem:reduction:var:profile}] Under \Cref{mod:het}, we use \eqref{eq:sim:diag:model:1} to express 
\begin{align*}
    X 
    \;=\;
    (X_1, \ldots, X_n)^\top
    \;=&\; 
    \big( U^\top  (D^{(1)})^{1/2} U Z_1, \ldots, U^\top  (D^{(n)})^{1/2} U Z_n \big)^\top
    \\
    \;\overset{d}{=}&\;
    \big( (D^{(1)})^{1/2} Z_1, \ldots, (D^{(n)})^{1/2} Z_n \big)^\top U
     \;,
\end{align*}
where we have used the orthogonal invariance of the standard Gaussian matrix $(Z_1, \ldots, Z_n)$. In this case 
\begin{align*}
    W_n \;=&\; 
    \mfrac{1}{n} X^\top X 
    \;=\; 
    U^\top 
    \Big( \underbrace{ \mfrac{1}{n} \msum_{i \leq n} (D^{(i)})^{1/2} Z_i Z_i^\top   (D^{(i)})^{1/2} }_{\eqqcolon W_\#} \Big) U \;,
    \\
    \bar \Sigma
    \;=&\;
    \mfrac{1}{n} \msum_{i \leq n} \Sigma_i
    \;=\;
    U^\top \Big( \underbrace{ \mfrac{1}{n} \msum_{i \leq n} D^{(i)} }_{\eqqcolon \bar \Sigma_\#} \Big) U\;.
\end{align*}
This allows us to express 
\begin{align*}
    \cR_B^\lambda(X;\beta)
    \;=&\;
     \Big\| 
        U^\top
        \bar \Sigma_\#^{1/2}
        U
        \big( 
            U^\top \big( W_{\#} + \lambda I_p \big)^{-1}  W_{\#} U - I_p \big)
        \beta 
    \Big\|^2 
    \\
    \;=&\;
    \Big\| 
        \bar \Sigma_\#^{1/2}
        \big( 
            ( W_{\#} + \lambda I_p)^{-1}  W_{\#} - I_p
        \big)
        (U \beta )
    \Big\|^2 
    \;,
    \\
    \cR_V^\lambda(X)
    \;=&\;
     \mfrac{\sigma^2_\epsilon}{n}
    \Tr\big( 
        U^\top \bar \Sigma_\# U \, U^\top W_{\#} ( W_{\#} + \lambda I_p)^{-2}  U
    \big)
    \\
    \;=&\;
    \mfrac{\sigma^2_\epsilon}{n}
    \Tr\big( 
        \bar \Sigma_\# \, W_{\#} ( W_{\#} + \lambda I_p)^{-2}
    \big)
    \;.
\end{align*}
Consider the (rescaled) data matrix 
\begin{align*}
    \mfrac{1}{\sqrt{n}} \, X_{\rm VP}
    \;=\;
    \mfrac{1}{\sqrt{n}}
     \big( (D^{(1)})^{1/2} Z_1, \ldots, (D^{(n)})^{1/2} Z_n \big)^\top
     \;,
\end{align*}
i.e.~the data matrix generated under \Cref{mod:vp} with the variance profile  $S = \frac{1}{n} ( D^{(1)} , \ldots, D^{(n)} )^\top = (S_1, \ldots, S_n)^\top$. Then we have 
\begin{align*}
    W_\# \;=&\; \mfrac{1}{n} X_{\rm VP}^\top X_{\rm VP} 
    &\text{ and }&&
    \bar \Sigma_\# \;=&\; \mfrac{1}{n} \msum_{i \leq n} (n S_i) \;,
\end{align*}
which allows us to write 
\begin{align*}
    \cR^\lambda(X;\beta) 
    \;=\; 
    \cR_B^\lambda(X;\beta) + \cR_V^\lambda(X)
    \;=\;
    \cR_B^\lambda\big(X_{\rm VP};  U \beta \big) + \cR_V^\lambda(X_{\rm VP})
    \;=\;
    \cR^\lambda(X_{\rm VP}; U\beta) \;.
\end{align*}

\vspace{1em}

Under \Cref{mod:dep}, we first apply \Cref{lem:model:equivalence} to obtain 
\begin{align*}
    \cR^\lambda(X;\beta) \;=\; \cR^\lambda(X_\star; U \beta)\;,
\end{align*}
where $X_\star$ is generated as in \Cref{mod:het} such that the covariance of its $i$-th row vector is given by
\begin{align*}
    \Sigma_i \;=&\; \msum_{m=1}^M \tilde D^{(m)}_{ii} D^{(m)}\;.
\end{align*}
Since $D^{(m)}$'s are diagonal, so is $\Sigma_i$ and therefore $X_\star$ is a Gaussian matrix with independent entries. In particular $X_{\rm VP}' \coloneqq  X_\star$ can be viewed as being generated from \Cref{mod:vp} with the variance profile 
\begin{align*}
    S' 
    \;=&\; 
    \mfrac{1}{n}
    (\Sigma_1, \ldots, \Sigma_n)^\top  
    \;=\;
    \mfrac{1}{n}
    \Big( 
        \msum_{m=1}^M \tilde D^{(m)}_{11} D^{(m)}, \ldots, 
        \msum_{m=1}^M \tilde D^{(m)}_{nn} D^{(m)}  
    \Big)^\top  
    \\
    \;=&\;
    \mfrac{1}{n}
    \msum_{m=1}^M 
    \big( 
        \tilde D^{(m)}_{11} D^{(m)}, \ldots, 
        \tilde D^{(m)}_{nn} D^{(m)}  
    \big)^\top  
    \;.
\end{align*}
A similar argument as the proof for (i) gives the desired statement.
\end{proof}

\section{Properties of the deterministic approximations} \label{sec:risk:interpret}

Our limiting expressions of the risk involve the quantities $r$ and $\partial r$, which are defined implicitly through a system of equations involving the variance profile $S$ and not obviously tractable. We will see that for the purpose of tracking degeneracies and the effect of multiple descents, it suffices to analyze directly the degeneracy structure of the $n \times p$ matrix $S$. 

\vspace{.5em}

\subsection{Analysis of \texorpdfstring{$r(0)$}{r(0)} in the bias term} \label{app:bias-support}

We use the definitions and notation from
\Cref{sec:vp-risk-main}. We first prove the matching characterization
of $\bJ_S$ in \Cref{lem:J-shp}.

\begin{proof}[Proof of \Cref{lem:J-shp}]  We will make use of ideas from the classical tool of alternating paths for handling maximum matchings: Given a matching $\cE$ for the graph $\cG_S$, an $\cE$-alternating path is a path in $\cG_S$ that alternates between edges outside and inside $\cE$; see e.g.~\cite{lovasz2009matching} for an introduction. We will also make use of following classical result:

\begin{theorem}[Berge's theorem \cite{berge1957two}]  \label{thm:berge}
$\cE$ is a maximum matching of $\cG_S$ if and only if there is no $\cE$-alternating path that starts and ends both on unmatched vertices of $\cE$.
\end{theorem}
    
Now fix a maximum matching $\cE^{\rm max}_S$ of $\cG_S$. We now create a directed graph $\hat \cH$ by orienting every edge outside $\cE^{\rm max}_S$ from $\cV_{\rm col}$ to $\cV_{\rm row}$, and every edge inside $\cE^{\rm max}_S$ from $\cV_{\rm row}$ to $\cV_{\rm col}$. Let $J(\hat \cH)$ be the set of column vertices reachable by traversing along the directed graph $\hat \cH$ starting from the set of column vertices unmatched by $\cE^{\rm max}_S$. We first observe that $\bJ_S = J(\hat \cH)$:
\begin{itemize}
    \item If $j \in J(\hat \cH)$, there is an $\cE^{\rm max}_S$-alternating path $\cP$ starting from some unmatched column vertex by $\cE^{\rm max}_S$. We can toggle the matching $\cE^{\rm max}$ along $\cP$ to create a maximum matching that leaves $j$ unmatched. This shows $J(\hat \cH) \subseteq \bJ_S$;
    \item If $j \in \bJ_S$, some maximum matching $\cE'$ leaves $j \in \cV_{\rm col}$ unmatched. If $\cE^{\rm max}_S$ leaves $j$ unmatched, then automatically we have $j \in J(\hat \cH)$. Otherwise, $\cE^{\rm max}_S$ leaves $j$ matched whereas $\cE'$ does not. In this case, the subgraph $\cA$ of $\cG_S$ formed by the symmetric difference $\cE^{\rm max}_S \triangle \cE'$ contains $j$. Let $\cA_j$ be the connected component of $\cA$ containing $j$, which is an $\cE^{\rm max}_S$-alternating path as well as an $\cE'$-alternating path. Moreover, the endpoints of the path $\cA_j$ cannot be both in $\cE'$ or $\cE^{\rm max}_S$: Otherwise it would create an alternating path between unmatched vertices by one of the maximum matchings and violate \Cref{thm:berge}. Now consider traversing the path $\cA_j$ from a vertex unmatched by $\cE^{\rm max}_S$ to $j$. Since it starts from an edge outside $\cE^{\rm max}_S$, this way of traversing $\cA_j$ follows the same direction as $\hat \cH$, which proves that $j \in J(\hat \cH)$. 
\end{itemize}

\vspace{.5em}

We next show that the induced submatrix of $S$ by $\bJ_S$, $S(\bJ_S)$, satisfies the row-wise strong Hall property.
Observe first that since $\bJ_S = J(\hat \cH)$, every row vertex $i \in I_{\bJ_S}(S)$ is reachable in $\hat \cH$ from the column indices unmatched by $\cE^{\rm max}_S$, and has to be matched by $\cE^{\rm max}_S$ by \Cref{thm:berge}. Moreover, the matched column vertex of every $i \in I_{\bJ_S}(S)$ is also reachable and therefore belongs to $J(\hat \cH) = \bJ(S)$. In other words, $I_{\bJ_S}(S)$ is matched to $\bJ_S$ by $\cE^{\rm max}_S$. Now for every non-empty $I' \subseteq I_{\bJ_S}(S)$, let $J'$ be the set of column indices matched to the row indices in $I'$ by $\cE^{\rm max}_S$, and denote the neighbors of $I'$ as 
\begin{align*}
    J_{I'} \;\coloneqq\; \big\{ j \in \bJ_S \,\big|\, S_{ij} > 0 \text{ for some } i \in I' \big\}\;.
\end{align*}
Then we have
\begin{align*}
    | I' | \;=\; | J' | \;\leq\; | J_{I'}  |\;
    \qquad\text{ and }\qquad 
    J' \;\subseteq\; J_{I'}\;.
\end{align*}
To show the row-wise strong hall property of $S(\bJ_S)$, it suffices to show that the above inequality is strict.  We seek to show the existence of some $j_* \in J_{I'}$ that is not in $J'$:
\begin{itemize}
    \item Consider any row index $i' \in I' \subseteq I_{\bJ_S}(S)$. The property $\bJ_S = J(\hat H)$ implies that there exists an $\cE^{\rm max}_S$-alternating path $\cP_*$ starting from a column index $j'$ unmatched by $\cE^{\rm max}_S$. Let $i_*$ be the first row index along the path $\cP_*$ that lies in $I'$, and $j_*$ be the column index right before $i_*$ on the path $\cP_*$.  Then $j_* \in J_{I'}$ since it is a neighbor of $i_*$. 
    \item Meanwhile, if $j_* \in J'$ also, then there exists some $i_{**} \in I'$ that (i) is matched to $j_*$ by $\cE^{\rm max}_S$, (ii) is distinct from $i_*$, and (iii) has the edge $i_{**} \rightarrow j_*$ in the alternating path $\cP_*$.  This contradicts the assumption that $i_*$ is the first row index in $I'$ encountered in the path $\cP_*$. Hence $j_* \not\in J'$.
\end{itemize}
This proves that $S(\bJ_S)$ satisfies the row-side strong Hall property. 

\vspace{.5em}

It remains to prove uniqueness and maximality. Let $J \subseteq \{1, \ldots, p\}$ be a column set whose induced sub-matrix $S(J)$ satisfies the row-side strong Hall property. It suffices to show that every such $J$ has to be a subset of $\bJ_S$. Suppose for contradiction that $B \coloneqq J \setminus \bJ_S \neq \emptyset$. Since $\bJ_S$ is the set of column indices unmatched by some maximum matching, any column index $b \in B$ needs to be matched by the maximum matching $\cE^{\rm max}_S$ above, say to a row index $i_b$. Now take $I_B = \{i_b \,|\, b \in B\}$, and denote the neighbors of $I_B$ as $J_B \coloneqq \{ j \in J \,|\, S_{ij} > 0 \text{ for some } i \in I_B\}$. Clearly $B \subseteq J_B$. Suppose that $B \subset J_B$. Then there exists some $b \in B$ with $i_b \in I_B$ such that $i_b$ has a neighbor in $\bJ_S$. But since $\bJ_S = J(\hat H)$ is the set of column vertices reachable in $\hat H$ from the column vertices unmatched by $\cE^{\rm max}_S$, this makes $b$ also reachable, which contradicts the starting assumption that $B$ is disjoint from $J_B$. Therefore we must have
\begin{align*}
    B \;=\; J_B\;.
\end{align*}This implies 
\begin{align*}
    | I_B| \;=\; |J_B| \;=\; |B| \;,
\end{align*} 
which violates the row-wise strong Hall property.  Therefore $J \subseteq \bJ_S$ and $\bJ_S$ is the unique maximal column set $J$ whose induced sub-matrix $S(J)$ satisfies the row-side strong Hall property.
\end{proof}

\begin{remark}[Computing $\bJ_S$] \label{rem:compute:JS}
$\bJ_S$ is computed in linear time from any single maximum matching, by the construction in the
proof of \Cref{lem:J-shp}: fix a maximum matching $\cE^{\rm max}_S$ of $\cG_S$, orient every
unmatched edge from columns to rows and every matched edge from rows to columns, and collect all
column vertices reachable by directed alternating paths from the columns left unmatched by
$\cE^{\rm max}_S$. By that proof the reachable column set equals $\bJ_S$. This is the standard
alternating-path form of the Dulmage--Mendelsohn decomposition~\citep{Dulmage_Mendelsohn_1958,
pothen1990computing}, and is available in standard software (e.g.~\texttt{dmperm} in
\textsc{Matlab}, \texttt{scipy.sparse.csgraph} in \textsc{Python}, and the \texttt{Matrix}
package in \textsc{R}).
\end{remark}

\vspace{.5em}

We are ready to characterize the support of $r(0)$ via $\bJ_S$.

\begin{proof}[Proof of \Cref{thm:bias-support}] 
{Recall that, for every $1\leq l\leq p$, the limit
$r_l(0)\coloneqq\lim_{\lambda\downarrow0}r_l(\lambda)$ exists in $[0,1]$.}
Denote the support of $\br(0)$ as 
\begin{align*}
    J_{\rm pos} \;\coloneqq\; 
    \big\{ l \in \{1, \ldots, p\} \,\big|\, r_l(0) > 0 \big\}\;. 
\end{align*}

\vspace{.5em}

We first seek to show that $\bJ_S \subseteq J_{\rm pos}$, i.e.~for every $l \in \bJ_S$, $r_l(0) > 0$. {The claim is trivial if $\bJ_S=\emptyset$, so suppose $\bJ_S\neq\emptyset$.} {Fix $\lambda\in(0,1]$.} Let $x_l \coloneqq \log r_l({\lambda}) \in (-\infty, 0]$, and rewrite {\eqref{eq:r-fixedpoint}} as 
\begin{align*}
     1
    \;=\; 
    e^{x_l}
    +
    \msum_{i=1}^n \mfrac{S_{il} e^{x_l}}{{\lambda +} \sum_{j=1}^p S_{ij} e^{x_j}}
    \;.
\end{align*}

This is the first-order optimality condition for minimizing the convex function 
\begin{align*}
    \Phi( \bx )
    \;\coloneqq&\;
    \msum_{l=1}^p ( e^{x_l} - x_l  ) 
    +
    \msum_{i=1}^n \log\Big( {\lambda +} \msum_{j=1}^p S_{ij} e^{x_j}\Big)
    \;,
\end{align*}
i.e.~$\bx$ is a minimizer of $\Phi(\bx)$. Now consider a modification of $\bx$ with the coordinates at $\bJ_S$ set to zero,
\begin{align*}
    \bx^{-\bJ_S} \;\coloneqq\; \big( x^{-\bJ_S}_l \big)_{l \leq p}
    \qquad 
    \text{ with }
    \quad 
    x^{-\bJ_S}_l  \;\coloneqq\;  x_l \, \ind_{\{ l \not\in \bJ_S\}}\;,
\end{align*}
and recall the notation $I_{\bJ_S}(S)$ and $S(\bJ_S)$ from \Cref{subsec:bias-support}. Given some row index {$i \in I_{\bJ_S}(S)$}, we also denote its neighbors in  $S(\bJ_S)$ as 
\begin{align*}
    \cN_i(\bJ_S) \;\coloneqq\; \{ j \in \bJ_S \,|\, S_{ij} > 0\}\;,
\end{align*}
and we fix
\begin{align*}
    j^{(i)} \;\in\; \argmax_{j \in \cN_i(\bJ_S)} x_j\;.
\end{align*}
Then 
\begin{align*}
    &\;
    \Phi( \bx ) 
    -
    \Phi( \bx^{-\bJ_S} )
    \\
    &\;=\;
    \msum_{l \in \bJ_S} 
    ( e^{x_l} - x_l  {-1}) 
    +
    \msum_{i=1}^n \log\Big( 
        \mfrac{
            {\lambda +} \sum_{j=1}^p S_{ij} e^{x_j}
        }
        {
            {\lambda +\sum_{j\in\bJ_S}S_{ij} + }\sum_{j \not \in \bJ_S} S_{ij} e^{x_j}
        }
    \Big)
    \\
    &\;=\;
    \msum_{l \in \bJ_S} 
    ( e^{x_l} - x_l  {-1}) 
    +
    \msum_{i \in I_{\bJ_S}(S)} \log\Big( 
        \mfrac{
            {\lambda +} \sum_{j=1}^p S_{ij} e^{x_j}
        }
        {
            {\lambda +\sum_{j\in\bJ_S}S_{ij} + }\sum_{j \not \in \bJ_S} S_{ij} e^{x_j}
        }
    \Big)
    \\
    &\;\geq\;
    \msum_{l \in \bJ_S} 
    {(-  x_l -1)}
    +
    \msum_{i \in I_{\bJ_S}(S)} \log\bigg( 
        \mfrac{
            {\min_{j \in \cN_i(\bJ_S)} S_{ij}} \,  e^{\max_{j \in \cN_i(\bJ_S)} x_j}
        }
        {
             {1+\sum_{j=1}^p S_{ij}}
        }
    \bigg)
    \\
    &\;=\;
     \underbrace{
    {
        -|\bJ_S|
        +
        \msum_{i \in I_{\bJ_S}(S)}
        \log\Big( 
            \mfrac{
                \min_{j \in \cN_i(\bJ_S)} S_{ij}
            }
            {
                1+\sum_{j=1}^p S_{ij}
            }
        \Big)
    }
    }_{\eqqcolon \, C(\bJ_S)}
    +
    \msum_{l \in \bJ_S} (- x_l)
    -
    \msum_{i \in I_{\bJ_S}(S)}
    \min_{j \in \cN_i(\bJ_S)} (- x_j)
    \\
    &\;=\;
     C(\bJ_S)
    +
    \msum_{l \in \bJ_S}  \, \mint_0^\infty \ind_{\{  - x_l \geq t \}}   dt  
    - 
    \msum_{i \in I_{\bJ_S}(S)} \mint_0^\infty \ind_{\{  \min_{j \in \cN_i(\bJ_S)} (-x_j) \geq t \}}   dt  
    \\
    &\;=\;
    C(\bJ_S)
    +
     \mint_0^\infty \# \{ j \in \bJ_S \,|\,  - x_j \geq t \}   dt  
    - 
    \mint_0^\infty  \# \Big\{ i \in I_{\bJ_S}(S) \,\Big|\, \min_{j \in \cN_i(\bJ_S)} (-x_j) \geq t \Big\}   dt  
    \;.
\end{align*}
Denoting the sets
\begin{align*}
    A_t \;\coloneqq&\; \{ j \in \bJ_S \,|\, - x_j \geq t  \}
    &\text{ and }&&
    I_t \;\coloneqq&\; \big\{ i \in I_{\bJ_S}(S) \,\big|\,  \cN_i(\bJ_S) \subseteq A_t  \big\}\;,
\end{align*}
we can rewrite the above as 
\begin{align*}
    \Phi( \bx ) 
    -
    \Phi( \bx^{-\bJ_S} ) 
    \;\geq\; 
    C(\bJ_S)
    +
     \mint_0^\infty \big( |A_t| - |I_t| \big)  dt \;.
\end{align*}
If $A_t = \emptyset$, then $I_t = \emptyset$. If $A_t \neq \emptyset$ and $I_t = \emptyset$, since $|A_t|$ and $|I_t|$ are both integers, $|A_t| - |I_t| \geq 1$. Suppose $I_t$ is non-empty. Recall that by \Cref{lem:J-shp}, the induced sub-matrix $S(\bJ_S)$ satisfies the row-side strong Hall property, which implies 
\begin{align*}
    |I_t| + 1
    \;\leq&\; 
    \big| \big\{ j \in \bJ_S \,\big|\, S_{ij} > 0 \text{ for some } i \in I_t \big\} \big|
    \\
    \;=&\;
    \big| \big\{ j \in \bJ_S \,\big|\, S_{ij} > 0 \text{ for some } i \in I_{\bJ_S}(S) \text{ with } \cN_i(\bJ_S) \subseteq A_t \big\} \big|
    \;\leq\; | A_t|\;,
\end{align*}
in which case $|A_t| - |I_t| \geq 1$ again. In summary, we have that $|A_t| - |I_t| \geq 1$ whenever $A_t \neq \emptyset$. Therefore 
\begin{align*}
    \Phi( \bx ) - \Phi( \bx^{-\bJ_S} )
    \;\geq&\;
    C(\bJ_S)
    +
    \mint_0^\infty  
    \ind_{\{ A_t \neq \emptyset \} } \,  dt   
    \\
    \;=&\;
    C(\bJ_S)
    +
    \mint_0^\infty  
    \ind_{\{ \max_{j \in \bJ_S}(-x_j) \geq t  \} } \,  dt   
    \\
    \;=&\;
    C(\bJ_S) - \min_{j \in \bJ_S} x_j
    \;.
\end{align*}
Since $\bx$ is a minimizer of $\Phi$, we get that
\begin{align*}
  \min_{j \in \bJ_S} \, \log r_j({\lambda}) \;=\;  \min_{j \in \bJ_S}\, x_j \;\geq\; C(\bJ_S) + \big( \Phi( \bx^{-\bJ_S} ) - \Phi( \bx ) \big) \;\geq\; C(\bJ_S)\;.
\end{align*}
For all $i \in I_{\bJ_S}(S)$, {by the definition of $I_{\bJ_S}(S)$, $\cN_i(\bJ_S)$ is non-empty, and hence $\min_{j\in\cN_i(\bJ_S)}S_{ij}>0$,} and therefore $C(\bJ_S) > - \infty$. This implies 
\begin{align*}
    \min_{j \in \bJ_S} r_j({\lambda}) {\;\geq\;
    e^{C(\bJ_S)}}  \;>\; 0\;.
\end{align*}
{Letting $\lambda\downarrow0$ and using
$r_j(0)=\lim_{\lambda\downarrow0}r_j(\lambda)$ gives}
\begin{align*}
    {r_j(0)
    \;\geq\;
    e^{C(\bJ_S)}
    \;>\;0
    \qquad
    \text{for every $j\in\bJ_S$.}}
\end{align*}
This proves that $\bJ_S \subseteq J_{\rm pos}$.

\vspace{.5em}

We now show that $J_{\rm pos} \subseteq \bJ_S$. By the proof of \Cref{lem:J-shp}, it suffices to show that the induced sub-matrix of $S$ by $J_{\rm pos}$, $S(J_{\rm pos})$, satisfies the row-side strong Hall property. Take any non-empty $I' \subseteq I_{J_{\rm pos}}(S)$, and denote its neighbors in $S(J_{\rm pos})$ as
\begin{align*}
    J' \;\coloneqq\; \big\{ j \in J_{\rm pos} \,\big|\, S_{ij} > 0 \text{ for some } i \in I'  \big\}\;.
\end{align*}
{For every $l\in J_{\rm pos}$, taking $\lambda\downarrow0$ in
\eqref{eq:r-fixedpoint} gives}
\begin{align*}
     1
    \;=\; 
    r_l(0)
    +
    \msum_{i{\in I_{J_{\rm pos}}}(S)} \mfrac{S_{il} r_l(0)}{\sum_{j {\in J_{\rm pos}}} S_{ij} r_j(0)}
    \quad 
    \qquad 
    \text{for all {$l\in J_{\rm pos}$}.  }
    \tagaligneq \label{eq:r:0:system_new}
\end{align*}
Summing \eqref{eq:r:0:system_new} over $l \in J'$, we have
\begin{align*}
    |J'|
    \;=&\; 
    \msum_{l \in J'} r_l(0)
    +
    \msum_{i\in {I_{J_{\rm pos}}(S)}}\mfrac{\sum_{l \in J'} S_{il} r_l(0)}{\sum_{j \in J_{\rm pos}} S_{ij} r_j(0)}
    \\
    \;\geq&\;
    \msum_{l \in J'} r_l(0)
    +
    \msum_{i \in I'} \mfrac{\sum_{l \in J'} S_{il} r_l(0)}{\sum_{j \in J_{\rm pos}} S_{ij} r_j(0)}
    \\
    \;=&\;
    \msum_{l \in J'} r_l(0)
    +
    |I'|
    \;.
\end{align*}
In the last line, we used that for every $i \in I'$, all of its neighbors {in $J_{\rm pos}$} are in $J'$. {Note that since $I'$ is non-empty, $J'$ is non-empty.} Since $J' \subseteq J_{\rm pos}$, {$\msum_{l \in J'} r_l(0) > 0$, and} the inequality above is strict, and therefore $|I'|{+1} \leq |J'|$. This proves that $S(J_{\rm pos})$ has the row-side strong Hall property and therefore $J_{\rm pos} \subseteq \bJ_S$. This then implies $J_{\rm pos} = \bJ_S$.   
\end{proof}

\vspace{.5em}

\subsection{Analysis of \texorpdfstring{$\partial \br(0)$}{dr(0)} in the variance term} \label{app:var-switch}

In the simple case of i.i.d.~isotropic data (\Cref{ex:all-ones}), 
the location of the double-descent peak is $p/n=1$, which is right before $\bJ_S$ switches from $\emptyset$ to the full set. 

It turns out that this intuition generalizes for describing the behavior of $\partial \br(0)$: $\partial \br(0)$ will be finite precisely when $r(0)$ is ``away'' from the boundary at which $\bJ_S$ switches values. 

Roughly speaking, the idea is that there is an unavoidable square sub-matrix that causes the eigenvalues of $X$ to accumulate near zero.  

To formalize this, we use the residual profile $\bS^{\rm res}$, the switching condition,
and the column-side strong Hall property defined in
\Cref{subsec:switching}. 
Recall that $\bI_S = I_{\bJ_S}(S)$. We will use as shorthands 
\begin{align*}
    \bJ_S^c \;\coloneqq&\; \{1,\ldots,p\} \setminus \bJ_S
    &\text{ and }&&
    \bI_S^c \;\coloneqq&\; \{1,\ldots,n\} \setminus \bI_S\;.
\end{align*}

\begin{proof}[Proof of \Cref{lem:switch-shp}] We first show that $\bS^{\rm res}$ satisfies the column-side (weak) Hall property, i.e.~for each $q$ with $1 \leq q \leq p'$, every set of $q$ columns of $\bS^{\rm res}$ contain non-zero entries in at least $q$ different rows. Consider the maximum matching $\cE^{\rm max}_S$ of $\cG_S$ and the directed graph $\hat H$ constructed in the proof of \Cref{lem:J-shp}. By the proof of \Cref{lem:J-shp}, $\bJ_S = J(\hat H)$, the set of column vertices reachable in $\hat H$ from the column vertices unmatched by $\cE^{\rm max}_S$. Then the matched indices of $\bJ_S^c$ by $\cE^{\rm max}_S$ must lie in $\bI_S^c$: Otherwise, some $j \in \bJ_S^c$ is matched to $i \in \bI_S$ and therefore would be reachable in $\hat H$ from the unmatched column vertices, which contradicts the maximality of $\bJ_S$. Moreover, by the definition of $\bJ_S$, every vertex in $\bJ^c_S$ must be matched by $\cE^{\rm max}_S$. Since the graph $\bS^{\rm res}$ is a restriction of $S$ to $\bI^c_S$ and $\bJ^c_S$, we can apply Hall's theorem \cite{hall1935representatives} to  $\bS^{\rm res}$ to obtain that  $\bS^{\rm res}$ satisfies the column-side Hall property, i.e.
\begin{align*}
    |J'| \;\leq\; \big| I_{J'}\big( \bS^{\rm res} \big) \big|
    \;\qquad\;
    \text{ for every non-empty $J' \subseteq \bJ^c_S$\,,}
    \tagaligneq
    \label{eq:switch:HP}
\end{align*}
where we recall that $I_{J'}(\bS^{\rm res}) = \{ i \in \bI^c_S \,|\, \bS^{\rm res}_{ij} > 0 \text{ for some } j \in J'\}$.

\vspace{.5em}

To show the lemma statement, it now suffices to show that the switching condition is equivalent to the existence of some non-empty $J' \subseteq \bJ^c_S$ such that the inequality in \eqref{eq:switch:HP} is an equality. First assume that such an $J'$ exists. Then every row in $I_{J'}( \bS^{\rm res} )$ is used to match $J'$ by every maximum matching of $\cG_{\bS^{\rm res}}$. Taking any $i \in I_{J'}( \bS^{\rm res} )$ gives the switching condition.

\vspace{.5em}

Conversely, suppose the switching condition holds, and let $i \in \bI_S^c$ be the row vertex used by every maximum matching of $\cG_{\bS^{\rm res}}$. Consider the leave-$\{i\}$-out residual graph $\cG_{\bS^{\rm res}} \setminus \{i\}$: Any maximum matching of the leave-$\{i\}$-out graph must not match all of $\bJ_S^c$, because otherwise it gives a maximum matching of the original  $\cG_{\bS^{\rm res}}$ that does not use $\{i\}$. Applying Hall's theorem \cite{hall1935representatives} gives that, there exists some $J' \subseteq \bJ^c_S$ such that 
\begin{align*}
    |J'| \;>\; \big| I_{J'}\big( \bS^{\rm res} \big) \setminus \{i\} \big|\;.
\end{align*}
Combining this with \eqref{eq:switch:HP} shows that we must have $i \in I_{J'}( \bS^{\rm res})$ and $|J'| = \big| I_{J'}\big( \bS^{\rm res} \big)  \big|$. This gives our desired $J'$ and shows that the column-side strong Hall property fails for $\bS^{\rm res}$.
\end{proof}

\begin{remark}[Checking the switching condition] \label{rem:compute:switch}
The switching condition can be checked in a single pass over one maximum matching of the
residual graph $\cG_{\bS^{\rm res}}$, whose rows are $\bI_S^c$ and columns are $\bJ_S^c$; no
matching needs to be recomputed after deleting a row. Call a residual row \emph{forced} if
every maximum matching of $\cG_{\bS^{\rm res}}$ matches it, so that the switching condition is
exactly the existence of a forced row.

\emph{Step 1 (a column-saturating matching).} Compute a maximum matching $M$ of
$\cG_{\bS^{\rm res}}$ that saturates every residual column $j\in\bJ_S^c$; one exists by the
weak column-side Hall inequality~\eqref{eq:switch:HP} in the proof of \Cref{lem:switch-shp}.

\emph{Step 2 (orientation).} Direct each edge according to $M$: an edge \emph{not} in $M$ from
its row endpoint to its column endpoint, and an edge \emph{in} $M$ from its column endpoint to
its row endpoint. (This is the orientation of \Cref{rem:compute:JS} with rows and columns
interchanged, because here we start the search from rows rather than columns.)

\emph{Step 3 (reachability).} Search the directed graph from the residual rows left unmatched
by $M$, and let $R_{\rm alt}\subseteq\bI_S^c$ be the residual rows it reaches. A directed path
from an unmatched row alternates non-matching and matching edges, so reaching a row $i$ exhibits
an $M$-alternating path from an unmatched row to $i$; toggling $M$ along it yields another
matching that still saturates $\bJ_S^c$ but leaves $i$ unmatched. Conversely, if some
column-saturating matching leaves $i$ unmatched, the component of its symmetric difference with
$M$ through $i$ is such a path. Hence a residual row $i$ can be left unmatched by \emph{some}
column-saturating maximum matching if and only if $i\in R_{\rm alt}$.

\emph{Step 4 (read off the answer).} The forced rows are therefore exactly
$\bI_S^c\setminus R_{\rm alt}$, and by \Cref{lem:switch-shp} the switching condition holds if and
only if this set is non-empty. The procedure is one maximum-matching computation plus one
linear-time graph search, and is the standard alternating-path form of the Dulmage--Mendelsohn
decomposition~\citep{Dulmage_Mendelsohn_1958, pothen1990computing}.
\end{remark}

\begin{proof}[Proof of \Cref{thm:var-switch}]
We prove the necessary and sufficient condition, and \eqref{eq:partial-r-residual} and \eqref{eq:partial-r-support} will emerge as a by-product of the proof. By \Cref{lem:switch-shp}, it suffices to show that $\partial \br(0)$ exists in $\R^p$ if and only if $\bS^{\rm res}$ satisfies the column-side strong Hall property. The proof is similar to that of \Cref{thm:bias-support}: One direction proceeds by summing the fixed-point equations corresponding to $\partial_\lambda \br(0)$ over suitable coordinates, whereas the other proceeds by identifying a convex potential function whose minimum can be computed by the fixed point equations in $\partial_\lambda \br(0)$. To this end, we recall from \eqref{eq:r-fixedpoint} that $\br(\lambda)$ is the solution to the system of equations
\begin{align*}
     1
    \;=\; 
    r_l(\lambda)
    +
    \msum_{i=1}^n \mfrac{S_{il} r_l(\lambda)}{\lambda + \sum_{j=1}^p S_{ij} r_j(\lambda)}
    \quad 
    \text{ and }
    \quad
    r_l(0)\;\in\; [0,1]
    \qquad 
    \text{for all $1 \leq l \leq p$.  }
    \tagaligneq \label{eq:r:lamb:system}
\end{align*}
Also recall from \Cref{thm:bias-support} that $\bJ_S$ is exactly the support of $\br(0)$. This motivates us to define two different sets of variables,
\begin{align*}
    q_k(\lambda) \;\coloneqq&\; \mfrac{r_k(\lambda)}{\lambda}
    \quad \text{ for } k \in \bJ_S^c 
    &\text{ and }&&
    u_j(\lambda) \;\coloneqq&\; \mfrac{r_j(\lambda) - r_j(0)}{\lambda}
    \quad \text{ for } j \in \bJ_S\;,
\end{align*} 
both of which will approximate the corresponding coordinates of $\partial_\lambda \br(0)$ as $\lambda \rightarrow 0^+$ provided that they exist. It is also useful to note that 
\begin{align*}
    \msum_{j \in \bJ_S} S_{ij} r_j(0) \;>\; 0
    \qquad 
    \text{ if and only if }
    \qquad 
    i \in \bI_S\;,
\end{align*}
where we have recalled the definition $\bI_S = I_{\bJ_S}(S)$.

\vspace{.5em}

Assume first that $\partial \br(0)$ exists in $\R^p$. Then $q_k(0) = \lim_{\lambda \rightarrow 0^+}q_k(\lambda)$ and $u_j(0) = \lim_{\lambda \rightarrow 0^+}u_j(\lambda)$ exist. We first rewrite \eqref{eq:r:lamb:system} for $l=k \in \bJ^c_S$ as 
\begin{align*}
    1 
    \;=&\; 
    \lambda q_k(\lambda)  
    + 
    \msum_{i \in \bI_S} \mfrac{ \lambda S_{ik} q_k(\lambda)}{
        \lambda 
        + 
        \sum_{j \in \bJ_S} S_{ij} r_j(\lambda)
        +
        \sum_{j \in \bJ_S^c} S_{ij} r_j(\lambda)
    }
    \\
    &\;
    + 
    \msum_{i \in \bI_S^c} \mfrac{ \lambda S_{ik} q_k(\lambda)}{
        \lambda 
        +
        \sum_{j \in \bJ_S} S_{ij} r_j(\lambda)
        +
        \sum_{j \in \bJ_S^c} S_{ij} r_j(\lambda)
    }
    \\
    \;=&\;
    \lambda q_k(\lambda)  
    + 
    \msum_{i \in \bI_S} \mfrac{ S_{ik} q_k(\lambda)}{
        1
        + 
        \sum_{j \in \bJ_S} S_{ij}\big( u_j(\lambda) +  \frac{r_j(0)}{\lambda}\big)
        +
        \sum_{j \in \bJ_S^c} S_{ij} q_j(\lambda)
    }
    \\
    &\;
    + 
    \msum_{i \in \bI_S^c} \mfrac{ S_{ik} q_k(\lambda)}{
        1 
        +
        \sum_{j \in \bJ_S^c} S_{ij} q_j(\lambda)
    }
    \;.
\end{align*}
Taking $\lambda \rightarrow 0^+$ and setting $v^0_k = \partial_\lambda r_k(0)$ gives \eqref{eq:partial-r-residual}. Now fix any non-empty $J' \subseteq \bJ_S^c$ and sum \eqref{eq:partial-r-residual} over $k \in J'$, which gives 
\begin{align*}
    |J'| 
    \;=\; 
    \msum_{i \in \bI_S^c}
        \mfrac{
            \sum_{j \in J'} S_{ij} v_j^0
        }{
            1+\sum_{j'\in \bJ_S^c}S_{ij'}v_{j'}^0
        }
    \;<\; 
    \big| I_{J'}(\bS^{\rm res}) \big| \;,
\end{align*}
where the inequality follows because the summands vanish for
$i\notin I_{J'}(\bS^{\rm res})$ and are smaller than $1$
otherwise. This is exactly the statement that the residual graph $\bS^{\rm res}$ satisfies the column-side strong Hall property.

\vspace{.5em}

Conversely, assume that  $\bS^{\rm res}$ satisfies the column-side strong Hall property. Consider the convex function in $\by = (y_k)_{k \in \bJ^c_S}$:
\begin{align*}
    \Psi(\by) \;\coloneqq\; \msum_{i \in \bI^c_S} \log \Big( 1 + \msum_{k \in \bJ_S^c} S_{ik} e^{y_k} \Big) - \msum_{k \in \bJ_S^c} y_k\;.
\end{align*}
Notice that \eqref{eq:partial-r-residual} is exactly the first-order optimality condition of $\Psi$ by identifying $v^0_k = e^{y_k}$, so if $\Psi$ has a finite minimizer, \eqref{eq:partial-r-residual} has a finite solution, {and once convergence of $\bq(\lambda)$ and uniqueness of this positive solution
are established, this solution is exactly
$(\partial_\lambda r_k(0))_{k\in\bJ_S^c}$}. To show that the convex function $\Psi$ has a finite minimizer, it suffices to prove the coercivity of $\Psi$, i.e.~{$\Psi(\by)\to\infty$ as $\|\by\|\to\infty$.} 
{
For $t\geq0$, let $J_t\coloneqq\{k\in\bJ_S^c:y_k>t\}$ and
$m_i\coloneqq\max(\{0\}\cup\{y_k:S_{ik}>0\})$ for $i\in\bI_S^c$.
Since $\bS^{\rm res}$ has finitely many positive entries, there exists
a constant $C<\infty$ such that
\[
    \Psi(\by)
    \geq
    \msum_{i\in\bI_S^c}m_i
    -
    \msum_{k\in\bJ_S^c}y_k
    -C.
\]
Writing $y_k^+=\max\{y_k,0\}$ and $y_k^-=\max\{-y_k,0\}$, we have
\[
    \msum_{i\in\bI_S^c}m_i-\msum_{k\in\bJ_S^c}y_k^+
    =
    \int_0^\infty
    \big(
        |I_{J_t}(\bS^{\rm res})|-|J_t|
    \big)\,dt.
\]
By the column-side strong Hall property, the integrand is at least $1$
whenever $J_t\neq\varnothing$. Hence
\[
    \Psi(\by)
    \geq
    \max_{k\in\bJ_S^c}y_k^+
    +
    \msum_{k\in\bJ_S^c}y_k^-
    -C,
\]
which tends to infinity as $\|\by\|\to\infty$.}
{
Therefore \eqref{eq:partial-r-residual} admits a finite positive solution $\bv^0$.
}

\vspace{.5em}

 We {next} show that if {$\bq(\lambda)\to\bv^0$ as $\lambda\downarrow0$ and} $\partial r_j(0)$ exists for every $j \in \bJ_S$, {$(\partial_\lambda r_j(0))_{j\in\bJ_S}$} has to satisfy \eqref{eq:partial-r-support}. For $l=j \in \bJ_S$, {since $S_{ij}=0$ for every $i\in\bI_S^c$, 
equation \eqref{eq:r:lamb:system} can be written as
$
    1
    \;=\;
    r_j(\lambda)
    +
    \msum_{i\in\bI_S}
    \mfrac{S_{ij}r_j(\lambda)}
    {\lambda+\sum_{j'=1}^p S_{ij'}r_{j'}(\lambda)}
    \;.$
} We take a difference between the equation \eqref{eq:r:lamb:system}  for a general $\lambda$ and the same equation with $\lambda=0$, followed by a division by $\lambda$:
\begin{align*}
    0
    \;=&\; 
    \mfrac{r_j(\lambda) - r_j(0)}{\lambda}
    +
    \msum_{i\in{\bI_S}}
    \mfrac{1}{\lambda}
    \bigg(
        \mfrac{S_{ij} r_j(\lambda)}{\lambda + \sum_{j'=1}^p S_{ij'} r_{j'}(\lambda)}
        -
        \mfrac{S_{ij} r_j(0)}{\sum_{j'=1}^p S_{ij'} r_{j'}(0)}
    \bigg) 
    \\
    \;=&\;  
    u_j(\lambda)
    +
    \msum_{i\in{\bI_S}} 
    \mfrac{1}{\lambda}
    \bigg(
        \mfrac{
            S_{ij} \big(\lambda u_j(\lambda) + r_j(0)\big) \sum_{j'=1}^p S_{ij'} r_{j'}(0)
            -
            S_{ij} r_j(0) \big( \lambda + \sum_{j'=1}^p S_{ij'} r_{j'}(\lambda) \big)
        }{ 
            \big( \lambda + \sum_{j'=1}^p S_{ij'} r_{j'}(\lambda) \big) \, \sum_{j'=1}^p S_{ij'} r_{j'}(0)
        }
    \bigg) 
    \\
    \;=&\;  
    u_j(\lambda)
    +
    \msum_{i\in{\bI_S}}
    \mfrac{1}{\lambda}
    \mfrac{1}{ \big( \lambda + \sum_{j'=1}^p S_{ij'} r_{j'}(\lambda) \big) \, \sum_{j'  \in \bJ_S} S_{ij'} r_{j'}(0)}
    \\
    &\hspace{5em}
    \times
    \Big(
        S_{ij} \big(\lambda u_j(\lambda) + r_j(0)\big) \msum_{j' \in \bJ_S} S_{ij'} r_{j'}(0)
    \\
    &\hspace{7em}
        -
        S_{ij} r_j(0) \Big( 
            \lambda 
            + 
            \msum_{j' \in \bJ_S} S_{ij'} \big( \lambda u_{j'}(\lambda) + r_{j'}(0)  \big)
    \\
    &\hspace{7em}
            + 
            \lambda \msum_{j' \in \bJ_S^c} S_{ij'} q_{j'}(\lambda) \Big)
    \Big) 
    \\
    \;=&\;  
    u_j(\lambda)
    +
    \msum_{i\in{\bI_S}}
    \mfrac{1}{ \big( \lambda + \sum_{j'=1}^p S_{ij'} r_{j'}(\lambda) \big) \, \sum_{j'  \in \bJ_S} S_{ij'} r_{j'}(0)}
    \\
    &\hspace{5em}
    \times
    \Big(
        S_{ij} u_j(\lambda) \msum_{j' \in \bJ_S} S_{ij'} r_{j'}(0)
    \\
    &\hspace{7em}
        -
        S_{ij} r_j(0) \Big( 
            1
            + 
            \msum_{j' \in \bJ_S} S_{ij'} u_{j'}(\lambda) 
            + 
            \msum_{j' \in \bJ_S^c} S_{ij'} q_{j'}(\lambda) \Big)
    \Big) 
    \;.
\end{align*}
{
Assuming that $\bq(\lambda)\to\bv^0$ as $\lambda\downarrow0$ and that
$u_j(\lambda)\to u_j(0)=\partial_\lambda r_j(0)$ for every $j\in\bJ_S$,
} we can take  $\lambda \rightarrow 0^+$ to get
\begin{align*}
    0
    \;=&\;  
    u_j(0)
    +
    \msum_{i\in{\bI_S}} 
    \mfrac{S_{ij}}{ \big( \sum_{j'=1}^p S_{ij'} r_{j'}(0) \big)^2}
    \\
    &\hspace{5em}
    \times
    \Big(
        u_j(0) \msum_{j' \in \bJ_S} S_{ij'} r_{j'}(0)
    \\
    &\hspace{7em}
        -
         r_j(0) \Big( 
            1
            + 
            \msum_{j' \in \bJ_S} S_{ij'} u_{j'}(0) 
            + 
            \msum_{j' \in \bJ_S^c} S_{ij'} {v_{j'}^0} \Big)
    \Big) 
    \;.
\end{align*}
By noting that $\bJ_S$ is the support of $\br(0)$ (\Cref{thm:bias-support}) and recalling the definition of $\bI_S$, the above rearranges to 
\begin{align*}
    &\;
    u_j(0) \Big( 1 + \msum_{i \in \bI_S} \mfrac{S_{ij}}{\sum_{j'=1}^p S_{ij'} r_{j'}(0)  }  \Big)
    \\
    &\;=\;
    \msum_{i \in \bI_S} 
    \mfrac{S_{ij} \, r_j(0)}{ \big( \sum_{j' \in \bJ_S} S_{ij'} r_{j'}(0) \big)^2}
     \Big( 
            1
            + 
            \msum_{j' \in \bJ_S} S_{ij'} u_{j'}(0) 
            + 
            \msum_{j' \in \bJ_S^c} S_{ij'} {v_{j'}^0} 
    \Big) 
    \;.
\end{align*}
This is exactly \eqref{eq:partial-r-support} with $u_j^0 = u_j(0)$ by noting that 
\eqref{eq:r:lamb:system} with $\lambda = 0$ implies, for $j \in \bJ_S$,
\begin{align*}
    1 \;=\; r_j(0) \Big( 1 + \msum_{i \in \bI_S} \mfrac{S_{ij}}{\sum_{j' \in \bJ_S} S_{ij'} r_{j'}(0)}\Big)
    \;.
    \tagaligneq \label{eq:r:zero:system}
\end{align*} 

\vspace{.5em}

{We next} show that \eqref{eq:partial-r-support} has a solution. For all $j \in \bJ_S$, since $r_j(0) > 0$ for all $j \in \bJ_S$, we may consider the variables $h^0_j \coloneqq \frac{u^0_j}{r_j(0)}$ and rewrite \eqref{eq:partial-r-support} into the linear system 
\begin{align*}
    \msum_{j' \in \bJ_S} H_{jj'} h^0_{j'} \;=\;  b_j
    \qquad\qquad 
    \text{ for } j \in \bJ_S\;,
\end{align*}
where 
\begin{align*}
    H_{jj'} 
    \;\coloneqq&\;  
    \ind_{\{j=j'\}} 
    \,-\, 
    \msum_{i \in \bI_S} \mfrac{S_{ij}  r_j(0) \,\times\, S_{ij'} r_{j'}(0) }{\big( \sum_{j'' \in \bJ_S} S_{ij''} r_{j''}(0) \big)^2}
    \;,
    \\
    b_j 
    \;\coloneqq&\;
    \msum_{i \in \bI_S} \mfrac{S_{ij}  r_j(0) }{\big( \sum_{j'' \in \bJ_S} S_{ij''} r_{j''}(0) \big)^2}
    \Big( 
        1 
        +
        \msum_{k \in \bJ_S^c} S_{ik} v^0_k
    \Big)
    \;.
\end{align*}
Now note that, for any vector $\bz = (z_j)_{j \in \bJ_S}$, the operator $\bH =(H_{jj'})_{j,j' \in \bJ_S}$ satisfies
\begin{align*}
    \bz^\top \bH \bz
    \;=&\;
    \msum_{j \in \bJ_S}
    z_j^2
    -
    \msum_{i \in \bI_S} \mfrac{\big(\sum_{j \in \bJ_S} S_{ij}  r_j(0) z_j \big)^2 }{\big( \sum_{j'' \in \bJ_S} S_{ij''} r_{j''}(0) \big)^2}
    \\
    \;\overset{\eqref{eq:r:zero:system}}{=}&\;
    \msum_{j \in \bJ_S}
    z_j^2 r_j(0) \Big( 
        1 + \msum_{i \in \bI_S} \mfrac{S_{ij}}{\sum_{j' \in \bJ_S} S_{ij'} r_{j'}(0)}    
    \Big)
    -
    \msum_{i \in \bI_S} \mfrac{\big(\sum_{j \in \bJ_S} S_{ij}  r_j(0) z_j \big)^2 }{\big( \sum_{j'' \in \bJ_S} S_{ij''} r_{j''}(0) \big)^2}
    \\
    \;=&\;
    \msum_{j \in \bJ_S}
    r_j(0) z_j^2 
    +
    \msum_{i \in \bI_S}
    \underbrace{\Big(
        \msum_{j \in \bJ_S}
        z_j^2   w_{ij}
        -
        \Big(
            \msum_{j \in \bJ_S} z_j w_{ij}
        \Big)^2
    \Big)
    }_{\eqqcolon V_i}
    \;,
\end{align*}
where we have denoted the weights 
\begin{align*}
    w_{ij} \;\coloneqq\; \mfrac{S_{ij} r_j(0)}{\sum_{j' \in \bJ_S} S_{ij'} r_{j'}(0)}    \;.
\end{align*}
Observe that $V_i$ is the variance of a discrete random variable that takes the value $z_j$ with probability $w_{ij}$, and is therefore non-negative. This implies 
\begin{align*}
    \bz^\top \bH \bz \;\geq\; \msum_{j \in \bJ_S}
    r_j(0) z_j^2 
    \;\geq\; 0
\end{align*}
with equality if and only if $\bz = \bzero$ since $\bJ_S$ is the support of $\br(0)$ by \Cref{thm:bias-support}. This proves that $\bH$ is a positive definite operator and therefore invertible. 
This in turn shows that, if $\bS^{\rm res}$ satisfies the column-side strong Hall property, then \eqref{eq:partial-r-support} has a {unique} solution{It remains to show that $\bq(\lambda)\to\bv^0$ as $\lambda\downarrow0$ and that
$\lim_{\lambda\downarrow0}u_j(\lambda)$ exists for every $j\in\bJ_S$.}

{
To show this, we reparametrize by writing
\begin{align*}
    r_j = a_j, \quad j\in\bJ_S,
    \qquad\text{and}\qquad
    r_k = \lambda q_k, \quad k\in\bJ_S^c.
\end{align*}
For $i\in\bI_S$, define
\begin{align*}
    D_i(\lambda,\ba,\bq)
    \;\coloneqq\;
    \msum_{j\in\bJ_S} S_{ij}a_j
    +
    \lambda\Big(
        1+\msum_{k\in\bJ_S^c}S_{ik}q_k
    \Big).
\end{align*}
Then, for $\lambda>0$, \eqref{eq:r:lamb:system} can be written as
\begin{align}
    0
    \;=&\;
    a_j
    +
    \msum_{i\in\bI_S}
    \mfrac{S_{ij}a_j}{D_i(\lambda,\ba,\bq)}
    -1,
    \qquad j\in\bJ_S,
    \label{eq:ift-support}
    \\
    0
    \;=&\;
    \lambda q_k
    +
    \lambda\msum_{i\in\bI_S}
    \mfrac{S_{ik}q_k}{D_i(\lambda,\ba,\bq)}
    +
    \msum_{i\in\bI_S^c}
    \mfrac{S_{ik}q_k}
    {1+\sum_{h\in\bJ_S^c}S_{ih}q_h}
    -1,
    \qquad k\in\bJ_S^c.
    \label{eq:ift-residual}
\end{align}
}
{
Let $F_j(\lambda,\ba,\bq)$ and $G_k(\lambda,\ba,\bq)$ denote the right-hand sides of
\eqref{eq:ift-support} and \eqref{eq:ift-residual}, respectively, and let
$\ba^0\coloneqq (r_j(0))_{j\in\bJ_S}$.
At $\lambda=0$, \eqref{eq:ift-support} reduces to the fixed-point equation for
$\ba^0$, while \eqref{eq:ift-residual} reduces to
\eqref{eq:partial-r-residual}, which is solved by $\bv^0$.
Hence
\[
    \bF(0,\ba^0,\bv^0)=\bzero,
    \qquad
    \bG(0,\ba^0,\bv^0)=\bzero.
\]
Moreover, all the denominators in \eqref{eq:ift-support} and
\eqref{eq:ift-residual} are strictly positive at
$(0,\ba^0,\bv^0)$. Therefore, the system is smooth in a neighborhood of
$(0,\ba^0,\bv^0)$.
}
{ Since $\ba^0$ and $\bv^0$ are positive, do the change of variables $x_j=\log a_j$, $j\in\bJ_S$, and $y_k=\log q_k$, $k\in\bJ_S^c$. Let \[ \bx^0=(\log r_j(0))_{j\in\bJ_S}, \qquad \by^0=(\log v_k^0)_{k\in\bJ_S^c}. \] The Jacobian of $(\bF,\bG)$ with respect to $(\bx,\by)$ at $(0,\bx^0,\by^0)$ is \begin{align*} \begin{pmatrix} \bH & \bzero \\ \bzero & \nabla^2\Psi(\by^0) \end{pmatrix}. \end{align*} 
At $\lambda=0$, $\bF$ is independent of $\by$ and $\bG$ is independent of $\bx$, so the two off-diagonal blocks vanish. The two diagonal blocks are $\bH$ and $\nabla^2\Psi(\by^0)$, respectively, where
$\big[\nabla^2\Psi(\by^0)\big]_{kk'}
=
\ind_{\{k=k'\}}
\sum_{i\in\bI_S^c}
\frac{S_{ik}v_k^0}
{1+\sum_{h\in\bJ_S^c}S_{ih}v_h^0}
-
\sum_{i\in\bI_S^c}
\frac{S_{ik}v_k^0\,S_{ik'}v_{k'}^0}
{\big(1+\sum_{h\in\bJ_S^c}S_{ih}v_h^0\big)^2}$.}
{
Recall that we have already shown that $\bH$ is positive definite.
It remains to show that $\nabla^2\Psi(\by^0)$ is positive definite.
}
{
For $i\in\bI_S^c$ and $k\in\bJ_S^c$, let
\begin{align*}
    p_{ik}
    \;\coloneqq\;
    \mfrac{S_{ik}v_k^0}
    {1+\sum_{h\in\bJ_S^c}S_{ih}v_h^0},
    \qquad
    p_{i0}
    \;\coloneqq\;
    \mfrac{1}
    {1+\sum_{h\in\bJ_S^c}S_{ih}v_h^0}.
\end{align*}
Then, for any $\bz=(z_k)_{k\in\bJ_S^c}$,
\begin{align*}
    \bz^\top\nabla^2\Psi(\by^0)\bz
    &=
    \msum_{i\in\bI_S^c}
    \left(
        \msum_{k\in\bJ_S^c}p_{ik}z_k^2
        -
        \Big(
            \msum_{k\in\bJ_S^c}p_{ik}z_k
        \Big)^2
    \right)
    \\
    &\geq
    \msum_{i\in\bI_S^c}
    p_{i0}
    \msum_{k\in\bJ_S^c}p_{ik}z_k^2.
\end{align*}
Since $p_{i0}>0$ and every column of $\bS^{\rm res}$ has a neighbor,
the last expression is strictly positive for $\bz\neq\bzero$.
Therefore, $\nabla^2\Psi(\by^0)$ is positive definite.
}
{
Together with the positive definiteness of $\bH$, this shows that the Jacobian of $(\bF,\bG)$ is invertible. By the implicit function theorem, there
exist smooth functions $\ba(\lambda)$ and $\bq(\lambda)$ in a
neighborhood of $\lambda=0$, with
$\ba(0)=\ba^0$ and $\bq(0)=\bv^0$, satisfying
\eqref{eq:ift-support} and \eqref{eq:ift-residual}.
Since $\ba^0$ and $\bv^0$ are positive, $\ba(\lambda)$ and
$\bq(\lambda)$ remain positive for all sufficiently small $\lambda>0$.
Define
\[
    \widetilde r_j(\lambda)=a_j(\lambda),\quad j\in\bJ_S,
    \qquad
    \widetilde r_k(\lambda)=\lambda q_k(\lambda),\quad k\in\bJ_S^c.
\]
Then $\widetilde{\br}(\lambda)$ is a positive solution of
\eqref{eq:r:lamb:system}. By uniqueness of the solution to
\eqref{eq:r:lamb:system}, $\widetilde{\br}(\lambda)=\br(\lambda)$.
Therefore, as $\lambda\downarrow0$,
\[
    \bq(\lambda)\longrightarrow\bv^0,
    \qquad
    u_j(\lambda)
    =
    \frac{a_j(\lambda)-a_j(0)}{\lambda}
    \longrightarrow a_j'(0),
    \quad j\in\bJ_S.
\]
Thus $\partial_\lambda\br(0)$ exists and is finite and satisfies
\eqref{eq:partial-r-residual} and \eqref{eq:partial-r-support}.
This finishes the proof.
}
\end{proof}

\subsection{Two-group profiles}
\label{sec:two-group-alt}

We first give the matching calculations for the profile with overlapping
column sets in \Cref{ex:two-group-overlapping}. We also consider the case in which the two column sets are disjoint.

\emph{The profile with overlapping column sets.}
The coordinates in $\cV_0$ are isolated and therefore belong to $\bJ_S$.
It remains to determine which columns in $\cV_1\cup\cV_2$ can be left
unmatched by a maximum matching. Note that
after discarding $\cV_0$, \Cref{asst:flat}
is straightforward to verify. For any nonempty column set
$T\subseteq\cV_1\cup\cV_2$, its set of neighboring rows satisfies
\begin{align}
    |I_T(S)|
    =
    \begin{cases}
        n_1,
        & T\subseteq\cV_1\setminus\cV_2,\\
        n_2,
        & T\subseteq\cV_2\setminus\cV_1,\\
        n_1+n_2,
        & \text{otherwise}.
    \end{cases}
    \label{eq:two-group-neighborhoods}
\end{align}

We first verify the possible $\bJ_S$  for different cases in \Cref{ex:two-group-overlapping}.

\emph{Case (i).}
Suppose
\begin{align*}
    p_1-p_{12}\le n_1,
    \qquad
    p_2-p_{12}\le n_2,
    \qquad
    p_1+p_2-p_{12}\le n_1+n_2.
\end{align*}
Then \eqref{eq:two-group-neighborhoods} gives
$|T|\le |I_T(S)|$ for every
$T\subseteq\cV_1\cup\cV_2$. Hall's theorem therefore gives a matching
that matches every column in $\cV_1\cup\cV_2$. Since the remaining
columns are isolated, every maximum matching matches every column in
$\cV_1\cup\cV_2$, and
\begin{align*}
    \bJ_S=\cV_0.
\end{align*}

\emph{Case (ii).}
Suppose
\begin{align*}
    p_1-p_{12}>n_1,
    \qquad
    p_2\le n_2.
\end{align*}
The columns in $\cV_1\setminus\cV_2$ are adjacent only to the first
group and therefore cannot all be matched. On the other hand, one can
match $n_1$ columns in $\cV_1\setminus\cV_2$ to the first group and
all $p_2$ columns in $\cV_2$ to the second group. This matching has
size $n_1+p_2$, which is maximal since any matching uses at most
$n_1$ columns from $\cV_1\setminus\cV_2$ and at most $p_2$ columns
from $\cV_2$. The columns in $\cV_1\setminus\cV_2$ have the same
neighborhood, so each can be left unmatched by a maximum matching.
Every maximum matching must match all columns in $\cV_2$, and hence
\begin{align*}
    \bJ_S
    =
    \cV_0\cup(\cV_1\setminus\cV_2).
\end{align*}

\emph{Case (iii).}
The same argument with the two groups exchanged shows that, if
\begin{align*}
    p_2-p_{12}>n_2,
    \qquad
    p_1\le n_1,
\end{align*}
then
\begin{align*}
    \bJ_S
    =
    \cV_0\cup(\cV_2\setminus\cV_1).
\end{align*}

\emph{Case (iv).}
Suppose
\begin{align*}
    p_1>n_1,
    \qquad
    p_2>n_2,
    \qquad
    p_1+p_2-p_{12}>n_1+n_2.
\end{align*}
Hall's condition on the row side holds: a set of rows contained in the
first group has neighborhood $\cV_1$, a set contained in the second
group has neighborhood $\cV_2$, and a set meeting both groups has
neighborhood $\cV_1\cup\cV_2$. Thus there is a matching that matches
all $n_1+n_2$ rows. Moreover, after deleting any one column in
$\cV_1\cup\cV_2$, the remaining neighborhoods of the first group,
the second group, and both groups together still have sizes at least
$n_1$, $n_2$, and $n_1+n_2$, respectively. Thus the resulting graph
still has a matching that matches every row.
Consequently every column in $\cV_1\cup\cV_2$ can be left unmatched by
a maximum matching, and
\begin{align*}
    \bJ_S
    =
    \cV_0\cup\cV_1\cup\cV_2.
\end{align*}

Note that these are all the possible cases. If $p_1-p_{12}>n_1$, then either
$p_2\le n_2$, giving case (ii), or $p_2>n_2$, giving case (iv). The same
argument applies with the two groups exchanged. If neither private column set
exceeds its corresponding row group, then either
$p_1+p_2-p_{12}\le n_1+n_2$, giving case (i), or the reverse inequality holds,
which gives case (iv).

We next verify the switching classification in
\Cref{ex:two-group-peaks}. Suppose $n_1=n_2=n/2$. By
\Cref{rem:decomposable}, the isolated coordinates $\cV_0$ can be
discarded when considering the switching condition.

If $\bJ_S=\cV_0$, the residual graph is the full graph on
$\cV_1\cup\cV_2$. By \eqref{eq:two-group-neighborhoods}, a nonempty
column set has $n/2$ neighboring rows if it is contained in
$\cV_1\setminus\cV_2$ or $\cV_2\setminus\cV_1$, and $n$
neighboring rows otherwise. Since all
columns in $\cV_1\cup\cV_2$ can be matched, the column-side strong Hall
property fails exactly when
\begin{align*}
    |\cV_1\setminus\cV_2|=n/2,
    \qquad\text{or}\qquad
    |\cV_2\setminus\cV_1|=n/2,
    \qquad\text{or}\qquad
    |\cV_1\cup\cV_2|=n.
\end{align*}
In the first case, the inequalities defining
$\bJ_S=\cV_0$ imply $p_2\le n/2$; in the second they imply
$p_1\le n/2$; and in the third they imply
$p_1-p_{12}\le n/2$ and $p_2-p_{12}\le n/2$. These are cases
(i)--(iii) of \Cref{ex:two-group-peaks}.

If
\begin{align*}
    \bJ_S
    =
    \cV_0\cup(\cV_1\setminus\cV_2),
\end{align*}
removing $\bJ_S$ and its neighboring rows removes the first group and
leaves the complete bipartite graph between the second group and
$\cV_2$. In this regime $p_2\le n/2$, so this residual graph satisfies
the switching condition exactly when $p_2=n/2$. This gives case (iv).
Similarly, if
\begin{align*}
    \bJ_S
    =
    \cV_0\cup(\cV_2\setminus\cV_1),
\end{align*}
the switching condition holds exactly when $p_1=n/2$, which gives
case (v). Finally, if
\begin{align*}
    \bJ_S
    =
    \cV_0\cup\cV_1\cup\cV_2,
\end{align*}
the residual graph is empty and the switching condition does not hold.
This proves the five cases in \Cref{ex:two-group-peaks}.

\emph{The profile with disjoint column sets.}
We next consider the case in which the two column sets are disjoint.

\begin{example}[Two-group profile with disjoint column sets]
\label{ex:two-group-disjoint}
Consider
\begin{align*}
    S
    \;=&\;
    \begin{psmallmatrix}
        \bone_{n_1\times p_1} & 0 & 0 \\
        0 & \bone_{n_2\times p_2} & 0
    \end{psmallmatrix}
    \;.
\end{align*}
Let
\begin{align*}
    \cV_1
    &\coloneqq \{1,\ldots,p_1\},\\
    \cV_2
    &\coloneqq \{p_1+1,\ldots,p_1+p_2\},\\
    \cV_0
    &\coloneqq \cV_{\rm col}\setminus(\cV_1\cup\cV_2).
\end{align*}
The variance graph is the disjoint union of the complete bipartite
graphs $K_{n_1,p_1}$ and $K_{n_2,p_2}$, together with the isolated
columns in $\cV_0$. Hence every maximum matching is the union of
maximum matchings in the two complete blocks, and
\begin{align*}
    |\cE_S^{\rm max}|
    =
    \min\{n_1,p_1\}+\min\{n_2,p_2\}.
\end{align*}
Since the columns within each block have the same neighborhood,
\begin{align*}
    \bJ_S
    =
    \cV_0
    \cup
    \big(\cV_1\,\ind_{\{p_1>n_1\}}\big)
    \cup
    \big(\cV_2\,\ind_{\{p_2>n_2\}}\big).
\end{align*}
By \Cref{thm:bias-support}, this is also the support of $\br(0)$.

Now suppose $n_1=n_2=n/2$. By \Cref{rem:decomposable}, the two
complete blocks are treated separately. For the block
$K_{n/2,p_a}$, the column-side strong Hall property holds when
$p_a<n/2$, fails when $p_a=n/2$, and the whole column block belongs to
$\bJ_S$ when $p_a>n/2$. Therefore the switching condition holds
exactly when
\begin{align*}
    p_1=n/2
    \qquad\text{or}\qquad
    p_2=n/2.
\end{align*}
\end{example}

\subsection{Asymptotic variance for replicated profiles}
\label{app:prototype-variance}

We now prove \Cref{cor:prototype-variance}. We first prove two properties of
the sequence $S^{(t)}$ defined in \eqref{eqn:prototype_definition}. We begin by
showing that if \Cref{asst:flat} holds for $\widetilde S$, then it also holds
for $S^{(t)}$ with constants independent of $t$.

\begin{lemma}\label{lem:prototype-flat}
Suppose $\widetilde S\in\R_+^{m\times k}$ satisfies \Cref{asst:flat}
with constants $s_*,\psi_1,\psi_2,L_1,L_2$. Then $S^{(t)}$ satisfies
\Cref{asst:flat} with the same constants for every $t\in\N$.
\end{lemma}

\begin{proof}
We verify the three conditions in \Cref{asst:flat}. Since $\widetilde S$
satisfies \Cref{asst:flat},
\begin{align*}
    \widetilde S_{ab}\leq \frac{s_*}{m+k},
    \qquad a\in[m],\ b\in[k].
\end{align*}
Thus every entry of $S^{(t)}$ satisfies
\begin{align*}
    S^{(t)}_{ij}
    \leq \frac{1}{t}\frac{s_*}{m+k}
    = \frac{s_*}{n_t+p_t}.
\end{align*}

For the lower bounds, the mixed-product property of the Kronecker product and
$(\bone_t\bone_t^\top)^L=t^{L-1}\bone_t\bone_t^\top$ give
\begin{align*}
    \big(S^{(t)}(S^{(t)})^\top\big)^L
    &=\frac1t
      \big(\widetilde S\widetilde S^\top\big)^L
      \otimes(\bone_t\bone_t^\top),\\
    \big((S^{(t)})^\top S^{(t)}\big)^L
    &=\frac1t
      \big(\widetilde S^\top\widetilde S\big)^L
      \otimes(\bone_t\bone_t^\top).
\end{align*}
Hence
\begin{align*}
    \big[\big(S^{(t)}(S^{(t)})^\top\big)^{L_1}\big]_{ii'}
    &\geq
      \frac1t\frac{\psi_1}{m+k}
      =\frac{\psi_1}{n_t+p_t},\\
    \big[\big((S^{(t)})^\top S^{(t)}\big)^{L_2}\big]_{jj'}
    &\geq
      \frac1t\frac{\psi_2}{m+k}
      =\frac{\psi_2}{n_t+p_t}.
\end{align*}
This proves the claim.
\end{proof}

We next show that the fixed-point solution for $S^{(t)}$ is obtained by
replicating the fixed-point solution for $\widetilde S$.
\begin{lemma}\label{lem:prototype-fixed-point}
For $\lambda>0$, let
$\widetilde{\br}(\lambda)
=(\widetilde r_b(\lambda))_{b\in[k]}$
be the unique solution to
\begin{align}
    1
    &=
    \widetilde r_b(\lambda)
    +
    \sum_{a=1}^m
    \frac{
        \widetilde S_{ab}\widetilde r_b(\lambda)
    }{
        \lambda+\sum_{c=1}^k
        \widetilde S_{ac}\widetilde r_c(\lambda)
    },
    \qquad b\in[k].
    \label{eq:r:prototype}
\end{align}
For every $t\in\N$, let
$\br^{(t)}(\lambda)\in\R^{p_t}$ denote the unique solution to
\eqref{eq:r-fixedpoint} corresponding to $S^{(t)}$. Then
\begin{align*}
    r^{(t)}_{(b-1)t+v}(\lambda)
    =
    \widetilde r_b(\lambda),
    \qquad b\in[k],\ v\in[t].
\end{align*}
\end{lemma}
\begin{proof}
The existence and uniqueness of $\widetilde{\br}(\lambda)$ follow from
\Cref{lem:r:eqns}, applied to $\widetilde S$. For fixed $t$, define
$\widehat{\br}^{(t)}(\lambda)\in\R^{p_t}$ by
\begin{align*}
    \widehat r^{(t)}_{(b-1)t+v}(\lambda)
    \coloneqq
    \widetilde r_b(\lambda),
    \qquad b\in[k],\ v\in[t].
\end{align*}
Index the rows of $S^{(t)}$ by $(a-1)t+u$, with $a\in[m]$ and $u\in[t]$. By
\eqref{eqn:prototype_definition},
\begin{align*}
    S^{(t)}_{(a-1)t+u,\,(b-1)t+v}
    =\frac1t\widetilde S_{ab}.
\end{align*}
Hence
\begin{align*}
    \sum_{c=1}^k\sum_{w=1}^t
    S^{(t)}_{(a-1)t+u,\,(c-1)t+w}
    \widehat r^{(t)}_{(c-1)t+w}(\lambda)
    =
    \sum_{c=1}^k
    \widetilde S_{ac}\widetilde r_c(\lambda).
\end{align*}
Therefore, for every $b\in[k]$ and $v\in[t]$,
\begin{align*}
    &\widehat r^{(t)}_{(b-1)t+v}(\lambda)
    +
    \sum_{a=1}^m\sum_{u=1}^t
    \frac{
        S^{(t)}_{(a-1)t+u,\,(b-1)t+v}
        \widehat r^{(t)}_{(b-1)t+v}(\lambda)
    }{
        \lambda+
        \sum_{c=1}^k\sum_{w=1}^t
        S^{(t)}_{(a-1)t+u,\,(c-1)t+w}
        \widehat r^{(t)}_{(c-1)t+w}(\lambda)
    }\\
    &\qquad=
    \widetilde r_b(\lambda)
    +
    \sum_{a=1}^m
    \frac{
        \widetilde S_{ab}\widetilde r_b(\lambda)
    }{
        \lambda+
        \sum_{c=1}^k
        \widetilde S_{ac}\widetilde r_c(\lambda)
    }
    =1.
\end{align*}
Thus $\widehat{\br}^{(t)}(\lambda)$ solves \eqref{eq:r-fixedpoint} for
$S^{(t)}$. The result follows from uniqueness in \Cref{lem:r:eqns}.
\end{proof}

Next, we express the degenerate column set and the switching condition for $S^{(t)}$ in terms of those of the prototype $\widetilde S$.
\begin{corollary}\label{cor:prototype-graph}
Let $\bJ_{\widetilde S}$ and $\bJ_{S^{(t)}}$ denote the degenerate column
sets of $\widetilde S$ and $S^{(t)}$, respectively. Then
\begin{align*}
    \bJ_{S^{(t)}}
    =
    \{(b-1)t+v:b\in\bJ_{\widetilde S},\ v\in[t]\}.
\end{align*}
Moreover, $S^{(t)}$ satisfies the switching condition if and only if
$\widetilde S$ does.
\end{corollary}

\begin{proof}
By \Cref{lem:prototype-fixed-point},
\begin{align*}
    r^{(t)}_{(b-1)t+v}(\lambda)
    =
    \widetilde r_b(\lambda),
    \qquad
    b\in[k],\ v\in[t],\ \lambda>0.
\end{align*}
Taking $\lambda\downarrow0$ and by \Cref{thm:bias-support},
$(b-1)t+v\in\bJ_{S^{(t)}}$ exactly when
$r^{(t)}_{(b-1)t+v}(0)>0$,
which is exactly when
$\widetilde r_b(0)>0$, or equivalently, when
$b\in\bJ_{\widetilde S}$. Therefore,
\begin{align*}
    \bJ_{S^{(t)}}
    =
    \{(b-1)t+v:b\in\bJ_{\widetilde S},\ v\in[t]\}.
\end{align*}
Similarly,
\begin{align*}
    \partial_\lambda r^{(t)}_{(b-1)t+v}(0)
    =
    \partial_\lambda\widetilde r_b(0),
\end{align*}
so $\partial\br^{(t)}(0)$ is finite if and only if
$\partial\widetilde{\br}(0)$ is finite. The statement about switching
follows from \Cref{thm:var-switch}.
\end{proof}

We can now combine these results with \Cref{thm:master} to prove
\Cref{cor:prototype-variance}.
\begin{proof}[Proof of \Cref{cor:prototype-variance}]
By \Cref{lem:prototype-flat}, $S^{(t)}$ satisfies \Cref{asst:flat}
uniformly in $t$, and
\begin{align*}
    \frac{p_t}{n_t}=\frac{k}{m}.
\end{align*}
Hence \Cref{thm:master} gives
\begin{align*}
    \cR_V^{\lambda_t}(X^{(t)})
    -
    \frac{\sigma_\epsilon^2}{n_t}
    \sum_{l=1}^{p_t}
    \bar\Sigma^{(t)}_{ll}
    \partial_\lambda r_l^{(t)}(\lambda_t)
    \xrightarrow{\mathrm{a.s.}}0.
\end{align*}

Indexing the columns by $(b-1)t+v$, with $b\in[k]$ and $v\in[t]$,
\begin{align*}
    \bar\Sigma^{(t)}_{(b-1)t+v,\,(b-1)t+v}
    &=
    \sum_{a=1}^m\sum_{u=1}^t
    S^{(t)}_{(a-1)t+u,\,(b-1)t+v}\\
    &=
    \sum_{a=1}^m\widetilde S_{ab}.
\end{align*}
Together with \Cref{lem:prototype-fixed-point}, this gives
\begin{align*}
    \frac1{n_t}
    \sum_{l=1}^{p_t}
    \bar\Sigma^{(t)}_{ll}
    \partial_\lambda r_l^{(t)}(\lambda_t)
    &=
    \frac1m
    \sum_{b=1}^k
    \left(\sum_{a=1}^m\widetilde S_{ab}\right)
    \partial_\lambda\widetilde r_b(\lambda_t).
\end{align*}

If the switching condition fails, \Cref{thm:var-switch} gives
$\partial_\lambda\widetilde{\br}(0)\in\R^k$, so the right-hand side
converges to the finite limit in \Cref{cor:prototype-variance}. If the
switching condition holds, then
$\partial_\lambda\widetilde r_b(0)=+\infty$ for some $b\in[k]$.
Since
\begin{align*}
    \sum_{a=1}^m\widetilde S_{ab}>0,
\end{align*}
the right-hand side diverges. The result follows from
\Cref{thm:master}.
\end{proof}

\section{Local law for the variance profile model with explicit \texorpdfstring{$z$}{z}-dependence} \label{appendix:local:law}

The main goal of this section is to adapt the local law results in \cite{alt2017local} to
\begin{itemize}
    \item have an explicit tracking of $z$-dependence, and
    \item derive a bound that allows for $z$ to be taken close to zero. 
\end{itemize}
We write as shorthand 
\begin{align*}
    \Sf 
    \;\coloneqq\; 
    \begin{psmallmatrix} 
        0 & S \\ S^\top & 0 
    \end{psmallmatrix}
    \;\in\; 
    \R^{(n+p)\times(n+p)}
    \;.
\end{align*}
Denote the upper half-plane of the complex plane as $\bbH \coloneqq \{ z \in \C \,|\, \Im(z) > 0\}$. Under \Cref{mod:vp,asst:flat}, \cite{alt2017local} shows that the limiting spectrum of $X / \sqrt{n}$ can be completely described by a deterministic function $\fM: \bbH \rightarrow \bbH^{n+p}$ that is the unique solution of the quadratic vector equation (QVE),
\begin{align*}
    - \diag\{ \fM(z) \}^{-1} \;=\; z  \, I_{n+p} + 
    \diag\{
    \Sf \, \fM(z)
    \}
    \;,
    \qquad 
    z \in \bbH\;,
    \tagaligneq
    \label{eq:combined_QVE}
\end{align*}
where, for any generic $\R^b$ vector $\bv=(v_l)_{1 \leq l \leq b}$, we denote the diagonal matrix
\begin{align}
    \diag\{\bv\} \;=\; \begin{psmallmatrix}
        v_1 & & \\ 
        & \ddots & \\
        & & v_b
    \end{psmallmatrix}
    \;.
\end{align}
The existence and uniqueness of $\fM$ are guaranteed by Theorem 2.1 and (3.6) of  \cite{alt2017local}, which follows from applying Theorem 2.1 of \cite{Ajanki_2019} to quadratic vector equations of the form \eqref{eq:combined_QVE}. We also define the resolvent of the Hermitialized empirical data matrix as
\begin{align*}
    G(z) \;\coloneqq&\; (H - z I_{n+p})^{-1}\;,
    \qquad 
    \text{ where }
    H \;\coloneqq\; \begin{psmallmatrix}
        0 & \frac{1}{\sqrt{n}} X \\
        \frac{1}{\sqrt{n}} X^\top & 0
    \end{psmallmatrix}
    \;\text{ and }\;
    z \in \bbH\;,
    \tagaligneq \label{eq:empirical:resolvent}
\end{align*}
and define $\gf$ as the random $\Hb \to (\C\setminus\{0\})^{n+p}$ function (where the stochasticity comes from the data) as 
\begin{align*}
    \gf(z) \;\coloneqq\; \big( G_{11}(z) \;, \ldots\;, G_{(n+p)(n+p)}(z) \big)\;.
\end{align*}

\vspace{.5em}

We seek to approximate the random resolvent matrix $G(z)$ by the deterministic matrix $\diag\{\fM(z)\}$. Our results rely on adapting the analyses of the QVE \eqref{eq:combined_QVE} in \cite{alt2017local}. One main difference is that most of their tools hold for $z$ bounded away from $0$, whereas our risk analysis requires taking $z$ to $0$ as a suitable function of $n$ and $p$ to approximate the pseudo-inverse $(X^\top X)^\dagger$. On the other hand, we are allowed to take $z$ to $0$ along any chosen trajectory and at a much slower rate in $n$, which allows us to have looser bounds than those established in \cite{alt2017local}. 

\vspace{.5em}

To state our variant of their local law result, we follow the notation of \cite{ajanki2017universalitygeneralwignertypematrices} --- whose tools are heavily used in \cite{alt2017local} --- and consider the following definition of stochastic domination:

\begin{definition}[Stochastic domination, Definition 1.6 of \cite{ajanki2017universalitygeneralwignertypematrices}] \label{defn:sto:dom} For two sequences of $n$-indexed non-negative random variables $\varphi_n$ and $\psi_n$, we say that $\varphi_n$ is stochastically dominated by $\psi_n$, denoted as 
\begin{align*}
    \varphi_n \;\prec\; \psi_n\;,
\end{align*}
if for every $\xi >0$ and $\kappa > 0$, there exists $n_{\xi, \kappa} > 0$ that depends  only on the constants $s_*$, $L_1$, $L_2$, $\psi_1$ and $\psi_2$ in \Cref{asst:flat}, such that for all $n \geq n_{\xi, \kappa}$,
\begin{align*}
    \P\big( \varphi_n \;\geq\; n^\xi \psi_n \big) \;\leq\; n^{-\kappa}\;.
\end{align*}
We also write $\varphi_n \preceq \psi_n$ if the above holds also with $\xi = 0$.
\end{definition}
 
 Note that the definition of stochastic domination in \cite{ajanki2017universalitygeneralwignertypematrices} additionally requires $n_{\xi, \kappa}$ to depend on a small and fixed tolerance parameter $\tilde \gamma \in (0,1)$ under the additional constraint that $\Im(z) \geq \frac{1}{(n+p)^{1-\tilde \gamma}}$, and their result holds for any fixed $\tilde \gamma$. We do introduce a condition $\Im(z) \geq \frac{1}{n^\upsilon}$ with $\upsilon > 0$, similar to \cite{ajanki2017universalitygeneralwignertypematrices}, but we will make the dependence of $\upsilon$ explicit in our bounds. Specifically we focus on the domain
\begin{align*}
    \D_{\upsilon, K} \;\coloneqq\; 
    \Big\{ 
        z \in \bbH 
        \;\Big|\;
        \Im(z) \;\geq\; \mfrac{1}{n^\upsilon}\;,
        \quad 
        | z| \;\leq\; K
    \Big\}    
    \;.
\end{align*}

\begin{theorem}[Local law for $H$] \label{thm:local:law:VP} Under \Cref{mod:vp} and \Cref{asst:flat}, there exists  a sufficiently large $K > \sqrt{s_*}$ that depends only on the constants $s_*$, $L_1$, $L_2$, $\psi_1$ and $\psi_2$ in \Cref{asst:flat} {and $\upsilon$}, such that {for every fixed} $\upsilon \in (0, {\frac{1}{24}})$, we have
\begin{align*}
    \max_{1 \leq l, j \leq n+p}
    \big| 
        G_{l,j}(z)
        -
        \fM_l(z) \, \ind_{\{l = j\}}
    \big|
    \;\prec\;
     {
    n^{-5\upsilon}
    }
    \;\qquad\;
    \text{ for } z \in \D_{\upsilon, K}\;.
\end{align*}
Moreover, the off-diagonal approximation error satisfies a sharper bound that 
\begin{align*}
    \max_{l \neq j} | G_{lj}(z) | 
    \;\prec\; 
    n^{ 3\upsilon - \frac{1}{2} }
    \;\qquad\;
    \text{ for } z \in \D_{\upsilon, K}\;.
\end{align*}
\end{theorem}

The rest of this section is devoted to the proof of \Cref{thm:local:law:VP}. To state the intermediate results,  we first note that \cite{alt2017local} applies the technique of matrix Dyson equation and uses \eqref{eq:combined_QVE} to approximate the self-consistent equations satisfied by the resolvent of the empirical data matrix. Specifically, (3.4a) and (3.4b) of \cite{alt2017local} show that $\gf$ satisfies a perturbed QVE
\begin{align*} 
    - 
    \diag \{ \gf(z) \}^{-1} 
    \;=\; 
    z \, I_{n+p}
    + 
    \diag\{ \Sf\, \gf(z) \} 
    + 
    \diag\{ \df(z) \}, 
    \qquad z \in \bbH,
    \tagaligneq \label{eq:perturbed_combined_QVE}
\end{align*}
where we refer to \cite{alt2017local} as well as our \Cref{sec:diag:error:to:perturb} below for the definition of the random  $\Hb \to \C^{n+p}$ perturbation function $\df$ (where the stochasticity comes from the data). In the analyses of both the QVE \eqref{eq:combined_QVE} and the perturbed QVE \eqref{eq:perturbed_combined_QVE}, one needs to study the stability of the equations. Observe that \eqref{eq:combined_QVE} can be rearranged to the vector equation
\begin{align*}
    \bF(z) \;=\; 0\;,
    \qquad
    \text{ where }
    \bF(z)
    \;\coloneqq&\;
    \Big( \mfrac{1}{\fM_1(z)} \,,\, \cdots \,,\, \mfrac{1}{\fM_{n+p}(z)}  \Big)^\top
    +
    z \, \bone_{n+p}
    +
    \fS
    \, \fM(z)
    \;.
    \tagaligneq \label{eq:bF:defn}
\end{align*}
The stability of $\fM(z)$ under a perturbation of the equation $\bF(z)$ is measured by the vector derivative
\begin{align*}
    \mfrac{\partial \fM(z)}{\partial \bF(z)}
    \;=&\;
    \Big( \mfrac{\partial \bF(z)}{\partial \fM(z)} \Big)^{-1}
    \;=\;
    \big( 
        - \diag\{ \fM(z) \}^{-2}
        + \fS 
    \big)^{-1}
    \\
    \;=&\;
    -
    \diag\{ | \fM(z)| \}
    \,
    \fB^{-1}(z)
    \,
    \diag\{ | \fM(z)| \}
    \;,
\end{align*}
where we have denoted 
\begin{align*}
    | \fM(z)|
    \;\coloneqq&\;
        \big( 
            |\fM_1(z)|
            \,,\, \ldots \,,\,
            |\fM_{n+p}(z)|
        \big)^\top
    \;,
    \\
    \fB(z)
    \;\coloneqq&\;
    \diag\Big\{ 
            \mfrac{|\fM_1(z)|^2}{\fM_1(z)^2}
            \,,\, \ldots \,,\,
            \mfrac{|\fM_{n+p}(z)|^2}{\fM_{n+p}(z)^2}
    \Big\}
    -
    \diag\{ | \fM(z)| \} \, \fS  \, \diag\{ | \fM(z)| \}
    \;.
    \tagaligneq \label{eq:defn:fB}
\end{align*}
The analysis of the inverse $\fB^{-1}$ plays a crucial role in the proof of \cite{alt2017local}, and we will track the $z$-dependence in this analysis explicitly. The results in the rest of this section are organized as follows:
\begin{itemize}
    \item The results in \Cref{sec:control:fM} control the size of $\fM(z)$ as well as the size $\partial_z \fM(z)$ through $\fB^{-1}(z)$;
    \item \Cref{sec:control:stability} provides a result that controls $\fB^{-1}(z)$, which in turn control stability of $\fM(z)$ under a perturbation of the QVE. As a corollary, this gives an explicit control of $\partial_z \fM(z)$; 
    \item \Cref{sec:diag:error:to:perturb} provides a result that approximate the diagonal of the resolvent, $\gf(z)$, by the QVE solution $\fM(z)$, with errors given in terms of the QVE perturbation $\df(z)$;
    \item \Cref{sec:perturb:offdiag:bound} provides results that bound the size of the QVE perturbation $\df(z)$ as well as the size of the off-diagonal entries of the resolvent $G(z)$;
    \item \Cref{sec:proof:local:law:VP} combines the intermediate results and proves \Cref{thm:local:law:VP}.
\end{itemize}

\vspace{1em}

\noindent 
\textbf{Additional notation in this section. } Given a vector $v \in \C^b$, we denote its coordinate-wise average and coordinate-wise inverse as 
\begin{align*}
    \langle v \rangle 
    \;\coloneqq&\; 
    \mfrac{1}{b} \msum_{l \leq b} v_l   
    &\text{ and }&&
    v^{-1} 
    \;\coloneqq&\; 
    \big(v_1^{-1}, \ldots, v_b^{-1} \big)^\top
    \;.
\end{align*}
For $u, v\in \C^b$, we also denote the coordinate-wise product as 
\begin{align*}
    u \odot v \;=\; ( u_1^* v_1, \ldots, u_b^* v_b  )^\top\;,
\end{align*}
where $*$ indicates complex conjugate. We also denote the averaged $l_2$ inner product as 
\begin{align*}
    \langle u, v \rangle_2 \;\coloneqq\; \mfrac{1}{b} \msum_{l=1}^b u^*_l v_l\;,
\end{align*}
{and denote its corresponding norm by
\begin{align*}
    \|v\|_2
    \;\coloneqq\;
    \langle v,v\rangle_2^{1/2}
    \;=\;
    \left(
        \mfrac{1}{b}\msum_{l=1}^b |v_l|^2
    \right)^{1/2}.
\end{align*}}
For a matrix $A \in \C^{b \times b}$, we follow the convention of \cite{alt2017local} {and denote the operator norm induced by
$\|\cdot\|_2$ as}
\begin{align*}
    \normtwo{A} 
    \;\coloneqq\;
    \sup_{v \neq 0} \mfrac{\|A v\|_2}{\| v\|_2}  
    \;,
\end{align*}
the operator norm induced by the $\infty$-norm as
\begin{align*}
    \norminf{A} 
    \;\coloneqq\;
    \sup_{v \neq 0} \mfrac{\|A v\|_\infty}{\| v\|_\infty}  
    \;,
\end{align*}
and similarly 
\begin{align*}
    \normtwoinf{A} 
    \;\coloneqq\;
    \sup_{v \neq 0} \mfrac{\|A v\|_\infty}{\| v\|_2}  \;.
\end{align*}
We also use $\lesssim$ and $\gtrsim$ to indicate inequalities up to a multiplicative positive and bounded constant, which only depends on $\gamma = \lim \frac{p}{n}$ and the constants $s_*$, $L_1$, $L_2$, $\psi_1$ and $\psi_2$ in \Cref{asst:flat}.

\subsection{Controls on \texorpdfstring{$\fM(z)$}{M(z)}} \label{sec:control:fM}

This subsection first provides controls on the entries of $|\fM(z)|$  in \Cref{lem:abs:fM:bound}, which are analogous to Lemma 3.1 of \cite{alt2017local} , which will hold under \Cref{asst:flat} on the variance profile $S$. The main difference is that we simplify the bounds to get rid of the dependency on the limiting spectral measure $\rho$ while seeking explicit $z$-dependence. We also provide a component-wise control in \Cref{lem:dz:fM:bound} on the derivative by bounding $\| \partial_z\fM(z)\|_\infty$, which is needed for both the proof of the local laws and the proof of our risk formula. This is in contrast to Lemma 3.8 of \cite{alt2017local}, which only controls that of the averaged component $\partial_z \langle \fM(z)\rangle$.

\vspace{.5em}

The first result concerns $|\fM(z)|$.

\begin{lemma}[Control on $|\fM(z)|$] \label{lem:abs:fM:bound} Suppose $\sup_{i \leq n, j \leq p} S_{ij} \leq \frac{s_*}{p+n}$. Then for any $z \in \bbH$,
    \begin{align*}
        \mfrac{\Im(z)}{\Im(z) \, |z|
        + s_*
        } 
        \;\leq\;
        \inf_{1 \leq l \leq n+p} \,\big| \fM_l(z) \big|
        \;\leq&\;
        \sup_{1 \leq l \leq n+p} \,\big| \fM_l(z) \big|  
        \;\leq\;  \mfrac{1}{\Im(z)}
        \;.
    \end{align*}
\end{lemma}

\begin{proof}[Proof of \Cref{lem:abs:fM:bound}] We first rewrite \eqref{eq:combined_QVE} as a system $n+p$ of equations: For every $1 \leq l \leq n$,
\begin{align*}
    - \mfrac{1}{\fM_l(z) }
    \;=\;
    z 
    -
    \msum_{j=1}^p 
    \mfrac{ 
        S_{lj}
    }{
        z + \sum_{i=1}^n S_{ij} \fM_i(z)
    }
    \;,
    \tagaligneq \label{eq:fM:n}
\end{align*}
and that for every $1 \leq l \leq p$,
\begin{align*}
    - \mfrac{1}{\fM_{n+l}(z) }
    \;=\;
    z 
    -
    \msum_{i=1}^n
    \mfrac{ 
        S_{il} 
    }{
        z + \sum_{j=1}^p S_{ij}  \fM_{n+j}(z)
    }
    \;.
    \tagaligneq \label{eq:fM:p}
\end{align*}    
We handle \eqref{eq:fM:p} first since the argument for \eqref{eq:fM:n} is similar. Taking the norm on both sides of \eqref{eq:fM:p} and applying the triangle inequality, we have that 
\begin{align*}
    \mfrac{1}{|\fM_{n+l}(z)|} 
    \;\leq&\;
    |z|
    +
    \msum_{i=1}^n
    \Big|
        \mfrac{ 
            S_{il} 
        }{
            z + \sum_{j=1}^p S_{ij}  \fM_{n+j}(z)
        }
    \Big|
    \\
    \;\overset{(a)}{\leq}&\;
    |z|
    +
    \msum_{i=1}^n
        \mfrac{ 
            S_{il} 
        }{
            \Im(z) 
            + 
            \sum_{j=1}^p S_{ij} 
            \, \Im(\fM_{n+j}(z))
        }
    \\
    \;\overset{(b)}{\leq}&\;
    |z| + \mfrac{\sum_{i \leq n} S_{il}}{\Im(z)}
    \;\overset{(c)}{\leq}\; 
       |z| + \mfrac{n s_*}{(n+p)\Im(z)}
    \;\leq\; 
    | z| + \mfrac{s_*}{\Im(z)}
    \;.
\end{align*}
In both $(a)$ and $(b)$, we have used that $S_{ij} \geq 0$ and that both $z$ and $\fM_{n+j}(z)$ lie in $\bbH$, the complex upper-half plane; in $(c)$, we used the condition that $\sup_{i \leq n, j \leq p} S_{ij} \leq \frac{s_*}{p+n}$. This implies the lower bound 
\begin{align*}
    |\fM_{n+l}(z)| \;\geq\;  \mfrac{\Im(z)}{\Im(z) \, |z| + s_* }
    \qquad 
    \text{ for every }
    \quad 
    1 \leq l \leq p\;.
\end{align*}
By a similar argument applied to \eqref{eq:fM:n}, we obtain that
\begin{align*}
    \inf_{1 \leq l \leq n+p} \,\big| \fM_l(z) \big| 
    \;\geq\;
     \mfrac{\Im(z)}{\Im(z) \, |z|
     + s_*
    } 
    \;.
\end{align*}

\vspace{.5em}

To provide the upper bound on the largest entry of $\fM(z)$, we take a multiplicative inverse of both sides of \eqref{eq:fM:p} to obtain 
\begin{align*}
    | \fM_{n+l}(z)|
    \;=&\;
    \Big|
    \Big( 
        - 
        z     
        +
        \msum_{i=1}^n
        \mfrac{ 
            S_{il}
        }{
            z + \sum_{j=1}^p S_{ij} \fM_{n+j}(z)
        }
    \Big)^{-1}
    \Big|
    \\
    \;\leq&\;
    \Big| 
        - 
        \Im(z)     
        +
        \msum_{i=1}^n
        \Im\Big(
        \mfrac{ 
            S_{il} 
        }{
            z + \sum_{j=1}^p S_{ij} \fM_{n+j}(z)
        }
        \Big)
    \Big|^{-1}
    \\
    \;=&\;
    \bigg| 
        - 
        \Im(z)     
        -
        \msum_{i=1}^n
        \mfrac{ 
            S_{il}\,
            \Im\big(  z + \sum_{j=1}^p S_{ij} \fM_{n+j}(z) \big)
        }{
            \big| 
                z + \sum_{j=1}^p S_{ij} \fM_{n+j}(z)
            \big|^2
        }
    \bigg|^{-1}
    \\
    \;\leq&\;
    \mfrac{1}{\Im(z)}
    \;,
\end{align*}
where we have again noted that the imaginary parts of $z$ and $\fM_{n+j}(z)$ are all positive. Applying the same argument to \eqref{eq:fM:n} gives 
\begin{align*}
    \sup_{1 \leq l \leq n+p}
    \,\big| \fM_l(z) \big| 
    \;\leq\; \mfrac{1}{\Im(z)}
    \;,
\end{align*}
which proves the required upper bound.
\end{proof}

\vspace{.5em}

The last result concerns the control on $\| \partial_z \fM(z) \|_\infty$.

\begin{lemma} \label{lem:dz:fM:bound}  Suppose $\sup_{i \leq n, j \leq p} S_{ij} \leq \frac{s_*}{p+n}$. Then for any $z \in \bbH$, provided that $\fB^{-1}(z)$ exists,
\begin{align*}
    \big\| \partial_z  \fM(z) \big\|_\infty  
    \;\leq\;
    \mfrac{1}{(\Im (z))^2} \, \| \fB^{-1}(z) \|_\infty
    \;.
\end{align*}
\end{lemma}

\begin{proof}[Proof of \Cref{lem:dz:fM:bound}] Recall from the notation of \eqref{eq:bF:defn} that the QVE can be rearranged to 
\begin{align*}
    \bF(z) \;=\; 0\;,
    \qquad
    \text{ where }
    \bF(z)
    \;\coloneqq&\;
    \Big( \mfrac{1}{\fM_1(z)} \,,\, \cdots \,,\, \mfrac{1}{\fM_{n+p}(z)}  \Big)^\top
    +
    z \, \bone_{n+p}
    +
    \fS
    \, \fM(z)
    \;,
\end{align*}
and
\begin{align*}
    \mfrac{\partial \fM(z)}{\partial \bF(z)}
    \;=&\;
    -
    \diag\{ | \fM(z)| \}
    \,
    \fB^{-1}(z)
    \,
    \diag\{ | \fM(z)| \}
    \;.
\end{align*}
By the implicit function theorem followed by the inverse function theorem, 
\begin{align*}
    \partial_z  \fM(z)
    \;=&\;
    - \Big( \mfrac{\partial \bF(z)}{\partial \fM(z)} \Big)^{-1} 
    \mfrac{\partial \bF(z)}{\partial z} 
    \;=\;
    - \mfrac{\partial \fM(z)}{\partial \bF(z)} \bone_{n+p}
    \\
    \;=&\;
    \diag\{ | \fM(z)| \}
    \,
    \fB^{-1}(z)
    \,
    | \fM(z)| 
    \;.
\end{align*}
This implies that 
\begin{align*}
    \big\| \partial_z  \fM(z) \big\|_\infty  
    \;\leq&\; 
    \| \fM(z) \|_\infty \, \| \fB^{-1}(z) \|_\infty \, \| \fM(z) \|_\infty %
    \;\leq\;
    \mfrac{1}{(\Im (z))^2} \, \| \fB^{-1}(z) \|_\infty%
    \;,
\end{align*}
where we have applied \Cref{lem:abs:fM:bound} in the second inequality.

\end{proof}

\subsection{Stability control} \label{sec:control:stability}

This subsection concerns the operator $\fB^{-1}(z)$, which measures the stability of $\fM(z)$ under a perturbation of the QVE. The main result is the following, which is essentially an adaptation of Lemma 3.5 of \cite{alt2017local} with explicit tracking of the dependence of $z$ and our bounds from \Cref{lem:abs:fM:bound} above.

\begin{lemma}[$z$-explicit version of Lemma 3.5 of \cite{alt2017local}]\label{lem:bound_B_inverse}
Under  \Cref{asst:flat},  %
for all $z \in \bbH$,
\begin{align}
    \normtwo{ \fB^{-1}(z) } 
    \;\leq&\;
    \mfrac{\Im(z)|z|+s_*}{(\Im(z))^2}\;,
    \label{eq:estimate_norm_B_inverse}
    \\
    \norminf{\Bf^{-1}(z)}
    \;\leq&\;
    1 
    + 
    \mfrac{s_*(\Im(z)|z|+s_*)}{(\Im(z))^4}.
    \label{eq:estimate_norm_B_inverse_infty} 
\end{align}
\end{lemma}

\vspace{.5em}

Plugging the bound \eqref{eq:estimate_norm_B_inverse_infty}  into \Cref{lem:dz:fM:bound} gives the following immediate corollary:
\begin{corollary} \label{cor:dz:fM:bound}  Under  \Cref{asst:flat}, 
for all $z \in \bbH$,
\begin{align*}
    \big\| \partial_z  \fM(z) \big\|_\infty  
    \;\leq\;
    \mfrac{{1}}{(\Im (z))^2} \,
    \Big(
    1 
    + 
    \mfrac{s_*(\Im(z)|z|+s_*)}{(\Im(z))^4}
    \Big)
    \;.
\end{align*}
\end{corollary}

\vspace{.5em}

{The rest of this subsection is devoted to the proof of
\Cref{lem:bound_B_inverse}. }
To facilitate the presentation, we follow \cite{alt2017local} to introduce 
\begin{align*}
    \fF(z)
    \;\coloneq&\;
    \diag\{ | \fM(z)| \} \, \fS  \, \diag\{ | \fM(z)| \}\;,
    \tagaligneq \label{eq:defn:fF}
\end{align*}
such that by \eqref{eq:defn:fB},
\begin{align*}
    \fB(z)
    \;=&\;
    \diag\Big\{ 
            \mfrac{|\fM_1(z)|^2}{\fM_1(z)^2}
            \,,\, \ldots \,,\,
            \mfrac{|\fM_{n+p}(z)|^2}{\fM_{n+p}(z)^2}
    \Big\}
    -
    \fF(z)
    \;.
\end{align*}

{Under \Cref{asst:flat}, by Lemma 3.3 of \cite{alt2017local}, $\fF(z)$ has a unique
non-negative eigenvector $\ff(z)$ associated with
$\|\fF(z)\|_2$, with normalization
\(
    \| \ff(z)\|_2=1.
\)
We use this normalization below.}

{We provide a bound on
$1-\|\fF(z)\|_2$ in \Cref{lem:fF}, which will be used to control
$\|\fB^{-1}(z)\|_2$. We then use \Cref{lem:norm_hat_F_inverse}
together with \Cref{lem:fS:norm} to obtain the corresponding control on
$\|\fB^{-1}(z)\|_\infty$.}

\begin{lemma} \label{lem:fF}
Under \Cref{asst:flat}, we can compute
\begin{align*} 
    \normtwo{ \Ff(z)} = 1- \frac{\Im(z) \; \langle \ff(z) \odot \abs{\Mf(z)} \rangle}{ \left\langle \ff(z) \odot \Im \Mf(z) \odot \abs{\Mf(z)}^{-1} \right\rangle }.
    \tagaligneq \label{eq:normtwo_F_explicit}
\end{align*}
 In particular, for any $z \in \bbH$, 
\begin{align*}
    0\;\leq\;
    \big(1- \normtwo{\Ff(z)}\big)^{-1} 
    \;\leq&\;
    \mfrac{ \Im(z) |z| + s_*}{(\Im(z))^2}
    \;.
    \tagaligneq \label{eq:estimate_norm_F}
\end{align*}
\end{lemma}

\begin{proof}[Proof of \Cref{lem:fF}] \eqref{eq:normtwo_F_explicit} is verbatim from Lemma 3.4 of \cite{alt2017local}. To obtain \eqref{eq:estimate_norm_F}, 
{
Since $\ff(z)$ is non-negative and
$
0<
\mfrac{\Im(\fM_l(z))}{|\fM_l(z)|}
\leq 1,
$
\eqref{eq:normtwo_F_explicit} gives
\begin{align*}
    1-\normtwo{\fF(z)}
    &=
    \Im(z)\,
    \mfrac{
        \big\langle \ff(z)\odot|\fM(z)|\big\rangle
    }{
        \big\langle
        \ff(z)\odot\Im(\fM(z))\odot|\fM(z)|^{-1}
        \big\rangle
    }
     \\
    &=
    \Im(z)\,
    \mfrac{
        \frac{1}{n+p}\sum_{l=1}^{n+p} f_l(z)|\fM_l(z)|
    }{
        \frac{1}{n+p}\sum_{l=1}^{n+p}
        f_l(z)\frac{\Im(\fM_l(z))}{|\fM_l(z)|}
    }
    \\
    &\geq
    \Im(z)\inf_{1\leq l\leq n+p}|\fM_l(z)|
    \\
    &\geq
    \mfrac{(\Im(z))^2}{\Im(z)|z|+s_*},
\end{align*}
where the last inequality follows from
\Cref{lem:abs:fM:bound}.
Moreover, since $|z|\geq \Im(z)$,
\[
    \mfrac{\Im(z)|z|+s_*}{(\Im(z))^2}\geq 1.
\]
Hence, \eqref{eq:estimate_norm_F} follows.
}
\end{proof}

\begin{lemma}[Lemma 3.7 of \cite{alt2017local}] \label{lem:norm_hat_F_inverse}
Let $R\colon \C^{n+p} \to \C^{n+p}$ be a linear operator and $D \colon \C^{n+p} \to \C^{n+p}$ a diagonal operator. 
If $R - D$ is invertible and $D_{xx} \neq 0$ for all $x =1, \ldots, n+p$ then
\begin{align*}
\norminf{(R - D )^{-1} } \leq \left(\inf_{x=1}^{n+p} \abs{D_{xx}}\right)^{-1} \left( 1 + \normtwoinf{R}\normtwo{(R - D )^{-1}} \right). 
\end{align*}
\end{lemma}

\vspace{.5em}

We also state a small lemma for controlling the norms of $\fS$.

\begin{lemma} \label{lem:fS:norm} Suppose $\sup_{i \leq n, j \leq p} S_{ij} \leq \frac{s_*}{p+n}$. Then 
\begin{align*}
    \max\big\{ \normtwoinf{\fS} \,,\, \norminf{\fS} \big\}
    \;\leq\; s_*\;.
\end{align*}
\end{lemma}

\begin{proof}[Proof of \Cref{lem:fS:norm}] Since $\fS = \begin{psmallmatrix}
        0 & S \\ S^\top & 0
    \end{psmallmatrix}$, we can write 
\begin{align*}
    \normtwoinf{\fS} 
    \;=&\;
    \sup_{\substack{v\in \mathbb{C}^{n+p}\\v \neq 0}} \mfrac{\| \fS v \|_\infty }{\| v \|_2}
    \;=\;
    \sup_{\substack{
v_1\in\C^n,\,
v_2\in\C^p\\(v_1^\top, v_2^\top) \neq 0}} \mfrac{ \max\{ \|S v_{2} \|_\infty \,,\, \|S^\top v_{1} \|_\infty \}  }{\big\| (v_1^\top, v_2^\top) \big\|_2}
    \\
    \;\overset{(a)}{\leq}&\;
    \sup_{\substack{
v_1\in\C^n,\,
v_2\in\C^p\\(v_1^\top, v_2^\top) \neq 0}}  \mfrac{ \max\{ \frac{s_*}{p+n} \sum_{j \leq {n}} {|}v_{1,j}{|}  \,,\, \frac{s_*}{p+n} \sum_{i \leq {p}} {|}v_{2,i}{|} \}  }{
        \big( 
            \frac{1}{n+p}
            \big( 
                \sum_{j \leq {n}} {|}v_{1,j}{|}^2 
                +
                \sum_{i \leq {p}} {|}v_{2,i}{|}^2
            \big)
        \big)^{1/2}
    }
    \;\overset{(b)}{\leq}\;   s_*\;.
\end{align*}
In $(a)$, we have applied the entrywise bound on $S$; $(b)$ follows by noting that 
\begin{align*}
    \mfrac{1}{{n}} \msum_{j \leq {n}} {|}v_{1,j}{|}
    \;\leq&\;
    \Big(\mfrac{1}{{n}} \msum_{j \leq {n}} {|}v_{1,j}{|}^2\Big)^{1/2}
    &\text{ and }&&
    \mfrac{1}{{p}} \msum_{i \leq {p}} {|}v_{2,i}{|}
    \;\leq&\;
    \Big(\mfrac{1}{{p}} \msum_{i \leq {p}} {|}v_{2,i}{|}^2\Big)^{1/2}
    \;.
\end{align*}
The second bound follows from a similar calculation:
\begin{align*}
    \| \fS \|_\infty
    \;\leq&\;
    \sup_{\substack{
v_1\in\C^n,\,
v_2\in\C^p\\(v_1^\top, v_2^\top) \neq 0}}  \mfrac{ \max\{ \frac{s_*}{p+n} \sum_{j \leq {n}} {|}v_{1,j}{|}  \,,\, \frac{s_*}{p+n} \sum_{i \leq {p}} {|}v_{2,i}{|} \}  }{
        \max\{ 
            \sup_{j' \leq {n}} |v_{1,j'}| 
            \,,\,  
            \sup_{i' \leq {p}} |v_{2,i'}| 
        \}
    }
    \;\leq\; s_*\;,
\end{align*}
where we now use that 
\begin{align*}
    \mfrac{1}{{n}} \msum_{j \leq {n}} {|}v_{1,j}{|}
    \;\leq&\;
    \sup_{j' \leq {n}} |v_{1,j'}| 
    &\text{ and }&&
    \mfrac{1}{{p}} \msum_{i \leq {p}} {|}v_{2,i}{|}
    \;\leq&\;
    \sup_{i' \leq {p}} |v_{2,i'}| 
    \;.
\end{align*}
\end{proof}

We are now ready to present the proof of \Cref{lem:bound_B_inverse}.

\begin{proof}[Proof of \Cref{lem:bound_B_inverse}]
Recall from \eqref{eq:defn:fB} and \eqref{eq:defn:fF} that
\begin{align*}
    \fB(z)
    \;=&\;
    \fU(z)-\fF(z),
\end{align*}
where $    \fU(z)
    \;\coloneqq\;
    \diag\Big\{
        \mfrac{|\fM_1(z)|^2}{\fM_1(z)^2},
        \ldots,
        \mfrac{|\fM_{n+p}(z)|^2}{\fM_{n+p}(z)^2}
    \Big\}$. Then we have
\begin{align*}
    \fB(z)
    \;=\;
    \fU(z)\big(I-\fU(z)^*\fF(z)\big).
\end{align*}
Since $\fU(z)$ is unitary,
\begin{align*}
    \normtwo{\fU(z)^*\fF(z)}
    \;=\;
    \normtwo{\fF(z)}.
\end{align*}
Moreover, \Cref{lem:fF} gives
\begin{align*}
    {
        1 - \normtwo{\fU(z)^*\fF(z)}
    }
    \;=\;
    1-\normtwo{\fF(z)}
    \;\geq\;
    \mfrac{(\Im(z))^2}{\Im(z)|z|+s_*}
    \;>\;0.
\end{align*}
Hence,
\begin{align*}
    \normtwo{\fU(z)^*\fF(z)}<1,
\end{align*}
and the Neumann series yields
\begin{align*}
    \big(I-\fU(z)^*\fF(z)\big)^{-1}
    \;=\;
    \sum_{k=0}^{\infty}
    \big(\fU(z)^*\fF(z)\big)^k.
\end{align*}
Since
\begin{align*}
    \fB^{-1}(z)
    =
    \big(I-\fU(z)^*\fF(z)\big)^{-1}\fU(z)^*
\end{align*}
and $\normtwo{\fU(z)^*}=1$, we obtain
\begin{align*}
    \normtwo{\fB^{-1}(z)}
    &\leq
    \mfrac{1}{1-\normtwo{\fF(z)}}
    \\
    &\leq
    \mfrac{\Im(z)|z|+s_*}{(\Im(z))^2}.
\end{align*}

To derive the second bound \eqref{eq:estimate_norm_B_inverse_infty}, we apply \Cref{lem:norm_hat_F_inverse} with 
\begin{align*}
    R \;=&\; \fF(z)
    &\text{ and }&&
    D \;=&\; \fF(z) + \fB(z) \;=\; \diag\big\{ \fM(z)^{-2} \, |\fM(z)|^2 \big\}
    \;,
\end{align*}
which gives 
\begin{align*}
    \norminf{\Bf^{-1}(z)}
    \;=&\;
    \norminf{(R - D  )^{-1}}
    \\
    \;\leq&\;
    \Big(  \inf_{l \leq n+p} \big| \fM_l(z)^{-2}  |\fM_l(z)|^2 \big| \Big)^{-1}
    \,
    \Big( 
        1 
        +
        \| \fF(z) \|_{2 \rightarrow \infty}
        \,
        \| \fB^{-1}(z) \|_2
    \Big)
    \\
    \;=&\;
    1 
    +
    \| \fF(z) \|_{2 \rightarrow \infty}
    \,
    \| \fB^{-1}(z) \|_2
    \;.
\end{align*}
Recall that $\fF(z) = \diag\{|\fM(z)|\} \, \fS \, \diag\{|\fM(z)|\}$, which allows us to control
\begin{align*}
    \normtwoinf{\Ff(z)} 
    \;\leq\; 
    \norm{\Mf(z)}_\infty^2 \normtwoinf{\Sf} 
    \;\overset{(a)}{\leq}\; 
    s_*\norm{\Mf(z)}_\infty^2
    \;\overset{(b)}{\leq}\;
    \mfrac{s_*}{(\Im(z))^2}
    \;.
\end{align*}
In $(a)$, we have used that $\normtwoinf{\Sf} \leq s_*$ by applying \Cref{lem:fS:norm} under \Cref{asst:flat}; in $(b)$, we have applied our \Cref{lem:abs:fM:bound}. Combining this with our bound on $\| \fB^{-1}(z)\|_2$ above gives 
\begin{align*}
    \norminf{\Bf^{-1}(z)}
    \;\leq &\;
    1 + \mfrac{s_*}{(\Im(z))^2} \,
    \| \fB^{-1}(z) \|_2
    \;\leq\;
    1+
    \mfrac{s_*(\Im(z)|z|+s_*)}{(\Im(z))^4}
    \;.
\end{align*}
\end{proof}

\subsection{Controlling the diagonal resolvent error by QVE perturbation} \label{sec:diag:error:to:perturb}

We seek to approximate $\gf(z)$, the diagonal of the resolvent $G(z)$ of the data matrix, by the solution $\fM(z)$ to the QVE \eqref{eq:combined_QVE}. As mentioned in \eqref{eq:perturbed_combined_QVE}, the random $ \Hb \to (\C\setminus\{0\})^{n+p}$ function  $\gf$ satisfies a perturbed QVE
\begin{align*} 
    - 
    \gf(z)^{-1} 
    \;=\; 
    z \, \bone_{n+p}
    + 
    \Sf\, \gf(z)
    + 
    \df(z) \;, 
    \qquad z \in \bbH,
\end{align*}
where the random perturbation vector $\df(z) = (\df_1(z), \ldots, \df_{n+p}(z))$ reads, for $1 \leq m \leq n$ and $1 \leq r \leq p$,
\begin{align*}
    \df_m(z) 
    \;\coloneqq&\;
    \msum_{k,l=1, k \neq l}^p \mfrac{X_{mk}}{\sqrt{n}} \, G^{(m)}_{n+k, n+l}(z) \,\mfrac{X_{ml} }{\sqrt{n}}
    +
    \msum_{k=1}^p \Big( \mfrac{| X_{mk}|^2}{n} - S_{mk} \Big) G^{(m)}_{n+k, n+k}(z)
    \\
    &\;
    -
    \msum_{k=1}^p S_{mk} \mfrac{G_{n+k,m}(z) \, G_{m,n+k}(z)}{\gf_m(z)}
    \;, 
    \\
    \df_{n+r}(z)
    \;\coloneqq&\;
    \msum_{j,t = 1, j \neq t}^n
    \mfrac{X_{jr}}{\sqrt{n}} \, G^{(n+r)}_{jt}(z) \, \mfrac{X_{tr}}{\sqrt{n}}
    +
    \msum_{j=1}^n
    \Big( \mfrac{| X_{jr}|^2}{n} - S_{jr} \Big) G^{(n+r)}_{jj}(z)
    \\
    &\;
    - 
    \msum_{j=1}^n S_{jr} \mfrac{G_{j,n+r}(z) \, G_{n+r,j}(z)}{\gf_{n+r}(z)}
    \;,
    \tagaligneq \label{eq:defn:perturb}
\end{align*}
and we have denoted the modification of $G$ after removal of the $k$-th row and $k$-th column of $H = \begin{psmallmatrix}
    0 & X /\sqrt{n} \\ X^\top / \sqrt{n} & 0
\end{psmallmatrix} $ as, for $1 \leq k \leq n+p$,
\begin{align*}
    G^{(k)}(z)
    \;\coloneqq\;
    \big( H^{(k)} - z I_{n+p} \big)^{-1}
    \;,
    \qquad 
    \text{ where }
    H^{(k)}_{lj} \;\coloneqq\;  H_{lj} \, \ind_{\{ l \neq k \,,\, j \neq k \}}
    \;.
\end{align*}
The formulas follow from Schur complement formula; see (3.4a) and (3.4b) of \cite{alt2017local} for a derivation, noting two notational differences that 
\begin{itemize}
    \item they handle a variance profile of dimension $p \times n$ whereas our variance profile $S$ is of dimension $n \times p$, and
    \item {
our variance profile is defined by $\operatorname{Var}(X_{ij}/\sqrt{n})=S_{ij}$
}, and therefore $X$ shows up in the formulas with an explicit rescaling factor of $1/\sqrt{n}$.
\end{itemize}
The main result of this section is stated next: Roughly speaking, it says that provided that the approximation error $\|\gf(z) - \fM(z)\|_\infty$ is bounded in some way, we can further sharpen the bound on $\|\gf(z) - \fM(z)\|_\infty$ to be a function of the size of the perturbation $\|\df(z)\|_\infty$ (which is subsequently controlled in \Cref{sec:perturb:offdiag:bound}).

\begin{lemma}[$z$-explicit version of the $\|\argdot\|_\infty$ bound in Lemma 3.9 of \cite{alt2017local}]
\label{lem:diag:error:to:perturb} Suppose $\Im(z) \leq \min\{1, s_*\}$. 
There exist some constants $C_B \geq 1$ and $C_0 > 0$, depending only on the constants $s_*$, $L_1$, $L_2$, $\psi_1$ and $\psi_2$ in \Cref{asst:flat}, such that almost surely
\begin{align*}
    &\;
    \| \gf(z) - \fM(z) \|_\infty \,\times\,\ind{\Big\{ \| \gf(z) - \fM(z) \|_\infty \,\leq\,
    {
        \mfrac{C_0 \, (\Im(z))^5}
        {\Im(z)|z|+s_*}
        }
    \Big\}}
    \\
    &\;\leq\;
    {
    \mfrac{
        3 C_B\big(\Im(z)|z|+s_*\big)
    }{
        (\Im(z))^6
    }
    }
    \, 
    \| \df(z) \|_\infty
    \;.
\end{align*}
\end{lemma}

As with the proof of Lemma 3.9 of \cite{alt2017local}, we will employ the following tool developed by \cite{Ajanki_2019}, which enables a refined control on $\|\gf(z) - \fM(z)\|_\infty$ whenever a crude control is already available:

\begin{lemma}[$\|\argdot\|_\infty$ bound of Lemma~5.11 of \cite{Ajanki_2019}]
\label{lem:gf:fM:diff}
Suppose $\gf', \df' \in \C^{n+p}$ with $\inf_{l \leq n+p} \abs{\gf'_l} > 0$ satisfy the perturbed QVE \eqref{eq:perturbed_combined_QVE} at some $z \in \Hb$, and assume
\begin{align*} 
    \norm{\Mf(z)}_\infty \le C'_\fM, \qquad \norminf{\fB^{-1}(z)} \le C'_\fB,
    \tagaligneq
    \label{eq:quantitative_bulk_stability_assumptions}
\end{align*}
for some constants $C'_\fM,C'_\fB \ge 1$. If
\begin{equation*}
    \norm{\gf' - \Mf(z)}_\infty \le \frac{1}{2 \max\{1, \norminf{\Sf}\} C'_\fM C'_\fB},
    \tagaligneq \label{eq:gf:fM:diff:init:threshold}
\end{equation*}
then we have
\begin{align*}
    \tagaligneq\label{eq:bulk_perturbations_inf}
    \norm{\gf' - \Mf(z)}_\infty &\le 3 C'_\fB (C'_\fM)^2 \norm{\df'}_\infty\;.
\end{align*}
\end{lemma}

\begin{proof}[Proof of \Cref{lem:diag:error:to:perturb}] The goal is to apply \Cref{lem:gf:fM:diff} with $\gf'=\gf(z)$ and $\df' = \df(z)$. We first note that by our \Cref{lem:abs:fM:bound} and the condition that $\Im(z) \leq \min\{1, s_*\}$, we can take the bound on $\norminf{ \fM(z) }$ in  \Cref{lem:gf:fM:diff} as 
\begin{align*}
    C'_\fM \;\coloneqq\;  \mfrac{1}{\Im(z)} 
    \;\geq\; 1 \;.
\end{align*}
{
Meanwhile, by \eqref{eq:estimate_norm_B_inverse_infty} in
\Cref{lem:bound_B_inverse}, we have
\begin{align*}
    \norminf{\fB^{-1}(z)}
    \;\leq&\;
    1
    +
    \mfrac{
        s_*\big(\Im(z)|z|+s_*\big)
    }{
        (\Im(z))^4
    }
    \\
    \;\overset{(a)}{\leq}&\;
    (1+s_*)
    \mfrac{
        \Im(z)|z|+s_*
    }{
        (\Im(z))^4
    }
    \;.
    \tagaligneq \label{eq:Binv:bound}
\end{align*}
}
In $(a)$, we have used {$\Im(z)\leq \min\{1,s_*\}$, which implies $
    (\Im(z))^4
    \leq
    s_*
    \leq
    \Im(z)|z|+s_*.$}

 {Let $C_B \;\coloneqq\; 1+s_*$. }In other words, we can take the bound on $\norminf{\fB^{-1}(z)}$ in  \Cref{lem:gf:fM:diff} as 
{\begin{align*}
    C'_\fB
    \;\coloneqq\;
    C_B
    \mfrac{
        \Im(z)|z|+s_*
    }{
        (\Im(z))^4
    }
    \;\geq\;1.
\end{align*}}
 We also note that by \Cref{lem:fS:norm}, 
\begin{align*}
    \| \fS \|_\infty \;\leq\; s_* \;.
\end{align*}
In this case, the threshold in \eqref{eq:gf:fM:diff:init:threshold} satisfies
{\begin{align*}
    \mfrac{1}{
        2 \max\{1,\norminf{\fS}\}
        C'_\fM C'_\fB
    }
    &=
    \mfrac{
        (\Im(z))^5
    }{
        2 \max\{1,\norminf{\fS}\}
        C_B
        \big(\Im(z)|z|+s_*\big)
    }
    \\
    &\geq
    \mfrac{
        C_0(\Im(z))^5
    }{
        \Im(z)|z|+s_*
    },
\end{align*}
where we can take
\begin{align*}
    C_0
    \;\coloneqq\;
    \mfrac{
        1
    }{
        2\max\{1,s_*\}C_B
    }
    \;>\;0.
\end{align*}}
This allows us to apply \Cref{lem:gf:fM:diff} to give that, almost surely,
\begin{align*}
    &\;
    \| \gf(z) - \fM(z) \|_\infty \,\times\,\ind{\Big\{ \| \gf(z) - \fM(z) \|_\infty \,\leq\,
    {\mfrac{
            C_0(\Im(z))^5
        }{
            \Im(z)|z|+s_*
        }}
    \Big\}}
    \\
    &\;\leq\;
    3 C_\fB' (C_\fM')^2 \, \| \df(z) \|_\infty
    \\
    &\;=\;
    {
        \mfrac{
        3 C_B
        \big(\Im(z)|z|+s_*\big)
    }{
        (\Im(z))^6
    }
    \|\df(z)\|_\infty}
    \;.
\end{align*}
\end{proof}

\subsection{Bounding the QVE perturbation and the off-diagonal terms of the resolvent} \label{sec:perturb:offdiag:bound} 

\Cref{lem:gf:fM:diff} allows us to control the error of approximating the diagonal of the resolvent, $\|\gf(z) - \fM(z)\|_\infty$, through the random perturbation $\df(z)$ (where the stochasticity comes from the data). One of the goals in this section is to control $\df(z)$, which is a $\bbH \rightarrow \C^{n+p}$ function, as its output dimension 
\begin{align*}
    n+p \;\rightarrow\; \infty\;.
\end{align*}
The other goal of this section is to control the off-diagonal terms of the resolvent $G(z)$, i.e.~to provide a bound on 
\begin{align*}
    \max_{l \neq j} | G_{lj}(z)|\;.
\end{align*}
This term appears alongside the diagonal resolvent approximation term in the quantity we ultimately seek to control:
\begin{align*}
    \max_{1 \leq l,l' \leq n+p} \big| G_{ll'}(z) - \fM_l(z) \, \ind_{\{ l = l' \} } \big|
    \;=&\;
    \max\Big\{ \| \gf(z) - \fM(z)\|_\infty \,,\; \max_{l \neq j} | G_{lj}(z)|  \Big\}
    \;.
    \tagaligneq \label{eq:separate:diag:offdiag:err}
\end{align*}

\vspace{.5em}

The main results of this section are the following stochastic bounds on the QVE perturbation $\|\df(z)\|_\infty$ --- which controls $\| \gf(z) -\fM(z) \|_\infty$ --- and the off-diagonal resolvent term $\max_{l \neq j} | G_{lj}(z)|$. Here, not only do we need to track the dependence of $z$, we also need to develop results that allow $z$ to be taken small. This leads to bounds that are slightly different from \cite{ajanki2017universalitygeneralwignertypematrices} and \cite{alt2017local}.

\begin{lemma}[Perturbation terms, variant of Lemma 2.1 of \cite{ajanki2017universalitygeneralwignertypematrices} / Lemma 4.6 of \cite{alt2017local}] \label{lem:perturb:bound} Under \Cref{asst:flat},
\begin{align*}
    \| \df(z)\|_\infty \;\prec\; \mfrac{1}{\sqrt{n} \, \Im(z)} + \mfrac{1}{n \, (\Im(z))^3}\;.
\end{align*}
\end{lemma}

\begin{lemma}[Off-diagonal terms, variant of Lemma 2.1 of \cite{ajanki2017universalitygeneralwignertypematrices} / Lemma 4.6 of \cite{alt2017local}] \label{lem:offdiag:bound} Under \Cref{asst:flat},
\begin{align*}
    \max_{l \neq j} | G_{lj}(z) | 
    \;\prec\; 
    \mfrac{1}{\sqrt{n}} \,\mfrac{1}{ \min\{1, \Im(z) \}^3 }\;.
\end{align*}
\end{lemma}

We present the proof of the two results one after another, since many of the calculations are shared.

\begin{proof}[Proof of \Cref{lem:perturb:bound}]  We recall the definition of the resolvent matrix \eqref{eq:empirical:resolvent} that 
\begin{align*}
    G(z) \;\coloneqq&\; (H - z I_{n+p})^{-1}\;,
    \qquad 
    \text{ where }
    H \;\coloneqq\; \begin{psmallmatrix}
        0 & \frac{1}{\sqrt{n}} X \\
        \frac{1}{\sqrt{n}} X^\top & 0
    \end{psmallmatrix}
    \;\text{ and }\;
    z \in \bbH\;.
\end{align*}
Also recall that the definition \eqref{eq:defn:perturb} of the perturbation vector $\df(z)$ involves the following matrix minor 
of $G$ after removal of the $k$-th row and $k$-th column of $H$, 
\begin{align*}
    G^{(k)}(z)
    \;\coloneqq\;
    \big( H^{(k)} - z I_{n+p} \big)^{-1}
    \;,
    \qquad 
    \text{ where }
    H^{(k)}_{lj} \;\coloneqq\;  H_{lj} \, \ind_{\{ l \neq k \,,\, j \neq k \}}
    \;.
\end{align*}
Recall from \eqref{eq:defn:perturb} that the perturbation vector $\df(z) = (\df_1(z), \ldots, \df_{n+p}(z))$ reads, for $1 \leq m \leq n$ and $1 \leq r \leq p$,
\begin{align*}
    \df_m(z) 
    \;\coloneqq&\;
    \msum_{k,l=1, k \neq l}^p \mfrac{X_{mk}}{\sqrt{n}} \, G^{(m)}_{n+k, n+l}(z) \,\mfrac{X_{ml} }{\sqrt{n}}
    +
    \msum_{k=1}^p \Big( \mfrac{| X_{mk}|^2}{n} - S_{mk} \Big) G^{(m)}_{n+k, n+k}(z)
    \\
    &\;
    -
    \msum_{k=1}^p S_{mk} \mfrac{G_{n+k,m}(z) \, G_{m,n+k}(z)}{G_{m,m}(z)}
    \;, 
    \\
    \df_{n+r}(z)
    \;\coloneqq&\;
    \msum_{j,t = 1, j \neq t}^n
    \mfrac{X_{jr}}{\sqrt{n}} \, G^{(n+r)}_{jt}(z) \, \mfrac{X_{tr}}{\sqrt{n}}
    +
    \msum_{j=1}^n
    \Big( \mfrac{| X_{jr}|^2}{n} - S_{jr} \Big) G^{(n+r)}_{jj}(z)
    \\
    &\;
    - 
    \msum_{j=1}^n S_{jr} \mfrac{G_{j,n+r}(z) \, G_{n+r,j}(z)}{G_{n+r,n+r}(z)}
    \;.
\end{align*}
We focus on controlling $\df_m(z)$ for $m \leq n$, as the argument for $\df_{n+r}(z)$ is similar. Notice that the random variables $(X_{mk})_{k \leq p}$ are independent of the matrix minor $G^{(m)}(z)$. By applying the same large deviation estimate as \cite{alt2017local,ajanki2017universalitygeneralwignertypematrices} (see, e.g.~Appendix C of \cite{erdos2013local}) to $(X_{mk})_{k \leq p}$, we can control the first term of $\df_m(z)$ as 
\begin{align*}
    \big|\df_m^{(1)}(z)\big|
    \;\coloneqq&\;
    \Big| \msum_{k,l=1, k \neq l}^p \mfrac{X_{mk}}{\sqrt{n}} \, G^{(m)}_{n+k, n+l}(z) \,\mfrac{X_{ml} }{\sqrt{n}}
    \Big|
    \\
    \;\prec&\;
    \Big( 
        \msum_{k,l=1, k \neq l}^p S_{mk} S_{ml} \big| G^{(m)}_{n+k, n+l}(z) \big|^2 
    \Big)^{1/2}
    \\
    \;\overset{(a)}{\preceq}&\;
    \mfrac{1}{n}
    \Big( 
        \msum_{k,l=1, k \neq l}^p \big| G^{(m)}_{n+k, n+l}(z) \big|^2 
    \Big)^{1/2}
    \\
    \;\overset{(b)}{{\leq}}&\;
    \mfrac{1}{n}
    \bigg( 
        \msum_{l=1}^p 
        \mfrac{\Im\Big( G^{(m)}_{n+l,n+l}(z)  \Big)}{\Im(z)}
    \bigg)^{1/2}
    \\
    \;\overset{(c)}{\preceq}&\;
    \mfrac{1}{\sqrt{n} \; \Im(z)}
    \;.
    \tagaligneq \label{eq:df:1}
\end{align*}
In $(a)$ above, we have recalled that $\max_{i,j} S_{ij} \leq \frac{s_*}{n+p}$ by \Cref{asst:flat}, that $\lim n/p \in(0,\infty)$ and  that the \Cref{defn:sto:dom} of stochastic domination allows us to absorb the factor of $s_*^{1/2}$; in $(b)$, we have noted that by applying the Ward identity to the resolvent $G^{(k)}$ (see e.g.~Lemma 2.3 of \cite{he2018isotropic}),
\begin{align*}
    \msum_{r \leq n+p} 
    \big| G^{(m)}_{rl}(z) \big|^2
    \;=\;
    \mfrac{\Im G^{(m)}_{ll}(z)}{\Im(z)}\;;
    \tagaligneq \label{eq:Ward}
\end{align*}
in $(c)$,  we have noted that since $G^{(k)}(z)$ is the resolvent of a real Hermitian matrix, its diagonal entries are bounded by its largest eigenvalues in norm, and therefore 
\begin{align*}
    \sup_{1 \leq l \leq n+p} 
     \,
     \big| G^{(k)}_{ll}(z) \big|
     \;\leq\; 
     \mfrac{1}{\Im(z)}
     \;.
     \tagaligneq \label{eq:resolvent:bound}
\end{align*}
Note that the main difference of our proof from that of \cite{alt2017local,ajanki2017universalitygeneralwignertypematrices} is the use of \eqref{eq:Ward} to compute $| G^{(m)}_{n+k, n+l}(z)|^2$ directly, which allows us to bound $\big|\df_m^{(1)}(z)\big|$ directly without imposing an additional indicator function. Now by a similar argument and using the Gaussianity of $X_{mk}$, the second term of $\df_m(z)$ can be bounded as 
\begin{align*}
    \big|\df_m^{(2)}(z)\big|
    \;\coloneqq&\;
    \Big| 
        \msum_{k=1}^p \Big( \mfrac{| X_{mk}|^2}{n} - S_{mk} \Big) G^{(m)}_{n+k, n+k}(z)
    \Big| 
    \\
    \;\prec&\;
    \Big( 
        \msum_{k=1}^p    
        S_{mk}^2 \,
        \big| G^{(m)}_{n+k, n+k}(z) \big|^2
    \Big)^{1/2}
    \\
    \;\preceq&\;
    \mfrac{1}{n} \, \Big( 
        \msum_{k=1}^p    
        \big| G^{(m)}_{n+k, n+k}(z) \big|^2
    \Big)^{1/2} 
    \\
    \;\overset{\eqref{eq:resolvent:bound}}{\preceq}&\;
    \mfrac{1}{\sqrt{n} \; \Im(z)}
    \;.
    \tagaligneq \label{eq:df:2}
\end{align*}
To handle the last term of $\df_m(z)$, 
\begin{align*}
    \big|\df_m^{(3)}(z)\big|
    \;\coloneqq&\;
    \Big| 
        \msum_{k=1}^p S_{mk} \mfrac{G_{n+k,m}(z) \, G_{m,n+k}(z)}{G_{m,m}(z)}
    \Big| 
    \;,
\end{align*}
we again deviate from the proof of \cite{alt2017local,ajanki2017universalitygeneralwignertypematrices} to use the second resolvent identity in Lemma 4.5 of \cite{erdos2013local}: For $l, j \neq k$, we have that
\begin{align*}
    G_{lk}(z) 
    \;=&\; 
    - G_{kk}(z) \msum_{1 \leq r \leq n+p; r \neq k} G_{lr}^{(k)}(z)\, H_{rk}
    \;,
    \\
    G_{kj}(z) 
    \;=&\; 
    - G_{kk}(z) \msum_{1 \leq r \leq n+p; r \neq k} H_{kr}\, G_{rj}^{(k)}(z)
    \;.
\end{align*}
Note that since $G(z)$ and $H$ are both {symmetric}, we can define
\begin{align*}
    Q_l^{(k)}(z)
    \;\coloneqq&\;
    \msum_{1 \leq r \leq n+p; r \neq k} G_{lr}^{(k)}(z)\, H_{rk}
    \;=\;
    \msum_{1 \leq r \leq n+p; r \neq k} H_{kr} \,G_{rl}^{(k)}(z)
    \;,
\end{align*}
and obtain 
\begin{align*}
    \big|\df_m^{(3)}(z)\big|
    \;=&\;
    \Big| 
        \msum_{k=1}^p 
        S_{mk}
        \, G_{m,m}(z)  \,
        \big( Q^{(m)}_{n+k}(z) \big)^2
    \Big| 
    \\
    \;\prec&\;
    \mfrac{1}{n}
        \msum_{k=1}^p 
        \, | G_{m,m}(z)|  \,
        {\big| Q^{(m)}_{n+k}(z) \big|}^2
    \;.
\end{align*}
Now by a standard large deviation bound from Appendix C of \cite{erdos2013local} again,
\begin{align*}
    {\big|Q^{(m)}_{n+k}(z)\big|}
    \;\prec&\;
    \Big(
        \msum_{1 \leq r \leq n+p; r \neq m} 
        \max_{i,j} S_{ij}
        \,\big| G^{(m)}_{r,n+k}(z) \big|^2
    \Big)^{1/2}
    \\
    \;\preceq&\;
    \mfrac{1}{\sqrt{n}}
    \Big(
        \msum_{1 \leq r \leq n+p; r \neq m} 
        \big| G^{(m)}_{r,n+k}(z) \big|^2
    \Big)^{1/2}
    \\
    \;\overset{\eqref{eq:Ward}}{\preceq}&\;
    \mfrac{1}{\sqrt{n}}
    \Bigg(
        \mfrac{\Im G_{n+k,n+k}^{(m)}(z)}{\Im(z)}
    \Bigg)^{1/2}
    \\
    \;\overset{\eqref{eq:resolvent:bound}}{\preceq}&\;
    \mfrac{1}{\sqrt{n} \; \Im(z)}
    \;.
\end{align*}
This implies 
\begin{align*}
     \big|\df_m^{(3)}(z)\big|
    \;\prec&\;
    \mfrac{1}{n}
        \msum_{k=1}^p 
        \, | G_{m,m}(z)|  \,
        \mfrac{1}{n \,(\Im(z))^2}
    \;\overset{(d)}{\preceq}\;
    \mfrac{1}{n \, (\Im(z))^3}\;.
    \tagaligneq \label{eq:df:3}
\end{align*}
In $(d)$ above, we have noted that the bound \eqref{eq:resolvent:bound} applies also to $|G_{m,m}(z)|$ since $G(z)$ is also the resolvent of a real Hermitian matrix. Combining the three bounds \eqref{eq:df:1}, \eqref{eq:df:2} and \eqref{eq:df:3} by the triangle inequality, we obtain that for $1 \leq m \leq n$,
\begin{align*}
    \big|\df_m(z)\big|
    \;\prec\; \mfrac{1}{\sqrt{n} \, \Im(z)} + \mfrac{1}{n \, (\Im(z))^3}\;.
\end{align*}
Also note that none of the bounds nor the large standard deviation bounds above depend on the index $m$, and note that the exact same calculation applies to $|\df_{n+r}(z)|$ for $1 \leq r \leq p$. Therefore we obtain the desired control that
\begin{align*}
    \| \df(z)\|_\infty \;\prec\; \mfrac{1}{\sqrt{n} \, \Im(z)} + \mfrac{1}{n \, (\Im(z))^3}\;.
\end{align*}
\end{proof}

\vspace{.5em}

\begin{proof}[Proof of \Cref{lem:offdiag:bound}] For $l \neq j$, we start with the off-diagonal Schur formula (see (2.3) of \cite{ajanki2017universalitygeneralwignertypematrices}) to obtain 
\begin{align*}
    G_{lj}(z) 
    \;=\;
    G_{ll}(z) \, G^{(l)}_{jj}(z) \,
    \Bigg( 
        - 
        H_{lj}
        +
        \sum_{\substack{1 \leq k_1, k_2 \leq n+p \\ k_1, k_2 \,\not\in\, \{l,j\}}}
        H_{lk_1} \, G^{(lj)}_{k_1k_2} \, H_{k_2j}
    \Bigg)
    \;,
\end{align*}
where we have denoted the analogue of $G^{(k)}(z)$ but with two indices $k_1,k_2$ removed as 
\begin{align*}
    G^{(k_1k_2)}(z)
    \;\coloneqq\;
    \big( H^{(k_1k_2)} - z I_{n+p} \big)^{-1}
    \;,
    \qquad 
    \text{ where }
    H^{(k_1k_2)}_{lj} \;\coloneqq\;  H_{lj} \, \ind_{\{ l, j \not\in \{k_1,k_2\} \}}
    \;.
\end{align*}
We first use the same argument as (2.20) of \cite{ajanki2017universalitygeneralwignertypematrices}: By noting that $H_{lj}$ in the first term is either zero or zero-mean Gaussian with variance $\leq s_*/(n+p)$, and applying the large deviation bound of Appendix C of \cite{erdos2013local} to the sum in the second term, we obtain 
\begin{align*}
    | G_{lj}(z)|^2 
    \;\prec\;
    | G_{ll}(z) |^2 \, | G^{(l)}_{jj}(z)|^2 \,
    \Bigg( 
        \mfrac{1}{n} 
        +
        \mfrac{1}{n^2} 
        \sum_{\substack{1 \leq k_1, k_2 \leq n+p \\ k_1, k_2 \not\in \{l,j\}}}
        \big| G^{(lj)}_{k_1k_2}  \big|^2
    \Bigg)
    \;.
\end{align*}
We now control this by observing that the Ward identity \eqref{eq:Ward} and the bound \eqref{eq:resolvent:bound} both apply to the resolvent $G^{(k_1k_2)}$, which gives that almost surely
\begin{align*}
    \sum_{\substack{1 \leq k_1, k_2 \leq n+p \\ k_1, k_2 \not\in \{l,j\}}}
        \big| G^{(lj)}_{k_1 k_2}  \big|^2
    \;\leq&\;
    \sum_{\substack{1 \leq k_1 \leq n+p \\ k_1 \not\in \{l,j\}}}
    \mfrac{\Im  G^{(lj)}_{k_1 k_1}(z) }{\Im(z)}
    \;\leq\;
    \mfrac{n+p}{(\Im(z))^2}\;.
\end{align*}
This implies 
\begin{align*}
    | G_{lj}(z)|^2 
    \;\prec&\;
    | G_{ll}(z) |^2 \, | G^{(l)}_{jj}(z)|^2 \,
    \Big( 
        \mfrac{1}{n} 
        +
        \mfrac{1}{n \, (\Im(z))^2} 
    \Big)
    \\
    \;\overset{\eqref{eq:resolvent:bound}}{\preceq}&\;
    \mfrac{1}{n} \mfrac{1}{ \min\{1, \Im(z) \}^6 }
    \;.
\end{align*}
Again, the bounds above are independent of the indices $l$ and $j$, which imply that 
\begin{align*}
    \max_{l \neq j} | G_{lj}(z) | 
    \;\prec\; 
    \mfrac{1}{\sqrt{n}} \,\mfrac{1}{ \min\{1, \Im(z) \}^3 }\;.
\end{align*}
\end{proof}

\subsection{Proof for the local law under the variance profile model} \label{sec:proof:local:law:VP}

We now have most of the key ingredients to prove the local laws in \Cref{thm:local:law:VP}:
\begin{itemize}
    \item \Cref{lem:gf:fM:diff} says that whenever the diagonal resolvent approximation error $\| \gf(z) - \fM(z)\|_\infty$ satisfies some crude bound, we can improve the bound to depend on the perturbation $\| \df(z)\|_\infty$;
    \item \Cref{lem:perturb:bound} controls  $\| \df(z)\|_\infty$;
    \item \Cref{lem:offdiag:bound} controls the off-diagonal terms $\max_{l \neq j} | G_{lj}(z) | $.
\end{itemize}
We adapt the proof strategy of \cite{alt2017local} (see their proof of Theorem 4.2) to control the difference 
\begin{align*}
    \Lambda(z) \;\coloneqq\; \| \gf(z) - \fM(z)\|_\infty\;.
\end{align*}
The main difference from \cite{alt2017local} is that, in their proof of Theorem 4.2, they focus on $z$ with $|z| \geq \delta$ for some fixed $\delta > 0$, whereas we need to allow for values of $z$ with vanishing $|z|$. Specifically, recall that we focus on the domain of $z$ given by 
\begin{align*}
    \D_{\upsilon, K} \;\coloneqq\; 
    \Big\{ 
        z \in \bbH 
        \;\Big|\;
        \Im(z) \;\geq\; \mfrac{1}{n^\upsilon}\;,
        \quad 
        | z| \;\leq\; K
    \Big\}    
    \;.
\end{align*}
Roughly speaking, our strategy is:

\begin{proplist}
    \item We first use \Cref{lem:gf:fM:diff} to identify some $\epsilon_n = o(1)$ such that for all $z \in \D_{\upsilon,K}$, whenever $\Lambda(z) \leq \epsilon_n$, we can sharpen the bound to $\Lambda(z) \prec \delta_n$ with overwhelming probability (see \Cref{defn:overwhelming:prob} below), where $\delta_n \ll \epsilon_n$. This step will use the lower bound on $\Im(z)$ in the definition of $\D_{\upsilon, K}$, and restricts the choice of $\upsilon$;
    \item Observe that (i) implies that, for every $z \in \D_{\upsilon,K}$, $\Lambda(z)$ lies outside $(\delta_n, \epsilon_n)$ with overwhelming probability;
    \item {
We find some $z_0 \in \D_{\upsilon,K}$ with $|z_0| = \Theta(1)$
such that \Cref{lem:gf:fM:diff} applies and gives
$\Lambda(z_0)\prec n^{-1/2}\ll\delta_n\ll\epsilon_n$;
}
    \item We use the continuity of the function $\Lambda$ to note that, since $\Lambda$ cannot cross the interval $(\delta_n, \epsilon_n)$, and we have found one $z_0 \in \D_{\upsilon,K}$ such that $\Lambda(z_0) \leq \delta_n$, then necessarily $\Lambda(z) \leq \delta_n$ for all values of $z \in \D_{\upsilon,K}$. 
\end{proplist} 
Step (iv) needs to be handled with care, so that the continuity argument still holds under taking the union bound of different events with overwhelming probability. To make this argument formal, we use the same technical tool used in \cite{alt2017local}, i.e.~Lemma A.2 of \cite{ajanki2017universalitygeneralwignertypematrices}. We first make precise our notion of overwhelming probability:

\begin{definition}[Overwhelming probability, Definition 1.9 of \cite{ajanki2017universalitygeneralwignertypematrices}] \label{defn:overwhelming:prob}
For a sequence of $n$-indexed events $A_n$, we say that $A_n$ holds asymptotically with overwhelming probability (a.w.o.p) if for every $\kappa > 0$, there exists $n_\kappa > 0$ that depends only on the constants $s_*$, $L_1$, $L_2$, $\psi_1$ and $\psi_2$ in \Cref{asst:flat} {and the tolerance parameter $\upsilon$}, such that for all $n \geq n_\kappa$,
\begin{align*}
    \P\big( A_n \big) \;\geq\; 1 - n^{-\kappa}\;.
\end{align*}
\end{definition}

We restate Lemma A.2 of \cite{ajanki2017universalitygeneralwignertypematrices} in our notation:

\begin{lemma}[Bound propagation, Lemma A.2 of \cite{ajanki2017universalitygeneralwignertypematrices}] \label{lem:bound:propagate}
Suppose $C_1,D_1,D_2,D_3$ and $\eps_1$ are positive constants, depending only on the constants $s_*$, $L_1$, $L_2$, $\psi_1$ and $\psi_2$ in \Cref{asst:flat}. Let $\D^{(n)} \subseteq \C$ be a sequence of connected subsets of the complex upper half plane with only polynomially growing diameter, 
\begin{align*}
    \sup\{|z_1-z_2|: z_1,z_2 \in\D^{(n)} \}\leq (n+p)^{D_1} \;.
\end{align*}
Let $\varphi = (\varphi^{(n)}(z): z \in  \D^{(n)})_{n \in \N}$ be a sequence of non-negative random functions and $\Phi^{(n)}:\D^{(n)} \to ((n+p)^{-D_3},\infty)$ a sequence of deterministic functions on these sets. Suppose they satisfy the following conditions:
\begin{itemize}
\item Uniformly for all $z_1,z_2 \in \D^{(n)}$,
\begin{align*}
    |\varphi^{(n)}(z_1)-\varphi^{(n)}(z_2)|+|\Phi^{(n)}(z_1)-\Phi^{(n)}(z_2)|\,\leq\, C_1\2 n^{D_2}|z_1-z_2|^{\eps_1} \,;
    \tagaligneq \label{eq:weak:continuity}
\end{align*}
\item Uniformly for all $z \in \D^{(n)}$ 
\begin{align*}
    \text{a.w.o.p.}\qquad 
    \varphi^{(n)} \,\notin\, \bigl[\2\Phi^{(n)}(z)-n^{-D_3},\2\Phi^{(n)}(z)\2\bigr]
    \,;
    \tagaligneq \label{eq:gap:cond}
\end{align*}
\item There is a sequence $z_0^{\2(n)} \in \D^{(n)}$ such that
\begin{align*}
    \text{a.w.o.p.}\qquad 
    \varphi^{(n)}(\2z_0^{\2(n)})\,\leq\, \Phi^{(n)}(\2z_0^{\2(n)})
    \,;
    \tagaligneq \label{eq:init:cond}
\end{align*}
\end{itemize}
Then the sequence $\varphi$ satisfies the bound
\begin{align*}
    \text{a.w.o.p.}\qquad \text{for all }\; z \in \D^{(n)}\;:\quad
\varphi^{(n)}(z)\,\leq\, \Phi^{(n)}(z)\,.    
\end{align*}
\end{lemma}

The next lemma provides the continuity required for step (iv) of the proof sketch above, i.e.~the continuity condition required for \eqref{eq:weak:continuity}. Note that as with \cite{alt2017local}, the continuity argument is not required to be uniform in $n$ and $p$, and the Lipschitz constant is allowed to grow with $n$ and $p$.

\begin{lemma}[Continuity] \label{lem:lipschitz} Suppose \Cref{asst:flat} holds. %
Then for any $z_1, z_2 \in \D_{\upsilon,K}$ with $K \geq \sqrt{s_*}$, we have that almost surely
\begin{align*}
    \big| \Lambda(z_1) - \Lambda(z_2) \big| 
    \;\lesssim&\;
    {
    (1+K^2)n^{6\upsilon}
    \,
    }
    |z_1 - z_2| 
    \;.
\end{align*}
\end{lemma}

\begin{remark} \Cref{lem:lipschitz} provides a Lipschitz bound in contrast to the H\"older continuity result used in the proof of Theorem 4.2 of Lemma 3.8 of \cite{alt2017local}. This comes at the cost of a looser dependence on $n$.
\end{remark}

\begin{proof}[Proof of \Cref{lem:lipschitz}] By the first resolvent identity (see Lemma 6, p.578 of \cite{dunford1988linear}),
\begin{align*}
    G(z_1) - G(z_2)  \;=\; (z_1 - z_2) \, G(z_1) \, G(z_2)\;.
\end{align*}
{Since $H$ is Hermitian, the resolvent satisfies
$\normtwo{G(z)} \leq \frac{1}{\Im(z)}$ for all $z\in\bbH$.
Moreover, for $z_1,z_2\in\D_{\upsilon,K}$,
$\Im(z_1),\Im(z_2)\geq n^{-\upsilon}$. Hence,} we have that almost surely
\begin{align*}
    \big\| \gf(z_1) - \gf(z_2) \big\|_\infty
    \;=&\;
    \max_{1\leq l \leq n+p} \big| G_{ll}(z_1) - G_{ll}(z_2) \big| 
     \\
     =&\;
    {|z_1-z_2|
    \max_{1\leq l\leq n+p}
    \big|
    (G(z_1)G(z_2))_{ll}
    \big|}\\
     \;\leq&\; 
     {
     |z_1-z_2|\,
     \normtwo{G(z_1)}
     \normtwo{G(z_2)}}
     \;\leq\; 
     n^{2\upsilon} \, |z_1 - z_2| 
    \;,
\end{align*}
Meanwhile by \Cref{cor:dz:fM:bound},
\begin{align*}
    \big\| \partial_z  \fM(z) \big\|_\infty  
    {\;\leq\;  
    \mfrac{1}{(\Im(z))^2} \,
    \Big(
        1
        +
        \mfrac{
            s_*(\Im(z)|z|+s_*)
        }{
            (\Im(z))^4
        }
    \Big)}
    \;.
\end{align*}
Denote $[z_1,z_2]$ as the linear path connecting $z_1$ and $z_2$. {Since $\D_{\upsilon,K}$ is convex,
$[z_1,z_2]$ is contained in $\D_{\upsilon,K}$.} We have 
\begin{align*}
    &\;\big\| \fM(z_1) - \fM(z_2) \big\|_\infty 
     \\
    \;\leq\;&
    {|z_1-z_2|
    \sup_{z\in[z_1,z_2]}
    \Bigg(
        \mfrac{1}{(\Im(z))^2}
        +
        \mfrac{
            s_*\big(\Im(z)|z|+s_*\big)
        }{
            (\Im(z))^6
        }
    \Bigg)
    }
    \\
    \;\overset{(a)}{\leq}\;&
    {\Big(
        n^{2\upsilon}
        +
        s_*(K^2+s_*)n^{6\upsilon}
    \Big)
    |z_1-z_2|
    }
    \\
    \;{\lesssim}\;&
    {(1+K^2)n^{6\upsilon}
    |z_1-z_2|
    \;.
    }
\end{align*}
In $(a)$, we have used that for any $z \in [z_1, z_2]$, $\Im(z)\leq|z| \leq K$ {and $\Im(z)\geq n^{-\upsilon}$}. . To obtain the desired bound, we use the reverse triangle inequality to compute 
\begin{align*}
    \big| \Lambda(z_1) - \Lambda(z_2) \big| 
    \;=&\;
    \Big| 
        \big\| \gf(z_1) - \fM(z_1) \big\|_\infty 
        \,-\,
        \big\| \gf(z_2) - \fM(z_2) \big\|_\infty 
    \Big| 
    \\
    \;\leq&\;
    \big\| \big(\gf(z_1) - \fM(z_1)\big) - \big(\gf(z_2) - \fM(z_2) \big) \big\|_\infty 
    \\
    \;\leq&\;
    \| \gf(z_1) - \gf(z_2) \|_\infty 
    \,+\,
    \| \fM(z_1) - \fM(z_2) \|_\infty 
    \\
    \;\lesssim&\;
    {
    (1+K^2)n^{6\upsilon}
    |z_1-z_2|
    \;.
    }
\end{align*}
\end{proof}

\vspace{.5em}

We are now ready to prove the local laws.

\begin{proof}[Proof of \Cref{thm:local:law:VP}] The proof relies on applying \Cref{lem:bound:propagate}. We let $K \geq \sqrt{s_*}$ be some sufficiently large constant that depends only on $s_*$, $L_1$, $L_2$, $\psi_1$ and $\psi_2$ in \Cref{asst:flat}, and whose lower bound is specified in the proof.

{
We first consider the regime
$\Im(z)\leq\min\{1,s_*\}$, where
\Cref{lem:diag:error:to:perturb} applies.
The complementary regime
$\Im(z)>\min\{1,s_*\}$ will be treated separately below using
\Cref{lem:gf:fM:diff}.
}To apply \Cref{lem:diag:error:to:perturb}, we first note that for $z \in \D_{\upsilon, K}$, {with
$\Im(z)\leq \min\{1,s_*\}$}, the RHS of the indicator function in \Cref{lem:diag:error:to:perturb} can be bounded from below as 
\begin{align*}
    {
    \mfrac{C_0(\Im(z))^5}
    {\Im(z)|z|+s_*}
    }
    \;\geq\;
    {
    \mfrac{C_0}{K+s_*}\,
    n^{-5\upsilon}
    }
    \;.
    \tagaligneq \label{eq:ind:RHS:bound}
\end{align*}
The first bound of \Cref{lem:diag:error:to:perturb} together with the notation $\Lambda(z) = \| \gf(z) - \fM(z)\|_\infty$ implies
\begin{align*}
    \Lambda(z) \,\times\,\ind{\Big\{ \Lambda(z) \,\leq\, 
    {
        \mfrac{C_0}{K+s_*}\,
        n^{-5\upsilon}
        }
    \Big\}}
    &\;\leq\;
    \Lambda(z)  \,\times\,\ind{\Big\{ \Lambda(z)  \,\leq\,
    {
        \mfrac{C_0(\Im(z))^5}
        {\Im(z)|z|+s_*}
        }
    \Big\}}
    \\
    &\;\leq\;
    {\mfrac{
        3C_B\big(\Im(z)|z|+s_*\big)
    }{
        (\Im(z))^6
    }}
    \, 
    \| \df(z) \|_\infty
    \\
    &\;\overset{(a)}{\lesssim}\;
    {
    (K+s_*)\,n^{6\upsilon}}
    \,\| \df(z) \|_\infty
    \\
    &\;\overset{(b)}{\prec}\;
    {
    (K+s_*)\,n^{6\upsilon}}
    \, 
    \Big( 
        \mfrac{1}{\sqrt{n} \, \Im(z)} + \mfrac{1}{n \, (\Im(z))^3}
    \Big)
    \\
    &\;\overset{(c)}{\prec}\;
    {
    (K+s_*)\,n^{6\upsilon}}
    \, 
    \Big( 
        n^{\upsilon - \frac{1}{2} }
        +
        n^{3\upsilon - 1}
    \Big)
    \\
    &\;=\;
     {
    (K+s_*)\,
    n^{7\upsilon-\frac12}}
    \, 
    \big( 
        1
        +
        n^{2\upsilon - \frac{1}{2} }
    \big)
    \\
    &\;\overset{(d)}{\prec}\;
    {
    (K+s_*)\,
    n^{7\upsilon-\frac12}
    }
    \;.
\end{align*}
In $(a)$ above, we have used that $z \in \D_{\upsilon, K}$ {and
$\Im(z)\leq\min\{1,s_*\}$}, and in $(b)$ we have applied \Cref{lem:perturb:bound}; {in $(c)$, we have used that
$\Im(z)\geq n^{-\upsilon}$; }in $(d)$, we have used that $\upsilon \leq \frac{1}{4}$ by assumption. We now define 
\begin{align*}
    \epsilon_n 
    \;\coloneqq&\; 
      {
    \mfrac{C_0}{K+s_*}\,n^{-5\upsilon}
    }
    &\text{ and }&&
    \delta_n 
    \;\coloneqq&\;  
    {
    (K+s_*)\,n^{7\upsilon-\frac12}
    }
    \;,
\end{align*}
and note that since  $\upsilon < {\frac{1}{24}}$, we have
\begin{align*}
    \epsilon_n \;\gg\; \delta_n\;.
    \tagaligneq \label{eq:K:restrict}
\end{align*}
{
Therefore, for $z\in\D_{\upsilon,K}$ with
$\Im(z)\leq\min\{1,s_*\}$, we have $\Lambda(z) \leq \epsilon_n$ implies $\Lambda(z) \prec \delta_n$.
}

{
For the complementary regime
$\Im(z)>\min\{1,s_*\}$, the bounds in
\Cref{lem:abs:fM:bound} and \Cref{lem:bound_B_inverse}
imply that both $\|\fM(z)\|_\infty$ and
$\|\fB^{-1}(z)\|_\infty$ are bounded by constants depending only on
$K$ and $s_*$. Since $\epsilon_n\to0$, for all sufficiently
large $n$, if we have $\Lambda(z)\leq\epsilon_n$,
the condition in \Cref{lem:gf:fM:diff} is satisfied. Together with
\Cref{lem:perturb:bound}, this gives
\begin{align*}
    \Lambda(z)
    \;\prec\;
    n^{-1/2}
    \;\prec\;
    \delta_n
    \;.
\end{align*}
}
Therefore, uniformly over all $z \in \D_{\upsilon, K}$, $\Lambda(z) \leq \epsilon_n$ implies $\Lambda(z) \prec \delta_n$, and therefore asymptotically with overwhelming probability, we can identify for example a forbidden set for $\Lambda(z)$ as
\begin{align*}
    \Lambda(z) \;\not\in\; %
    {
    \big[
        \epsilon_n - n^{7\upsilon-\frac{1}{2}},
        \,\epsilon_n
    \big]
    }
    \;.
\end{align*}
This shows that \eqref{eq:gap:cond} in \Cref{lem:bound:propagate} is satisfied with 
\begin{align*}
    \varphi^{(n)}(z)
    \;=\;
    \Lambda(z)
    \;,
    \qquad
    \Phi^{(n)}(z) \;=\; \epsilon_n\;,
    \qquad 
    D_3 \;=\; {
    \mfrac{1}{2}-7\upsilon
    }
    \;,
    \qquad 
    \D^{(n)} \;=\; \D_{\upsilon, K} \;.
\end{align*}
Meanwhile under the same choice of $\varphi^{(n)}$, $\Phi^{(n)}$ and $\D^{(n)}$, the continuity condition \eqref{eq:weak:continuity}  is satisfied because of \Cref{lem:lipschitz}, with 
\begin{align*}
    D_2 \;=\; {
    6\upsilon
    }%
    \;,
    \qquad 
    \varepsilon_1
    \;=\;
    1\;.
\end{align*}
Finally to check \eqref{eq:init:cond}, it suffices to find some $z_0 \in \D_{\upsilon, K}$ with $\Lambda(z_0) \ll \epsilon_n$. Consider the crude bound 
\begin{align*}
    \Lambda(z)
    \;\leq\;
    \| \gf(z) \|_\infty
    \,+\,
    \| \fM(z) \|_\infty
    \;\leq\; 
    \mfrac{2}{\Im(z)}\;,
\end{align*}
where we have combined the generic resolvent bound \eqref{eq:resolvent:bound} and \Cref{lem:abs:fM:bound} by the triangle inequality. Then there exists some $z_0 \in \bbH$ with a sufficiently large $\Im(z_0) > 0$ that depends only on $s_*$, $L_1$, $L_2$, $\psi_1$ and $\psi_2$ in \Cref{asst:flat}, such that 
\begin{align*}
{
    \|\fM(z_0)\|_\infty \leq 1,
    \qquad
    \|\fB^{-1}(z_0)\|_\infty \leq 2,
    }
    \qquad
    \Lambda(z_0)
    \;\leq\; 
    \mfrac{2}{\Im(z_0)}
    \;\leq\;
    {
    \mfrac{1}{4\max\{1,s_*\}}
    }
    \;.
\end{align*}
We take $K \geq |z_0|$ such that $z_0 \in \D_{\upsilon,K}$. Now applying {\Cref{lem:gf:fM:diff}},
\begin{align*}
    \Lambda(z_0)
    \;\leq&\;
    {6\| \df(z_0) \|_\infty}
    \;\overset{(a)}{\prec}\; \mfrac{1}{\sqrt{n}}
    \;.
\end{align*}
{
In $(a)$, we have applied \Cref{lem:perturb:bound} and used that
$\Im(z_0)$ is a fixed positive constant.
} In particular since $\upsilon > 0$, 
\begin{align*}
    \mfrac{1}{\sqrt{n}} \;\ll\; \delta_n \;=\; 
    {
    (K+s_*)n^{7\upsilon-\frac{1}{2}}
    }
    \;.
\end{align*}
which then implies  \eqref{eq:init:cond} is fulfilled since
\begin{align*}
    \Lambda(z_0) \;\prec\; \delta_n \;\ll\; \epsilon_n\;.
\end{align*}
Applying \Cref{lem:bound:propagate} gives that for all $z \in \D_{\upsilon, K}$, 
\begin{align*}
    \text{a.w.o.p.}
    \qquad
    \|\gf(z) - \fM(z) \|_\infty \;=\; \Lambda(z) \;\leq\; %
    \epsilon_n 
    \;=\;
    {
    \mfrac{C_0}{K+s_*}\,
    n^{-5\upsilon}
    }
    \tagaligneq \label{eq:awop:propagated:bound}
\end{align*}
and therefore 
\begin{align*}
    \max_{1 \leq l \leq n+p}
    \big| 
        G_{l,l}(z)
        -
        \fM_l(z) 
    \big|
    \;=&\;
    \|\gf(z) - \fM(z) \|_\infty 
    \\
    \;\prec&\;
    {
    \mfrac{1}{K+s_*}\,
    n^{-5\upsilon}
    }
    \;\prec\;
     {
    n^{-5\upsilon}
    }%
    \;.%
\end{align*}
Now note that the off-diagonal bound follows by \Cref{lem:offdiag:bound} and that $\Im(z) \geq \frac{1}{n^\upsilon}$ for $z \in \D_{\upsilon, K}$:
\begin{align*}
    \max_{l \neq j} | G_{lj}(z) | 
    \;\prec\; 
    \mfrac{1}{\sqrt{n}} \,\mfrac{1}{ \min\{1, \Im(z) \}^3 }
    \;\prec\;
    n^{ 3\upsilon - \frac{1}{2} }
    \;\ll\;
    \delta_n \;\ll\; \epsilon_n
    \;.
\end{align*}
Therefore by the triangle inequality, we obtain the desired bound that 
\begin{align*}
    \max_{1 \leq l, j \leq n+p}
    \big| 
        G_{l,j}(z)
        -
        \fM_l(z) \, \ind_{\{l = j\}}
    \big|
    \;=&\;
    \max\Big\{ 
        \|\gf(z) - \fM(z) \|_\infty
        \,,\,
        \max_{l \neq j} | G_{lj}(z) | 
    \Big\}
    \\
    \;\prec&\;
    {
    n^{-5\upsilon}
    }
    \;.
\end{align*}
\end{proof}

\section{Risk control for the variance profile model} \label{appendix:risk}
    
Recall from \eqref{eq:R:formula} that the risk of ridge regression at penalty $\lambda > 0$ is given by $\cR^\lambda(X) = \cR_B^\lambda(X) + \cR_V^\lambda(X)$, where, for $W_n = \frac{1}{n} X^\top X$ and $\bar \Sigma = \frac{1}{n} \sum_{i \leq n} \Var[X_i]$, we have expressed
\begin{align*}
    \cR_B^\lambda(X)
    \;=&\;
        \big\| \bar \Sigma^{1/2}
        \big( (W_n + \lambda I_p)^{-1}  W_n - I_p \big)
        \beta \big\|^2 \,,
    \\
    \cR_V^\lambda(X)
    \;=&\;
    \mfrac{\sigma^2_\epsilon}{n}
    \Tr\big( 
        \bar \Sigma \, W_n ( W_n + \lambda I_p)^{-2}
    \big)
    \;.
\end{align*}
In this section, the data matrix $X$ is given as in \Cref{mod:vp} under \Cref{asst:flat}, with the variance profile $S$. We use the standard Hermitization trick to form 
\begin{align*}
    H \;\coloneqq&\; \begin{psmallmatrix}
        0 & \frac{1}{\sqrt{n}} X \\
        \frac{1}{\sqrt{n}} X^\top & 0
    \end{psmallmatrix}
    &\text{ and }&&
    \fS \;\coloneqq&\; \begin{psmallmatrix}
        0 & S \\
        S^\top & 0
    \end{psmallmatrix}
    \;.
\end{align*}

Recall from \eqref{eq:empirical:resolvent} the notation for the resolvent
\begin{align*}
    G(z) \;\coloneqq&\; (H - z I_{n+p})^{-1}\;.
\end{align*}
Notice that 
\begin{align*}
    G(z) 
    \;=&\; 
    (H - z I_{n+p})^{-1} 
    \;=\; 
    (H + z I_{n+p}) (H^2 - z^2 I_{n+p})^{-1}
    \\
    \;=&\;
    \begin{psmallmatrix}
        z I_n & \frac{1}{\sqrt{n}}  X \\
       \frac{1}{\sqrt{n}}  X^\top & z I_p
    \end{psmallmatrix}
    \begin{psmallmatrix}
        \frac{1}{n}  X X^\top - z^2 I_n & 0 \\
        0 & \frac{1}{n} X^\top X - z^2 I_p
    \end{psmallmatrix}^{-1}
    \\
    \;=&\;
    \begin{psmallmatrix}
        z ( \frac{1}{n} X X^\top - z^2 I_n)^{-1} & \frac{1}{\sqrt{n}} X ( \frac{1}{n} X^\top X - z^2 I_p)^{-1} \\
        \frac{1}{\sqrt{n}} X^\top ( \frac{1}{n} X X^\top - z^2 I_n)^{-1} & z ( \frac{1}{n} X^\top X - z^2 I_p)^{-1}
    \end{psmallmatrix}
    \;,
\end{align*}
which implies that {taking the branch
$\sqrt{-\lambda}=i\sqrt{\lambda}\in\bbH$, }we may compute the resolvent of $W_n = \frac{1}{n} X^\top X$ that appears in the risk above as
\begin{align*}
    \big( W_n + \lambda I_p \big)^{-1}
    \;=\;
    \mfrac{G_2(\sqrt{-\lambda})}{\sqrt{-\lambda}}\;,
    \tagaligneq \label{eq:resolvent:Wn}
\end{align*}
where we have defined the $p \times p$ matrix $G_2(z)$ by 
\begin{align*}
    (G_2(z))_{l,l'}  \;\coloneqq\; G_{n+l,n+l'}(z) 
    \text{ for } 1 \leq l,l' \leq p\;.
\end{align*}

\vspace{.5em}

Our results in \Cref{appendix:local:law} allow us to approximate $G(z)$ by $\diag\{\fM(z)\}$, where the deterministic function $\fM: \bbH \rightarrow \bbH^{n+p}$ satisfies the quadratic vector equation (QVE)
\begin{align*}
    - \diag\{ \fM(z) \}^{-1} \;=\; z  \, I_{n+p} + 
    \diag\big\{
    \begin{psmallmatrix}
        0 & S \\
        S^\top & 0
    \end{psmallmatrix}  \, \fM(z)
    \big\}
    \;,
    \qquad 
    z \in \bbH
    \;,
    \tagaligneq \label{eq:QVE}
\end{align*}
and where, for any generic $\R^b$ vector $\bv=(v_l)_{1 \leq l \leq b}$, we denote the diagonal matrix
\begin{align}
    \diag\{\bv\} \;=\; \begin{psmallmatrix}
        v_1 & & \\ 
        & \ddots & \\
        & & v_b
    \end{psmallmatrix}
    \;.
\end{align}
The existence and uniqueness of $\fM$ are guaranteed by Theorem 2.1 and (3.6) of  \cite{alt2017local}. Theorem 2.1 of \cite{Ajanki_2019} additionally guarantees that, for each $1 \leq l \leq n+p$, there exists a symmetric probability measure $\rho_l$ on $\R$ such that 
\begin{align*}
    \fM_l(z) \;=\; \mint_\R \mfrac{1}{t-z} \rho_l(dt)\;,
    \tagaligneq \label{eq:coordinatewise:stieltjes:rep}
\end{align*}

\vspace{.5em}

The results in the main text are given in terms of a function $\br: [0,\infty) \rightarrow \C^p$ that is explicitly defined as, for $1 \leq l \leq p$,
\begin{align*}
    r_l(\lambda) \;\coloneqq\; -\sqrt{-\lambda} \; \fM_{n+l}(\sqrt{-\lambda})
    \;.
    \tagaligneq \label{defn:r}
\end{align*}
Part of the proof lies in showing that $\br$ is the unique solution to the system of equations that, for every $1 \leq l \leq p$,
\begin{align*} 
    1
    \;=\; 
    r_l(\lambda)
    +
    \msum_{i=1}^n \mfrac{S_{il} r_l(\lambda)}{\lambda + \sum_{j=1}^p S_{ij} r_j(\lambda)}
    \;,
\end{align*}
subject also to the constraint that
\begin{align*}
    r_l(0) \;\coloneqq\; \lim_{\lambda \downarrow 0} r_l(\lambda) \;\in\; [0,1]
    \qquad 
    \text{for every $1 \leq l \leq p$. }
\end{align*}

\vspace{.5em}

Our main theoretical results are \Cref{thm:bias} and \Cref{thm:var} below, which provide the limiting expressions on the bias and the variance respectively.

\begin{theorem}[Bias term] \label{thm:bias} Consider \Cref{mod:vp} with \Cref{asst:flat}, where $s_*$, $L_1$, $L_2$, $\psi_1$ and $\psi_2$ as well as $\| \beta \|^2{ > 0}$ are universal constants. Suppose in addition that $\| \beta \|_{l_1} \leq n^{\frac{1}{4} - \delta}$ and $\lambda \in (n^{-\alpha} , \lambda_0)$ for some fixed $\delta, \alpha { > 0}$ {and $\lambda_0 \in (0,1)$.}
Then there exists some universal constant $\alpha_0 > 0$ such that, as long as $\alpha < \alpha_0$, 
\begin{align*}
    &\;
    \Bigg| 
        \,
        \cR_B^\lambda(X)  - 
        \sum_{l=1}^p \beta_l^2 e_l^\top
        \Big(
            \diag\{ \br(\lambda)\}^{-1} 
            -
                \diag\{ \br(\lambda)\}
                \,
                S^\top 
                \diag\{ \lambda \bone_n + S \br(\lambda) \}^{-2}
                S
        \Big)^{-1} 
        \bar \Sigma \br(\lambda) 
    \Bigg| 
    \\
    &\;\xrightarrow{\rm a.s.}\; 0\;.
\end{align*}
$e_l$ is the $l$-th standard basis vector in $\R^p$ and $\br(\lambda) = (r_l(\lambda))_{l \leq p}$ is the solution to the system of equations that, for every $1 \leq l \leq p$,
\begin{align*} 
    1
    \;=\; 
    r_l(\lambda)
    +
    \msum_{i=1}^n \mfrac{S_{il} r_l(\lambda)}{\lambda + \sum_{j=1}^p S_{ij} r_j(\lambda)}
    \;,
    \tagaligneq \label{eq:r:main}
\end{align*}
defined on $\lambda > 0$. Moreover for every $1 \leq l \leq p$, $r_l(0) \coloneqq \lim_{\lambda \downarrow 0} r_l(\lambda)$ exists with 
\begin{align*}
    r_l(0) \;\in\; [0,1]\;,
\end{align*}
and there exists a symmetric probability measure $\nu_l$ on $\R$ such that for every $\lambda > 0$,
\begin{align*}
    r_l(\lambda) 
    \;=&\;
    \mint_\R 
        \mfrac{\lambda}{t^2 + \lambda} 
        \, 
    \nu_l(dt)
    &\text{ and thus }&&
    r_l(0) \;=&\; \nu_l(\{0\})
    \;.
\end{align*}
\end{theorem}

\begin{theorem}[Variance term] \label{thm:var} Consider \Cref{mod:vp} with \Cref{asst:flat}, where $s_*$, $L_1$, $L_2$, $\psi_1$ and $\psi_2$ are universal constants. Suppose in addition that $\lambda \in (n^{-\alpha} , \lambda_0)$ for some fixed $\alpha, \lambda_0 > 0$. Then there exists some universal constant $\alpha_1 > 0$ such that, as long as $\alpha < \alpha_1$, 
\begin{align*}
    \Big| 
        \cR_V^\lambda(X) 
        - 
        \mfrac{\sigma^2_\epsilon}{n}
        \,
        \msum_{l=1}^p
        \,\bar \Sigma_{ll}\,
        \partial r_l(\lambda)
    \Big| 
    \;\xrightarrow{\rm a.s.}\; 0\;.
\end{align*}
Moreover, $\partial r_l(0) \coloneqq \lim_{\lambda \downarrow 0} \partial r_l(\lambda)$ exists, with 
\begin{align*}
    \partial r_l(0)
    \;=\;
    \mint_{\R\setminus\{0\}}
        \mfrac{1}{t^2}
        \,
        \nu_l(dt)
    \;\in\; [0,\infty]\;.
\end{align*}
\end{theorem}

\vspace{.5em}

The rest of this section is dedicated to the proofs of \Cref{thm:bias} and \Cref{thm:var}, and is organized as follows:
\begin{itemize}
    \item \Cref{appendix:properties:r} includes properties of $
    \br$;
    \item \Cref{appendix:proof:stability:z:eta} controls the stability of the variance profile $S$ and the matrix resolvent $\fM(z)$ under a small perturbation of the data matrix;
    \item \Cref{appendix:proof:bias} proves the result on the bias term of the risk (\Cref{thm:bias});
    \item  \Cref{appendix:proof:var} proves the result on the variance term of the risk (\Cref{thm:var}).
\end{itemize}
The proof strategy adapts some ideas from the i.i.d.~case in \cite{hastie2022surprises}, although the general variance profile model introduces several technical difficulties not present in the i.i.d.~model. These are handled by the local laws for \Cref{mod:vp}, which we derive in \Cref{thm:local:law:VP} in \Cref{appendix:local:law}.

\subsection{Properties of \texorpdfstring{$\br$}{r}} \label{appendix:properties:r}

In this section, we extend the definition \eqref{defn:r} of $\br(\lambda)$, i.e.
\begin{align*}
    r_l(\lambda) \;\coloneqq\; -\sqrt{-\lambda} \, \fM_{n+l}\big(\sqrt{-\lambda} \big)
    \;,
\end{align*}
to all $\lambda \in \C \setminus (-\infty, 0]$; note in particular that for this range of $\lambda$, $z=\sqrt{-\lambda} \,\in\, \bbH$.

\vspace{.5em}

Our first result shows that $\br(\lambda)$ is the unique solution to the system of equations in \Cref{thm:bias} over the domain $\lambda \in \C \setminus (-\infty, 0]$, which implies that it is a solution to the system of equations over the domain $\lambda > 0$.

\begin{lemma} \label{lem:r:eqns} $\br(\lambda)$ is the unique solution to the system of equations \eqref{eq:r:main} extended to the domain $\lambda \in \C \setminus (-\infty,0]$, i.e.~that for every $1 \leq l \leq p$,
\begin{align*} 
    1
    \;=\; 
    r_l(\lambda)
    +
    \msum_{i=1}^n \mfrac{S_{il} r_l(\lambda)}{\lambda + \sum_{j=1}^p S_{ij} r_j(\lambda)}
    \;.
\end{align*}
    
\end{lemma}

\begin{proof}[Proof of \Cref{lem:r:eqns}] Write $\fM^{(2)}(z) \coloneqq (\fM_{n+l}(z))_{1 \leq l \leq p}$. We first rewrite the QVE \eqref{eq:QVE} as 
\begin{align*}
    -
    \diag\{ \fM^{(2)}( z) \}^{-1} 
    \;=\; 
    z  \, I_p
    -
     S^\top 
     \Big( z I_n + \diag\big\{ S \, \fM^{(2)}(z) \big\} \Big)^{-1}
    \;,
    \qquad 
    z \in \bbH
    \;.
\end{align*}
Since $r_l(\lambda) = -\sqrt{-\lambda} \, \fM^{(2)}_l(\sqrt{-\lambda})$, the above is equivalent to 
\begin{align*}
    \diag\{ \br(\lambda) \}^{-1} 
    \;=\; 
     I_p 
    +
     S^\top \,
     \big( 
        \lambda \, I_n 
        +  
        \diag\{ S \, \br(\lambda) \} 
    \big)^{-1}
    \;,
\end{align*}
which is equivalent to the system of equations \eqref{eq:r:main} that, for every $1 \leq l \leq p$,
\begin{align*} 
    1
    \;=\; 
    r_l(\lambda)
    +
    \msum_{i=1}^n \mfrac{S_{il} r_l(\lambda)}{\lambda + \sum_{j=1}^p S_{ij} r_j(\lambda)}
    \;.
\end{align*}
Uniqueness follows from the uniqueness of $\fM(z)$ as a solution to the QVE over $z \in \bbH$.
\end{proof}

The next result concerns the range of values $\br$ and $\partial \br$ can take. We will denote 
\begin{align*}
    \nu_l \;\coloneqq\; \rho_{n+l}   \;,
\end{align*}
where $\rho_l$ is the symmetric probability measure in the Stieltjes transform representation \eqref{eq:coordinatewise:stieltjes:rep} of $\fM_l(z)$.

\begin{lemma} \label{lem:r:values} For $\lambda \in \C \setminus (-\infty,0]$ and for all $1 \leq l \leq p$, the following statements hold:
\begin{proplist}
    \item $r_l(\lambda)  = \int_\R \frac{\lambda}{t^2 + \lambda} \, \nu_l(dt)$;
    \item $r_l(0) \coloneqq \lim_{\lambda \rightarrow 0} r_l(\lambda)$ exists and $r_l(0) = \nu_l(\{0\}) \in [0,1]$;
    \item $\partial r_l(\lambda) =  \int_{\R} \frac{t^2}{(t^2 + \lambda)^2} \, \nu_l(dt)$;
    \item $\partial r_l(0) =\lim_{\lambda\rightarrow 0}\partial r_l(\lambda)$ exists and $ \partial r_l(0) = \int_{\R\setminus\{0\}}
        \frac{1}{t^2}
        \,
        \nu_l(dt) \in [0,\infty]$.
\end{proplist}
    
\end{lemma}

\begin{proof}[Proof of \Cref{lem:r:values}] We first use the Stieltjes transform representation \eqref{eq:coordinatewise:stieltjes:rep} to obtain
\begin{align*}
    r_l(\lambda) 
    \;=\;
    - \sqrt{-\lambda} \; \fM_{n+l}\big(\sqrt{-\lambda}\big) 
    \;=\;
    \mint_\R 
        \mfrac{- \sqrt{-\lambda}}{t - \sqrt{-\lambda}} \, \nu_l(dt)
    \;.
\end{align*}
Since $\nu_l$ is symmetric, we may replace the integrand by its even part in $t$:
\begin{align*}
    r_l(\lambda) 
    \;=&\;
    \mfrac{1}{2}
    \mint_\R 
          \mfrac{- \sqrt{-\lambda}}{t - \sqrt{-\lambda}}  + \mfrac{- \sqrt{-\lambda}}{- t - \sqrt{-\lambda}} 
        \, 
    \nu_l(dt)
    \;=\;
    \mint_\R 
        \mfrac{\lambda}{t^2 + \lambda} 
        \, 
    \nu_l(dt)
    \;.
\end{align*} 
Meanwhile, $r_l(0) = \lim_{\lambda \rightarrow 0} r_l(\lambda)$ exists by noting that
\begin{align*}
    r_l(0) 
    \;=\; 
    \nu_l( \{0 \}) \;\in\; \R^+_0\;.
\end{align*}
Moreover by taking $\lambda \rightarrow 0$ in the equation in \Cref{lem:r:eqns}, we get that $r_l(0)$ must satisfy the equation 
\begin{align*}
    1
    \;=\; 
    r_l(0)
    +
    \msum_{i=1}^n \mfrac{S_{il} r_l(0)}{ \sum_{j=1}^p S_{ij} r_j(0)}
    \;=\;
    r_l(0)
    \Big( \msum_{i=1}^n \mfrac{S_{il}}{ \sum_{j=1}^p S_{ij} r_j(0)} \Big)
    \;.
\end{align*}
Since $S_{il}$'s and $r_j(0)$'s are all non-negative, this implies $r_l(0) \leq 1$ for all $1 \leq l \leq p$. This gives 
\begin{align*}
    r_l(0) \in [0,1] \qquad \text{ for all } 1 \leq l \leq p\;.
\end{align*}
For results concerning the derivative, we can differentiate under the integral sign to obtain 
\begin{align*}
    \partial r_l(\lambda)
    \;=&\;
    \partial \Big( 
        \mint_\R 
            \mfrac{\lambda}{t^2 + \lambda} 
            \, 
        \nu_l(dt)
    \Big)
    \;=\;
    \mint_{\R}
        \mfrac{t^2}{(t^2 + \lambda)^2} 
        \,
        \nu_l(dt)\;.
\end{align*}
As $\lambda \rightarrow 0$, the integrand converges pointwise to $\ind_{\{t\neq 0\}}t^{-2}$. Therefore dominated convergence gives
\begin{align*}
    \partial r_l(0)
    \;\coloneqq\;
    \lim_{\lambda\rightarrow 0}\partial r_l(\lambda)
    \;=\;
    \mint_{\R\setminus\{0\}}
        \mfrac{1}{t^2}
        \,
        \nu_l(dt)
    \;\in\; [0,\infty]\;.
\end{align*}
\end{proof}

\subsection{Perturbation bounds on the variance profile \texorpdfstring{$S$}{S} and the matrix resolvent limit \texorpdfstring{$\fM$}{M}} \label{appendix:proof:stability:z:eta}

A key proof ingredient of Hastie et al.~\cite{hastie2022surprises} in the i.i.d.~case is to study the matrix resolvent under a suitably perturbed data model. Our proof also uses this idea. As such, we need to verify that the local law (\Cref{thm:local:law:VP}) applies also to the perturbed variance profile model, and that the matrix resolvent limit, $\fM$, is well-behaved under the perturbation. Specifically, we recall $\bar \Sigma = \frac{1}{n} \sum_{i \leq n} \Var[X_i]$ and, for a small $\eta > 0$, we consider the perturbed data matrix 
\begin{align*}
    X_\eta \;\coloneqq\; X  \big(I_p + \eta \bar \Sigma \big)^{-1/2} \;.
\end{align*}
The first result shows that $X_\eta$ is also a data matrix under \Cref{asst:flat} that satisfies \Cref{asst:flat}, which makes the local law (\Cref{thm:local:law:VP}) applicable. In particular if $\eta$ is bounded from above, $X_\eta$ satisfies \Cref{asst:flat} with its parameters independent of $\eta$.

\begin{lemma} \label{lem:asst:VP} Suppose $X$ under the variance profile model (\Cref{mod:vp})  satisfies \Cref{asst:flat} with parameters $s_*, \psi_1, \psi_2, L_1, L_2$. Then $X_\eta$ is also a data matrix under \Cref{mod:vp}, whose variance profile satisfies \Cref{asst:flat} with parameters 
\begin{align*}
    s_* \,,\, \mfrac{\psi_1}{(1+\eta s_*)^{2L_1}} \,,\, \mfrac{\psi_2}{(1+\eta s_*)^{2L_2}} \,,\, L_1 \,,\, L_2\;.
\end{align*}
\end{lemma}

We will make use of the following result, where we inherit the notation of \Cref{appendix:local:law} to write, for a matrix $A \in \C^{b \times b}$,
\begin{align*}
    \normtwo{A} 
    \;\coloneqq&\;
    \sup_{v \neq 0} \mfrac{\|A v\|}{\| v\|}  
    &\text{ and }&&
    \norminf{A} 
    \;\coloneqq&\;
    \sup_{v \neq 0} \mfrac{\|A v\|_\infty}{\| v\|_\infty}  
    \;.
\end{align*}

\begin{lemma} \label{lem:Sigma:op:bound} Under \Cref{mod:vp} with \Cref{asst:flat}, $\max\big\{ \| \bar \Sigma \|_2 \,,\, \|\bar\Sigma\|_\infty\big\} \leq s_*$.
\end{lemma}

\begin{proof}[Proof of \Cref{lem:Sigma:op:bound}] The $\| \argdot \|_2$ bound follows by computing
\begin{align*}
    \| \bar \Sigma \|_2
    \;=\;
    \Big\| \mfrac{1}{n} \msum_{i=1}^n \Var[X_i] \Big\|_2
    \;\leq&\;
    \msum_{i=1}^n \Big\| \mfrac{\Var[X_i]}{n} \Big\|_2
    \;\overset{(a)}{=}\;
    \msum_{i=1}^n \max_{1 \leq j \leq p} S_{ij} 
    \\
    \;\leq&\; 
    \mfrac{n}{p+n} s_*
    \;\leq\; 
    s_*\;.
\end{align*}
In $(a)$, we have noted that $\Var[X_i]$ is a diagonal matrix under \Cref{mod:vp}. The $\| \argdot \|_\infty$ bound is identical by noting also that $\Var[X_i]$ is a diagonal matrix under \Cref{mod:vp}. 
\end{proof}

\begin{proof}[Proof of \Cref{lem:asst:VP}] We first note that under \Cref{mod:vp}, each $\Var[X_i]$ is a diagonal matrix and so is $\bar \Sigma$. The variance profile $S^\eta$ corresponding to $X_\eta$ is given entry-wise as 
\begin{align*}
    S^\eta_{ij} 
    \;=\; 
    \Var\Big[X_{ij} \mfrac{1}{(1 + \eta \bar \Sigma_{jj}  )^{1/2}}  \Big] 
    \;=\; 
    \mfrac{S_{ij}}{1+ \eta \, \bar \Sigma_{jj}}
    \;\overset{(a)}{\leq}\; 
    \mfrac{s_*}{1+ \eta \, \bar \Sigma_{jj}}
    \;\leq\; 
    s_*
    \;,
    \tagaligneq \label{eq:S:eta:calc}
\end{align*}
where we have used the first condition of \Cref{asst:flat} in $(a)$ above. To verify  the second and third condition, we first observe that for any $d_1, d_2 \in \N$, if the $d_1 \times d_2$ matrices $A, B, A', B'$ satisfy that $A \geq B \geq 0$ entrywise and $A' \geq B' \geq 0$ entrywise, then 
\begin{align*}
    \big( A A' \big)_{ij} \;=\; \msum_{k} A_{ik} A'_{kj} \;\geq\; \big( B B' \big)_{ij}  \;\geq\; 0 \;,
\end{align*}
i.e.~$AA' \geq BB' \geq 0$ entrywise. Meanwhile by \eqref{eq:S:eta:calc} and  \Cref{lem:Sigma:op:bound},
\begin{align*}
    S^\eta \;\geq\; \mfrac{1}{1 + \eta s_*} S \;\geq\; 0 \qquad \text{ entrywise. }
\end{align*}
Under \Cref{asst:flat}, $(SS^\top)^{L_1} \geq \frac{\psi_1}{n+p}$ entrywise and  $(S^\top S)^{L_2} \geq \frac{\psi_2}{n+p}$ entrywise. As such 
\begin{align*}
    \big(S^\eta (S^\eta)^\top\big)^{L_1} 
    \;\geq&\; 
    \mfrac{1}{(1+ \eta s_*)^{2L_1}} \big(SS^\top\big)^{L_1} 
    \;\geq\; 
    \mfrac{\psi_1 }{(1+ \eta s_*)^{2L_1} (n+p)}
    \quad 
    \text{ entrywise, }
    \\
    \big( (S^\eta)^\top (S^\eta)\big)^{L_2} 
    \;\geq&\; 
    \mfrac{1}{(1+ \eta s_*)^{2L_2}} \big(S^\top S\big)^{L_2} 
    \;\geq\; 
    \mfrac{ \psi_2 }{(1+ \eta s_*)^{2L_2} (n+p)}
    \quad 
    \text{ entrywise. }
\end{align*}
This concludes the proof.
\end{proof}

\vspace{.5em}

\Cref{lem:asst:VP} shows that $X_\eta$ falls under \Cref{mod:vp}, which allows us to inherit the definition of $\fM$ for its perturbed counterpart. Specifically, we denote the corresponding deterministic approximation from \Cref{thm:local:law:VP} as 
\begin{align*}
    \fM^\eta(z)\;,
\end{align*}
which is the unique solution to the modified QVE (c.f.~\eqref{eq:combined_QVE})
\begin{align*}
    - \diag\{ \fM^\eta(z) \}^{-1} \;=\; z  \, I_{n+p} + 
    \diag\{
    \Sf^\eta \, \fM^\eta(z)
    \}
    \;,
    \qquad 
    z \in \bbH\;,
    \tagaligneq \label{eq:QVE:eta}
\end{align*}
where
\begin{align*}
    \fS^\eta \;\coloneqq\; 
    \begin{psmallmatrix} 
        0 & S (I_p + \eta \bar \Sigma)^{-1} \\
        (I_p + \eta \bar \Sigma)^{-1} S^\top & 0
    \end{psmallmatrix}
    \;\in\;
    \R^{(n+p) \times (n+p)}\;.
\end{align*}
The next result controls the $\eta$-derivatives of $\fM^\eta$, which quantify the stability of $\fM^\eta(z)$ under the perturbation $\eta$. We use the notation $\lesssim$ from \Cref{appendix:local:law} below, which indicates inequalities up to a multiplicative positive and bounded constant that depends only on $\gamma = \lim \frac{p}{n}$ and the constants $s_*$, $L_1$, $L_2$, $\psi_1$ and $\psi_2$ in \Cref{asst:flat}.

\begin{lemma}[Derivatives of $\fM^\eta$] \label{lem:Meta:derivatives} Suppose $\eta \in (0, 1/s_*)$. For $z \in \bbH$ with $|z| \leq 1$, there exists a constant $c > 0$ that  depends only on $s_*, \psi_1, \psi_2, L_1, L_2$ from \Cref{asst:flat} such that 
\begin{align*}
    \max\Big\{ \big\| \fM^\eta (z) \|_\infty \,,\, \big\| \partial_\eta \fM^\eta(z)  \big\|_\infty  \,,\, \big\| \partial_\eta^2 \fM^\eta(z)  \big\|_\infty  \,,\, \big\| \partial_\eta^3 \fM^\eta(z)  \big\|_\infty  \Big\}
    \;\lesssim\;
    (\Im(z))^{-c}\;.
\end{align*}
\end{lemma}

\begin{proof}[Proof of \Cref{lem:Meta:derivatives}] We first rewrite \eqref{eq:QVE:eta} as 
\begin{align*}
    0
    \;=&\; 
    \Big( \mfrac{1}{\fM_1^\eta(z)} \,,\, \cdots \,,\, \mfrac{1}{\fM_{n+p}^\eta(z)}  \Big)^\top
    +
    z \, \bone_{n+p}
    +
    \fG^\eta
    \, \fM^\eta(z)
    \;\eqqcolon\;
    \bF^\eta(z)
    \;,
\end{align*}
which is the perturbed version of $\bF(z)$ defined in \eqref{eq:perturbed_combined_QVE}. Following the notation of \eqref{eq:defn:fB}, we can compute
\begin{align*}
    \Big( \mfrac{\partial \bF^\eta(z)}{\partial \fM^\eta(z)} \Big)^{-1}
    \;=&\;
    \big( 
        - \diag\{ \fM^\eta(z) \}^{-2}
        + \fG^\eta
    \big)^{-1}
    \tagaligneq \label{eq:F:M:inverse:compute}
    \\
    \;=&\;
    -
    \diag\{ | \fM^\eta(z)| \}
    \;
    \fB_\eta^{-1}(z)
    \;
    \diag\{ | \fM^\eta(z)| \}
    \;,
\end{align*}
where we have denoted 
\begin{align*}
    \fB_\eta(z)
    \;\coloneqq&\;
    \diag\bigg\{ 
            \mfrac{|\fM_1^\eta(z)|^2}{\fM_1^\eta(z)^2}
            \,,\, \ldots \,,\,
            \mfrac{|\fM_{n+p}^\eta(z)|^2}{\fM_{n+p}^\eta(z)^2}
    \bigg\}
    -
    \diag\{ | \fM^\eta(z)| \} \, \fG^\eta(z)  \, \diag\{ | \fM^\eta(z)| \}
    \;.
\end{align*}
In particular, since $\eta \in (0,1/s_*)$, the perturbed data matrix $X^\eta$ also satisfies \Cref{asst:flat} with some modified parameters that depend only on $s_*, \psi_1, \psi_2, L_1, L_2$ by \Cref{lem:asst:VP}, and therefore the bound on $\fB$ from \Cref{lem:bound_B_inverse} applies directly to $\fB_\eta$ and yields that, for some constant $c'_L$ that depends only on $s_*, \psi_1, \psi_2, L_1, L_2$,
\begin{align}
    \norminf{ \fB_\eta^{-1}(z) }
    \;\lesssim&\;
    {
    1 
    + 
    \mfrac{s_*(\Im(z)|z|+s_*)}{(\Im(z))^4}.
    }
    \;,
    \label{eq:estimate_norm_B_inverse:eta}
\end{align}
where we recall the notation from \Cref{appendix:local:law} that for a matrix $A \in \C^{b \times b}$,
\begin{align*}
    \norminf{A} \;\coloneqq\; \sup_{v \neq 0} \mfrac{\| A v \|_\infty}{\|v \|_\infty}\;.
\end{align*}
Also note that by \Cref{lem:abs:fM:bound},
\begin{align*}
    \big\| \diag\{ | \fM^\eta(z)| \}  \big\|_{\infty} 
    \;=\; 
    \sup_{1 \leq l \leq n+p} | \fM_l(z)| \;\leq\; \mfrac{1}{\Im(z)}\;.
    \tagaligneq \label{eq:Meta:bound}
\end{align*}
This implies that 
\begin{align*}
    \Big\| 
        \, \Big( \mfrac{\partial \bF^\eta(z)}{\partial \fM^\eta(z)} \Big)^{-1} \,
    \Big\|_\infty
    \;\lesssim\;
    \mfrac{1}{(\Im(z))^2}
    + 
    {
        \mfrac{s_*(\Im(z)|z|+s_*)}{(\Im(z))^6}
    }
    \;.
    \tagaligneq \label{eq:derivative:F:M:bound}
\end{align*}
To use this to control the stability of $\fM^\eta$, we observe that by the implicit function theorem,
\begin{align*}
    \partial_\eta \fM^\eta(z) 
    \;=&\; 
    -
    \Big( \mfrac{\partial \bF^\eta}{\partial \fM^\eta(z)} \Big)^{-1} 
    \mfrac{\partial \bF^\eta}{\partial \eta}
    \;.
    \tagaligneq \label{eq:derivative:IFT:bound}
\end{align*}
Moreover,
\begin{align*}
    \Big\| \mfrac{\partial \bF^\eta}{\partial \eta} \Big\|_\infty
    \;=&\;
    \Big\|
        \begin{psmallmatrix} 
            0 & S (I_p + \eta \bar \Sigma)^{-2}  \bar \Sigma \\
            \bar \Sigma (I_p + \eta \bar \Sigma)^{-2} S^\top & 0
        \end{psmallmatrix}
        \, \fM^\eta(z)
    \Big\|_\infty
    \tagaligneq \label{eq:F:eta:compute}
    \\
    \;\overset{\eqref{eq:Meta:bound}}{\leq}&\;
    \Big\| 
        \begin{psmallmatrix} 
            0 & S (I_p + \eta \bar \Sigma)^{-2}  \bar \Sigma \\
            \bar \Sigma (I_p + \eta \bar \Sigma)^{-2} S^\top & 0
        \end{psmallmatrix}
    \Big\|_\infty
    \,
    \mfrac{1}{\Im(z)} 
    \\
    \;\leq&\;
    \mfrac{2}{\Im(z)} \, \big\| S (I_p + \eta \bar \Sigma)^{-2}  \bar \Sigma  \big\|_\infty
    \\
    \;\leq&\;
    \mfrac{2}{\Im(z)} \, \| S \|_\infty \, \big\| (I_p + \eta \bar \Sigma)^{-2} \big\|_\infty \, \|  \bar \Sigma  \|_\infty
    \\
    \;\overset{(a)}{\leq}&\;
    \mfrac{2 s_*^2}{\Im(z)}
    \;.
    \tagaligneq \label{eq:derivative:F:eta:bound}
\end{align*}
In $(a)$ above, we have used that $\|\bar \Sigma \|_\infty \leq s_*$ by \Cref{lem:Sigma:op:bound}, that since $\bar \Sigma$ is diagonal,
\begin{align*}
    \big\| (I_p + \eta \bar \Sigma)^{-2} \big\|_\infty  \;=\; \max_{1 \leq j \leq p} \mfrac{1}{(1+ \eta \bar \Sigma_{jj} )^2} \;\leq\; 1
    \;,
\end{align*}
and that since $\sup_{i \leq n, j \leq p} S_{ij}  \leq \frac{s_*}{p+n}$,
\begin{align*}
    \| S \|_\infty 
    \;=&\; 
    \sup_{v \in \R^p \,,\, \| v \|_\infty = 1} \| S v \|_\infty
    \;=\;
    \sup_{v \in \R^p \,,\, \| v \|_\infty = 1} \, \max_{1 \leq i \leq n}  \Big| \msum_{j =1 }^p S_{ij} v_j  \Big|
    \\
    \;\leq&\;
    \sup_{v \in \R^p \,,\, \| v \|_\infty = 1} \, \mfrac{s_*}{p+n} \msum_{j =1 }^p | v_j | 
    \;\leq\;
     \mfrac{s_* p}{p+n}  \sup_{v \in \R^p \,,\, \| v \|_\infty = 1} \, \| v \|_\infty 
     \;\leq\;
     s_*\;.
\end{align*}
Combining \eqref{eq:derivative:F:M:bound}, \eqref{eq:derivative:IFT:bound} and \eqref{eq:derivative:F:eta:bound}, we obtain the first desired bound
\begin{align*}
    \big\| \partial_\eta \fM^\eta(z)  \big\|_\infty 
    \;\leq&\;
    \Big\| 
        \, \Big( \mfrac{\partial \bF^\eta(z)}{\partial \fM^\eta(z)} \Big)^{-1} \,
    \Big\|_\infty 
    \, 
    \Big\| \mfrac{\partial \bF^\eta}{\partial \eta} \Big\|_\infty
    \\
    \;\lesssim&\;
    {
         \mfrac{1}{(\Im(z))^3}
        + 
        {
            \mfrac{s_*(\Im(z)|z|+s_*)}{(\Im(z))^7}
        }
    }
    \\
    \;\overset{(b)}{\lesssim}&\; (\Im(z))^{- c_1}
    \;,
\end{align*}
where $c_1 > 0$ is some constant that depends only on $s_*, \psi_1, \psi_2, L_1, L_2$. Note that $(b)$ is the only place where we have used $|z| \leq 1$. For the second derivative, we can apply the implicit function theorem and use an analogous argument to note that there exist polynomial functions $q_2: \R^3 \rightarrow \R$  such that 
\begin{align*}
    \big\| \partial_\eta^2 \fM^\eta(z)  \big\|_\infty 
    \;\lesssim&\;
    \Big\| 
        \, \Big( \mfrac{\partial \bF^\eta(z)}{\partial \fM^\eta(z)} \Big)^{-1} \,
    \Big\|_\infty
    \,\times\,
    p_2\Big( 
        \Big\| \mfrac{\partial^2 \bF^\eta(z)}{\partial \fM^\eta(z)^2} \big( \partial_\eta \fM^\eta(z) \otimes \partial_\eta   \fM^\eta(z)^\top \big) \Big\|_\infty  
        \,,\,
    \\
    &\hspace{11em} 
        \Big\| \partial_{\fM^\eta(z)}\partial_\eta \bF^\eta(z) \,\big( \partial_\eta   \fM^\eta(z)^\top  \big) \Big\|_\infty
        \,,\,
    \\
    &\hspace{11em} 
        \big\| \partial_\eta^2 \bF^\eta(z) \big\|_\infty
    \Big)
    \\
    \;\lesssim&\;
    (\Im(z))^{- c_2}
\end{align*}
for some constant $c_2 > 0$ that depends only on $s_*, \psi_1, \psi_2, L_1, L_2$. A similar bound holds for the third derivative:
\begin{align*}
    \big\| \partial_\eta^3 \fM^\eta(z)  \big\|_\infty \;\lesssim&\; (\Im(z))^{- c_3}\;.
\end{align*}
Taking $c= \max\{1, c_1, c_2, c_3\}$ and applying the triangle inequality on the three derivative bounds above as well as \eqref{eq:Meta:bound} yield the desired bound.
\end{proof}

\subsection{Proof of \texorpdfstring{\Cref{thm:bias}}{the bias theorem}}  \label{appendix:proof:bias}

The objective is to approximate
\begin{align*}
    \cR_B^\lambda(X)
    \;=\; 
    \big\| \bar \Sigma^{1/2}
        \big( (W_n + \lambda I_p)^{-1}  W_n - I_p \big)
        \beta 
    \big\|^2 
    \;.
\end{align*}
We will take $\lambda \in (n^{-2\upsilon}, \lambda_0)$ for the majority of the proof, where $\upsilon >0$ is chosen as in \Cref{thm:local:law:VP}, and restrict the range of values of $\lambda$ later. The proof consists of the following steps:
\begin{proplist}
    \item Rewrite $\cR_B^\lambda(X)$ as the $\eta$-derivative of some function $\cF^\lambda(\eta)$, which is linear in the perturbed matrix resolvent 
    \begin{align*}
        \big(W_n + \lambda I_p + \eta \bar \Sigma\big)^{-1}
        \;,
    \end{align*}
    and can therefore be approximated by the local law derived in \Cref{thm:local:law:VP};
    \item Extend the approximation of $\cF^\lambda(\eta)$ to that of $\partial_\eta \cF^\lambda(\eta)$ by applying a one-sided Taylor expansion;
    \item Rewriting the limiting quantity from the local law, which involves $\fM$, in terms of $\br(\lambda)$.
\end{proplist} 

\vspace{.5em}

\noindent 
\textbf{Step (i): Relate $\cR^\lambda_B(X)$ to $\cF^\lambda(\eta)$, for which a deterministic approximation is available.} We follow the strategy of \cite{hastie2022surprises} to introduce a perturbation parameter 
\begin{align*}
     \eta \;\in\; \Big(0 \,,\,  \mfrac{1}{s_*} \Big)\;,
     \tagaligneq \label{eq:eta:range}
\end{align*}
and consider the random function  
\begin{align*}
    \cF^\lambda(\eta)
    \;\coloneqq&\;
    \lambda
    \,
    \beta^\top 
     \big( W_n + \lambda I_p + \lambda \eta \bar \Sigma \big)^{-1} \beta 
     \\
     \;=&\;
    \lambda
    \beta_\eta^\top 
    \Big( W_\eta + \lambda I_p \Big)^{-1} \beta_\eta
    \;=\;
    \lambda
    \beta_\eta^\top 
    \Big( \mfrac{1}{n} X_\eta^\top X_\eta + \lambda I_p \Big)^{-1} \beta_\eta
     \;,
\end{align*}
where we have denoted 
\begin{align*}
    \beta_\eta \;\coloneqq&\; (I_p + \eta \bar \Sigma)^{-1/2} \beta\;,
    \qquad 
    W_\eta \;\coloneqq\; 
    (I_p + \eta \bar \Sigma)^{-1/2} 
    W_n 
    (I_p + \eta \bar \Sigma)^{-1/2} 
    \;,
    \\
    X_\eta \;\coloneqq&\;  X (I_p + \eta \bar \Sigma)^{-1/2} \;.
\end{align*}
Then we can express the bias term of the risk as
\begin{align*}
    \cR_B^\lambda(X)
    \;=&\;
    \big\| \bar \Sigma^{1/2}
            \big( (W_n + \lambda I_p)^{-1}  W_n - I_p \big)
            \beta 
        \big\|^2 
    \\
    \;=&\;
    \lambda^2 \, \big\| \bar \Sigma^{1/2} \big( W_n + \lambda I_p \big)^{-1} \beta \big\|^2
    \;=\;
    -
    \partial \cF^\lambda(0) \;.
    \tagaligneq \label{eq:bias:as:derivative}
\end{align*}

\vspace{1em}

To apply the local law, we first use a similar calculation as \eqref{eq:resolvent:Wn} to obtain
\begin{align*}
    ( W_\eta + \lambda I_p )^{-1} 
    \;=\;
    \mfrac{G_2^\eta\big( \sqrt{-\lambda} \big)}{\sqrt{-\lambda}}
    \;,
\end{align*}
where we have defined the $\eta$-perturbed matrix resolvents
\begin{align*}
    \big(G_2^\eta(z) \big)_{l,l'} 
    \;\coloneqq&\; 
    G_{n+l,n+l'}^\eta(z) 
    \text{ for } 1 \leq l,l' \leq p\;,
    \\
    G^\eta(z)
    \;\coloneqq&\;
    (H^\eta - z I_{n+p})^{-1}\;,
    \quad 
    \text{ where }
    H^\eta \;\coloneqq\; \begin{psmallmatrix}
        0 & \frac{1}{\sqrt{n}} X_\eta \\
        \frac{1}{\sqrt{n}} X_\eta^\top & 0
    \end{psmallmatrix}
    \;.
\end{align*}
This allows us to rewrite 
\begin{align*}
    \cF^\lambda(\eta)
    \;=\;
    - \sqrt{-\lambda}
    \, 
    \beta_\eta^\top 
    G_2^\eta\big( \sqrt{-\lambda} \big)  \beta_\eta
    \;.
\end{align*}
Observe that our local law (\Cref{thm:local:law:VP}) can be applied to replace $G^\eta(\sqrt{-\lambda})$ by $\fM^\eta(\sqrt{-\lambda})$: By \Cref{lem:asst:VP}, $X_\eta$ is a data matrix under \Cref{mod:vp} that satisfies \Cref{asst:flat}, so \Cref{thm:local:law:VP} applies to $G^\eta(\sqrt{-\lambda})$. In particular, the bounding constants in \Cref{thm:local:law:VP} are independent of $\eta$ since we have chosen $\eta$ to be bounded in \eqref{eq:eta:range} (see the discussion before \Cref{lem:asst:VP}). We denote the deterministic analogue of $\cF^\lambda(\eta)$ as 
\begin{align*}
    F^\lambda(\eta) 
    \;\coloneqq\;  
    - \sqrt{-\lambda} \, \msum_{l=1}^p (\beta_\eta)_l^2  \, \fM_{n+l}^\eta( \sqrt{-\lambda} )
    \;.
    \tagaligneq \label{eq:defn:F}
\end{align*}
We also note that since $\lambda \in \big( n^{-2\upsilon}, \lambda_0)$ for some universal constant $\lambda_0 > 0$,
\begin{align*}
    \Im\big(\sqrt{-\lambda} \big)
    \;=\; 
    \lambda^{1/2} 
    \;\geq\; n^{-\upsilon}
    \;,
    \qquad\text{ and }\qquad 
    \big| \sqrt{-\lambda} \big| = \lambda^{1/2}
    \;\leq\; \lambda_0^{1/2}\;,
\end{align*}
and in particular $\sqrt{-\lambda} \in \D_{\upsilon, \lambda_0^{1/2}}$ defined in \Cref{thm:local:law:VP}. Then 
\begin{align*}
    \big| 
        \cF^\lambda(\eta) 
        - 
        F^\lambda(\eta) 
    \big|
    &\;=\;
    \lambda^{1/2}
    \Big| 
        \msum_{l,l'=1}^p
         (\beta_\eta)_l  (\beta_\eta)_{l'}
        ( G_2^\eta(\sqrt{-\lambda}))_{ll'}
        - 
        \msum_{l=1}^p (\beta_\eta)_l^2  \, 
        \fM_{n+l}^\eta(\sqrt{-\lambda} )
    \Big|
    \\
    &\;=\;
    \lambda^{1/2}
    \,
    \Big| 
        \msum_{l,l'=1}^p
         (\beta_\eta)_l  (\beta_\eta)_{l'}
        \Big(
         ( G_2^\eta(\sqrt{-\lambda}))_{ll'}
         -
         \fM_{n+l}^\eta(\sqrt{-\lambda} )
         \,
         \ind_{\{l=l'\}}
         \Big)
    \Big|
    \\
    &\;\leq\;
    \lambda^{1/2}
    \,
        \msum_{l=1}^p
         (\beta_\eta)_l^2  
        \Big|
         ( G_2^\eta(\sqrt{-\lambda}))_{ll}
         -
         \fM_{n+l}^\eta(\sqrt{-\lambda} )
         \Big|
    \\
    &
    \hspace{2em}
    +
     \lambda^{1/2}
    \,
        \msum_{l \neq l'}
         | (\beta_\eta)_l | \, | (\beta_\eta)_{l'} |
         \,
        \big| G_2^\eta(\sqrt{-\lambda})_{ll'} \big|
    \\
    &\;\overset{(a)}{\prec}\;
    \lambda^{1/2}
    \,
    \Big(
    \| \beta_\eta \|^2
    \,
    {n^{-5\upsilon}}
    +
    \| \beta_\eta \|_{l_1}^2
    \,
    n^{3 \upsilon - \frac{1}{2}}
    \Big)
    \;.
    \tagaligneq \label{eq:bias:local:law:err}
\end{align*}
In $(a)$ above, we have used both the generic bound in \Cref{thm:local:law:VP} and the sharper off-diagonal approximation.

\vspace{1em}

\noindent 
\textbf{Step (ii): Extend the approximation of $\cF^\lambda(\eta)$ to $\partial \cF^\lambda(\eta)$, for which a deterministic approximation is available.}  To turn this into a control on the derivative  $\partial \cF^\lambda(\eta)$, we fix a sufficiently small constant $\eta_0 \leq 1/(2s_*)$ and note that, first note that uniformly over $\lambda \in (n^{-2\upsilon}, \lambda_0)$ and $\eta \in [0, \eta_0)$, the function $\eta \mapsto \fM^\eta(\sqrt{-\lambda})$ is analytic (by e.g.~applying the implicit function theorem as in the proof of \Cref{lem:Meta:derivatives}). 

{
Define
\begin{align*}
    e(\eta)
    \;\coloneqq\;
    \cF^\lambda(\eta)-F^\lambda(\eta).
\end{align*}
For any $t>0$ such that $2t\leq\eta_0$, a one-sided Taylor expansion gives
\begin{align*}
    e'(0)
    \;=&\;
    \mfrac{-3e(0)+4e(t)-e(2t)}{2t}
    +
    O\Bigg(
        t^2
        \sup_{0\leq\eta\leq2t}
        |e'''(\eta)|
    \Bigg).
\end{align*}
Therefore,
\begin{equation}\label{eqn:taylor_exp_bound_at_0}
    |e'(0)|
    \;\lesssim\;
    \mfrac{1}{t}
    \sup_{0\leq\eta\leq2t}|e(\eta)|
    +
    t^2
    \sup_{0\leq\eta\leq2t}|e'''(\eta)|
    \;.
\end{equation}
}

We first control the third derivatives $\partial_\eta^3 \cF^\lambda(\eta)$ and  $\partial_\eta^3 F^\lambda(\eta)$ with respect to $\eta$. We inherit the notation of \Cref{appendix:local:law} to write $\normtwo{A} = \sup_{v \neq 0} \frac{\|A v\|}{\| v\|}$. First note that for some universal constant $c_1 > 0$, almost surely
\begin{align*}
    \big| \partial_\eta^3 \cF^\lambda(\eta) \big|
    \;\leq&\;
    c_1 \,\lambda^4 \, \| \beta \|^2 
    \Big\| 
     \Big( W_n + \lambda I_p +
      \lambda \eta \bar \Sigma \Big)^{-1} \Big\|_2^4 
     \| \bar \Sigma \|_2^3
    \\
    \;\leq&\;
        c_1 \| \beta \|^2  \| \bar \Sigma \|_2^3
    \, 
    \;\leq\;
        c_1 \| \beta \|^2  s_*^3
    \;,
    \tagaligneq \label{eq:bias:third:empirical:F:control}
\end{align*}
where we have used \Cref{lem:Sigma:op:bound} in the last inequality. On the other hand, 
\begin{align*}
    \big| \partial_\eta^3 F^\lambda(\eta) \big|
    \;\lesssim&\;
    \lambda^{1/2} 
    \, 
    \Big( 
        \msum_{l=1}^p  (\beta_\eta)_l^2    
    \Big)  
    \, 
    \big\| \partial_\eta^3 \, \fM_{n+l}^\eta( \sqrt{-\lambda} ) \big\|_\infty 
     \\
     &\;
     +
     \lambda^{1/2}  
     \,
    \Big( 
        \msum_{l=1}^p \big| \partial_\eta (\beta_\eta)_l^2    \big|
    \Big)
     \,  
     \big\| \partial_\eta^2 \, \fM_{n+l}^\eta( \sqrt{-\lambda} ) \big\|_\infty 
     \\
     &\;
     +
     \lambda^{1/2}  
     \,
    \Big( 
        \msum_{l=1}^p \big| \partial_\eta^2 (\beta_\eta)_l^2 \big|
    \Big)
     \, 
     \big\| \partial_\eta \, \fM_{n+l}^\eta( \sqrt{-\lambda} ) \big\|_\infty  
     \\
     &\;
     +
     \lambda^{1/2}  
     \,
    \Big( 
        \msum_{l=1}^p \big| \partial_\eta^3 (\beta_\eta)_l^2 \big|
    \Big)
     \, 
     \big\| \fM_{n+l}^\eta( \sqrt{-\lambda} ) \big\|_\infty 
     \\
     \;\overset{(b)}{\lesssim}&\;
     \lambda^{ \frac{1-c}{2} }
    \, \msum_{l=1}^p
    \Big(
        (\beta_\eta)_l^2
        +
        \big| \partial_\eta (\beta_\eta)_l^2 \big|
        +
        \big| \partial_\eta^2 (\beta_\eta)_l^2 \big|
        + 
        \big| \partial_\eta^3 (\beta_\eta)_l^2 \big|
    \Big)
    \tagaligneq \label{eq:bias:third:pop:F:control:first}
\end{align*}
for some constant $c > 0$ that depends only on $s_*, \psi_1, \psi_2, L_1, L_2$ from \Cref{asst:flat}. In $(b)$ above, we have applied \Cref{lem:Meta:derivatives} to control the $\fM^\eta$ terms {while noting that $|\lambda| \leq 1$ implies $|\sqrt{-\lambda}| \leq 1$}. Now note that since  $\bar \Sigma$ is diagonal, 
\begin{align*}
    (\beta_\eta)_l \;=\; e_l^\top (I_p + \eta \bar \Sigma)^{-1/2} \beta
    \;=\;
    \mfrac{1}{\sqrt{1+\eta \bar \Sigma_{ll}}} \beta_l \;.
    \tagaligneq \label{eq:beta:eta}
\end{align*}
Therefore for $k=0,1,2,3$, noting additionally that $\| \bar \Sigma\|_\infty \leq s_*$ by \Cref{lem:Sigma:op:bound},
\begin{align*}
    \msum_{l=1}^p \big| \partial_\eta^k (\beta_\eta)_l^2 \big|
    \;=&\;
    \msum_{l=1}^p \beta_l^2  \, \Big| \partial_\eta^k \Big( \mfrac{1}{1+\eta \bar \Sigma_{ll}} \Big) \Big|
    \;\lesssim\;
    \| \beta \|^2 \, s_*^k \;.
\end{align*}
Substituting this into \eqref{eq:bias:third:pop:F:control:first} gives
\begin{align*}
    \big| \partial_\eta^3 F^\lambda(\eta) \big|
    \;\lesssim&\;
    \lambda^{ \frac{1-c}{2} } \| \beta \|^2
    \;.
    \tagaligneq \label{eq:bias:third:pop:F:control}
\end{align*}
{
Since $2t\leq\eta_0$, \eqref{eq:bias:third:empirical:F:control}
and \eqref{eq:bias:third:pop:F:control} imply
\begin{align*}
    \sup_{0\leq\eta\leq2t}|e'''(\eta)|
    \;\lesssim\;
    \|\beta\|^2
    \Big(
        s_*^3+\lambda^{\frac{1-c}{2}}
    \Big).
\end{align*}
}

{
We then control the first term $\sup_{0\leq\eta\leq2t}|e(\eta)|$ in the bound \cref{eqn:taylor_exp_bound_at_0}.
For each fixed $\eta\in[0,2t]$, \eqref{eq:bias:local:law:err} gives
\begin{align*}
    |e(\eta)|
    \;\prec\;
    \lambda^{1/2}
    \Big(
        \|\beta\|^2 n^{-5\upsilon}
        +
        \|\beta\|_{l_1}^2
        n^{3\upsilon-\frac12}
    \Big),
\end{align*}
where we have used
$\|\beta_\eta\|\leq\|\beta\|$ and
$\|\beta_\eta\|_{l_1}\leq\|\beta\|_{l_1}$.
}
{
Let $
    E_n
    \;\coloneqq\;
    \lambda^{1/2}
    \Big(
        \|\beta\|^2 n^{-5\upsilon}
        +
        \|\beta\|_{l_1}^2 n^{3\upsilon-\frac12}
    \Big).
$
Let $\mathcal G_n$ be a grid of $[0,2t]$ with mesh size at most
$n^{-D}$, where $D>0$ is sufficiently large. Since $|\mathcal G_n|=O(n^D)$ and the stochastic
domination in \eqref{eq:bias:local:law:err} is uniform over
$\eta\in[0,\eta_0]$, for any $\xi,\kappa>0$ we have, by a union bound,
\begin{align*}
    \P\Big(
        \max_{\eta\in\mathcal G_n}|e(\eta)|
        >n^\xi E_n
    \Big)
    &\leq
    \msum_{\eta\in\mathcal G_n}
    \P\Big(
        |e(\eta)|>n^\xi E_n
    \Big)
    \\
    &{\lesssim}
    n^D n^{-(\kappa+D+1)}
    \\
    &\leq
    n^{-\kappa}
\end{align*}
for all sufficiently large $n$. Hence
\begin{align*}
    \max_{\eta\in\mathcal G_n}|e(\eta)|
    \;\prec\;
    E_n.
\end{align*}
}
{
For any $\eta\in[0,2t]$, let $\eta^\circ\in\mathcal G_n$
be a nearest grid point. Then
\begin{align*}
    |e(\eta)|
    \;\leq\;
    |e(\eta^\circ)|
    +
    n^{-D}
    \sup_{\eta'\in[0,\eta_0]}|e'(\eta')|.
\end{align*}
By the same arguments as in
\eqref{eq:bias:third:empirical:F:control} and
\eqref{eq:bias:third:pop:F:control}, we obtain
\begin{align*}
    \sup_{\eta\in[0,\eta_0]}
    \big|\partial_\eta\cF^\lambda(\eta)\big|
    \lesssim
    \|\beta\|^2, \quad
    \sup_{\eta\in[0,\eta_0]}
    \big|\partial_\eta F^\lambda(\eta)\big|
    \lesssim
    \lambda^{\frac{1-c}{2}}\|\beta\|^2.
\end{align*}
Therefore, by the triangle inequality,
\begin{align*}
    \sup_{\eta\in[0,\eta_0]}|e'(\eta)|
    &\leq
    \sup_{\eta\in[0,\eta_0]}
    \big|\partial_\eta\cF^\lambda(\eta)\big|
    +
    \sup_{\eta\in[0,\eta_0]}
    \big|\partial_\eta F^\lambda(\eta)\big|
    \\
    &\lesssim
    \|\beta\|^2
    \big(
        1+\lambda^{\frac{1-c}{2}}
    \big)
    \\
    &\lesssim
    \|\beta\|^2 n^{(c-1)\upsilon},
\end{align*}
where in the last inequality we have used
$\lambda\geq n^{-2\upsilon}$ and taken
$c\geq1$ in \Cref{lem:Meta:derivatives}.
}{
Hence,
\begin{align*}
    \sup_{0\leq\eta\leq2t}|e(\eta)|
    \;\prec\;
    E_n
    +
    \|\beta\|^2 n^{(c-1)\upsilon-D}.
\end{align*}
Since $\lambda\geq n^{-2\upsilon}$,
\begin{align*}
    E_n
    \;\geq\;
    \|\beta\|^2 n^{-6\upsilon}.
\end{align*}
Therefore, choosing $D>(c+5)\upsilon$, we obtain
\begin{align*}
    \sup_{0\leq\eta\leq2t}|e(\eta)|
    \;\prec\;
    E_n
    \;.
    \tagaligneq \label{eq:bias:local:law:uniform}
\end{align*}
}
{
Combining \eqref{eqn:taylor_exp_bound_at_0} with
\eqref{eq:bias:local:law:uniform},
\eqref{eq:bias:third:empirical:F:control}, and
\eqref{eq:bias:third:pop:F:control}, we obtain
\begin{align*}
    \big| \cR_B^\lambda(X)
    - \big( -\partial F^\lambda(0) \big) \big|
    =
    |e'(0)|
    \;\prec\;
    \mfrac{E_n}{t}
    +
    t^2
    \|\beta\|^2
    \Big(
        s_*^3+\lambda^{\frac{1-c}{2}}
    \Big).
\end{align*}
Choose
\begin{align*}
    t_*
    \;=\;
    \left(
        \mfrac{
            E_n
        }{
            \|\beta\|^2
            \big(
                s_*^3+\lambda^{\frac{1-c}{2}}
            \big)
        }
    \right)^{1/3},
    \qquad
    t
    \;=\;
    \min\left\{
        t_*,
        \mfrac{\eta_0}{2}
    \right\}.
\end{align*}
which satisfies $2t\leq\eta_0$ for all sufficiently large $n$
under the condition {$\|\beta\|>0$ is fixed,} $\|\beta\|_{l_1}\leq n^{\frac14-\delta}$ and the choice
$\upsilon\leq\delta/3$, we obtain that 
\begin{align*}
    &\;
    \big| \cR_B^\lambda(X)
    - \big( -\partial F^\lambda(0) \big) \big|
    \\
    &\;\prec\;
    \lambda^{1/3}
    \Big(
        \|\beta\|^{2/3}s_*
        +
        \lambda^{\frac{1-c}{6}}
        \|\beta\|^{2/3}
    \Big)
    \Big(
        \|\beta\|^{4/3}n^{-10\upsilon/3}
        +
        \|\beta\|_{l_1}^{4/3}
        n^{2\upsilon-\frac13}
    \Big).
\end{align*}
}
We now use the condition that $s_*$, $L_1$, $L_2$, $\psi_1$ and $\psi_2$ from \Cref{asst:flat} as well as $\| \beta \|^2$ are universal constants, and so is $c > 0$. Then 
\begin{align*}
    &\;
    \big| \cR_B^\lambda(X) - \big( -\partial F^\lambda(0) \big) \big| 
    \\
    &\;\prec\;
    \lambda^{1/3}
    \,
    \big(
        1
        +
        \lambda^{(1-c)/6} 
    \big) 
    \,
    \big( 
        n^{-  {10} \upsilon / 3}
        +
        \| \beta \|_{l_1}^{4/3}
        \,
        n^{ 2 \upsilon - \frac{1}{3} }
    \big)
    \;.
\end{align*}
Now use the condition that  $\| \beta \|_{l_1} \leq n^{\frac{1}{4} - \delta}$ for some fixed $\delta > 0$. We choose $\upsilon \in (0, \frac{1}{24})$ from
\Cref{thm:local:law:VP} such that
$\upsilon \leq \frac{\delta}{3}$, which implies 
\begin{align*}
    &\;
    \big| \cR_B^\lambda(X) - \big( -\partial F^\lambda(0) \big) \big| 
    \\
    &\;\prec\;
    \lambda^{1/3}
    \,
    \big(
        1
        +
        \lambda^{(1-c)/6} 
    \big) 
    \,
    \big( 
        n^{-  10 \upsilon / 3}
        +
        n^{\frac{1}{3} - \delta}
        \,
        n^{ \frac{2}{3} \delta - \frac{1}{3} }
    \big)
    \\
     &\;=\;
    \big(
        \lambda^{1/3}
        +
        \lambda^{1/2 - c/6 } 
    \big) 
    \,
    \big( 
        n^{-  10 \upsilon / 3}
        +
        n^{ - \delta /3}
    \big)
    \;.
\end{align*}
{
Finally, taking $c>3$ in \Cref{lem:Meta:derivatives}}, provided that $\lambda \in (n^{-\alpha}, \lambda_0)$ with 
\begin{align*}
    \alpha 
    \;<\; 
    {\min}
    \Big\{ 
        2 \upsilon
        \,,\,    
        \mfrac{
            {10} \upsilon / 3
        }
        {
            \frac{c}{6} - \frac{1}{2} 
        }
        \,,\,
        \mfrac{
            \delta / 3
        }
        {
            \frac{c}{6} - \frac{1}{2} 
        }
    \Big\}
    \;\eqqcolon\;
    \alpha_0\;,
\end{align*}
we obtain the stochastic domination statement that for some universal constant $c' > 0$,
\begin{align*}
    \big| \cR_B^\lambda(X) - \big( -\partial F^\lambda(0) \big) \big|  \;\prec\; n^{- c'}\;.
\end{align*}
By the Borel-Cantelli lemma, we obtain that 
\begin{align*}
    \big| \cR_B^\lambda(X) - \big( -\partial F^\lambda(0) \big) \big| \;\xrightarrow{\rm a.s.}\; 0\;.
    \tagaligneq \label{eq:risk:bias:conv}
\end{align*}

\vspace{.5em}

\noindent 
\textbf{Step (iii): Rewriting $\partial F^\lambda(0)$ in terms of $\br(\lambda)$.} 
We first differentiate \eqref{eq:defn:F} to compute 
\begin{align*}
    \partial F^\lambda(\eta) 
    \;=&\;  
    - \sqrt{-\lambda} \, \msum_{l=1}^p 
    \partial_\eta \Big(
    (\beta_\eta)_l^2  \, \fM_{n+l}^\eta( \sqrt{-\lambda} )
    \Big) 
    \\
    \;\overset{\eqref{eq:beta:eta}}{=}&\;
    - \sqrt{-\lambda} \, \msum_{l=1}^p \beta_l^2 
    \, \partial_\eta \Big(
    \mfrac{1}{1+ \eta \bar \Sigma_{ll}}  \, \fM_{n+l}^\eta( \sqrt{-\lambda} )
    \Big) 
    \;.
\end{align*}
To compute $\big( \partial_\eta \, \fM_{n+l}^\eta(z) \big)_{1 \leq l \leq p}$, we denote the $\C^n$ and $\C^p$ vectors 
\begin{align*}
    \bu^\eta(\lambda) 
    \;\coloneqq&\; 
    \big(  - \sqrt{-\lambda} \; \fM_l^\eta(\sqrt{-\lambda}) \big)_{1 \leq l \leq n}
    &\text{ and }&&
    \br^\eta(\lambda) 
    \;\coloneqq&\; 
    \Big(
        \mfrac{
            - \sqrt{-\lambda} \; \fM_{n+l}^\eta(\sqrt{-\lambda}) 
        }{
            1 + \eta \bar \Sigma_{ll}
        }
    \Big)_{1 \leq l \leq p}\;,
\end{align*}
which allows us to rewrite 
\begin{align*}
    \partial F^\lambda(\eta) 
    \;=&\;  
     \msum_{l=1}^p \beta_l^2 
    \partial_\eta \, r^\eta_l(\lambda) 
    \;.
    \tagaligneq \label{eq:partial:F:lambda:eta}
\end{align*}
Observe that the modified QVE  \eqref{eq:QVE:eta} at $z=\sqrt{-\lambda}$ reads
\begin{align*}
    &\;
    -
    \begin{psmallmatrix}
        \diag\{ \frac{\bu^\eta(\lambda)}{-\sqrt{-\lambda}}  \}^{-1} & \\
        & \diag\{ \frac{\br^\eta(\lambda)}{-\sqrt{-\lambda}} (1+\eta \bar \Sigma_{ll})  \}^{-1}
    \end{psmallmatrix}
    \\
    &\;=\;
     \sqrt{-\lambda} \, I_{n+p} 
    + 
    \diag\bigg\{
        \begin{psmallmatrix}
            I_n & \\
            & (I_p + \eta \bar \Sigma)^{-1}
        \end{psmallmatrix}
        \, \Sf \, 
        \begin{psmallmatrix}
            I_n & \\
            & (I_p + \eta \bar \Sigma)^{-1}
        \end{psmallmatrix}
        \begin{psmallmatrix}
        \frac{\bu^\eta( \lambda)}{-\sqrt{-\lambda}} \\
        \frac{\br^\eta( \lambda)}{-\sqrt{-\lambda}} (1+\eta \bar \Sigma_{ll}) 
    \end{psmallmatrix}
    \bigg\}\;,
\end{align*}
which rearranges to 
\begin{align*}
    \begin{psmallmatrix}
        I_n & \\
        & (I_p + \eta \bar \Sigma)^{-1}
    \end{psmallmatrix}
    \;=\;
    \begin{psmallmatrix}
        \diag\{ \bu^\eta(\lambda)  \} & \\
        & \diag\{ \br^\eta(\lambda) \}
    \end{psmallmatrix}
    \bigg( 
      I_{n+p} 
    + 
    \mfrac{1}{\lambda}
    \diag\bigg\{
        \begin{psmallmatrix}
            I_n & \\
            & (I_p + \eta \bar \Sigma)^{-1}
        \end{psmallmatrix}
        \, \Sf \, 
        \begin{psmallmatrix}
        \bu^\eta( \lambda) \\
        \br^\eta( \lambda)
    \end{psmallmatrix}
    \bigg\}
    \bigg)
    \;.
\end{align*}
This implies
\begin{align*}
    I_n
    \;=&\;
    \diag\{ \bu^\eta(\lambda)\}
    + 
    \mfrac{1}{\lambda} 
    \diag\{ \bu^\eta(\lambda)\}
    \diag\big\{ 
        S 
        \,
        \br^\eta(\lambda) 
    \big\}
    \\
    \;=&\;
    \mfrac{1}{\lambda} 
    \diag\{\bu^\eta(\lambda)\} 
    \,
        \diag\big\{ 
            \lambda \bone_n 
            +
            S \, \br^\eta(\lambda)  
        \big\}
    \;,
    \\
    I_p
    \;=&\;
    \diag\{ \br^\eta(\lambda) \}
    \, (I_p + \eta \bar \Sigma )
    + 
    \mfrac{1}{\lambda}
    \, 
    \diag\{ \br^\eta(\lambda) \}
    \,
    \diag\big\{ 
        S^\top
        \diag\{ \bu^\eta(\lambda) \} \bone_n
    \big\}
    \\
    \;=&\;
    \diag\{ \br^\eta(\lambda) \}
    \, (I_p + \eta \bar \Sigma )
    +
    \diag\{ \br^\eta(\lambda) \}
    \,
    \diag\Big\{
    S^\top 
    \diag\{ \lambda \bone_n + S \br^\eta(\lambda) \}^{-1}
    \bone_n
    \Big\}
    \;,
    \tagaligneq \label{eq:r2} 
\end{align*}
where $\bone_n \in \R^n$ is the all one vector. Differentiating \eqref{eq:r2} with respect to $\eta$ gives
\begin{align*}
    0
    \;=&\;
    \diag\{ \partial_\eta \br^\eta(\lambda)\}
    \, 
    (I_p + \eta \bar \Sigma )
    +
    \diag\{ \br^\eta(\lambda)\}
    \, 
    \bar \Sigma
    \\
    &\;
    +
    \diag\{ \partial_\eta \br^\eta(\lambda)\}
    \,
    \diag\Big\{
        S^\top 
        \diag\{ \lambda \bone_n + S \br^\eta(\lambda) \}^{-1}
        \bone_n
    \Big\}
    \\
    &\;
    -
    \diag\{ \br^\eta(\lambda)\}
    \,
    \diag\Big\{
    S^\top 
    \diag\{ \lambda \bone_n + S \br^\eta(\lambda) \}^{-2}
    S \, 
    \diag\{ \partial_\eta \br^\eta(\lambda) \} \bone_n
    \Big\}
    \\
    \;\overset{\eqref{eq:r2}}{=}&\;
    \diag\{ \partial_\eta \br^\eta(\lambda)\}
    \, 
    \diag\{ \br^\eta(\lambda)\}^{-1} 
    +
    \diag\{ \br^\eta(\lambda)\} \, 
    \bar \Sigma
    \\
    &\;
    -
    \diag\{ \br^\eta(\lambda)\}
    \,
    \diag\Big\{
    S^\top 
    \diag\{ \lambda \bone_n + S \br^\eta(\lambda) \}^{-2}
    S \, \partial_\eta \br^\eta(\lambda) 
    \Big\}
    \;.
\end{align*}
Noting that $\bar \Sigma$ is diagonal, we obtain 
\begin{align*}
    0
    \;=&\;
    \diag\{ \br^\eta(\lambda)\}^{-1} 
    \, 
    \partial_\eta \br^\eta(\lambda)
    +
    \bar \Sigma \, \br^\eta(\lambda)
    \\
    &\;
    -
    \Big( 
        \diag\{ \br^\eta(\lambda)\}
        \,
        S^\top 
        \diag\{ \lambda \bone_n + S \br^\eta(\lambda) \}^{-2}
        S 
    \Big) \, \partial_\eta \br^\eta(\lambda) 
    \;,
\end{align*}
which rearranges to give 
\begin{align*}
    \partial_\eta \br^\eta(\lambda) 
    \;=&\;
    -
    \Big( 
        \diag\{ \br^\eta(\lambda)\}^{-1} 
        -
            \diag\{ \br^\eta(\lambda)\}
            \,
            S^\top 
            \diag\{ \lambda \bone_n + S \br^\eta(\lambda) \}^{-2}
            S
    \Big)^{-1} 
    \bar \Sigma \br^\eta(\lambda) 
    \;.
    \tagaligneq \label{eq:partial:eta:r}
\end{align*}
Now notice that by construction,
\begin{align*}
    \br^\eta(\lambda) \,\big|_{\eta =0 } \;=\; \br(\lambda)\;,
\end{align*}
where we recall from \eqref{defn:r} that $r_l(\lambda) = -\sqrt{-\lambda} \, \fM_{n+l}(\sqrt{-\lambda})$. Substituting the above two expressions into \eqref{eq:partial:F:lambda:eta} gives 
\begin{align*}
    - \partial F^\lambda(0) 
    \;=&\;  
    -
     \msum_{l=1}^p \beta_l^2  \,
    \partial_\eta \, r^\eta_l(\lambda) \,  \big|_{\eta =0 } 
    \\
    \;=&\;
    \sum_{l=1}^p \beta_l^2 e_l^\top
    \Big(
        \diag\{ \br(\lambda)\}^{-1} 
        -
            \diag\{ \br(\lambda)\}
            \,
            S^\top 
            \diag\{ \lambda \bone_n + S \br(\lambda) \}^{-2}
            S
    \Big)^{-1} 
    \bar \Sigma \br(\lambda) 
    \;.
\end{align*}
Plugging this into \eqref{eq:risk:bias:conv} gives the desired statement that
\begin{align*}
    &\;
    \Bigg| 
        \,
        \cR_B^\lambda(X) - 
        \sum_{l=1}^p \beta_l^2 e_l^\top
        \Big(
            \diag\{ \br(\lambda)\}^{-1} 
            -
                \diag\{ \br(\lambda)\}
                \,
                S^\top 
                \diag\{ \lambda \bone_n + S \br(\lambda) \}^{-2}
                S
        \Big)^{-1} 
        \bar \Sigma \br(\lambda) 
    \Bigg| 
    \\
    &\;\xrightarrow{\rm a.s.}\; 0\;.
\end{align*}
The properties of $\br(\lambda)$ follow from \Cref{lem:r:eqns,lem:r:values}.

\qed

\subsection{Proof of \texorpdfstring{\Cref{thm:var}}{the variance theorem}} \label{appendix:proof:var}
The objective is to approximate 
\begin{align*}
    \cR_V^\lambda(X)
    \;=&\;
    \mfrac{\sigma^2_\epsilon}{n}
    \Tr\big( 
        \bar \Sigma \, W_n ( W_n + \lambda I_p)^{-2}
    \big)
    \;.
\end{align*}
Similar to \Cref{thm:bias}, we first take $\lambda \in (n^{-2\upsilon}, \lambda_0)$ and define 
\begin{align*}
    \cF_V^\lambda(X) 
    \;\coloneqq\;
    \mfrac{\sigma^2_\epsilon}{n} 
    \Tr\big( 
        \lambda \bar \Sigma
        ( W_n + \lambda I_p )^{-1} \big)
\end{align*}
such that 
\begin{align*}
    \cR_V^\lambda(X)
    \;=&\;
    \mfrac{\sigma^2_\epsilon}{n} \Tr\big( 
         \bar \Sigma
        \, W_n \, ( W_n + \lambda I_p )^{-2} \big)
    \\
    \;=&\;
    \mfrac{\sigma^2_\epsilon}{n} \Tr\big( 
        \bar \Sigma
        \, ( W_n + \lambda I_p )^{-1} \big)
    -
    \mfrac{\sigma^2_\epsilon}{n} \Tr\big( 
        \lambda \bar \Sigma
        \, ( W_n + \lambda I_p )^{-2} \big)
    \\
    \;=&\;
    \partial_\lambda \, \cF_V^\lambda(X) \;.
\end{align*}
Recall from \eqref{eq:resolvent:Wn} that $G_2(\sqrt{-\lambda}) / \sqrt{-\lambda} = (W_n+\lambda I_p)^{-1}$, which implies 
\begin{align*}
    \cF_V^\lambda(X) 
    \;=\;
    \mfrac{\sigma^2_\epsilon}{n} 
    \Tr\Big( 
        \lambda  \bar \Sigma
        \mfrac{ G_2(\sqrt{-\lambda}) }{\sqrt{-\lambda}} 
    \Big)
    \;\overset{(a)}{=}\;
    \mfrac{\sigma^2_\epsilon}{n} 
     \msum_{l=1}^p
        \bar \Sigma_{ll}
        \big(
            -
            \sqrt{-\lambda} \,
            G_{n+l, n+l}(\sqrt{-\lambda}) 
        \big)
    \;.
\end{align*}
In $(a)$, we have recalled that $\bar \Sigma$ is diagonal under \Cref{mod:vp}. Define the deterministic approximation 
\begin{align*}
    F^\lambda_V \;\coloneqq\; 
         \mfrac{\sigma^2_\epsilon}{n}
            \,
            \msum_{l=1}^p
            \, \bar \Sigma_{ll}\,
            \big( - \sqrt{-\lambda} \, 
                \fM_{n+l}(\sqrt{-\lambda})
            \big)\;.
\end{align*}
As with the proof of \Cref{thm:bias}, since $\lambda \in (n^{-2\upsilon}, \lambda_0)$, $\sqrt{-\lambda} \in \D_{\upsilon;\lambda_0^{1/2}}$, and the local law (\Cref{thm:local:law:VP}) applies to give
\begin{align*}
    \big| 
         \cF^\lambda_V(X) 
         -
         F^\lambda_V
    \big|
    \;&\leq\;
    \mfrac{\sigma^2_\epsilon}{n}
    \msum_{l=1}^p
    | \bar \Sigma_{ll}|
    \,
    \lambda^{1/2}
    \,
    \Big| 
         G_{n+l, n+l}(\sqrt{-\lambda}) - \fM_{n+l}(\sqrt{-\lambda}) 
    \Big|
    \\
    \;&\prec\; 
    \sigma^2_\epsilon \, \lambda^{1/2} \, n^{ - {5} \upsilon }
    \;.
\end{align*}
In the last line, we have also used \Cref{lem:Sigma:op:bound} to note that $| \bar \Sigma_{ll}| \leq s_*$. Note in particular that the local law holds uniformly over $\lambda \in (n^{-2\upsilon}, \lambda_0)$. 

{
Following the argument in Appendix A.3~of \cite{hastie2022surprises},
we extend the above approximation to a complex neighborhood of
$\lambda$. For $\zeta$ with $\Re(\zeta)>0$, define
\begin{align*}
    e_V(\zeta)
    \;\coloneqq\;
    \cF_V^\zeta(X)-F_V^\zeta,
\end{align*}
where we choose $\sqrt{-\zeta}$ with
$\Im(\sqrt{-\zeta})>0$.
Both $\cF_V^\zeta(X)$ and $F_V^\zeta$ are analytic in $\{\zeta\in\C:\Re(\zeta)>0\}$.
Since
$\cR_V^\lambda(X)=\partial_\lambda\cF_V^\lambda(X)$,
\begin{align*}
    \big|
        \cR_V^\lambda(X)
        -
        \partial_\lambda F_V^\lambda
    \big|
    \;=\;
    |e_V'(\lambda)|.
\end{align*}
By Cauchy's integral formula on the circle
$|\zeta-\lambda|=\lambda/2$,
\begin{align*}
    |e_V'(\lambda)|
    \;\leq\;
    \mfrac{2}{\lambda}
    \sup_{|\zeta-\lambda|=\lambda/2}
    |e_V(\zeta)|.
\end{align*}
}
{
For $\zeta$ on this circle,
\begin{align*}
    \Re(\zeta)
    \;\geq\;
    \mfrac{\lambda}{2},
    \qquad
    \mfrac{\lambda}{2}
    \;\leq\;
    |\zeta|
    \;\leq\;
    \mfrac{3\lambda}{2}.
\end{align*}
Consequently,
\begin{align*}
    \Im\big(\sqrt{-\zeta}\big)
    &=
    \left(
        \mfrac{|\zeta|+\Re(\zeta)}{2}
    \right)^{1/2}
    \;\geq\;
    \left(\mfrac{\lambda}{2}\right)^{1/2},
    \\
    \big|\sqrt{-\zeta}\big|
    &=
    |\zeta|^{1/2}
    \;\leq\;
    \left(\mfrac{3\lambda}{2}\right)^{1/2}.
\end{align*}
If $\lambda\geq n^{-\alpha}$ for some $\alpha<2\upsilon$, then
$\Im(\sqrt{-\zeta})\geq n^{-\upsilon}$ for all sufficiently large
$n$, uniformly over $|\zeta-\lambda|=\lambda/2$.
Hence \Cref{thm:local:law:VP} applies uniformly on this circle and gives
\begin{align*}
    \sup_{|\zeta-\lambda|=\lambda/2}
    |e_V(\zeta)|
    \;\prec\;
    \sigma_\epsilon^2
    \lambda^{1/2}
    n^{-5\upsilon}.
\end{align*}
Therefore,
\begin{align*}
    \big|
        \cR_V^\lambda(X)
        -
        \partial_\lambda F_V^\lambda
    \big|
    =
    |e_V'(\lambda)|
    \;\prec\;
    \mfrac{\sigma_\epsilon^2}{\lambda^{1/2}}
    n^{-5\upsilon}.
\end{align*}
}
As such, provided that $\lambda \in (n^{-\alpha}, \lambda_0)$ with 
\begin{align*}
    \alpha \;<\; 2 \upsilon%
\end{align*}
we obtain the stochastic domination statement that for {$c'' = 5\upsilon-\alpha/2>0$}, %
\begin{align*}
    \big| 
        \cR_V^\lambda(X)
        - 
        \partial_\lambda F^\lambda_V
    \big|  \;\prec\; n^{- c''}\;.
\end{align*}
By the Borel-Cantelli lemma, we obtain that 
\begin{align*}
    \big| 
        \cR_V^\lambda(X)
        - 
        \partial_\lambda F^\lambda_V
    \big|  \;\xrightarrow{\rm a.s.}\; 0\;.
\end{align*}
To obtain the expression in terms of $\br(\lambda)$, we recall from \eqref{defn:r} that $r_l(\lambda) = -\sqrt{-\lambda} \, \fM_{n+l}(\sqrt{-\lambda})$. This implies
\begin{align*}
    \partial_\lambda F^\lambda_V 
    \;=&\; 
    \mfrac{\sigma^2_\epsilon}{n}
    \,
    \msum_{l=1}^p
    \, \bar \Sigma_{ll}\,
    \partial_\lambda 
    \big( - \sqrt{-\lambda} \, 
        \fM_{n+l}(\sqrt{-\lambda})
    \big)
    \;=\;
    \mfrac{\sigma^2_\epsilon}{n}
    \,
    \msum_{l=1}^p
    \, \bar \Sigma_{ll}\, 
    \partial  r_l(\lambda)
    \;.
\end{align*}
The properties of $\partial r_l$ follow from \Cref{lem:r:values}.

{
Combining this with \eqref{eq:risk:bias:conv}, and recalling that
$\alpha_0\leq2\upsilon$, we may take $\alpha_*=\alpha_0$ in
{ \Cref{thm:var}}.
}

\qed

\section{Additional simulations} \label{sec:additional-experiments}

We collect further diagonal-support experiments that realise the degeneracy mechanism of
\Cref{sec:vp-risk-main} coordinate-wise, together with approximate-degeneracy and full-rank
controls. Throughout, $n$ is fixed and the risk is Monte-Carlo averaged over $100$ independent
draws of a two- or three-group design with the flat signal $\beta=\bone_p/\sqrt p$; the peak
locations quoted below are the specialisation of the switching rule \Cref{thm:var-switch} to
the two-group design worked out in \Cref{ex:two-group-peaks}.

\subsection{Diagonal support: two and three groups}

\Cref{fig:app-two-group,fig:app-two-group-reverse} vary the support overlap of two groups and
\Cref{fig:app-three-group} adds a third group; in each case the empirical peaks fall at the
thresholds of \Cref{ex:two-group-peaks} (and its three-group analogue), with the number of
peaks growing as the supports are made to switch at distinct aspect ratios.

\begin{figure}[htbp]
    \centering
    \includegraphics[width=\linewidth]{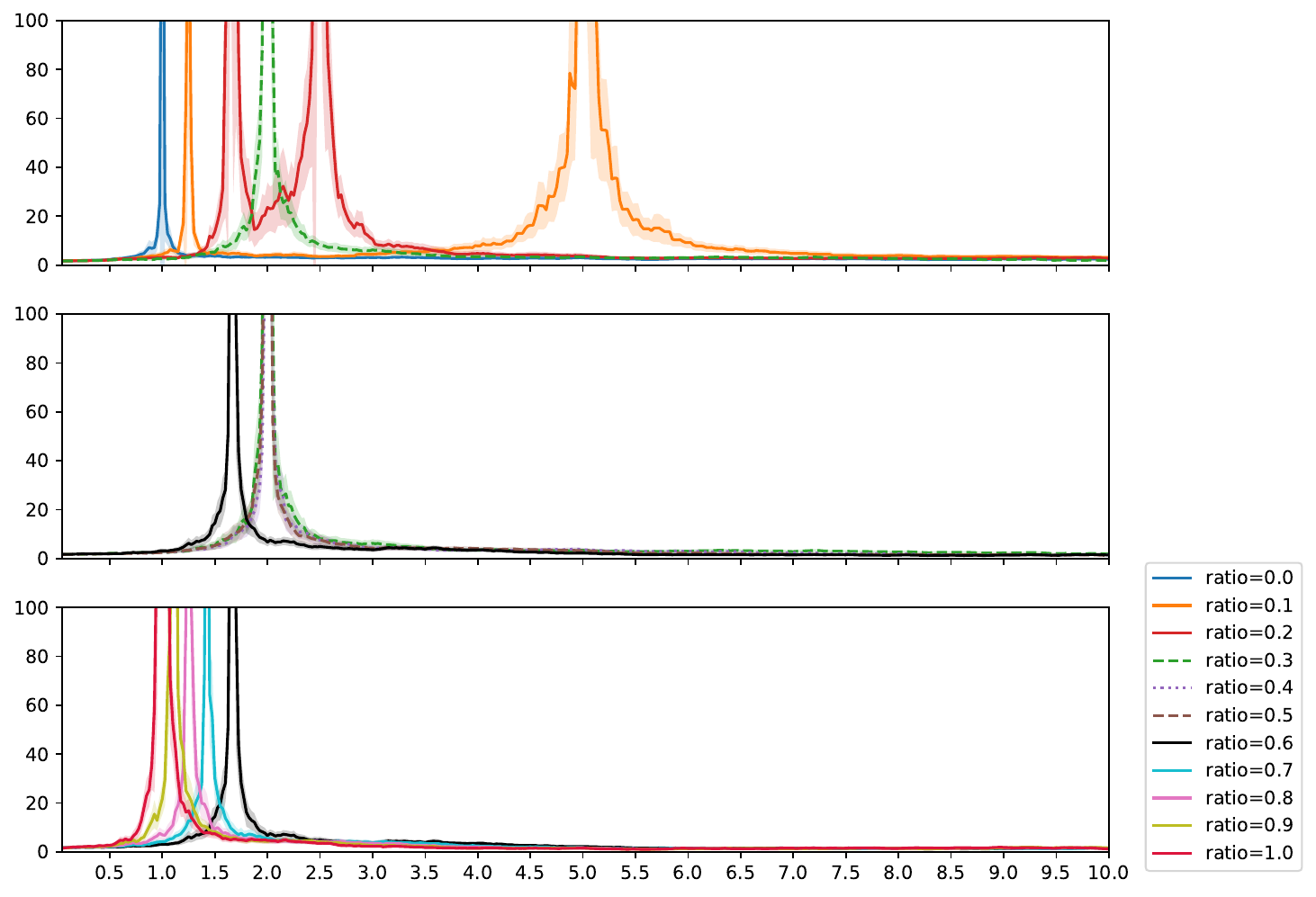}
    \caption{Two groups of $n/2$ observations, $n=200$. Group $1$ is $\cN(0,\diag\{1,\dots,1,0,\dots,0\})$
    with the top $p/2$ coordinates active; group $2$ is $\cN(0,\diag\{1,\dots,1,0,\dots,0\})$ with the
    top $\rho p$ coordinates active. Test risk against $p/n$; each curve is one value of $\rho$
    (\texttt{ratio}$=\rho$).}
    \label{fig:app-two-group}
\end{figure}

\begin{figure}[htbp]
    \centering
    \includegraphics[width=\linewidth]{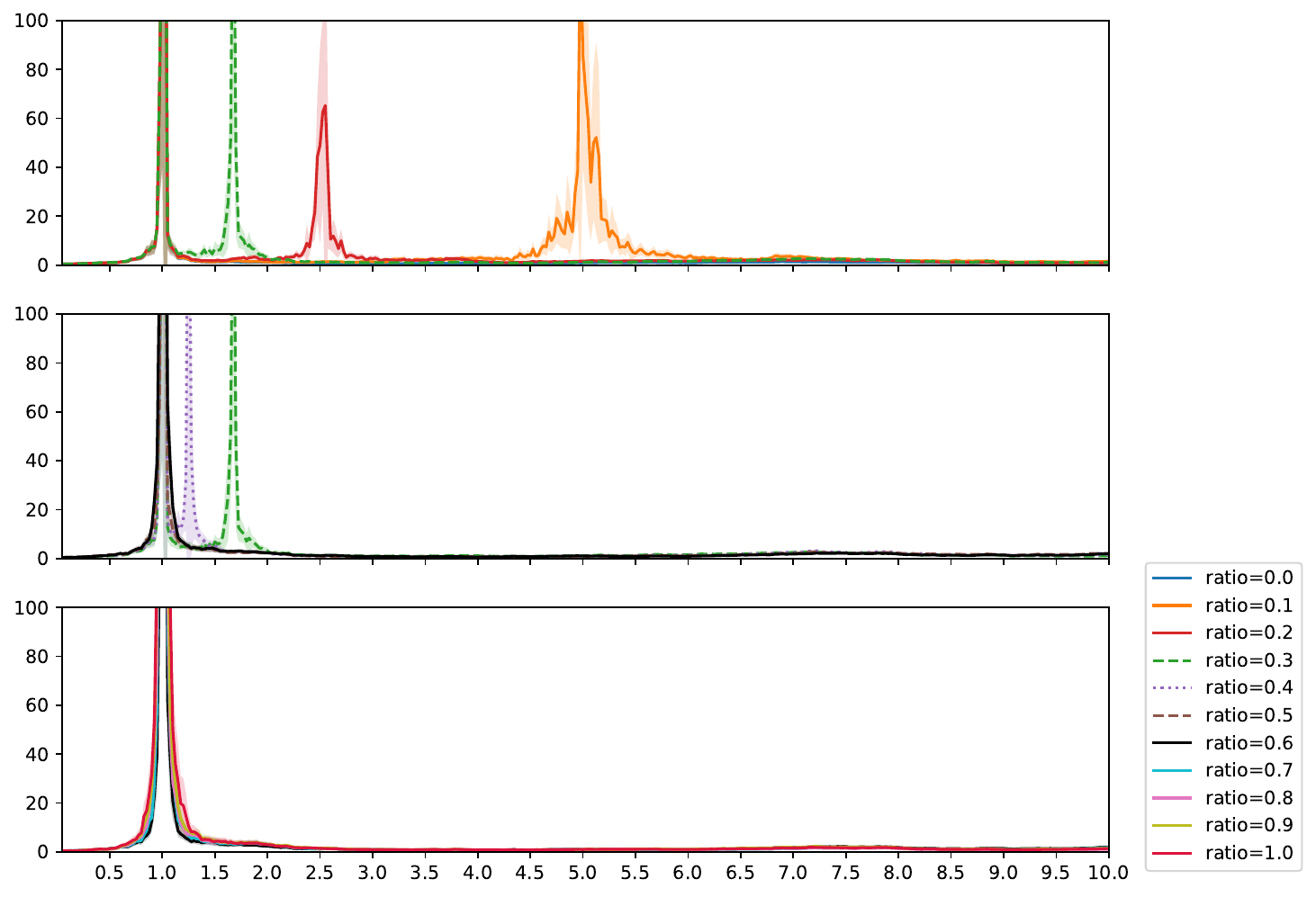}
    \caption{As \Cref{fig:app-two-group}, but group $2$ activates its \emph{bottom} $\rho p$
    coordinates, so the two supports overlap only partially. Test risk against $p/n$, $n=200$;
    each curve is one value of $\rho$ (\texttt{ratio}$=\rho$).}
    \label{fig:app-two-group-reverse}
\end{figure}

\begin{figure}[htbp]
    \centering
    \includegraphics[width=\linewidth]{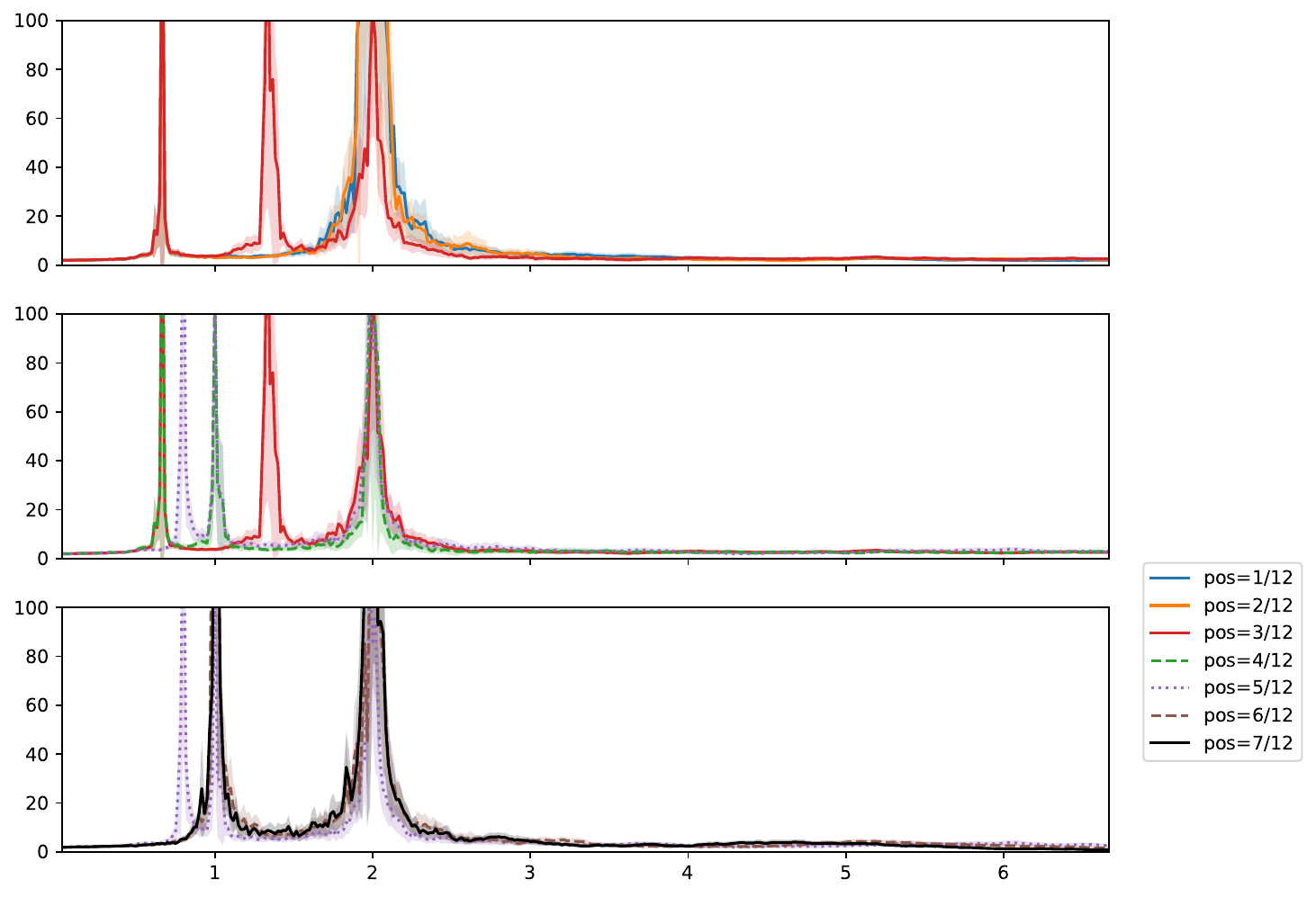}
    \caption{Three groups of $n/3$ observations, $n=300$. Group $1$ activates the top $p/3$
    coordinates, group $2$ the bottom $p/2$, and group $3$ a middle band of width $p/6$ starting at
    coordinate $\rho p$. Test risk against $p/n$; each curve is one value of $\rho$
    (\texttt{pos}$=\rho$).}
    \label{fig:app-three-group}
\end{figure}

\subsection{Approximate degeneracy: small in place of zero eigenvalues}

To probe robustness to \emph{exact} zeros, \Cref{fig:app-diminish,fig:app-diminish-beyond}
replace the vanishing eigenvalues by a small scale $\sigma^2$ (and its reciprocal), and
\Cref{fig:app-diminish-overlap,fig:app-diminish-overlap-beyond} repeat this with overlapping
supports. As $\sigma\to0$ the near-degenerate directions reproduce the exact-zero peaks of
\Cref{fig:app-two-group}; for moderate $\sigma$ the peaks soften, quantifying the
exact-zero idealisation.

\begin{figure}[htbp]
    \centering
    \includegraphics[width=0.75\linewidth]{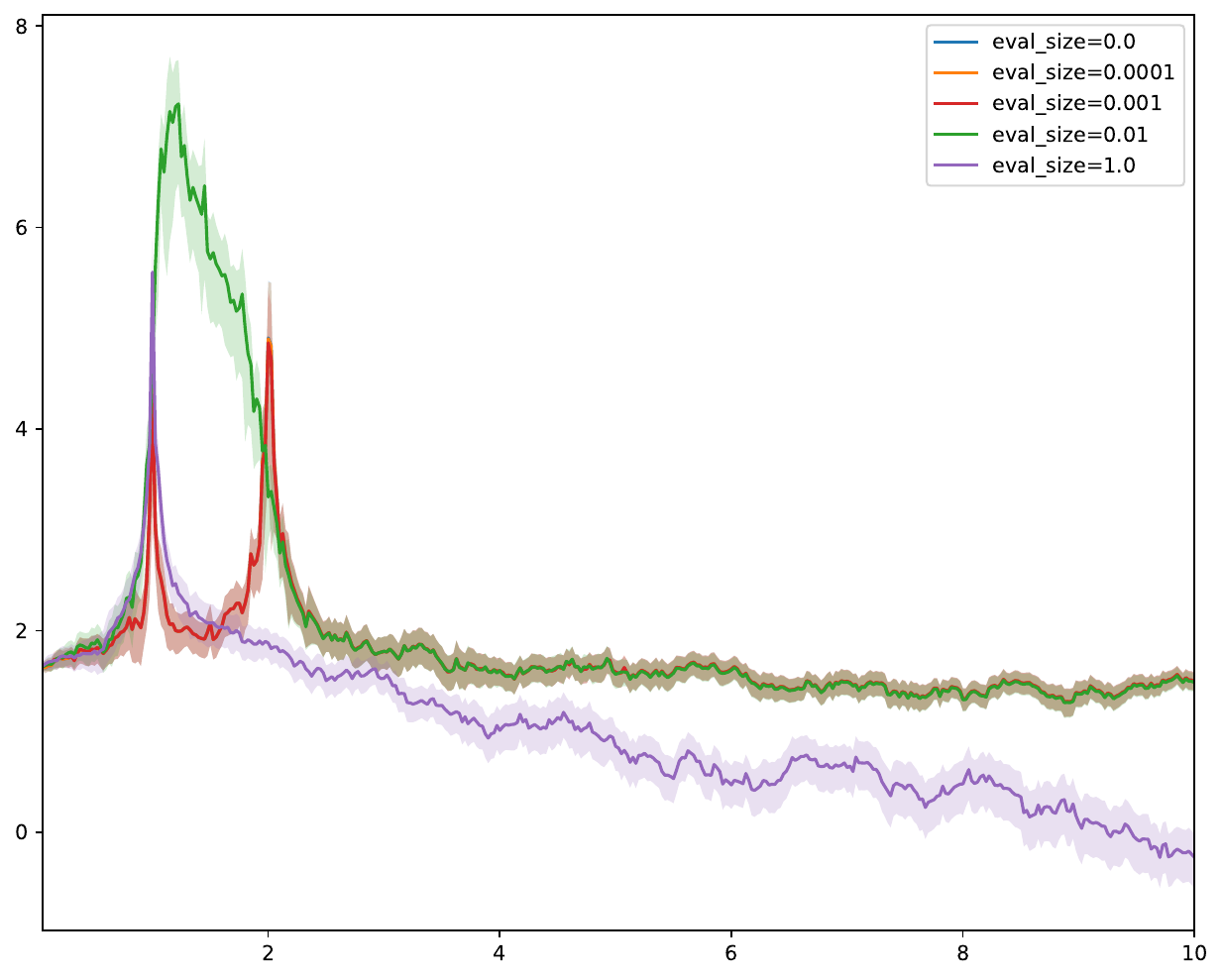}
    \caption{Two groups of $n/2$, $n=200$; test risk ($\log_{10}$ scale) against $p/n$. Group $1$
    has covariance $\diag\{1,\dots,1,\sigma^2,\dots,\sigma^2\}$ (top $p/2$ equal to $1$); group $2$
    has $\diag\{\sigma^2,\dots,\sigma^2,1,\dots,1\}$ (bottom $p/4$ equal to $1$). Each curve is one
    value of $\sigma$ (\texttt{eval\_size}$=\sigma$).}
    \label{fig:app-diminish}
\end{figure}

\begin{figure}[htbp]
    \centering
    \includegraphics[width=0.75\linewidth]{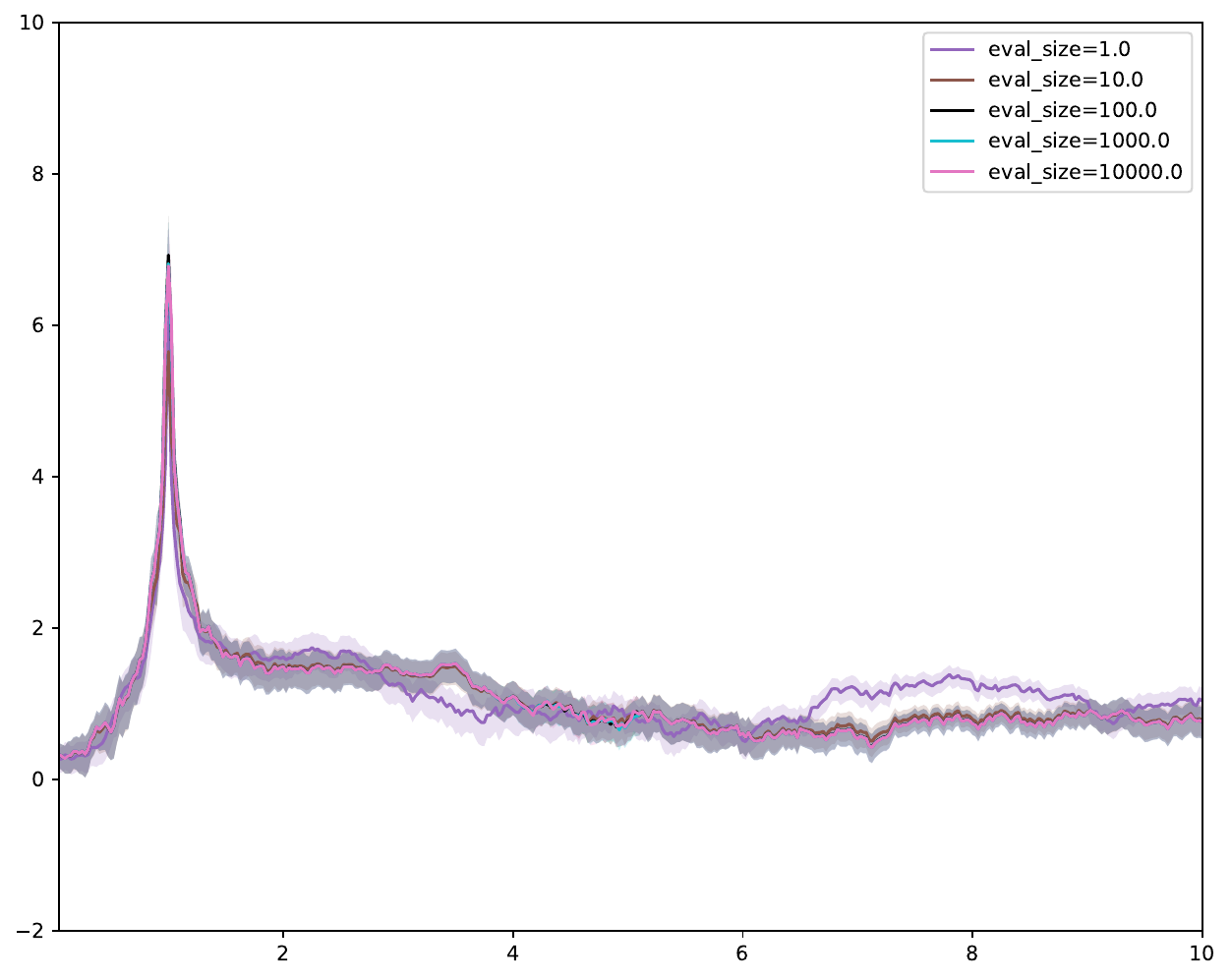}
    \caption{As \Cref{fig:app-diminish} with the small scale inverted: group $1$ covariance
    $\diag\{\sigma^{-2},\dots,\sigma^{-2},1,\dots,1\}$ (top $p/2$), group $2$
    $\diag\{1,\dots,1,\sigma^{-2},\dots,\sigma^{-2}\}$ (bottom $p/4$). Test risk ($\log_{10}$) against
    $p/n$, $n=200$; each curve is one $\sigma$ (\texttt{eval\_size}$=\sigma$).}
    \label{fig:app-diminish-beyond}
\end{figure}

\begin{figure}[htbp]
    \centering
    \includegraphics[width=0.75\linewidth]{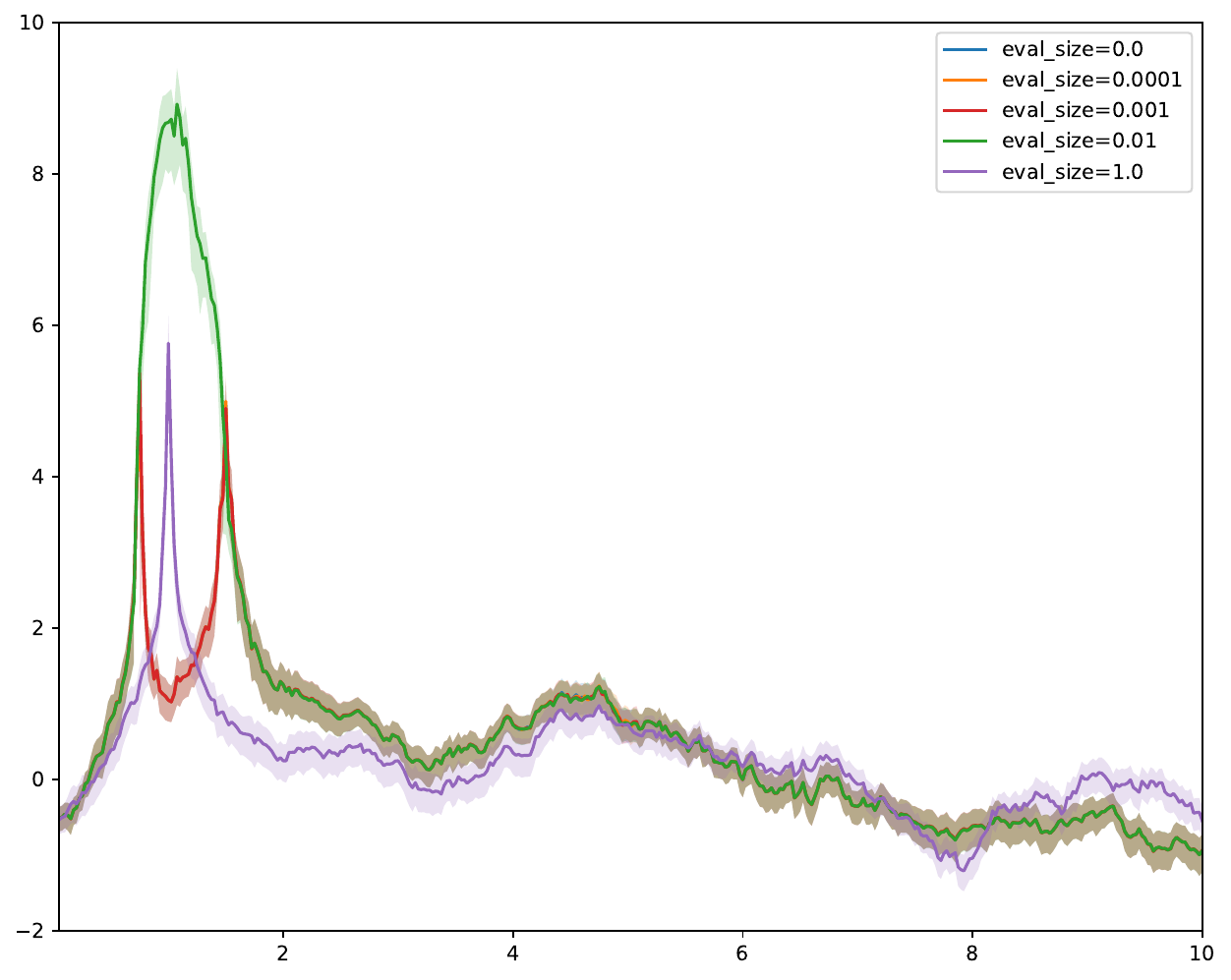}
    \caption{Overlapping-support version of \Cref{fig:app-diminish}: group $1$ activates its top
    $p/3$ coordinates and group $2$ its bottom $3p/4$, with the remaining directions carrying the
    small scale $\sigma^2$. Test risk ($\log_{10}$) against $p/n$, $n=200$; each curve is one
    $\sigma$ (\texttt{eval\_size}$=\sigma$).}
    \label{fig:app-diminish-overlap}
\end{figure}

\begin{figure}[htbp]
    \centering
    \includegraphics[width=0.75\linewidth]{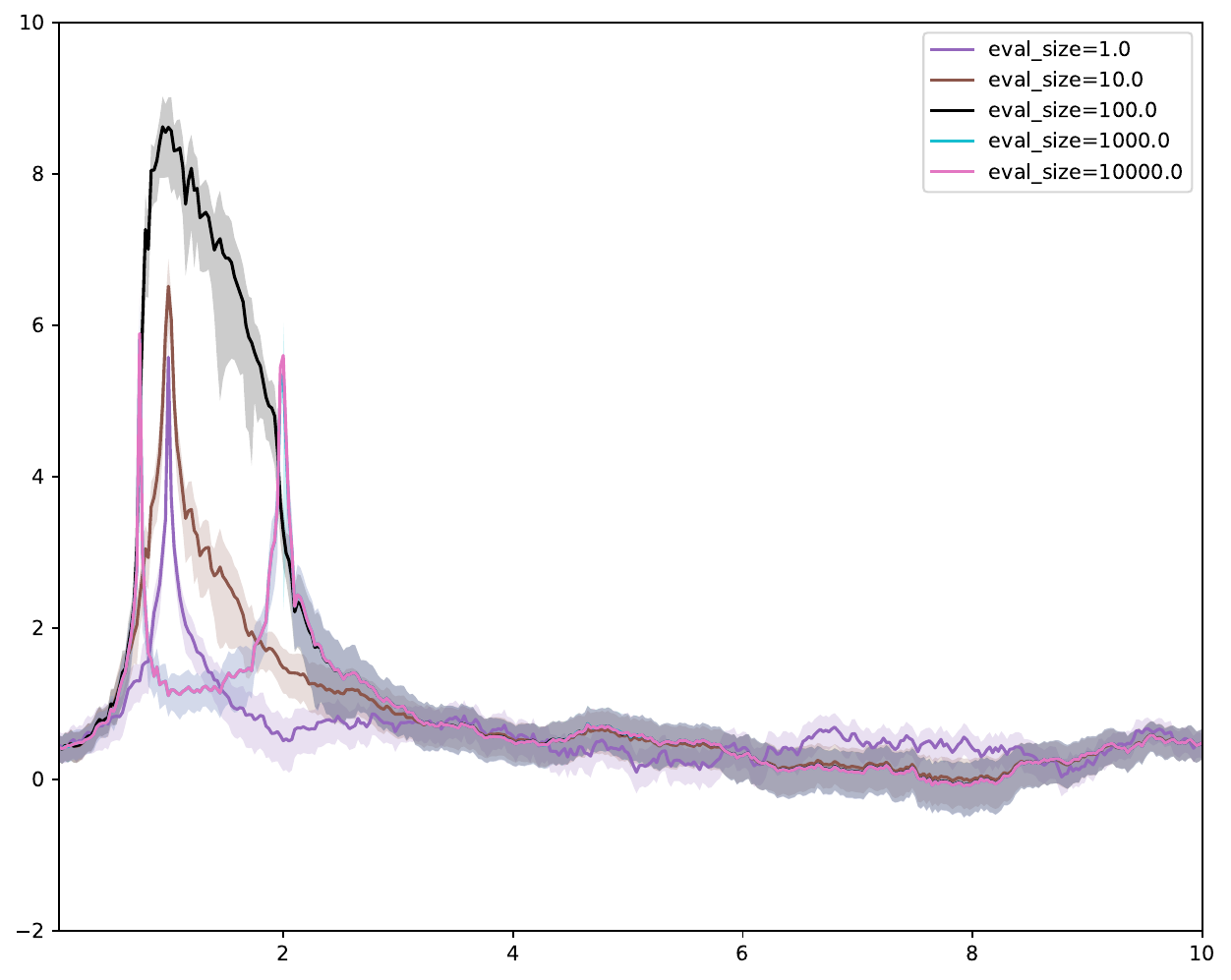}
    \caption{As \Cref{fig:app-diminish-overlap} with the small scale inverted to $\sigma^{-2}$.
    Test risk ($\log_{10}$) against $p/n$, $n=200$; each curve is one $\sigma$
    (\texttt{eval\_size}$=\sigma$).}
    \label{fig:app-diminish-overlap-beyond}
\end{figure}

\subsection{Full-rank, non-commuting controls}

Finally, \Cref{fig:app-nondiag,fig:app-nondiag-two} use full-rank, non-simultaneously
diagonalizable covariances (a rotation by $\theta$), and \Cref{fig:app-nondiag-size} a
full-rank sanity check with varying scale. Consistent with \Cref{sec:non:sim:diag:but:pos},
these show only the classical peak at $p/n=1$: without spectral degeneracy, heterogeneity adds
no descents.

\begin{figure}[htbp]
    \centering
    \includegraphics[width=0.75\linewidth]{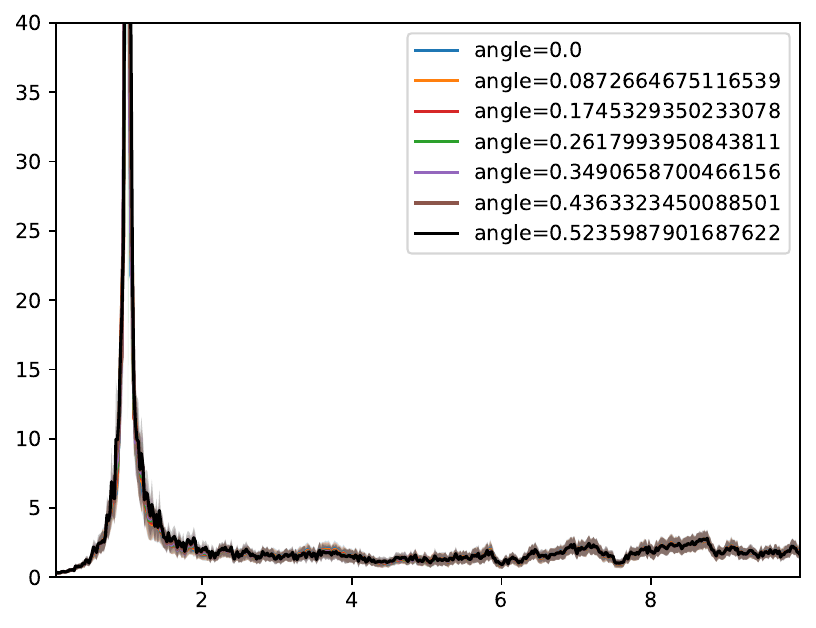}
    \caption{Full-rank, non-commuting control. Two groups of $n/2$, $n=200$. Group $1$ has covariance $\diag\{1,\dots,1,4,\dots,4\}$ (top $p/2$ equal to $1$); group $2$ has covariance $\big(\begin{smallmatrix}1 & \sin 2\theta\\ \sin 2\theta & 1\end{smallmatrix}\big)\otimes I_{p/2}$,$\theta\in[0,\pi/6]$. Test risk against $p/n$; each curve is one $\theta$ (\texttt{angle}$=\theta$).    
    }
    \label{fig:app-nondiag}
\end{figure}

\begin{figure}[htbp]
    \centering
    \includegraphics[width=0.75\linewidth]{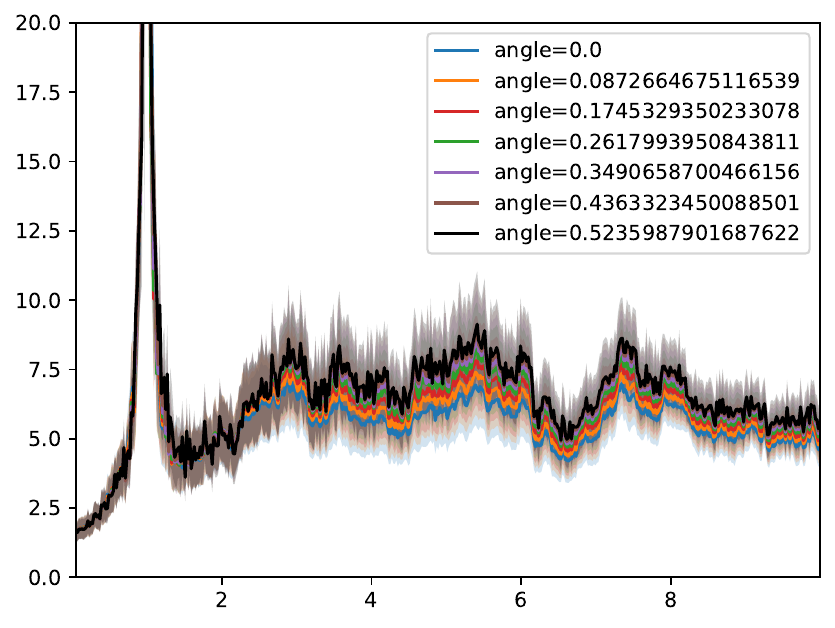}
    \caption{As \Cref{fig:app-nondiag}, but group $1$ is isotropic, covariance $I_{p/2}$. Test risk against $p/n$, $n=200$; each curve is one $\theta$ (\texttt{angle}$=\theta$).
    }
    \label{fig:app-nondiag-two}
\end{figure}

\begin{figure}[htbp]
    \centering
    \includegraphics[width=0.75\linewidth]{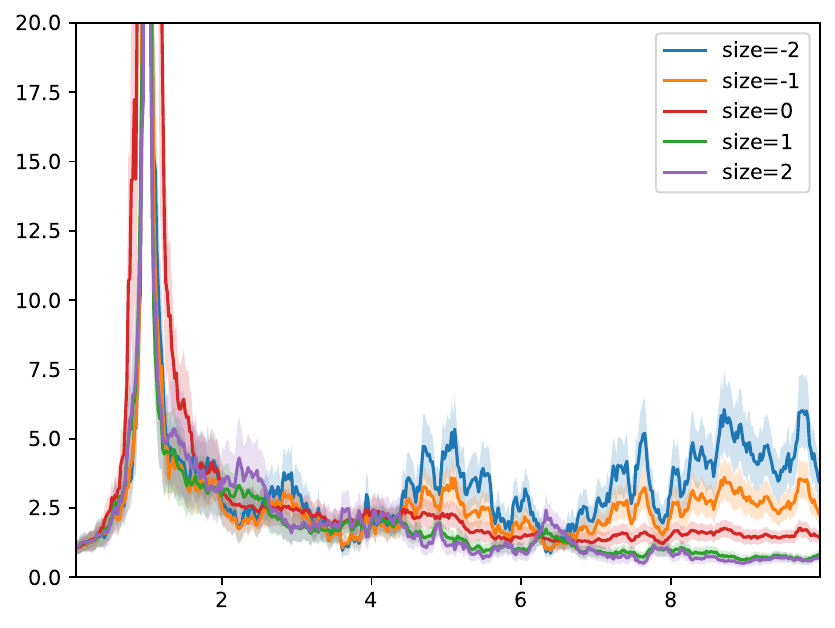}
    \caption{Full-rank sanity check with varying scale. Two groups of $n/2$, $n=200$. Group $1$ is
    isotropic, $\cN(0,I_p)$; group $2$ has covariance $\diag\{1,\dots,1,\sigma^2,\sigma^2\}$. Test
    risk against $p/n$; each curve is one $\sigma$ (\texttt{size}$=\sigma$).}
    \label{fig:app-nondiag-size}
\end{figure}

\end{document}